\documentclass{article}

\usepackage{comment}
\usepackage{geometry}
\usepackage{tabularx}
\def\version{3 February 2024}

\nonstopmode

\usepackage{pgfplots}
 \usepgfplotslibrary{fillbetween}
 \usetikzlibrary{patterns}
 \pgfplotsset{compat=1.17}

\usepackage{ifpdf}

\newcommand{\ac}[1]{\noindent\textcolor{red}
{{\rm [\![}\mbox{\sc{AC}$\blacktriangleright\!\!\blacktriangleright$}: {#1}{\rm ]\!]}}}
\newcommand{\nb}[1]{\textcolor{blue}
{{\rm [\![}\mbox{\sc{NB}$\blacktriangleright\!\!\blacktriangleright$}: #1{\rm
]\!]}}}

\newcommand{\ak}[1]{\noindent\textcolor{MyDarkGreen}
{{\rm [\![}\mbox{\sc{AC}$\blacktriangleright\!\!\blacktriangleright$}: {#1}{\rm ]\!]}}}

\renewcommand{\ac}[1]{}
\renewcommand{\ak}[1]{}
\renewcommand{\nb}[1]{}

\usepackage{mathtools}
\usepackage{latexsym,epsfig,bm}
\usepackage{upgreek}
\usepackage{mathrsfs}
\usepackage{amssymb}

\usepackage{upref}
\usepackage{esint}

\usepackage[english]{babel}

\usepackage{ulem} 
\usepackage{comment}

\PassOptionsToPackage{usenames,dvipsnames}{color}
\usepackage{color}
\usepackage{graphicx}

\newcommand{\longhookdownarrow}{\mathrel{\rotatebox[origin=c]{-90}{$\longhookrightarrow$}}}

\newcommand{\longhooknearrow}{\mathrel{\rotatebox[origin=c]{45}{$\longhookrightarrow$}}}

\newcommand{\longhooksearrow}{\mathrel{\rotatebox[origin=c]{-45}{$\longhookrightarrow$}}}

\definecolor{MyDarkBlue}{rgb}{0,0.08,0.45}
\definecolor{MyDarkGreen}{rgb}{0,0.7,0}
\definecolor{MyGreen}{rgb}{0,0.7,0.1}
\definecolor{Pomegranade}{rgb}{0.6,0.1,0.15}
\definecolor{purple}{rgb}{0.6,0.1,0.15}
\usepackage[backref=none]{hyperref}
\hypersetup{pdfborder={0 0 0},
 colorlinks,
 urlcolor={MyDarkBlue},
 linkcolor={MyDarkBlue},
 citecolor={MyDarkBlue},
 breaklinks=true}
\providecommand{\url}[1]{\small\textcolor{blue}{#1}}
\usepackage{doi}
\providecommand{\eprint}[1]{}
\renewcommand{\eprint}[1]{arXiv:\href{http://arxiv.org/abs/#1}{#1}}

\providecommand{\eqref}[1]{{\rm (\ref{#1})}}
\providecommand{\itref}[1]{{\it (\ref{#1})}}

\newcommand{\unity}{\textrm{{\usefont{U}{fplmbb}{m}{n}1}}}

\newcommand{\fra}[2]{{#1}/{#2}}
\newcommand{\ds}{\displaystyle}

\providecommand{\longhookrightarrow}{\lhook\joinrel\longrightarrow}

\newcommand{\dom}{\mathfrak{D}}
\newcommand{\range}{\mathfrak{R}}
\newcommand{\ran}{\mathfrak{R}}
\renewcommand{\ker}{\mathop{\mathbf{ker}}}
\newcommand{\coker}{\mathop{\mathbf{coker}}}

\newcommand{\jj}{\mathrm{i}}

\newcommand{\frakL}{\mathfrak{L}}
\newcommand{\frakM}{\mathfrak{M}}

\newcommand{\scrB}{\mathscr{B}}
\newcommand{\scrC}{\mathscr{C}}
\newcommand{\scrD}{\mathscr{D}}
\newcommand{\scrK}{\mathscr{K}}

\newcommand{\scrQ}{\mathscr{Q}}

\newcommand{\calB}{\mathcal{B}}
\newcommand{\calC}{\mathcal{C}}
\newcommand{\calD}{\mathcal{D}}
\newcommand{\calK}{\mathcal{K}}

\newcommand{\calR}{\mathcal{R}}
\newcommand{\calS}{\mathcal{S}}

\newcommand{\bfD}{\mathbf{D}}
\newcommand{\bfE}{\mathbf{E}}
\newcommand{\bfF}{\mathbf{F}}

\newcommand{\doma}{\dom(\hatA)}
\newcommand{\bfG}{\mathbf{G}}

\newcommand{\bfX}{\mathbf{X}}

\newcommand{\e}{\bm{e}}

\newcommand{\vargamma}{\mathrm{g}}
\newcommand{\hatA}{\hat{A}}

\newcommand{\hatB}{\hat{B}}

\newcommand{\chB}{\mathcal{B}}

\newcommand{\notyet}[1]{{}}

\newcommand{\rank}{\mathop{\mathrm{rank}}}

\newcommand{\supp}{\mathop{\mathrm{supp}}}
\newcommand{\p}{\partial}
\newcommand{\at}[1]{\!\!\upharpoonright\sb{{#1}}}

\newcommand{\R}{\mathbb{R}}
\newcommand{\C}{\mathbb{C}}

\newcommand{\N}{\mathbb{N}}\newcommand{\Z}{\mathbb{Z}}
\newcommand{\Abs}[1]{\left\vert#1\right\vert}
\newcommand{\abs}[1]{\vert #1 \vert}
\newcommand{\bigabs}[1]{\big\vert #1 \big\vert}

\newcommand{\Norm}[1]{\Big\Vert #1 \Big\Vert}

\newcommand{\norm}[1]{\Vert #1 \Vert}

\newcommand{\sothat}{\,\,{\rm :}\ \ }

\newcommand{\slim}{\mathop{\mbox{\rm s-lim}}}

\newcommand{\cnor}{\rho_0}

\DeclareMathSymbol{\varGamma}{\mathord}{letters}{"00}
\DeclareMathSymbol{\varDelta}{\mathord}{letters}{"01}
\DeclareMathSymbol{\varTheta}{\mathord}{letters}{"02}
\DeclareMathSymbol{\varLambda}{\mathord}{letters}{"03}
\DeclareMathSymbol{\varXi}{\mathord}{letters}{"04}
\DeclareMathSymbol{\varPi}{\mathord}{letters}{"05}
\DeclareMathSymbol{\varSigma}{\mathord}{letters}{"06}
\DeclareMathSymbol{\varUpsilon}{\mathord}{letters}{"07}
\DeclareMathSymbol{\varPhi}{\mathord}{letters}{"08}
\DeclareMathSymbol{\varPsi}{\mathord}{letters}{"09}
\DeclareMathSymbol{\varOmega}{\mathord}{letters}{"0A}

\usepackage{amsthm,xpatch}
\makeatletter
\xpatchcmd{\@thm}{.}{\,}{}{}
\makeatother

\makeatletter
\def\th@plain{%
 \thm@notefont{}
 \itshape 
}
\def\th@definition{%
 \thm@notefont{}
 \normalfont 
}
\makeatother

\theoremstyle{plain}
\newtheorem{lemma}{Lemma}[section]
\newtheorem{theorem}[lemma]{Theorem}
\newtheorem{corollary}[lemma]{Corollary}

\theoremstyle{definition}
\newtheorem{definition}[lemma]{Definition}

\newtheorem{assumption}{Assumption}

\theoremstyle{remark}

\newtheorem{remark}{Remark}[section]
\newtheorem{example}[remark]{Example}

\newcounter{step}

\makeatletter\@addtoreset{equation}{section}
\makeatletter\@addtoreset{figure}{section}
\makeatletter\@addtoreset{lemma}{section}
\makeatother

\newcommand{\dist}{\mathop{\rm dist}\nolimits}
\renewcommand{\Re}{\mathop{\rm{R\hskip -1pt e}}\nolimits}
\renewcommand{\Im}{\mathop{\rm{I\hskip -1pt m}}\nolimits}

\renewenvironment{abstract}
  {\list{}{\listparindent 0.0cm%
  \baselineskip 0pt
  \setlength{\leftmargin}{1.5cm}
  \setlength{\rightmargin}{1cm}
  }%
  \item\relax \hskip -10pt {\sc Abstract.}\ \footnotesize}
  {\endlist}

\newcommand{\shortto}[1][3pt]{\mathrel{%
  \hbox{\hskip 1pt\rule[\dimexpr\fontdimen22\textfont2-.2pt\relax]{#1}{.4pt}}%
  \mkern-6mu\hbox{\usefont{U}{lasy}{m}{n}\symbol{41}}}}

\begin{document}

\title{
Virtual levels
and virtual states
of linear operators in Banach spaces.
Applications to Schr\"{o}dinger operators
}

\author{
{\sc Nabile Boussa{\"\i}d}
\\
{\it\small Laboratoire Math\'{e}matiques de Besan\c{c}on,
Universit\'e Franche-Comt\'e,
Besan\c{c}on,\,France}
\\[2ex]
{\sc Andrew Comech}
\\
{\it\small Texas A\&M University, College Station, Texas, USA
}
}

\date{\version}

\maketitle

\begin{abstract}
We develop a general approach to virtual levels
in Banach spaces.
We show that
virtual levels
admit several characterizations
which are essentially equivalent:
{\it (1)}
there are corresponding \emph{virtual states}
(from a certain larger space);
{\it (2)}
there is no limiting absorption principle in their vicinity
(e.g. no weights such that the ``sandwiched'' resolvent
is uniformly bounded);
{\it (3)}
an arbitrarily small perturbation can produce an eigenvalue.

We provide applications to Schr\"odinger
operators
with nonselfadjoint
nonlocal potentials and in any dimension,
deriving resolvent estimates
in the neighborhood of the threshold
when the corresponding operator has no virtual level there.
\end{abstract}

\section{Introduction}
The term \emph{virtual levels} takes its origin from
the study of scattering of neutrons on protons by Eugene Wigner
\cite{wigner1933streuung},
just a year after the discovery of the neutron by James Chadwick.
While a proton and a neutron with parallel spins
form a deuteron
(Deuterium's nucleus),
the binding energy of particles with antiparallel spins
is near zero,
and it was not clear
for some time
whether the corresponding state
is \emph{real} or \emph{virtual},
that is, whether the binding energy was positive or negative;
see, for instance, \cite{fermi1935recombination},
where the word ``virtual'' appears first
(it later turned out that this state was virtual indeed).
The presence of this virtual state
manifests itself in the increase of the scattering cross-section
at
low energies,
when slow incoming neutrons try to couple with protons.
Mathematically, virtual levels
correspond to particular singularities
of the resolvent at the essential spectrum;
this idea goes back to Julian Schwinger \cite{schwinger1960field}.

Let us start by noticing
that the following properties of the threshold $z_0=0$
of the Schr\"odinger operator $H=-\Delta+V(x)$,
with
$x\in\R^d$, $d\ge 1$,
and $V\in C_{\mathrm{comp}}(\R^d)$,
seem related:

\noindent
{\it
(P1)\;
The equation $H\psi=z_0\psi$ has a nonzero solution from $L^2$
or from a certain larger space;

\noindent
(P2)\;
The resolvent $R(z)=(H-z I )^{-1}$
has no limit
in weighted spaces as $z\to z_0$;

\noindent
(P3)\;
Under some arbitrarily small perturbation,
an eigenvalue can bifurcate from $z_0$.
}
\smallskip

\noindent
For example, Properties~{\it (P1) -- (P3)}
are satisfied for $H=-\p_x^2$
in $L^2(\R)$
considered with domain $\dom(H)=H^2(\R)$,
near the point $z_0=0$.
Indeed, the equation $-\p_x^2\psi=0$
has a bounded solution
$\psi(x)=1$;
while non-$L^2$,
at least it is uniformly bounded.
The integral kernel of the resolvent
$R_0^{(1)}(z)=(-\p_x^2-z I )^{-1}$,
$z\in\C\setminus\overline{\R_{+}}$,
contains a singularity at $z=0$:
\begin{eqnarray}\label{free-resolvent-1d}
R_0^{(1)}(x,y;z)=
\frac{e^{-\abs{x-y}\sqrt{-z}}}{2\sqrt{-z}},
\qquad
x,\,y\in\R,
\qquad
z\in\C\setminus\overline{\R_{+}},
\qquad
\Re\sqrt{-z}>0,
\end{eqnarray}
and has no limit
(even in weighted spaces) as $z\to 0$.
Finally, under a small perturbation,
an eigenvalue may bifurcate from the
threshold (see e.g. \cite{simon1976bound}).
Indeed, the perturbed operator
$H_\vargamma=-\p_x^2-\vargamma\unity_{[-1,1]}$, $0<\vargamma\ll 1$,
has the eigenvalue
$E_\vargamma=-\vargamma^2+o(\vargamma^2)\in(-\vargamma,0)$.
In this example,
one says that $z_0=0$
is a \emph{virtual level}
(sometimes called a \emph{zero energy resonance}
or \emph{threshold resonance});
the corresponding non-$L^2$
solution $\psi(x)=1$
is a \emph{virtual state}.

On the contrary,
Properties~{\it (P1) -- (P3)}
are not satisfied for $H=-\Delta$
in $L^2(\R^3)$,
with $\dom(H)=H^2(\R^3)$,
near the threshold $z_0=0$.
Regarding~{\it (P1)},
we notice that, for certain compactly supported potentials,
nonzero solutions to $(-\Delta+V)\psi=0$
can behave like the Green function,
$\sim \abs{x}^{-1}$ as $\abs{x}\to\infty$,
and one expects that this is what virtual states should look like,
while $\Delta\psi=0$ has no such solutions
by Liouville's theorem for harmonic functions,
so there are no virtual states;
the integral kernel of
$R_0^{(3)}(z)=(-\Delta -z I )^{-1}$,
\begin{eqnarray}\label{free-resolvent-3d}
R_0^{(3)}(x,y;z)=
\frac{e^{-\abs{x-y}\sqrt{-z}}}{4\pi\abs{x-y}},
\qquad
x,\,y\in\R^3,
\qquad
z\in\C\setminus\overline{\R_{+}},
\qquad
\Re\sqrt{-z}>0,
\end{eqnarray}
remains pointwise bounded as $z\to 0$
and
converges in the space of mappings
$L^2_s(\R^3)\to L^2_{-s'}(\R^3)$, $s,\,s'>1/2$, $s+s'\ge 2$
(this follows from \cite[Lemma 2.1]{nirenberg1973null};
see also Theorem~\ref{theorem-3d} below),
contrary to~{\it (P2)};
finally, a small real-valued perturbation $V$
such that $\norm{\langle x\rangle^2 V}_{L^\infty}\le 1/4$
cannot produce negative eigenvalues.
Indeed, the relation
$z\psi=(-\Delta+V)\psi$ with $\psi\in L^2(\R^3,\R)$
and $z<0$
implies that
$\psi\in H^2(\R^3,\R)$
and then leads to a contradiction:
\[
z\int_{\R^3}
\abs{\psi(x)}^2\,dx
=
\int_{\R^3}\bar\psi(x)(-\Delta+V(x))\psi(x)\,dx
\ge
\int_{\R^3}\bar\psi(x)\Big(-\Delta-\frac{1}{4\abs{x}^2}\Big)\psi(x)\,dx
\ge 0,
\]
by the Hardy inequality;
thus, Property {\it (P3)} also fails.
(The absence of bifurcations of complex eigenvalues
from $z_0=0$
under complex perturbations
will follow from the theory that we develop
in the present work;
see Theorem~\ref{theorem-b} below.)

We are going to show that
Properties~{\it (P1) -- (P3)}
are essentially equivalent,
being satisfied when $z_0$ is an eigenvalue of $H$
or, more generally, a \emph{virtual level};
this allows us to relate
the concept of virtual levels to the limiting absorption principle (LAP).
The idea of introducing
a small absorption into the wave equation
for specifying a particular solution
goes back to Ignatowsky \cite{ignatowsky1905reflexion}
and is closely related to the Sommerfeld radiation condition
\cite{sommerfeld1912greensche}.
This approach was further developed in
\cite{sveshnikov1950radiation,
povzner1950method,
povzner1953expansion,
gelfand1955eigenfunction,
berezanskii1957eigenfunction,
ikebe1960eigenfunction,
birman1961spectrum,
eidus1962principle,
vainberg1966principles}.
Eventually the limiting absorption principle
takes the form of estimates in certain spaces
satisfied by the limit of the resolvent
at the essential spectrum
in
\cite[Lemma 6.1]{rejto1969partly},
\cite[Theorem 2.2]{agmon1970spectral},
and \cite[Appendix A]{agmon1975spectral};
see also
\cite[Chapter 4]{kuroda1978introduction} and \cite{benartzi1987limiting}.
For example, the LAP is available
for the Laplacian
when the spectral parameter approaches the bulk of the
essential spectrum.
By \cite[Theorem A.1]{agmon1975spectral},
the resolvent
\begin{eqnarray}\label{r0-ls}
R_0^{(d)}(z)=
(-\Delta-z I )^{-1}:\;L^2_{s}(\R^d)\to L^2_{-s'}(\R^d),
\quad
z\in\C\setminus\overline{\R_{+}},
\quad
d\ge 1,
\quad
s,\,s'>1/2,
\end{eqnarray}
is bounded uniformly in $z\in\varOmega\setminus\overline{\R_{+}}$,
for any $\varOmega\subset\C$
such that
$\overline\varOmega\not\ni \{0\}$,
and has limits as $z\to z_0\pm \jj 0$, $z_0>0$,
$z\in\varOmega$.
Let us also mention
\cite{agmon1998perturbation},
where the analytic continuation of the resolvent
through the essential spectrum
is considered as a mapping
$\bfE\to\bfF$;
the \emph{resonances}
are defined as poles of this continuation
which are inevitable
by choosing smaller $\bfE$
and larger $\bfF$,
with dense continuous embeddings
$\bfE\hookrightarrow\bfX\hookrightarrow\bfF$.

The resolvent
$R_0^{(3)}(z)=(-\Delta-z I )^{-1}$
of the Laplace operator in $\R^3$
remains bounded
as a mapping $L^2_s(\R^3)\to L^2_{-s'}(\R^3)$
uniformly
in $z\in\C\setminus\overline{\R_{+}}$
and has a limit even as $z\to z_0=0$
as long as $s+s'\ge 2$.
A similar boundedness of the resolvent in an open neighborhood of the threshold
$z_0=0$ persists in higher dimensions.
Such a boundedness is absent
for the resolvent of the free Laplace operator
in dimensions $d\le 2$
because of the presence of a virtual level at $z_0=0$
but becomes available if some
perturbation $V$ is added
\cite{simon1976bound,bolle1988threshold,jensen2001unified}.
Motivated by the above considerations,
given a closed operator $A$ in a Banach space $\bfX$
and Banach spaces $\bfE,\,\bfF$
with continuous embeddings
$\bfE
\mathop{\longhookrightarrow}\limits\sp\imath
\bfX
\mathop{\longhookrightarrow}\limits\sp\jmath
\bfF$,
we call
a point $z_0\in\sigma\sb{\mathrm{ess}}(A)$
a \emph{regular point of the essential spectrum}
if there is the LAP
near this point, in the sense that
$\jmath\circ(A-z I_\bfX)^{-1}\circ\imath$
converges
as $z\to z_0$,
$z\in\varOmega\subset\C\setminus\sigma(A)$,
in the weak
operator topology of $\scrB(\bfE,\bfF)$
(or in the weak$^*$ operator topology if $\bfF$
has pre-dual).
We call $z_0$
a \emph{virtual level of rank $r\in\N$}
(relative to $(\bfE,\bfF,\varOmega)$)
if this point is not regular, but becomes regular
under a perturbation by an operator of rank $r$
(but not smaller).
Corresponding virtual states of $A$ are
elements from a certain vector space
$\frakM_{\bfE,\bfF,\varOmega}(A-z_0 I_\bfX)$, also of dimension $r$.

We can state
the underlying idea of our approach
as follows:
to study the LAP and properties of virtual states
of a particular operator $A$,
we study the LAP for some conveniently chosen
reference operator $A_1$,
which is a compact or relatively compact perturbation of $A$.
If $z_0$ is a regular point of $A$,
then both $A$ and $A_1$ share the same LAP regularity.
If $z_0$ is a virtual level of $A$,
then properties of the resolvent of $A_1$
allow one to describe properties of virtual states of $A$:
they belong to the range of
the limit operator,
$\lim\sb{z\to z_0,\,z\in\varOmega}\jmath\circ(A_1-z I )^{-1}\circ\imath$
(in the weak or weak$^*$ operator topology
of $\scrB(\bfE,\bfF)$).

\medskip

Let us provide some history of the study of virtual levels.
This concept was addressed by
Schwinger \cite{schwinger1960field},
Birman \cite{birman1961spectrum},
Faddeev \cite{faddeev1963mathematical},
Simon \cite{simon1973resonances,simon1976bound},
Vainberg \cite{vainberg1968analytical,vainberg1975short},
Yafaev \cite{yafaev1974theory,yafaev1975virtual},
Rauch \cite{rauch1978local},
and Jensen and Kato \cite{jensen1979spectral},
with the focus on
Schr\"odinger operators in three dimensions.
Higher dimensions were considered
in \cite{jensen1980spectral,yafaev1983scattering,jensen1984spectral}.
An approach to more general
symmetric differential operators was developed in
\cite{weidl1999remarks}.
The nonselfadjoint
Schr\"odinger operators in three dimensions
appeared in \cite{cuccagna2005bifurcations}.
Dimensions $d\le 2$
require special attention since the free Laplace operator
has a virtual level at zero (see \cite{simon1976bound}).
The one-dimensional case is covered in \cite{bolle1985complete,bolle1987scattering}.
The approach from the latter article
was further developed
in \cite{bolle1988threshold} to two dimensions
(if $\int\sb{\R^2} V(x)\,dx\ne 0$)
and then
in \cite{jensen2001unified,jensen2001unified-2,jensen2004erratum}
(with this condition dropped)
who give a general approach
in all dimensions,
with the regularity of the resolvent
formulated
via the weights which are square roots of the potential
(and consequently not optimal).
There is an interest in the subject due to
the dependence of dispersive estimates on the presence
of virtual levels at the threshold point,
see e.g. \cite{jensen1979spectral,yafaev1983scattering,erdogan2004dispersive,schlag2005dispersive,yajima2005dispersive}
in the context of Schr\"{o}dinger operators;
the Dirac operators are treated in
\cite{boussaid2006stable,boussaid2008asymptotic,erdogan2017dirac,erdogan2019dispersive}.
The need to have the limiting absorption principle --
the resolvent estimates in the vicinity of the essential
spectrum -- appeared in the study of linear stability
of nonlinear Dirac equation in \cite{linear-b,opus}.

In the context of positive-definite symmetric operators,
a dichotomy similar to the one we discussed above
-- either having a particular Hardy-type inequality
or existence of a null state
of the quadratic form corresponding to the
operator -- is obtained
in \cite[\S5.1]{weidl1999remarks}
as a generalization of
Birman's approach \cite[\S1.7]{birman1961spectrum}
(which was
based on closures of the space with respect to
quadratic forms
corresponding to symmetric positive-definite operators,
in the spirit of the extension theory
\cite{krein1947theory,vishik1952general}).
This is directly related to the research
on subcritical and critical Schr\"odinger operators
\cite{simon1981large,murata1986structure,pinchover1990criticality,gesztesy1991critical,pinchover2006ground,pinchover2007ground,devyver2014spectral,lucia2018criticality}
and Hardy-type inequalities
\cite{florkiewicz1999some}.
A similar approach was developed for the
description of virtual levels
of selfadjoint Schr\"odinger operators
in dimensions
$d\le 2$ in \cite[Theorem 2.3]{barth2021absence}.

We point out that although the concept 
of \emph{threshold resonances}
is directly related to virtual levels,
the location of virtual levels in the nonselfadjoint case
is no longer restricted to thresholds:
nonselfadjoint perturbations can pull out eigenvalues
from any point of the essential spectrum.
Let us also mention the phenomenon of \emph{spectral singularities}
\cite{naimark1954investigation,schwartz1960some,pavlov1966nonselfadjoint,lyantse1967differential,guseinov2009concept,konotop2019designing}
(for a more general setting, see \cite{nagy1986operators}):
we note that
selfadjoint operators have no spectral singularities,
although they could have virtual levels at threshold points;
this shows that these two concepts differ.
We also note that
the concept of $\varepsilon$-\emph{pseudospectrum} \cite{landau1975szego},
which is a neighborhood of the spectrum where the resolvent
is of norm $1/\varepsilon$, $\varepsilon>0$,
does not capture the phenomenon of virtual levels,
giving the same pseudospectrum for the Laplace operator
in dimensions $d\le 2$ (virtual level at zero)
and $d\ge 3$ (no virtual level).

\medskip

So far, the virtual levels
have been eluding a rigorous treatment.
In particular, the following questions in \cite{jensen2001unified}
have remained unanswered:

\smallskip
\noindent
$\bullet$
\emph{Treatment of virtual levels of nonselfadjoint potentials.}
In spite of several attempts at nonselfadjoint case
\cite{murata1982asymptotic,jensen2001unified,jensen2001unified-2,jensen2004erratum,cuccagna2005bifurcations},
there was no general approach developed so far.
We give a general treatment of virtual levels
in Banach spaces, which is equally applicable
to nonselfadjoint operators
in Hilbert spaces.

\smallskip
\noindent
$\bullet$
\emph{Virtual levels for nonlocal potentials.}
We consider
general relatively compact perturbations.

\smallskip
\noindent
$\bullet$
\emph{Resolvent estimates
in two dimensions.}
The estimates in \cite{jensen2001unified} in 2D case
were derived
for
the Schr\"odinger operators
$-\Delta+V$
sandwiched with the weights $\abs{V}^{1/2}$.
Such estimates are not optimal (e.g. in the case when $V$ is
exponentially decreasing
or compactly supported);
we obtain resolvent estimates
for the two-dimensional case
in Section~\ref{sect-2d}.

\smallskip

Besides answering the above questions,
we give a rigorous definition of a virtual level
(Definition~\ref{def-virtual})
as a point of the essential spectrum where
the limiting absorption principle
(the existence
of the limit of the resolvent
at the point of the essential spectrum,
as an operator in certain spaces)
is not available, but becomes available
after adding a finite rank perturbation
(instead of finite rank perturbations,
one can as well consider relatively compact perturbations;
see Section~\ref{sect-relatively-compact}).
The limit of the resolvent at the essential
spectrum is considered in the weakest possible sense:
in the weak operator topology or in the weak$^*$ operator
topology when it is available.
The minimal rank of such a perturbation
is related to the dimension of the space of virtual states
(Theorem~\ref{theorem-m}),
hence the virtual level
could be characterized as
a point of the essential spectrum which has
corresponding \emph{virtual states}.
We show that the virtual level
could also be characterized as a point
that produces eigenvalues
under an arbitrarily small perturbation
(Theorem~\ref{theorem-b}).
The phenomenon of virtual levels
makes sense not only at a threshold point
but also at any point of the essential spectrum;
there is a direct relation between
a virtual state corresponding to a virtual level at $z_0>0$
and the Sommerfeld radiation condition
(see Remark~\ref{remark-sommerfeld} below).
We study dependence on the choice
of regularizing spaces $\bfE,\,\bfF$ (Theorem~\ref{theorem-a})
and show how to obtain
$\tilde\bfE\to\tilde\bfF$
resolvent estimates
on a perturbed operator
when there is no embedding
$\tilde\bfE\hookrightarrow\bfX\hookrightarrow\tilde\bfF$
(Theorem~\ref{theorem-factorization}).
We also develop a simple construction
for deriving resolvent estimates
(Theorem~\ref{theorem-construction})
applicable e.g. for the Laplace operator
in $\R^2$ (see Section~\ref{sect-2d})
and study virtual levels of the adjoint operators
(Theorem~\ref{theorem-adjoint}).

In Sections~\ref{sect-1d},~\ref{sect-2d}, and~\ref{sect-free}
we apply this approach to study LAP estimates
and properties of virtual states
of Schr\"odin\-ger operators
with nonselfadjoint
relatively compact perturbations
in one, two, and higher dimensions.
In particular, in the two-dimensional case
(Section~\ref{sect-2d}),
we show that if $V$ is the operator of multiplication
by a nonzero, nonnegative function with sufficiently fast decay
($L^\infty_\rho(\R^2)$ with $\rho>2$),
then $(-\Delta+V-z I)^{-1}$
is bounded as a mapping $L^1_{0,\mu}(\R^2)\to L^\infty(\R^2)$, $\mu>1$,
uniformly in $z\in\C\setminus\overline{\R_{+}}$,
and converges (in the weak$^*$ operator topology)
as $z\to z_0=0$;
here $L^1_{0,\mu}(\R^2)=\{u\in L^1_{\mathrm{loc}}\sothat
u|\ln(1+\langle x\rangle)|^\mu\in L^1(\R)\}$.
We give an elementary proof
that an arbitrary nonnegative nonzero perturbation,
with sufficient decay at infinity,
destroys a virtual level at $z_0=0$ (Lemma~\ref{lemma-2d-pr}).
We also show that
in the case of relatively compact perturbations
virtual states in dimensions $d\le 2$ are not necessarily
$L^\infty(\R^d)$;
see Example~\ref{example-1d} and Example~\ref{example-2d}.

\subsection{Notations}
\label{sect-notation}
We adopt the standard conventions
\[
\mathbb{N}=\{1,\,2,\,\dots\},
\quad
\mathbb{N}_0=\{0\}\cup\N,
\quad
\mathbb{R}_{\pm}=\{x\in\R;\,\pm x>0\},
\quad
\mathbb{C}_{\pm}=\{z\in\C;\,\pm\Im z>0\}.
\]
For $\varOmega\subset\C$, we denote
\[
\overline{\varOmega}=\varOmega\cup\p\varOmega,
\qquad
\varOmega^*
=\big\{\zeta\in\C\sothat\bar\zeta\in\varOmega\big\}.
\]
In the complex plane, $\mathbb{D}_\delta(z_0)$
is the open disk of radius $\delta>0$ centered at $z_0$;
$\mathbb{D}_\delta$ is the open disk
centered at the origin;
$\mathbb{D}$ is the unit open disk centered at the origin.
We will use the notation
\[
\cnor=\C\setminus\overline{\R_{+}}
\]
for the resolvent set of $-\Delta$
in $L^2(\R^d)$.

For $x\in\R^d$, $d\in\N$,
$\langle x\rangle$
denotes
$\sqrt{1+x^2}$
and the operator of multiplication by $\sqrt{1+x^2}$.

All Banach spaces are assumed to be over $\C$.
Given Banach spaces $\bfE$ and $\bfF$,
the spaces of closed, bounded, compact,
and bounded finite rank operators
from $\bfE$ to $\bfF$ are denoted, respectively, by
$\scrC(\bfE,\bfF)$,
$\scrB(\bfE,\bfF)$,
$\scrB_0(\bfE,\bfF)$,
$\scrB_{00}(\bfE,\bfF)$.
For operators in $\bfE$
(from $\bfE$ to itself),
we write
$\scrC(\bfE)$ instead of $\scrC(\bfE,\bfE)$,
$\scrB(\bfE)$ instead of $\scrB(\bfE,\bfE)$, etc.
Whenever $\bfE$ and $\bfF$ are isomorphic,
we write $\bfE\cong\bfF$.
The open ball of radius $r>0$ 
in a Banach space $\bfE$
centered at $\phi_0\in\bfE$
is denoted by
$\mathbb{B}_r\big(\phi_0,\bfE\big)$.
The open ball centered at the origin is denoted by
$\mathbb{B}_r\big(\bfE\big)$,
and the open ball in $\R^d$ is denoted by $\mathbb{B}^d_r$.

For a Banach space $\bfE$,
the identity mapping $\phi\mapsto\phi$, $\phi\in\bfE$
is denoted by $I_\bfE$.
Following Kato \cite{kato1995perturbation},
we define $\bfE^*$,
the adjoint space to $\bfE$,
to be the linear space of all bounded complex-valued forms
$\xi$ on $\bfE$
which are antilinear:
\[
\xi:\,\bfE\to\C,
\qquad
\xi(\alpha \phi+\beta \chi)=\bar\alpha\xi(\phi)+\bar\beta\xi(\chi),
\qquad\forall\alpha,\,\beta\in\C,
\quad\forall \phi,\,\chi\in\bfE.
\]
According to this definition,
the coupling of a Banach space with its adjoint,
\begin{eqnarray}\label{xx}
\langle\,,\,\rangle_\bfE:\;
\bfE^*\times\bfE\to\C,
\qquad
(\xi,\phi)\mapsto\langle\xi,\phi\rangle_\bfE=\xi(\phi),
\end{eqnarray}
is sesquilinear:
it is linear in the first argument
and antilinear in the second argument.
For $\phi\in\bfE$ and $\eta\in\bfF^*$,
we define
$\phi\otimes\eta\in\scrB_{00}(\bfF,\bfE)$
by
$\phi\otimes\eta:\,\Psi\mapsto
\phi\overline{\langle\eta,\Psi\rangle_{\bfF}}$,
$\Psi\in\bfF$.

One denotes by
$J_\bfE:\,\bfE\hookrightarrow\bfE^{**}$
the James map, which is
the canonical embedding into the bidual of $\bfE$.
We note that
\[
\langle J_\bfE\phi,\xi\rangle_{\bfE^*}
=
\overline{\langle\xi,\phi\rangle_\bfE},
\qquad
\forall\phi\in\bfE,
\qquad
\forall\xi\in\bfE^*.
\]
If a Banach space $\bfF$ is a topological
dual of some other Banach space
(that, is, if it has a pre-dual),
then that other space will be denoted by $\bfF_*$.
If there are several different pre-duals,
then $\bfF_*$ denotes any one of them
which is then fixed in this article.

For
$d\in\N$, $1\le p\le\infty$,
and $s,\,\mu\in\R$,
we define the spaces $L^p_{s,\mu}(\R^d)$ by
\begin{eqnarray}\label{lpsmu}
&
L^p_{s,\mu}(\R^d)
=\big\{
u\in L^p_{\mathrm{loc}}(\R^d)\sothat
\langle x\rangle^s|\ln(1+\langle x\rangle)|^\mu
u
\in L^p(\R^d)
\big\};
\nonumber
\\[1ex]
&
\norm{u}_{L^p_{s,\mu}}
=
\norm{\langle x\rangle^s|\ln(1+\langle x\rangle)|^\mu u}_{L^p}.
\end{eqnarray}
We also define
\[
L^p_s(\R^d)=L^p_{s,0}(\R^d).
\]
According to \eqref{xx},
for measurable functions $u$ and $v$ such that
$u v\in L^1(\R^d)$,
we denote
\begin{eqnarray}\label{xxdx}
\langle u,v\rangle
=\int_{\R^d}u(x)\bar v(x)\,dx.
\end{eqnarray}
We also use the same notation for the
$\C$-linear-antilinear coupling
$\mathscr{D}'(\R^d)\times\mathscr{D}(\R^d)\to\C$
with $u\in\mathscr{D}'(\R^d)$, $v\in\mathscr{D}(\R^d)$,
where $\mathscr{D}(\R^d)$ is the space of
compactly supported test functions
with its
standard topology
and
$\mathscr{D}'(\R^d)$ its
topological
dual, the space of distributions.

\medskip
\noindent
{\sc Acknowledgments.\ }
We are most grateful to
Gregory Berkolaiko,
Kirill Cherednichenko,
Audrey Fovelle,
Uwe Franz,
Bill Johnson,
Alexander V. Kiselev,
Gilles Lancien,
Mark Malamud,
Alexander Nazarov,
Alexandre Nou,
Yehuda Pinchover,
Thomas Schlumprecht,
Vladimir Sloushch,
Tatiana Suslina,
Cyril Tintarev,
Boris Vainberg,
and Dmitrii Yafaev
for their attention and advice.
We are indebted to Roman Romanov
for mentioning to us Example~\ref{example-roman}.
We are also most grateful to
Fritz Gesztesy
for his
encouragement of the project,
for many important references,
and for advice with
constructing resolvents for
regularized Schr\"odinger operators in dimensions one and two.

\section{Virtual levels and virtual states in Banach spaces}
\label{sect-main-results}
\subsection{General theory and examples}

Below, we always assume that
$\bfX$ is an infinite-dimensional complex Banach space
and that
$A\in\scrC(\bfX)$
is a closed linear operator
with domain $\dom(A)\subset\bfX$.

We say that $z\in\C$
belongs to the point spectrum
$\sigma\sb{\mathrm{p}}(A)$
if
there is $\psi\in\dom(A)\setminus\{0\}$ such that $(A-z I_\bfX)\psi=0$;
we say that
$z$
is from the discrete spectrum
$\sigma\sb{\mathrm{d}}(A)$
if
it is an isolated point in $\sigma(A)$
and $A-z I_\bfX$ is a Fredholm operator,
or, equivalently,
if the corresponding Riesz projection is of finite rank.
We define the essential spectrum by
\begin{eqnarray}\label{def-ess}
\sigma\sb{\mathrm{ess}}(A)
=\sigma(A)\setminus\sigma\sb{\mathrm{d}}(A).
\end{eqnarray}
Let us mention that, according to
\cite[Appendix B]{hundertmark2007exponential}
(see also \cite[Theorem III.125]{opus}),
the definition \eqref{def-ess}
of the essential spectrum
coincides with
$\sigma\sb{\mathrm{ess},5}(A)$
from \cite[\S I.4]{edmunds2018spectral};
see also Remark~\ref{remark-ess} below.

In our framework, we consider
two Banach spaces $\bfE$ and $\bfF$
with continuous embeddings
\[
\bfE\mathop{\longhookrightarrow}\limits\sp{\imath}
\bfX\mathop{\longhookrightarrow}\limits\sp{\jmath}\bfF
\]
which are not necessarily dense.

\begin{definition}[Virtual levels]
\label{def-virtual}
Let $A\in\scrC(\bfX)$
and let $\varOmega\subset\C\setminus\sigma(A)$
be
such that $\sigma\sb{\mathrm{ess}}(A)\cap\p\varOmega\ne\emptyset$.
We say that a point
$z_0\in\sigma\sb{\mathrm{ess}}(A)\cap\p\varOmega$
is a
\emph{point of the essential spectrum of rank $r\in\N_0$
relative to $(\bfE,\bfF,\varOmega)$}
if it is the smallest nonnegative integer
for which there is
an operator
$\calB\in\scrB_{00}(\bfF,\bfE)$
of rank $r$
such that
$\varOmega\cap\sigma(A+B)\cap\mathbb{D}_\delta(z_0)=\emptyset$
with some $\delta>0$,
where $B=\imath\circ\calB\circ\jmath\in\scrB_{00}(\bfX)$,
and such that the following limit exists:
\begin{eqnarray}\label{lim}
(A+B-z_0 I_\bfX)^{-1}_{\bfE,\bfF,\varOmega}:=
\lim\sb{z\to z_0,\,z\in\varOmega\cap\mathbb{D}_\delta(z_0)}
\jmath\circ(A+B-z I_\bfX)^{-1}\circ\imath
:\;\bfE\to\bfF,
\end{eqnarray}
where the limit is considered in the weak or weak$^*$
(when $\bfF$ has a pre-dual)\footnote{Here and everywhere
below,
the assumptions about
convergence in the weak$^*$ operator topology
of $\scrB(\bfE,\bfF)$
(and, similarly, about
the weak$^*$ convergence in the Banach space $\bfF$)
should be discarded when $\bfF$
does not have a pre-dual.}
operator topology of $\scrB(\bfE,\bfF)$.

\begin{itemize}
\item
If $r=0$,
so that there is a limit \eqref{lim} with $B=0$,
so that
$\jmath\circ(A-z I_\bfX)^{-1}\circ\imath:\,\bfE\to\bfF$
converges
as $z\to\infty$, $z\in\varOmega$,
in the
weak or weak$^*$
operator topology,
then $z_0$ is called
a \emph{regular point of the essential spectrum
relative to $(\bfE,\bfF,\varOmega)$}.
We will also say that the resolvent of $A$
\emph{satisfies the limiting absorption principle
at $z_0$
relative to $(\bfE,\bfF,\varOmega)$}
(and in the corresponding topology).

\item
If $r\ge 1$,
then $z_0$ is called
an \emph{exceptional point of rank $r$}
relative to $(\bfE,\bfF,\varOmega)$.
We will also say that
$z_0$ is a \emph{virtual level}
of rank $r$
relative to $(\bfE,\bfF,\varOmega)$.

\item
If $\Psi\in\bfF\setminus\{0\}$
is in $\range\big((A+B-z_0 I_\bfX)^{-1}_{\bfE,\bfF,\varOmega}\big)$,
where $B=\imath\circ\calB\circ\jmath$, with some $\calB\in\scrB_{00}(\bfF,\bfE)$,
is such that
the limit \eqref{lim} exists in the weak or weak$^*$ operator topology,
and satisfies $(\hatA-z_0 I_\bfF)\Psi=0$,
then $\Psi$ is called a \emph{virtual state} of $A$
relative to $(\bfE,\bfF,\varOmega)$
corresponding to $z_0$.
\end{itemize}

\end{definition}

\begin{remark}
By Theorem~\ref{theorem-m}~\itref{theorem-m-1},
$\range\big((A+B-z_0 I_\bfX)^{-1}_{\bfE,\bfF,\varOmega}\big)$ does not depend on
$B=\imath\circ\calB\circ\jmath$,
$\calB\in\scrB_{00}(\bfF,\bfE)$,
as long as the limit
\eqref{lim} exists in the weak or weak$^*$ operator topology.
\end{remark}

\begin{remark}\label{Remark:DifferentLAP}
Example~\ref{example-shift} (see below)
shows that there are situations when
the LAP holds
for the resolvent of $A$
in the weak$^*$ but not weak operator topology
(relative to some $(\bfE,\bfF,\varOmega)$).
Under the assumption that $A$ has a certain
closed extension onto $\bfF$
(see Assumption~\ref{ass-virtual} below),
it follows from Corollary~\ref{Cor:DifferentConvergence2}
below that 
when the LAP holds for the resolvent of $A$ in one
of uniform, strong, weak, or weak$^*$ operator topologies
but not another,
there is no $\calB\in\scrB_{0}(\bfF,\bfE)$
such that LAP would hold for the resolvent of
$A+\imath\circ\calB\circ\jmath$ in that other topology.
It also follows from that corollary
that if the limit \eqref{lim}
exists in the uniform or strong or weak operator topology
for some $\calB\in\scrB_{00}(\bfF,\bfE)$
and that if the minimal rank of such $\calB$
equals $r\in\N_0$,
then the minimal rank of $\calB'\in\scrB_{00}(\bfF,\bfE)$
so that the limit \eqref{lim} holds
in a weaker operator topology
(strong or weak or weak$^*$ operator topology)
is also equal to $r$.
\end{remark}

\begin{remark}\label{remark-smaller}
We point out that
if $z_0$ is of rank $r\in\N$
relative to $(\bfE,\bfF,\varOmega)$
and if $\calB\in\scrB_{00}(\bfF,\bfE)$,
$\rank\calB=r$,
is such that the limit
\eqref{lim} exists in the weak or weak$^*$ operator topology,
then
$B=\imath\circ\calB\circ\jmath\in\scrB_{00}(\bfX)$ also
satisfies $\rank B=r$.
Indeed,
if $r_1:=\rank B\le r$
satisfies $r_1<r$,
then
we could choose $\calB_1\in\scrB_{00}(\bfF,\bfE)$
so that $\calB\at{\ran(\jmath)}=\calB_1\at{\ran(\jmath)}$,
$\rank\calB_1=r_1$,
hence the limit
\eqref{lim} still exists
in the weak or weak$^*$ operator topology, respectively,
with
$B=\imath\circ\calB_1\circ\jmath$,
where $\calB_1$ is of rank $r_1<r$,
in contradiction to the definition of $r$.
\end{remark}

\begin{remark}\label{rem:FirstRemark}
If
$z_0$ is
a regular point of the essential spectrum
relative to $(\bfE,\bfF,\varOmega)$
and $\phi\in\bfE$ is such that $(A-z_0 I_\bfX)\imath(\phi)=0$,
then $(A-z I_\bfX)\imath(\phi)=(z_0-z)\imath(\phi)$
and so
\[
\jmath\circ\imath (\phi)
=(z_0-z)\jmath\circ(A-z I_\bfX)^{-1}\circ\imath(\phi),
\qquad
\forall z\in\varOmega,
\]
which in the limit $z\to z_0$ yields $\phi= 0$.
Thus,
$(A-z_0 I_\bfX)\at{\imath(\bfE)}$ is injective.
Therefore, a regular point of the essential spectrum
relative to $(\bfE,\bfF,\varOmega)$ is not an
eigenvalue associated to an eigenvector of $A$
in $\imath(\bfE)$.
\end{remark}

\begin{remark}\label{remark-ess}
By \cite[\S I.4]{edmunds2018spectral},
one distinguishes
five different types of the essential spectra:
\[
\sigma\sb{\mathrm{ess},1}(A)\subset
\sigma\sb{\mathrm{ess},2}(A)\subset
\sigma\sb{\mathrm{ess},3}(A)\subset
\sigma\sb{\mathrm{ess},4}(A)\subset
\sigma\sb{\mathrm{ess},5}(A)
=:
\sigma\sb{\mathrm{ess}}(A),
\]
where
$\sigma\sb{\mathrm{ess},1}(A)$
is defined as $z\in\sigma(A)$
such that either $\range(A-z I_\bfX)$ is not closed
or both $\ker(A-z I_\bfX)$ and $\coker(A-z I_\bfX)=\bfX/\range(A-z I_\bfX)$
are infinite-dimensional
(this definition of the essential spectrum was used by T.\,Kato
in \cite{kato1995perturbation}).
The spectrum
$\sigma\sb{\mathrm{ess},2}(A)$
is the set of points $z\in\sigma(A)$
such that either $\range(A-z I_\bfX)$ is not closed
or $\ker(A-z I_\bfX)$ is infinite-dimensional;
$\sigma\sb{\mathrm{ess},3}(A)$
and $\sigma\sb{\mathrm{ess},4}(A)$
are, respectively,
the sets of points $z\in\sigma(A)$ such that
$A-z I$ is not Fredholm
and such that
$A-z I$ is not Fredholm of index zero.
The spectrum $\sigma\sb{\mathrm{ess},5}(A)$
is defined as
the union of $\sigma\sb{\mathrm{ess},1}(A)$
with the connected components of
$\C\setminus\sigma\sb{\mathrm{ess},1}(A)$
which do not intersect the resolvent set of $A$
(this definition of the essential spectrum is due to F.\,Browder;
see \cite[Definition 11]{browder1961spectral}).
It follows that
if $z_0\in\sigma\sb{\mathrm{ess}}(A)$
satisfies $z_0\in\p\varOmega$,
with $\varOmega\subset\C\setminus\sigma(A)$,
then $z_0\in\sigma\sb{\mathrm{ess},1}(A)$.
We remind that, by the Weyl theorem
(see \cite[Theorem IX.2.1]{edmunds2018spectral}),
the essential spectra
$\sigma\sb{\mathrm{ess},k}(A)$,
$1\le k\le 4$,
remain invariant with respect to
relatively compact perturbations,
although this is not necessarily so for
$\sigma\sb{\mathrm{ess},5}(A)$.
Let us mention that
the essential spectrum
$\sigma\sb{\mathrm{ess},2}(A)$
could be characterized as the values of $z\in\C$
for which one can construct the Weyl sequences.
We also remind that
if the essential spectrum
$\sigma\sb{\mathrm{ess},5}(A)$
does not contain
open subsets of $\C$,
then all the essential spectra
$\sigma\sb{\mathrm{ess},k}(A)$,
$1\le k\le 5$, coincide.
For more details, see \cite[\S I.4]{edmunds2018spectral}.

We note that
if $z_0\in\sigma\sb{\mathrm{ess}}(A)\cap\p\varOmega$,
$\varOmega\subset\C\setminus\sigma(A)$,
then $z_0\in\sigma\sb{\mathrm{ess},1}(A)$
and $\varOmega\subset \C\setminus\sigma\sb{\mathrm{ess},1}(A)$.
\end{remark}

We assume that $A\in\scrC(\bfX)$ has a closable extension onto $\bfF$,
in the following sense:
\begin{assumption}
\label{ass-virtual}
The operator
$A\in\scrC(\bfX)$,
considered as a mapping
$\bfF\to\bfF$,
\begin{eqnarray}\label{def-a-f-f}
\dom(A\sb{\bfF\shortto\bfF}):=\jmath(\dom(A)),
\qquad
A\sb{\bfF\shortto\bfF}:\,
\Psi\mapsto\jmath(A\,\jmath^{-1}(\Psi)),
\end{eqnarray}
is closable in $\bfF$, with closure $\hatA\in\scrC(\bfF)$
and domain
$\doma\supset
\dom(A\sb{\bfF\shortto\bfF}):=\jmath\big(\dom(A)\big)$.

If, further, $\bfF$ has a pre-dual $\bfF_*$
and one considers convergence of the resolvent
in the weak$^*$ operator topology of $\scrB(\bfE,\bfF)$,
then we additionally assume that the graph of
$\hatA$ is not only closed but also weak$^*$-closed.

\end{assumption}

We recall that weakly closed subset in a Banach
space are exactly the strongly closed subsets which are convex
\cite[Theorem~III.9]{brezis2010functional}.
Closed linear subspaces are thus weakly closed;
applied to graphs of linear operators,
this shows that closed operators are weakly closed,
and
vice versa.
We note that
weak$^*$-closed operators are also weakly closed
(and hence are closed),
while the converse is not necessarily true.

\begin{remark}
The hypothesis in Assumption~\ref{ass-virtual} that
a mapping which is closed is also
weak$^*$-closable is not redundant.
Assume that
$\bfF$ is such that a pre-dual
$\bfF_*$ is not a Grothendieck space\footnote{A Grothendieck space is a Banach space $\bfG$ such that
the weak$^*$ sequential convergence in $\bfG^*$
implies weak sequential convergence in $\bfG^*$.
Examples of non-Grothendieck spaces are $\ell^1(\N)$ and $c_0(\N)$.
\label{footnote-g}
}.
Then there exists a sequence $(u_j)_{j\in\N}$ in $\bfF$ which is convergent
to $0\in\bfF$
in the weak$^*$ topology but not
in the weak topology.
Hence there exists
$\eta\in\bfF^*\setminus J_{\bfF_*}(\bfF_*)$
such that
\[
\langle
\eta, u_j
\rangle_{\bfF}\not\to 0.
\]
By the Banach--Steinhaus theorem, $(u_j)_{j\in\N}$ is bounded and hence, up to a subsequence,
$\langle\eta, u_j \rangle_{\bfF}\to c\neq 0$.
We fix $v\in\bfF\setminus\{0\}$
and define the operator
$\hatA=v\otimes\eta\in\scrB_{00}(\bfF)$
by
$\hatA:\,u\mapsto
v\overline{\langle\eta, u \rangle_{\bfF}}$;
one has
$\hatA u_j\to \bar{c} v \neq 0$.
Said differently, $\hatA$ is not weak$^*$-closable,
while it is closed and weakly closed
since it is a bounded operator.
\end{remark}

In the applications of the theory of virtual levels
and virtual states
to differential operators
it is useful to be able to consider
relatively compact perturbations,
allowing in place of
$\calB\in\scrB_{00}(\bfF,\bfE)$ in
Definition~\ref{def-virtual}
operators
$\calB:\,\bfF\to\bfE$
which are $\hatA$-compact
(in the sense of Definition~\ref{def-compact}
below).

\begin{lemma}\label{Lem:AboundJ0}
Let Assumption~\ref{ass-virtual} be satisfied.
Let
\[
\updelta_A:\;\dom(A)\hookrightarrow\bfX,
\qquad
\updelta_{\hatA}:\;\doma\hookrightarrow\bfF
\]
be the canonical embeddings;
here $\dom(A)$ and $\doma$
are considered as the Banach spaces
endowed with the
corresponding graph norms
\eqref{gn} and
\begin{eqnarray}\label{gn-hat}
\norm{x}_{\doma}
:=
\norm{x}_\bfF+\norm{\hatA x}_\bfF,
\qquad
x\in\doma.
\end{eqnarray}
There is a unique
map
$\hat\jmath\in\scrB\big(\dom(A),\doma\big)$
such that
\[
\updelta_{\hatA}\circ\hat\jmath=\jmath\circ\updelta_A:\;\dom(A)\to\bfF.
\]
\end{lemma}

\begin{figure}[ht]
\begin{picture}(0,80)(-220,0)
\linethickness{0.33pt}
\put(-35,27){$\updelta_A$}
\put(-17,27){$\longhookdownarrow$}
\put( 63,27){$\updelta_{\hatA}$}
\put( 53,27){$\longhookdownarrow$}
\put(23,64){$\hat\jmath$}
\put(23,11){$\jmath$}
\put(-30,55){$\dom(A)\ \ \xrightarrow{\hspace{2em}}  \ \,\doma$}
\put(-20,0){$\bfX\ \ \ \ \xhookrightarrow{\hspace{2em}}  \ \ \,\bfF $}
\end{picture}
\caption{The map $\hat\jmath:\,\dom(A)\to\doma$
is defined by
$\updelta_{\hatA}\circ\hat\jmath=\jmath\circ\updelta_A$.}
\label{fig-embeddings}
\end{figure}

\begin{proof}
The map $\hat\jmath$ is the canonical embedding
$\dom(A)\hookrightarrow\doma$ resulting in the restriction of $\jmath$ to $\doma$. Its uniqueness follows from the density of $\dom(A)$. Let us show it is bounded:
For any $\psi\in\dom(A)$,
\[
\norm{\hat\jmath(\psi)}_{\doma}
=
\norm{\jmath(A \psi)}_\bfF+\norm{\jmath(\psi)}_\bfF
\leq \norm{\jmath}\sb{\bfX\shortto\bfF}(\norm{A \psi}_\bfX+\norm{\psi}_\bfX)=\norm{\jmath}\sb{\bfX\shortto\bfF}\norm{\psi}_{\dom(A)}.
\]
This implies that $\hat\jmath$ is bounded from $\dom(A)$ to $\doma$.
\end{proof}
\begin{remark}\label{remark-aboundj}
We note that
if $A\in\scrC(\bfX)$ satisfies Assumption~\ref{ass-virtual}
and $\calB:\,\bfF\to\bfE$ is $\hatA$-compact,
then $B=\imath\circ\calB\circ\jmath$
is $A$-compact.
Indeed,
the unit ball in $\dom(A)$ with respect to the graph norm of $A$,
\begin{eqnarray}\label{gn}
\norm{\psi}_{\dom(A)}
:=
\norm{\psi}_\bfE+\norm{A \psi}_\bfE,
\qquad
\psi\in\dom(A),
\end{eqnarray}
is mapped
by $\jmath$ into a
unit ball in $\doma$ with respect to the graph norm of $\hatA$.
Since
$\calB\circ\updelta_{\hatA}:\,\doma\to\bfE$
is compact,
we conclude that
so is
$\imath\circ\calB\circ\jmath\circ\updelta_A
=\imath\circ\calB\circ\updelta_{\hatA}\circ\hat\jmath:
\,\dom(A)\to\bfX
$
(cf. Lemma~\ref{Lem:AboundJ0});
that is,
$B=\imath\circ\calB\circ\jmath:\,\bfX\to\bfX$
is $A$-compact.
By \cite[Theorem IV.1.11]{kato1995perturbation},
one has $A+\imath\circ\calB\circ\jmath\in\scrC(\bfX)$.
The identity
\[
(\hatA+\jmath\circ\imath\circ\calB)\jmath(\psi)=\jmath\circ(A+\imath\circ\calB\circ\jmath)\psi,
\qquad
\forall \psi\in\dom(A),
\]
shows that $A+\imath\circ\calB\circ\jmath$
also satisfies Assumption~\ref{ass-virtual}.
\end{remark}

\begin{lemma}\label{lemma-uniform}
Let $A\in\scrC(\bfX)$
satisfy Assumption~\ref{ass-virtual},
let $\varOmega\subset\C\setminus\sigma(A)$,
and let
$z_0\in\sigma\sb{\mathrm{ess}}(A)\cap\p\varOmega$.
Assume that there is a convergence
\begin{equation}\label{Eq:LimInF}
\jmath\circ(A-z I_\bfX)^{-1}\circ\imath
\to
(A-z_0 I_\bfX)^{-1}_{\bfE,\bfF,\varOmega}:\;\bfE\to\bfF
\qquad
z\to z_0,\quad z\in\varOmega,
\end{equation}
in the weak or weak$^*$ operator topology of $\scrB(\bfE,\bfF)$.
\begin{enumerate}
\item
\label{lemma-uniform-1}
There is $\delta>0$ such that
the operator family
$\big\{\jmath\circ(A-z I_\bfX)^{-1}\circ\imath\big\}_{z\in\varOmega\cap\mathbb{D}_\delta(z_0)}$
is uniformly bounded in $\scrB(\bfE,\bfF)$.
\item
\label{lemma-uniform-2}
Let $\delta>0$ be as in Part~\itref{lemma-uniform-1}.
There is a unique operator family
$\{\rho(z)\}_{z\in\varOmega\cap\mathbb{D}_\delta(z_0)}$ in
$\scrB\big(\bfE,\dom(A))$
such that 
\[
\updelta_A\circ\rho(z)=
(A-z I_\bfX)^{-1}\circ\imath,
\qquad
z\in\C\setminus\sigma(A),
\]
and the family
$\big\{\hat\jmath\circ\rho(z)
\big\}_{z\in\varOmega\cap\mathbb{D}_\delta(z_0)}$
is uniformly bounded in $\scrB\big(\bfE,\doma\big)$,
where $\doma$
is considered as the Banach space
endowed with the corresponding graph norm \eqref{gn-hat}.
\item
\label{lemma-uniform-3}
If
$\jmath\circ(A-z I_\bfX)^{-1}\circ\imath$
converges as $z\to z_0$, $z\in\varOmega$
in the
strong or uniform operator topology
of $\scrB(\bfE,\bfF)$, then
$\hat\jmath\circ\rho(z)
:\,\bfE\to\doma$
converges as $z\to z_0$, $z\in\varOmega$
in the
strong
or
uniform operator topology,
respectively,
of $\scrB\big(\bfE,\doma\big)$.
\item
\label{lemma-uniform-4}
Let $\delta>0$ be as in Part~\itref{lemma-uniform-1}.
If $\calB:\,\bfF\to\bfE$ is $\hatA$-compact,
then the operator family
$\big\{\calB\circ\jmath\circ(A-z I_\bfX)^{-1}\circ\imath
\big\}_{z\in\varOmega\cap\mathbb{D}_\delta(z_0)}$
is collectively compact, in the following sense:
\[
\bigcup\sb{z\in\varOmega\cap\mathbb{D}_\delta}
\calB
\circ\jmath\circ(A-z I_\bfX)^{-1}
\circ\imath\big(\mathbb{B}_1(\bfE)\big)
\mbox{ is precompact}.
\]
\end{enumerate}
\end{lemma}

Above, $\updelta_A:\,\dom(A)\hookrightarrow\bfX$
and
$\hat\jmath:\,\dom(A)\to\doma$
are from Lemma~\ref{Lem:AboundJ0};
see Figure~\ref{fig-embeddings}.

\begin{proof}
By the Banach--Steinhaus theorem,
for each $\phi\in\bfE$ there is $\delta>0$ such that
\[
\big\{\jmath\circ(A-z I_\bfX)^{-1}\circ\imath(\phi)
\big\}_{z\in\varOmega\cap\mathbb{D}_\delta(z_0)}
\]
is bounded in $\bfF$
uniformly in
$z\in\varOmega\cap\mathbb{D}_\delta(z_0)$.
The proof of Part~\itref{lemma-uniform-1}
follows by one more application
of the Banach--Steinhaus theorem,
now with respect to $\phi\in\bfE$.

For Part~\itref{lemma-uniform-2},
we notice that for any
$z\in\C\setminus\sigma(A)$,
for each given
$\phi\in\bfE$,
the mapping
$\phi\mapsto(A-z I_\bfX)^{-1}\imath(\phi)\in\dom(A)$
defines the map
$\rho(z):\,\bfE\to\dom(A)$.
One has
$\rho(z)\in\scrB(\bfE,\dom(A))$
since both $(A-z I_\bfX)^{-1}\imath(\phi)$ and
\[
A(A-z I_\bfX)^{-1}\imath(\phi)
=\imath(\phi)+z(A-z I_\bfX)^{-1}\imath(\phi)
\]
are in $\bfX$.
The uniform boundedness
of
$\hat\jmath\circ\rho(z)$
in $\scrB\big(\bfE,\doma\big)$
for $z\in\varOmega\cap\mathbb{D}_\delta(z_0)$
follows from the boundedness of
$\jmath\circ(A-z I_\bfX)^{-1}\circ\imath(\phi)$
and of
\begin{eqnarray}
\hatA
\jmath\circ(A-z I_\bfX)^{-1}\circ\imath(\phi)
=
A_{\bfF\shortto\bfF}
\jmath\circ(A-z I_\bfX)^{-1}\circ\imath(\phi)
=
\jmath\circ A(A-z I_\bfX)^{-1}\circ\imath(\phi)
\nonumber
\\
=\jmath\circ I_\bfX\circ\imath(\phi)
+z\jmath\circ(A-z I_\bfX)^{-1}\circ\imath(\phi)
\label{if-you-wish}
\end{eqnarray}
in $\bfF$,
uniformly in $z\in\varOmega\cap\mathbb{D}_\delta(z_0)$,
for each $\phi\in\bfE$.
We took into account the relation
$
A_{\bfF\shortto\bfF}\circ\jmath
=
\jmath\circ A
$
and the identity
$
A(A-z I_\bfX)^{-1}=I_\bfX+z(A-z I_\bfX)^{-1}$
for $z\not\in\sigma(A)$.
We conclude that
$\big\{\hat\jmath\circ\rho(z)
\big\}_{z\in\varOmega\cap\mathbb{D}_\delta(z_0)}$
is uniformly bounded in $\scrB\big(\bfE,\doma\big)$.
We note that the operator family
$\{\rho(z)\}_{z\in\varOmega\cap\mathbb{D}_\delta(z_0)}$
is unique since $\updelta_A$ is an embedding.

Let us prove Part~\itref{lemma-uniform-3}.
Assume that there is the convergence of $\calR(z)$
as $z\to z_0$, $z\in\varOmega$,
in the uniform operator topology
of $\scrB(\bfE,\bfF)$.
To prove the convergence of $\hat\jmath\circ\rho(z)$
as $z\to z_0$, $z\in\varOmega$,
in the uniform operator topology
of $\scrB\big(\bfE,\doma\big)$,
we will show that the family
$\big\{\hat\jmath\circ\rho(z)\in\scrB\big(\bfE,\doma\big)
\big\}_{z\in\varOmega}$
is Cauchy as
$z\to z_0$
(in the sense that
for any $\varepsilon>0$ there is $\delta>0$
such that for any
$z,\,z'\in\varOmega\cap\mathbb{D}_\delta(z_0)$
one has
$\norm{\hat\jmath\circ\rho(z)-\hat\jmath\circ\rho(z')
}_{\bfE\shortto\doma}
<\varepsilon$).
We have:
\begin{eqnarray}
\nonumber
\|\hat\jmath\circ\rho(z)
-\hat\jmath\circ\rho(z')\|_{\bfE\shortto\doma}
=\|\calR(z)-\calR(z')\|_{\bfE\shortto\bfF}
+\|\hatA(\calR(z)-\calR(z'))\|_{\bfE\shortto\bfF}
\\
\label{z-z}
=\|\calR(z)-\calR(z')\|_{\bfE\shortto\bfF}
+\|z\calR(z)-z'\calR(z')\|_{\bfE\shortto\bfF};
\end{eqnarray}
above, we used the identity
$\hatA\calR(z)=\jmath\circ\imath+z \calR(z)$,
$z\in\varOmega$.
Since the families
$\{z\}_{z\in\varOmega}$
and $\big\{\calR(z)\in\scrB(\bfE,\bfF)\big\}_{z\in\varOmega}$
are Cauchy as $z\to z_0$,
the product $z\calR(z)$
is Cauchy
in $\scrB(\bfE,\bfF)$ as $z\to z_0$, $z\in\varOmega$;
thus, the right-hand side of \eqref{z-z}
tends to zero as $z,\,z'\to z_0$
and $z,\,z'\in\varOmega$.

Similarly, for a fixed $\phi\in\bfE$,
one proves that as long as the family
$\big\{\calR(z)\phi\big\}_{z\in\varOmega}$,
is Cauchy in $\bfF$ as $z\to z_0$,
the set
$\big\{\hat\jmath\circ\rho(z)\phi\in\doma\big\}_{z\in\varOmega}$
is Cauchy as $z\to z_0$;
thus, the convergence of $\calR(z)$
in the strong operator topology
$\scrB(\bfE,\bfF)$
leads to the convergence of $\hat\jmath\circ\rho(z)$
in the strong operator topology
of $\scrB\big(\bfE,\doma\big)$.

Let us prove Part~\itref{lemma-uniform-4}.
By Part~\itref{lemma-uniform-2},
we have:
\begin{eqnarray}\label{uu}
\bigcup_{z\in\varOmega\cap\mathbb{D}_\delta(z_0)}
\calB
\circ\jmath\circ(A-z I_\bfX)^{-1}\circ\imath
\big(\mathbb{B}_1(\bfE)\big)
=\calB\circ\updelta_A
\Big(
\mathop{\bigcup}\limits_{z\in\varOmega\cap\mathbb{D}_\delta(z_0)}
\rho(z)\mathbb{B}_1(\bfE)\Big),
\end{eqnarray}
where the set in the parentheses is bounded in $\doma$
since
the operators
$\rho(z)\in\scrB\big(\bfE,\doma\big)$ are uniformly bounded
for $z\in\varOmega\cap\mathbb{D}_\delta(z_0)$
by Part~\itref{lemma-uniform-2},
while
$\calB\circ\updelta_A\in\scrB_0\big(\doma,\bfF\big)$
since $\calB$ is $\hatA$-compact.
It follows that \eqref{uu} is a precompact set,
completing the proof.
\end{proof}

\begin{example}\label{example-laplace-2d}
Let $A=-\Delta$ in $\bfX=L^2(\R^2)$, $\dom(A)=H^2(\R^2)$.
For any $s,\,s'\ge 0$,
let $\bfE=L^2_s(\R^2)$
and $\bfF=L^2_{-s'}(\R^2)$, the subspaces of
the space of distributions $\scrD'(\R^2)$.
Any sequence $(u_j)_{j_\in\N}$ in $H^2(\R^2)$ such that
$\jmath(u_j)\mathop{\longrightarrow}\limits\sp{\bfF}0$
also satisfies
$u_j\mathop{\longrightarrow}\limits\sp{\scrD'(\R^2)} 0$,
and so the convergence
$\jmath(A u_j)\mathop{\longrightarrow}\limits\sp{\bfF}v$
implies $v=0$.
Hence $A$ is closable as a mapping $\bfF\to\bfF$.
The integral kernel of
$R_0^{(2)}(z)=(-\Delta-z I )^{-1}$,
$z\in\cnor:=\C\setminus\overline{\R_{+}}$,
is given by
\[
R_0^{(2)}(x,y;z)
=
\frac{\jj}{4}H_0^{(1)}(\zeta\abs{x-y}),
\quad
\mbox{where}
\quad
\zeta\in\C_{+},\quad \zeta^2=z,
\]
with $H_0^{(1)}(\upzeta)=J_0(\upzeta)+\jj Y_0(\upzeta)$
the Hankel function of the first kind
which has the asymptotic behavior
(see \cite[Ch.~9]{AS72})

\begin{eqnarray}\label{h01}
H_0^{(1)}(\upzeta)
=
\begin{cases}
1+\frac{2\jj}{\pi}
\big(\ln\frac{\upzeta}{2}+\gamma\big)
+o(1),
&0<\abs{\upzeta}\ll 1,
\\
\sqrt{\frac{2}{\pi\upzeta}}
\,e^{\jj(\upzeta-\frac{1}{4}\pi)},
&\abs{\upzeta}\gg 1,
\end{cases}
\qquad
\upzeta\in\C,
\quad
\abs{\arg\upzeta}<\pi,
\end{eqnarray}
with Euler's constant $\gamma\approx 0.577$.
This leads to
\begin{eqnarray}\label{r02}
R_0^{(2)}(x,y;z)
=-\frac{1}{2\pi}
\ln(\zeta\abs{x-y})+O(1),
\quad
\zeta\in\C_{+},\quad \zeta^2=z,
\quad
\abs{x-y}\ll\abs{z}^{-1/2},
\end{eqnarray}
showing that
$(-\Delta-z I )^{-1}:\,L^2_s(\R^2)\to L^2_{-s'}(\R^2)$,
with arbitrarily large $s,\,s'\ge 0$,
cannot be bounded uniformly in
$z\in\mathbb{D}\setminus\overline{\R_{+}}$,
thus $z_0=0$ is not a regular point of the essential spectrum
relative to
$\big(L^2_s(\R^2),L^2_{-s'}(\R^2),\cnor\big)$.
\end{example}

\begin{example}\label{Ex:FreeDim3}
Let $A=-\Delta$ in $\bfX=L^2(\R^3)$, $\dom(A)=H^2(\R^3)$.
Fix $s,\,s'>1/2$
and let $\bfE=L^2_s(\R^3)$
and $\bfF=L^2_{-s'}(\R^3)$.
As in Example~\ref{example-laplace-2d},
$A$ is closable as a mapping $\bfF\to\bfF$.
As follows from \cite[Appendix A]{agmon1975spectral},
for any $z_0>0$,
the resolvent $(-\Delta-z I )^{-1}$ converges
as $z\to z_0\pm \jj 0$
in the uniform operator topology of continuous mappings
$L^2_s(\R^3)\to L^2_{-s'}(\R^3)$.
The two limits differ;
the integral kernels of the limiting operators
$(-\Delta-z_0 I )^{-1}_{L^2_s,L^2_{-s'},\C_\pm}$
are given,
respectively,
by
$e^{\pm\jj\abs{x-y}\sqrt{z_0}}/(4\pi\abs{x-y})$,
$x,\,y\in\R^3$.
It follows that $z_0>0$ is a regular point
of the essential spectrum of $-\Delta$
relative to $\big(L^2_s(\R^3),L^2_{-s'}(\R^3),\C_{\pm}\big)$.
Moreover,
there is a limit of the resolvent
as $z\to z_0=0$, $z\in\cnor$,
in the
strong operator topology of
$\scrB\big(L^2_s(\R^3),L^2_{-s'}(\R^3)\big)$,
$s,\,s'>1/2$, $s+s'\ge 2$
(see \cite[Lemma 2.1]{nirenberg1973null}
and also Theorem~\ref{theorem-3d} below),
hence $z_0=0$ is also a regular point of the essential spectrum
relative to
$\big(L^2_s(\R^3),L^2_{-s'}(\R^3),\C_{\pm}\big)$.
\end{example}

\begin{example}\label{example-d}
Consider the differential operator
$A=-\jj\p_x +V:\,L^2(\R)\to L^2(\R)$,
$\dom(A)=H^1(\R)$,
with $V$ the operator of multiplication by
$V\in L^1(\R)\cap L^2(\R)$.
We note that $A$
has a closable extension onto
$\bfF=L^2_{-s}(\R)$
(cf. Example~\ref{example-laplace-2d}).
The solution to the problem
\[
(-\jj\p_x+V-z I )u=f\in L^2_s(\R),
\qquad
z\in\C_{+},
\]
with $s>1/2$,
is given by
\begin{eqnarray}\label{u-is-f}
u(x)=\jj\int^x_{-\infty}e^{\jj z(x-y)-\jj W(x)+\jj W(y)}f(y)\,dy,
\;\;
\mbox{with}\;\;W(x):=\int^x_{-\infty} V(y)\,dy,
\;\; W\in L^\infty(\R).
\end{eqnarray}
By \eqref{u-is-f},
for each $z\in\C_{+}$,
the mapping
$(A-z I )^{-1}:\,f\mapsto u$
is continuous
from $L^2_s(\R)$ to $L^2_{-s}(\R)$,
with the
bound uniform in $z\in\C_{+}$.
Moreover,
for each $z_0\in\R$ there exists a limit
$(A-z_0 I )^{-1}_{L^2_s,L^2_{-s},\C_{+}}=
\lim\limits_{z\to z_0,\ z\in\C_+}
(A-z I )^{-1}$
in the strong operator topology of
$\scrB\big(L^2_s(\R^2),L^2_{-s}(\R)\big)$;
thus, any $z_0\in\R$
is a regular point of the essential spectrum
of $A$
relative to $\big(L^2_s(\R),L^2_{-s}(\R),\C_{+}\big)$
and similarly
relative to $\big(L^2_s(\R),L^2_{-s}(\R),\C_{-}\big)$.
As the matter of fact, one can see that
\eqref{u-is-f} defines an operator 
$L^1(\R)\to L^\infty(\R)$
bounded uniformly in $z\in\C_{+}$
and for any $z_0\in\R$
converges
as $z\to z_0$, $z\in\C_{+}$,
in the strong operator topology.
We note that the framework
of Definition~\ref{def-virtual}
does not apply to this situation
since there are no embeddings
of $L^1$ into $L^2$ and of $L^2$ into $L^\infty$;
we will return to this issue in
Section~\ref{sect-factorization}.
\end{example}

\begin{example}\label{example-shift}
Consider the left shift
$L:\,\ell^2(\N)\to\ell^2(\N)$,
$(x_1,x_2,x_3,\dots)
\mapsto
(x_2,x_3,x_4,\dots)$,
with
$
\sigma(L)=\sigma\sb{\mathrm{ess}}(L)
=\overline{\mathbb{D}}$.
The operator $L$ has a closed extension to $c_0(\N)$
as a bounded operator from the dense subspace $\ell^2(\N)$.
The matrix representations of $L-z I$ and $(L-z I)^{-1}$,
$z\in\C\setminus\overline{\mathbb{D}}$,
are given by
\[
L-z I=
{\footnotesize
\begin{bmatrix}
-z&1&0&\cdots\\
0&-z&1&\cdots\\
0&0&-z&\cdots\\
\vdots&\vdots&\vdots&\ddots
\end{bmatrix}
},
\ \ z\in\C;
\qquad
(L-z I )^{-1}
=
-
{\footnotesize
\begin{bmatrix}
z^{-1}&z^{-2}&z^{-3}&\cdots\\
0&z^{-1}&z^{-2}&\cdots\\
0&0&z^{-1}&\cdots\\
\vdots&\vdots&\vdots&\ddots
\end{bmatrix}
},
\ \ z\in\C\setminus\overline{\mathbb{D}}.
\]
  From the above representation,
one has
$\abs{((L-z I )^{-1}x)_i}\le
\abs{z^{-1}x_i}+\abs{z^{-2}x_{i+1}}+\dots\le\norm{x}_{\ell^1}$,
and moreover $\lim_{i\to\infty}((L-z I )^{-1}x)_i=0$,
for any $x\in\ell^1(\N)\subset\ell^2(\N)$
and any $z\in\C\setminus\overline{\mathbb{D}}$,
hence
$(L-z I )^{-1}$
defines a continuous linear mapping
$\ell^1(\N)\to c_0(\N)$,
with the norm bounded (by one) uniformly
in $z\in\C$, $\abs{z}>1$.
For any $\abs{z_0}=1$,
$(L-z I )^{-1}$
converges
as $z\to z_0$, $\abs{z}>1$,
in the strong operator topology
(but not in the uniform operator topology)
of $\scrB\big(\ell^1(\N),c_0(\N)\big)$.
It follows that any
of the boundary points of the spectrum
of $L$
(i.e., any $z_0\in\C$ with $\abs{z_0}=1$)
is a regular point of the essential spectrum
relative to
$\big(\ell^1(\N),c_0(\N),\C\setminus\overline{\mathbb{D}}\big)$.

For the right shift $R$ in $\ell^2(\N)$,
it has a unique bounded extension to $\ell^\infty(\N)$ 
as the adjoint of the bounded restriction of the operator $L$ to $\ell^1(\N)$. 
Similarly to $L$, 
one has:
\[
R-z I=
{\footnotesize
\begin{bmatrix}
-z&0&0&\cdots\\
1&-z&0&\cdots\\
0&1&-z&\cdots\\
\vdots&\vdots&\vdots&\ddots
\end{bmatrix}
},
\ \ z\in\C;
\quad
(R-z I )^{-1}
=
-
{\footnotesize
\begin{bmatrix}
z^{-1}&0&0&\cdots\\
z^{-2}&z^{-1}&0&\cdots\\
z^{-3}&z^{-2}&z^{-1}&\cdots\\
\vdots&\vdots&\vdots&\ddots
\end{bmatrix}
},
\ \ z\in\C\setminus\overline{\mathbb{D}};
\]
one can see that
$
\abs{((R-z I)^{-1}x)_i}
=\abs{z^{-i}x_1+\dots+z^{-1}z_i}
\le\norm{x}_{\ell^1}
$,
hence $(R-z I)^{-1}$ is bounded uniformly
in $z\in\C\setminus\overline{\mathbb{D}}$
in the $\ell^1\to\ell^\infty$ norm
and moreover converges in the weak$^*$ operator topology
of $\scrB\big(\ell^1(\N),\ell^\infty(\N)\big)$,
hence any $z_0\in\p\mathbb{D}$
is a regular point of the essential spectrum of $R$
relative to
$\big(\ell^1(\N),\ell^\infty(\N),\C\setminus\overline{\mathbb{D}}\big)$.
We note that there is no
convergence of $(R-z I)^{-1}$
in the weak operator topology of
$\scrB\big(\ell^1(\N),\ell^\infty(\N)\big)$
and in
the weak operator topology of
$\scrB\big(\ell^1(\N),c_0(\N)\big)$
as $z\to z_0\in\p\mathbb{D}$, $\abs{z}>1$.
\end{example}

\begin{example}\label{example-shift-2}
Let
$L:\,\ell^2(\Z)\to\ell^2(\Z)$,
$(\dots,x_{-1},x_0,\dots)
\mapsto
(\dots,x_0,x_1,\dots)$,
be the left bilateral shift,
with
$
\sigma(L)=\sigma\sb{\mathrm{ess}}(L)
=\p\mathbb{D}$.
The operator $L$ has a bounded extension
from $\ell^2(\Z)$ onto $c_0(\Z)$.
Since the adjoint of $L$ has a bounded restriction
onto
$\ell^1(\Z)$, it has a unique bounded extension onto
$\ell^\infty(\Z)$ or any of its invariant subspaces.
The matrix representation of
$(L-z I )^{-1}$
for $z\in\C\setminus\overline{\mathbb{D}}$
and $z\in\mathbb{D}$
is given, respectively, by
\begin{eqnarray}\label{res-out}
(L-z I )^{-1}
=
-
{\footnotesize
\begin{bmatrix}
\ddots&\vdots&\vdots&\vdots&\cdots\\
\cdots&z^{-1}&z^{-2}&z^{-3}&\cdots\\
\cdots&0&z^{-1}&z^{-2}&\cdots\\
\cdots&0&0&z^{-1}&\cdots\\
\vdots&\vdots&\vdots&\vdots&\ddots
\end{bmatrix}
},
\qquad
z\in\C\setminus\overline{\mathbb{D}};
\end{eqnarray}
\begin{eqnarray}\label{res-in}
(L-z I )^{-1}
=
{\footnotesize
\begin{bmatrix}
\ddots&\vdots&\vdots&\vdots&\vdots&\cdots\\
\cdots&0&0&0&0&\cdots\\
\cdots&1&0&0&0&\cdots\\
\cdots&z&1&0&0&\cdots\\
\cdots&z^2&z&1&0&\cdots\\
\cdots&\vdots&\vdots&\vdots&\vdots&\ddots
\end{bmatrix}
},
\qquad
z\in\mathbb{D}.
\end{eqnarray}
  From the above representation,
one has
$\abs{((L-z I )^{-1}x)_i}\le\norm{x}_{\ell^1}$,
with $\lim_{i\to\pm\infty}((L-z I )^{-1}x)_i=0$
for any $x\in\ell^1(\Z)\subset \ell^2(\Z)$
and any $z\in\C\setminus\overline{\mathbb{D}}$
(note that the convergence in the limit
$((L-z I )^{-1}x)_i\to 0$ as $i\to -\infty$
is not uniform for $\abs{z}\to 1$);
$(L-z I )^{-1}$
defines a continuous linear mapping
$\ell^1(\Z)\to\ell^\infty(\Z)\cap
\{x\in\ell^\infty(\Z)\sothat \lim\sb{i\to+\infty}x_i=0\}$,
with the norm bounded (by one) uniformly
in $z\in\C$, $\abs{z}>1$.
For any $\abs{z_0}=1$,
the resolvent
$(L-z I)^{-1}$
converges
as $z\to z_0$, $\abs{z}>1$,
and also
as $z\to z_0$, $\abs{z}<1$,
in the
weak$^*$ operator topology
of $\scrB\big(\ell^1(\Z),\ell^\infty(\Z)\big)$.
It follows that any
of the boundary points of the spectrum
of $L$
(i.e., any $z_0\in\C$ with $\abs{z_0}=1$)
is a regular point of the essential spectrum
relative to
$\big(\ell^1(\Z),\ell^\infty(\Z),\C\setminus\overline{\mathbb{D}}\big)$
and relative to
$\big(\ell^1(\Z),\ell^\infty(\Z),\mathbb{D}\big)$.

Let us point out that the convergence of
$(L-z I)^{-1}:\,\ell^1(\Z)\to\ell^\infty(\Z)$
as $z\to z_0\in\p\mathbb{D}$, $z\in\mathbb{D}$
(and also
as $z\to z_0\in\p\mathbb{D}$,
$z\in\C\setminus\overline{\mathbb{D}}$)
holds in the
weak$^*$ operator topology
but not in the
weak operator topology:
\[
(L-z I)^{-1}\e_0=\sum_{i\ge 1}\e_i z^i
\mbox{ does not converge to }
(L-z_0 I)^{-1}_{\ell^1(\Z),\ell^\infty(\Z),\mathbb{D}}\e_0
=\sum\sb{i\ge 1}\e_i z_0^i
\]
neither strongly (in $\ell^\infty(\Z)$) nor weakly
(here $\{\e_i\}_{i\in\Z}$ is the standard basis in $\ell^2(\Z)$).
\end{example}

\subsection{Jensen--Kato lemma:
limit of the resolvent as an inverse}

The following key lemma
is an abstract version of 
Lemma 2.4 in
\cite{jensen1979spectral}
(formulated for the Laplace operator in $\R^3$)
and in
\cite{jensen1980spectral,jensen1984spectral}
(for higher dimensions).

\begin{lemma}[Jensen--Kato lemma]
\label{lemma-jk}
Let the embeddings
$\bfE\mathop{\longhookrightarrow}\limits\sp{\imath}
\bfX\mathop{\longhookrightarrow}\limits\sp{\jmath}
\bfF$ be continuous,
and let $A\in\scrC(\bfX)$ satisfy Assumption~\ref{ass-virtual}.
Let $\varOmega\subset\C\setminus\sigma(A)$
and let $z_0\in\sigma\sb{\mathrm{ess}}(A)\cap\p\varOmega$.
Assume that there exists a limit
\[
\jmath\circ(A-z I_\bfX)^{-1}\circ\imath
\to
(A-z_0 I_\bfX)^{-1}_{\bfE,\bfF,\varOmega}
:\;\bfE\to\bfF,
\qquad
z\to z_0,\quad z\in\varOmega,
\]
in the weak or weak$^*$ operator topology of $\scrB(\bfE,\bfF)$.
Then:
\begin{enumerate}
\item
\label{lemma-jk-jk1}
There exists $\alpha\in\scrC(\bfF,\bfE)$
with $\dom(\alpha)=\range\big((A-z_0 I_\bfX)^{-1}_{\bfE,\bfF,\varOmega}\big)$
such that
\[
\jmath\circ\imath\circ\alpha
=(\hatA-z_0I_\bfF)
\at{\range\big((A-z_0 I_\bfX)^{-1}_{\bfE,\bfF,\varOmega}\big)}
\]
and such that
$\alpha:\,\dom(\alpha)\to\bfE$
is invertible with the
inverse $(A-z_0 I_\bfX)^{-1}_{\bfE,\bfF,\varOmega}$.
In particular,
\begin{eqnarray*}
&
\ker\big((A-z_0 I_\bfX)^{-1}_{\bfE,\bfF,\varOmega}\big)
=
\{0_\bfE\},
\\[1ex]
&
\range
\big((A-z_0 I_\bfX)^{-1}_{\bfE,\bfF,\varOmega}\big)
\subset
\doma,
\\[1ex]
&
(\hatA-z_0 I_\bfF)
(A-z_0 I_\bfX)^{-1}_{\bfE,\bfF,\varOmega}
=
\jmath\circ\imath,
\\[1ex]
&
(A-z_0 I_\bfX)^{-1}_{\bfE,\bfF,\varOmega}
\circ(\jmath\circ\imath)^{-1}\circ
(\hatA-z_0 I_\bfF)
\at{\range\big((A-z_0 I_\bfX)^{-1}_{\bfE,\bfF,\varOmega}\big)}
=
I_{\ran((A-z_0 I_\bfX)^{-1}_{\bfE,\bfF,\varOmega})}.
\end{eqnarray*}
\item
\label{lemma-jk-jk2}
If $\phi\in\bfE$, $\imath(\phi)\in\dom(A)$,
and $(A-z_0 I_\bfX)\imath(\phi)\in\imath(\bfE)$,
then
$
\jmath\circ\imath(\phi)\in\range\big((A-z_0 I_\bfX)^{-1}_{\bfE,\bfF,\varOmega}\big).
$
Moreover,
\[
\jmath\circ\imath(\phi)
=
(A-z_0 I_\bfX)^{-1}_{\bfE,\bfF,\varOmega}\phi_1
\quad
\mbox{with}
\quad
\phi_1=\imath^{-1}\circ(A-z_0 I_\bfX)\imath(\phi).
\]
\end{enumerate}
\end{lemma}

\begin{proof}
Below, for brevity, we will use the following notations:
\begin{eqnarray}\label{def-r-z0}
R(z)&=&(A-z I_\bfX)^{-1}:\;\bfX\to\bfX,
\qquad
z\in\C\setminus\sigma(A);
\\[1ex]
\calR(z)&=&\jmath\circ(A-z I_\bfX)^{-1}\circ\imath:\;\bfE\to\bfF,
\qquad
z\in\C\setminus\sigma(A);
\\[1ex]
\calR(z_0)&=&(A-z_0 I_\bfX)^{-1}_{\bfE,\bfF,\varOmega}
:\,\bfE\to\bfF.
\end{eqnarray}
Under the assumptions of the lemma, $\calR(z_0)$ is in $\scrB(\bfE,\bfF)$ and therefore its graph $\mathcal{G}(\calR(z_0))$ is closed in $\bfE\times\bfF$.
Let $\phi\in\bfE$
and denote $\Psi_z:=\jmath\circ R(z)\imath(\phi)\in\bfF$,
$z\in\varOmega$;
then
$\jmath\circ\imath (\phi)=\jmath\circ(A-z I_\bfX)\circ R(z)\imath(\phi)=(\hatA-z I_\bfF)\Psi_z=\jmath\circ\imath(\phi)$
and
\[
\Psi_z\to\calR(z_0)\phi
\in
\bfF
\qquad
z\to z_0,\quad z\in\varOmega,
\]
in the weak or weak$^*$ topology of $\bfF$.
Then,
since $\Psi_z$, $z\in\varOmega\cap \mathbb{D}_\delta(z_0)$ with some $\delta>0$, are uniformly bounded (cf.~Lemma~\ref{lemma-uniform}), one has
$
\slim_{z\to z_0,\,z\in\varOmega}
(\hatA-z_0I_\bfF)\Psi_z
=\jmath\circ\imath(\phi)\in\bfF$.

We claim that
\begin{eqnarray}\label{rg}
(\calR(z_0)\phi,\jmath\circ\imath(\phi))\in\mathcal{G}\big(\hatA-z_0 I_\bfF\big).
\end{eqnarray}
Indeed, if
the convergence
$\jmath\circ(A-z I_\bfX)^{-1}\circ\imath$ as $z\to z_0$,
$z\in\varOmega$, is in the weak operator topology,
then the graph of $\hatA$ is closed by
Assumption~\ref{ass-virtual}
(and then also closed in the weak topology).
If $\bfF$ has a pre-dual and
we assume instead that
there is the convergence \eqref{Eq:LimInF}
in the weak$^*$ operator topology of $\scrB(\bfE,\bfF)$,
the inclusion \eqref{rg} holds
since the graph of $\hatA$
is weak$^*$-closed by Assumption~\ref{ass-virtual}.

By \eqref{rg}, we have
\begin{eqnarray}\label{t-is}
\big(I_\bfF\times(\jmath\circ\imath)\big)
\big(\mathcal{G}(\calR(z_0))^T\big)
\subset \mathcal{G}\big(\hatA-z_0 I_\bfF\big),
\end{eqnarray}
where
$
(\,\cdot\,,\,\cdot\,)^T:\;
\bfE\times\bfF\to\bfF\times\bfE$,
$
(\phi,\Psi)\mapsto(\Psi,\phi)$.
It follows from \eqref{t-is} that
\begin{eqnarray}\label{i-f-t}
\ker(\calR(z_0))=\{0\},
\qquad
\ran(\calR(z_0))\subset\doma,
\qquad
(\hatA-z_0 I_\bfF)\at{\ran(\calR(z_0))}\subset\range(\jmath\circ\imath),
\end{eqnarray}
and that
\[
\mathcal{G}(\alpha):=\mathcal{G}(\calR(z_0))^T\subset\bfF\times\bfE
\]
is a graph of a closed linear mapping
$\alpha:\,\bfF\to\bfE$
which
can be characterized by the condition that for any
$\Psi \in\dom(\alpha):=\range(\calR(z_0))$
one has
\begin{eqnarray}\label{o-h}
\jmath\circ\imath\circ\alpha(\Psi)=(\hatA-z_0 I_\bfF)\Psi.
\end{eqnarray}
The relations \eqref{i-f-t}, \eqref{o-h},
together with the relations
$
\alpha\circ\calR(z_0)=I_{\bfE}$,
$\calR(z_0)\circ\alpha=I_{\range(\calR(z_0))}$,
conclude the proof of Part~\itref{lemma-jk-jk1}.

Let us prove Part~\itref{lemma-jk-jk2}.
If $(A-z_0 I_\bfX)\imath(\phi)=\imath(\phi_1)$
with some $\phi_1\in\bfE$, then
$(A-z I_\bfX)\imath(\phi)=\imath(\phi_1+(z_0-z)\phi)$
and hence
\[
\jmath\circ\imath (\phi)
=\jmath\circ(A-z I_\bfX)^{-1}\circ\imath\left(\phi_1+(z_0-z)\phi\right),
\qquad
\forall z\in\varOmega,
\]
which in the limit $z\to z_0$
yields the desired relation
$\jmath\circ\imath (\phi)=(A-z_0 I_\bfX)^{-1}_{\bfE,\bfF,\varOmega}\phi_1$.
\end{proof}

\begin{example}\label{Example:Canonical}
Let
$\bfE\hookrightarrow\bfX\hookrightarrow\bfF$,
let $A\in\scrC(\bfX)$ satisfy Assumption~\ref{ass-virtual},
and assume that $A$ satisfies the limiting absorption principle
at some $z\in\p\sigma(A)\cap\p\varOmega$,
$\varOmega\subset\C\setminus\sigma(A)$,
in the sense that there exists a limit
$
(A-z_0 I_\bfX)^{-1}_{\bfE,\bfF,\varOmega}
:=\lim\sb{z\to z_0,\,z\in\varOmega}(A-z I_\bfX)^{-1}
$
in the weak
or weak$^*$
topology of $\scrB(\bfE,\bfF)$.
Let $\phi\in\bfE$, $\phi\ne 0$,
and pick
$\calB\in\scrB_{00}(\bfF,\bfE)$
such that
$\calB\Psi=-\phi$,
with
$\Psi:=(A-z_0 I_\bfX)^{-1}_{\bfE,\bfF,\varOmega}\phi\in\bfF$.
Then
\[
(\hatA+\jmath\circ\imath\circ\calB-z_0 I_\bfF)\Psi
=
(\hatA-z_0 I_\bfF)(\hatA-z_0 I_\bfF)^{-1}_{\bfE,\bfF,\varOmega}
\phi
+\jmath\circ\imath\circ\calB\Psi
=0
\]
by Lemma~\ref{lemma-jk}~\itref{lemma-jk-jk1},
thus $\Psi\in\bfF$ is a virtual state
of $A+\imath\circ\calB\circ\jmath$ at $z_0$
relative to $(\bfE,\bfF,\varOmega)$.
\end{example}

\begin{theorem}
\label{theorem-lap}
Let $\bfE$ and $\bfF$ be Banach spaces with
continuous embeddings
$\bfE\mathop{\longhookrightarrow}\limits\sp{\imath}\bfX
\mathop{\longhookrightarrow}\limits\sp{\jmath}\bfF$
and let $A\in\scrC(\bfX)$ satisfy Assumption~\ref{ass-virtual}.
Let $\varOmega\subset\C\setminus\sigma(A)$
and let
$z_0\in\sigma\sb{\mathrm{ess}}(A)\cap\p\varOmega$.
Assume that
the mapping
$\jmath\circ(A-z I_\bfX)^{-1}\circ\imath:\,\bfE\to\bfF$,
$z\in\varOmega$,
has a limit as $z\to z_0$, $z\in\varOmega$,
in the weak$^*$ or weak or strong or uniform
operator topology of $\scrB(\bfE,\bfF)$,
denoted by
$(A-z_0 I_\bfX)^{-1}_{\bfE,\bfF,\varOmega}:\,\bfE\to\bfF$.
Assume that
$\calB:\,\bfF\to\bfE$ is $\hatA$-compact
(in the sense of Definition~\ref{def-compact}
below;
here $\hatA\in\scrC(\bfF)$ is the closure of
the extension of $A$ onto $\bfF$
from Assumption~\ref{ass-virtual})
and denote
\[
B=\imath\circ\calB\circ\jmath:\,\bfX\to\bfX,
\qquad
\hatB=\jmath\circ\imath\circ\calB:\,\bfF\to\bfF.
\]
Then:
\begin{enumerate}
\item
\label{theorem-lap-1}
For each $\lambda\in\C\setminus\{0\}$,
\[
\ker\big(\lambda I_\bfE-\calB(A-z_0 I_\bfX)^{-1}_{\bfE,\bfF,\varOmega}\big)
\cong
\Big\{
\Psi\in
\range\big((A-z_0 I_\bfX)^{-1}_{\bfE,\bfF,\varOmega}\big)
\sothat
(\hatA-\lambda^{-1}\hatB-z_0 I_\bfF)\Psi=0
\Big\},
\]
and
\[
\dim\ker\big(
\big(\lambda I_\bfF-(A-z_0 I_\bfX)^{-1}_{\bfE,\bfF,\varOmega}\calB\big)^k\big)
=
\dim\ker\big(\big(
\lambda I_\bfE-\calB(A-z_0 I_\bfX)^{-1}_{\bfE,\bfF,\varOmega}\big)^k\big),
\quad
\forall k\in\N
\]
with the sequence
$\ker\big(
\big(\lambda I_\bfE-\calB(A-z_0 I_\bfX)^{-1}_{\bfE,\bfF,\varOmega}\big)^k\big)
$
stationary for $k\geq k_0$ for some
$k_0\in\N$.
\item
\label{theorem-lap-2}
The following statements are equivalent:
\begin{enumerate}
\item
\label{theorem-lap-2a}
There is no nonzero solution to
$(\hatA+\hatB-z_0 I_\bfF)\Psi=0$,
$\Psi\in\range\big((A-z_0 I_\bfX)^{-1}_{\bfE,\bfF,\varOmega}\big)$;
\item
\label{theorem-lap-2b}
$-1\not\in\sigma
\big(\calB(A-z_0 I_\bfX)^{-1}_{\bfE,\bfF,\varOmega}\big)$;
\item
\label{theorem-lap-2c}
$-1\not\in\sigma\sb{\mathrm{p}}
\big((A-z_0 I_\bfX)^{-1}_{\bfE,\bfF,\varOmega}\calB\big)$;
\item
\label{theorem-lap-2d}
There is $\delta>0$ such that
$\varOmega\cap\mathbb{D}_\delta(z_0)\subset\C\setminus\sigma(A+B)$,
and there is a convergence
\[
\jmath\circ(A+B-z I_\bfX)^{-1}\circ\imath
\to (A+B-z_0 I_\bfX)^{-1}_{\bfE,\bfF,\varOmega}
:\;\bfE\to\bfF,
\qquad
z\to z_0,\quad
z\in\varOmega\cap\mathbb{D}_\delta(z_0),
\]
in the weak$^*$, weak, or strong or uniform, respectively,
operator topology
of $\scrB(\bfE,\bfF)$;
\item
\label{theorem-lap-2e}
There is $\delta>0$ such that
$\varOmega\cap\mathbb{D}_\delta(z_0)\subset\C\setminus\sigma(A+B)$,
and the mapping
$\jmath\circ(A+B-z I_\bfX)^{-1}\circ\imath:\;\bfE\to\bfF$
is uniformly bounded in $z\in\varOmega\cap\mathbb{D}_\delta(z_0)$.
\end{enumerate}
\end{enumerate}
\end{theorem}
By Definition~\ref{def-virtual},
when the equivalent statements in
Theorem~\ref{theorem-lap}~\itref{theorem-lap-2}
are satisfied,
$z_0$ is a regular point of the 
essential spectrum of $A+B$;
otherwise, it is a virtual level.

Although $\calB$ is $\hatA$-compact, $(A-z_0 I_\bfX)^{-1}_{\bfE,\bfF,\varOmega}\calB$
is not necessarily in $\scrB_0(\bfF)$ and a nonzero spectral value is not necessarily
from the discrete spectrum.
However, by Lemmata~\ref{Lem:AboundJ0}--\ref{lemma-uniform},
the statement~\itref{theorem-lap-2c}
of Theorem~\ref{theorem-lap}
is equivalent to 
$-1\not\in\sigma\big(\rho(z_0)\calB\circ\updelta_{\hatA}\big)$
since $\calB\circ\updelta_{\hatA}$ is compact.

\begin{remark}
\label{Rem:LimitCase}
Let us point out that in Example~\ref{example-shift-2},
the perturbation
$B=-T$,
where $T:\,\e_1\mapsto\e_0$, $\e_j\mapsto 0$ for $j\ne 1$,
does not satisfy
statement~\itref{theorem-lap-2a} of Theorem~\ref{theorem-lap}.
Indeed,
given any $z_0\in\p\mathbb{D}$,
one can see that
$\Psi\in\ell^\infty(\Z)$
with $\Psi_n=z_0^n$ for $n\ge 1$, $\Psi_n=0$ for $n\le 0$
satisfies
$(L-T-z_0 I)\Psi=0$,
and one can check that
$\Psi
=(L-z_0 I)^{-1}_{\ell^1,\ell^\infty,\mathbb{D}}\phi
\in\range\big((L-z_0 I)^{-1}_{\ell^1,\ell^\infty,\mathbb{D}}\big)$,
with $\phi_n=0$ for $n\ne 0$ and $\phi_0=1$
(cf. the construction from Example~\ref{Example:Canonical}).
Similarly,
the statement~\itref{theorem-lap-2d}
of Theorem~\ref{theorem-lap}
is
not satisfied:
since
$\sigma(L-T)=\overline{\mathbb{D}}$,
there would be no $\delta>0$ such that
$\mathbb{D}\cap\mathbb{D}_\delta(z_0)\subset\C\setminus\sigma(L-T)$.
\end{remark}

\begin{proof}
We again use the notations
$R(z):\,\bfX\to\bfX$
for $z\in\varOmega$
and $\calR(z):\,\bfE\to\bfF$
for $z\in\{z_0\}\sqcup\varOmega$
from \eqref{def-r-z0}.
Let us show that the space
$\ker\big(\lambda I_{\bfE}-\calB\calR(z_0)\big)$
is isomorphic to
the space of solutions to the problem
\begin{eqnarray}\label{sd}
(\hatA-\lambda^{-1}\hatB-z_0 I_\bfF)\Psi=0,
\qquad
\Psi\in\range(\calR(z_0)).
\end{eqnarray}
Let us assume that there is a solution to \eqref{sd}.
Substituting $\Psi=\calR(z_0)\phi$,
with $\phi\in\bfE$
(which is nonzero if so is $\Psi$),
we use Lemma~\ref{lemma-jk}~\itref{lemma-jk-jk1}
to arrive at
\begin{eqnarray}\label{psi-a-z-0}
\phi-\lambda^{-1}\calB\calR(z_0)\phi=0.
\end{eqnarray}
Going backwards,
we apply $\jmath\circ\imath$ to the above,
use Lemma~\ref{lemma-jk}~\itref{lemma-jk-jk1}
to rewrite \eqref{psi-a-z-0} as
\[
(\hatA-z_0 I_\bfF)\calR(z_0)\phi-\lambda^{-1}\hatB\calR(z_0)\phi=0,
\]
and denote
$\Psi=\calR(z_0)\phi\in\doma\subset\bfF$,
arriving at a solution to \eqref{sd}.
This produces the isomorphism
\[
\ker\big(\lambda I_{\bfE}-\calB\calR(z_0)\big)
\stackrel{\cong}{\longrightarrow}
\big\{\Psi\in
\range(\calR(z_0))
\sothat
(\hatA-\lambda^{-1}\hatB-z_0 I_\bfF)\Psi=0\big\},
\]
given by $\phi\mapsto\Psi=\calR(z_0)\phi$,
with the inverse given by $\Psi\mapsto\phi=\lambda^{-1}\calB\Psi$.

Due to compactness of the operator
$\calB\calR(z_0)$,
one has
$\sigma(\calB\calR(z_0))\setminus\{0\}\subset\sigma\sb{\mathrm{d}}(\calB\calR(z_0))$,
which implies
the finiteness of the dimension of
$\ker(\lambda I_\bfE-\calB\calR(z_0))$
for each $\lambda\ne 0$.

Let $\lambda\in\C\setminus\{0\}$,
$k\in\N$,
and assume that
$\Psi\in\ker\big((\lambda I_\bfF-\calR(z_0)\calB)^k\big)$.
One can see that $\Psi\in\range(\calR(z_0))
\subset\doma\subset\dom(\calB)$,
and $\phi:=\calB\Psi\in\bfE$ satisfies
$\phi\in\ker\big((\lambda I_\bfE-\calB\calR(z_0))^k\big)$.
The relation
$(\lambda I_\bfF-\calR(z_0)\calB)^k\Psi=0$
allows us to express
$\Psi$ in terms of $\calB\Psi$,
showing that
$\calB\Psi=0$ if and only if $\Psi=0$.
This shows that
\[
\dim\ker\big((\lambda I_\bfF-\calR(z_0)\calB)^k\big)
\le
\dim\ker\big((\lambda I_\bfE-\calB\calR(z_0))^k\big).
\]
Similarly, assume that
$\phi\in\ker\big((\lambda I_\bfE-\calB\calR(z_0))^k\big)$,
with some $k\in\N$.
Then $\Psi=\calR(z_0)\phi\in\bfF$ satisfies
$\Psi\in\ker\big((\lambda I_\bfF-\calR(z_0)\calB)^k\big)$.
Using the relation
$(\lambda I_\bfE-\calB\calR(z_0))^k\phi=0$
to express $\phi$ in terms of $\calR(z_0)\phi$,
one concludes that $\phi=0$ if and only if $\calR(z_0)\phi=0$
(this also follows from Lemma~\ref{lemma-jk}~\itref{lemma-jk-jk1}),
and hence
$
\dim\ker\big((\lambda I_\bfE-\calB\calR(z_0))^k\big)
\le
\dim\ker\big((\lambda I_\bfF-\calR(z_0)\calB)^k\big)
$,
and thus these dimensions are equal.

Due to the compactness of $\calB\calR(z_0)$, each
subspace
$\ker\big((\lambda I_\bfE-\calB\calR(z_0))^k\big)$, $k\in\N$,
is finite-dimensional,
and the increasing sequence of closed subspaces
$\ker\big((\lambda I_\bfE-\calB\calR(z_0))^k\big)$
becomes stationary for $k\ge k_0$, with some value $k_0\in\N$.
Indeed,
if this sequence did not become stationary,
then for each $k\in\N$
there would exist
$u_k\in\ker\big((\lambda I_\bfE-\calB\calR(z_0))^k\big)$,
$\norm{u_k}_{\bfE}=1$,
and by Riesz's lemma we could assume that
\begin{eqnarray}\label{nahevari}
\dist_\bfE
\big(u_k,
\ker\big((\lambda I_\bfE-\calB\calR(z_0))^{k-1}\big)\big)\geq \fra{1}{2}
\qquad
\forall k\ge 2.
\end{eqnarray}
Thus, for all $j,\,k\in\N$ such that $j<k$
one would have:
\begin{eqnarray*}
\norm{\calB\calR(z_0) u_k-\calB\calR(z_0) u_j}_\bfE
&=&
\norm{\lambda u_k
-(\lambda I_\bfE-\calB\calR(z_0))u_k-\calB\calR(z_0) u_j}_\bfE
\\
&\ge&
\dist_\bfE
\big(
\lambda u_k,\ker\big((I_\bfE-\calB\calR(z_0))^{k-1}\big)
\big)
\geq \fra{\abs{\lambda}}{2}>0,
\end{eqnarray*}
where we used the inclusions
$(\lambda I_\bfE-\calB\calR(z_0))u_k$, $\calB\calR(z_0) u_j\in
\ker\big((I_\bfE-\calB\calR(z_0))^{k-1}\big)$
and the inequality \eqref{nahevari};
therefore,
the sequence $\big(\calB\calR(z_0) u_j\in\bfE\big)_{j\in\N}$
could not contain a convergent subsequence,
in contradiction to $\calB\calR(z_0)$ being a compact operator.
This completes Part~\itref{theorem-lap-1}.

\smallskip

Let us prove Part~\itref{theorem-lap-2}.
The equivalence
\itref{theorem-lap-2a}
$\Leftrightarrow$
\itref{theorem-lap-2b}
$\Leftrightarrow$
\itref{theorem-lap-2c}
follows from Part~\itref{theorem-lap-1} with $\lambda=-1$.

We will write the rest of the proof
for the case when
$\jmath\circ(A-z I_\bfX)^{-1}\circ\imath\to
(A-z I_\bfX)^{-1}_{\bfE,\bfF,\varOmega}
$
in the weak operator topology;
the proof in the case of weak$^*$ operator topology
is exactly the same.

To prove
\itref{theorem-lap-2d}
$\Rightarrow$
\itref{theorem-lap-2b},
we assume that
there is $\delta>0$ such that
\begin{eqnarray}
\label{omega-in-c}
\varOmega\cap\mathbb{D}_\delta(z_0)\subset\C\setminus\sigma(A+B)
\end{eqnarray}
and
$\jmath\circ(A+B-z I_\bfX)^{-1}\circ\imath:\,\bfE\to\bfF$
has a limit
in the weak operator topology,
denoted by
$(A+B-z_0 I_\bfX)^{-1}_{\bfE,\bfF,\varOmega}:\,\bfE\to\bfF$.
Using the identity
\[
A+B-z I_\bfX
=
\big(I_\bfX+B R(z)\big)
(A-z I_\bfX)
:\;\bfX\to\bfX,
\qquad
z\in\varOmega\subset\C\setminus\sigma(A),
\]
and the inclusion
\eqref{omega-in-c},
we have:
\begin{eqnarray}\label{a-z}
\jmath\circ (A-z I_\bfX)^{-1}\circ\imath
=
\jmath\circ (A+B-z I_\bfX)^{-1}\circ\imath\circ
\big(I_\bfE+\calB\circ\jmath\circ R(z)\circ\imath\big)
:\,\bfE\to\bfF,
\end{eqnarray}
valid for $z\in\varOmega\cap\mathbb{D}_\delta(z_0)$.
We note that
$\calB\circ\updelta_{\hatA}\in\scrB_0\big(\doma,\bfE\big)$
and
$\hat\jmath\circ\rho(z)\circ\imath\in\scrB\big(\bfE,\doma\big)$
(cf. Lemma~\ref{Lem:AboundJ0}),
hence the last factor in the right-hand side
of \eqref{a-z}
converges strongly
by Lemma~\ref{lemma-br}.
Taking the limit of the right-hand side of \eqref{a-z}
$z\to z_0$,
$z\in\varOmega\cap\mathbb{D}_\delta(z_0)$,
in the weak operator topology,
we arrive at
\[
(A-z_0 I_\bfX)^{-1}_{\bfE,\bfF,\varOmega}
=
(A+B-z_0 I_\bfX)^{-1}_{\bfE,\bfF,\varOmega}
\big(I_\bfE+\calB\calR(z_0)\big)
:\;\bfE\to\bfF,
\quad
z\in\C\setminus\sigma(A).
\]
By Lemma~\ref{lemma-jk},
the left-hand side
has
zero kernel on $\bfE$;
we conclude that
$-1\not\in\sigma(\calB\calR(z_0))$,
completing the proof of
\itref{theorem-lap-2d}
$\Rightarrow$
\itref{theorem-lap-2b}.

Let us prove
\itref{theorem-lap-2b}
$\Rightarrow$
\itref{theorem-lap-2d}.
We start with the relation
\begin{equation}\label{b-a-z-1}
A+B-z I_\bfX=(A-z I_\bfX)\big(I_\bfX+(A-z I_\bfX)^{-1}B\big),
\qquad z\in\varOmega\cap\mathbb{D}_\delta(z_0).
\end{equation}
For the inclusion
$\varOmega\cap\mathbb{D}_\delta(z_0)\subset\C\setminus\sigma(A+B)$,
it suffices to show that the second factor
in the right-hand side of \eqref{b-a-z-1}
is invertible in $\dom(A)$.

By Lemma~\ref{lemma-kz}~\itref{lemma-kz-1},
there is $\delta>0$ and $\epsilon>0$ such that
$\sigma(\calB\calR(z))\cap\mathbb{D}_\epsilon(-1)
=\emptyset$
for $z\in\varOmega\cap\mathbb{D}_\delta(z_0)$.
There is the identity
\begin{eqnarray}\label{t-i-1}
\calB\circ\big(I_\bfF+\jmath\circ(A-z I_\bfX)^{-1}
\circ\imath\circ\calB\big)
=
\big(I_\bfE+\calB\circ\jmath\circ(A-z I_\bfX)^{-1}
\circ\imath\big)
\circ\calB:\;
\bfF\to\bfE,
\end{eqnarray}
valid for $z\in\varOmega\cap\mathbb{D}_\delta(z_0)$.
This identity shows that
$I_\bfF+\jmath\circ(A-z I_\bfX)^{-1}\circ\imath\circ\calB$
is injective.
Indeed,
if $\Psi\in
\ker\big(I_\bfF+\jmath\circ(A-z I_\bfX)^{-1}\circ\imath\circ\calB\big)$,
then one can see from \eqref{t-i-1} that
either
$\Psi\in\ker(\calB)$,
hence
$\big(I_\bfF+\jmath\circ(A-z I_\bfX)^{-1}\circ\imath\circ\calB\big)\Psi
=\Psi$, implying that $\Psi=0$,
or
else $\calB\Psi\ne 0$ leads to
$\calB\Psi\in\ker\big(I_\bfE+\calB\circ\jmath\circ(A-z I_\bfX)^{-1}
\circ\imath\big)=\{0\}$
(the last equality is due to
Lemma~\ref{lemma-kz}~\itref{lemma-kz-1} with $\lambda_0=-1$),
a contradiction.
Thus,
for each $z\in\varOmega\cap\mathbb{D}_\delta(z_0)$,
the operator
$I_{\doma}+\hat\jmath\circ\rho(z)\circ\imath\circ\calB
\circ\updelta_{\hatA}
:\,\doma\to\doma$,
being a Fredholm operator of index zero
(due to
$\calB\circ\updelta_{\hatA}$
being compact),
has a bounded inverse.
Then the identity
\[
\big(I_{\doma}+\hat\jmath\circ\rho(z)\circ\imath\circ\calB
\circ\updelta_{\hatA}
\big)
\circ\hat\jmath
=\hat\jmath\circ\big(I_{\dom(A)}+\rho(z)\circ B
\circ\updelta_A\big),
\qquad
z\in\varOmega\cap\mathbb{D}_\delta(z_0),
\]
shows that
$I_{\dom(A)}+\rho(z)\circ B\circ\updelta_A$
is injective.
Since
$\rho(z)\circ B\circ\updelta_A:\,\dom(A)\to\dom(A)$
is compact,
the operator
$I_{\dom(A)}+\rho(z)\circ B\circ\updelta_A$
is invertible
as an injective Fredholm operator of index zero.
Thus, indeed
the second factor in the identity
\eqref{b-a-z-1}
is invertible in $\dom(A)$,
showing that
$\varOmega\cap\mathbb{D}_\delta(z_0)\subset\C\setminus\sigma(A+B)$.

For $z\in\varOmega\cap\mathbb{D}_\delta(z_0)$,
we rewrite \eqref{a-z} as
\begin{eqnarray}\label{a-z-1}
\jmath\circ (A+B-z I_\bfX)^{-1}\circ\imath
=
\calR(z)
\circ\big(I_\bfE+\calB\calR(z)\big)^{-1}
:\,\bfE\to\bfF.
\end{eqnarray}

\begin{lemma}
\label{lemma-br-uniform}
Assume that the operator family
$\big\{
\jmath\circ(A-z I_\bfX)^{-1}\circ\imath\big\}_{z\in\varOmega}$
converges as $z\to z_0$
in the uniform operator topology of $\scrB(\bfE,\bfF)$.
Then the operator family
$\big\{
\calB\circ\jmath\circ(A-z I_\bfX)^{-1}\circ\imath
\big\}_{z\in\varOmega}$
converges as $z\to z_0$
in the
uniform operator topology of $\scrB(\bfE)$
to the operator
$\calB(A-z_0 I_\bfX)^{-1}_{\bfE,\bfF,\varOmega}\in\scrB_0(\bfE)$.
\end{lemma}

\begin{proof}
Since $\jmath\circ(A-z I_\bfX)^{-1}\circ\imath$
converges as $z\to z_0$, $z\in\varOmega$,
in the
uniform operator topology
of $\scrB(\bfE,\bfF)$,
by Lemma~\ref{lemma-uniform}~\itref{lemma-uniform-3}
this convergence also holds
in the
uniform operator topology of
$\scrB\big(\bfE,\doma\big)$.
Together with $\hatA$-boundedness of $\calB$,
this implies  that
$\calB\circ\jmath\circ(A-z I_\bfX)^{-1}\circ\imath$
converges as $z\to z_0$, $z\in\varOmega$,
in the
uniform operator topology
of $\scrB(\bfE)$.
By Lemma~\ref{lemma-uniform}~\itref{lemma-uniform-4},
there is $\delta>0$ such that the family
$\big\{\calB\circ\jmath\circ(A-z I_\bfX)^{-1}\circ\imath
\in\scrB_0(\bfE)
\big\}_{z\in\varOmega\cap\mathbb{D}_\delta(z_0)}$
is collectively compact.
Therefore, by Lemma~\ref{lemma-br},
$\calB(A-z_0 I_\bfX)^{-1}_{\bfE,\bfF,\varOmega}$ is compact.
\end{proof}

We recall that, due to our assumptions,
$\calR(z)=
\jmath\circ(A-z I_\bfX)^{-1}\circ\imath:\,\bfE\to\bfF$
converges in the
weak$^*$ or weak or strong or uniform
operator topology of $\scrB(\bfE,\bfF)$
as $z\to z_0$, $z\in\varOmega$.
By Lemma~\ref{lemma-uniform}~\itref{lemma-uniform-4}
and
Lemma~\ref{lemma-br}, the operator family
$\big\{\calB\calR(z)\big\}_{z\in\varOmega}$
converges
as $z\to z_0$
in the strong operator topology of $\scrB(\bfE)$,
or -- by Lemma~\ref{lemma-br-uniform} --
in the uniform operator topology of $\scrB(\bfE)$
if there is likewise convergence of $\calR(z)$.
Then, since
$-1\not\in\sigma\big(\calB(A-z_0 I_\bfX)^{-1}_{\bfE,\bfF,\varOmega}\big)$
by assumption of Part~\itref{theorem-lap-2b},
Lemma~\ref{lemma-kz}~\itref{lemma-kz-1}
yields the convergence of
$(I_\bfE+\calB\circ\calR(z))^{-1}\in\scrB(\bfE)$
as $z\to z_0$, $z\in\varOmega\cap\mathbb{D}_\delta(z_0)$,
in the strong operator topology of $\scrB(\bfE)$
or in the uniform operator topology
in the case of likewise convergence of $\calR(z)$.
It follows that
as $z\to z_0$, $z\in\varOmega\cap\mathbb{D}_\delta(z_0)$,
the right-hand side of
\eqref{a-z-1}
converges
in the weak$^*$ or weak or strong or uniform,
respectively,
operator topology of $\scrB(\bfE,\bfF)$,
to
\begin{eqnarray}\label{l-u-c-zero-new}
(A+B-z_0 I_\bfX)^{-1}_{\bfE,\bfF,\varOmega}
=
\calR(z_0)
\big(I_\bfE+\calB\calR(z_0)\big)^{-1}
:\;\bfE\to\bfF.
\end{eqnarray}
This completes the proof of the implication
\itref{theorem-lap-2b}
$\Rightarrow$
\itref{theorem-lap-2d}.

The implication
\itref{theorem-lap-2d}
$\Rightarrow$ \itref{theorem-lap-2e}
follows from Lemma~\ref{lemma-uniform}.

Let us prove the implication
\itref{theorem-lap-2e}
$\Rightarrow$ \itref{theorem-lap-2b}.
We assume that there exist $\delta>0$ and $C>0$
such that
$\varOmega\cap\mathbb{D}_1(z_0)
\subset\C\setminus\sigma(A+B)$
and
\[
\norm{\jmath\circ (A+B-z I_\bfX)^{-1}\circ\imath}_{\bfE\shortto\bfF}\leq C,\qquad z\in\varOmega\cap\mathbb{D}_\delta(z_0).
\]
Let $\phi\in\bfE$ be such that
$\calB(A-z_0 I_\bfX)^{-1}_{\bfE,\bfF,\varOmega}\phi=-\phi$;
we need to prove that $\phi=0$.
Pick a sequence $(z_j\in\varOmega\cap\mathbb{D}_\delta(z_0))_{j\in\N}$ such that $z_j\to z_0$.
Using the relation
\[
\jmath\circ (A-z_j I_\bfX)^{-1}\circ\imath
=
\jmath\circ (A+B-z_j I_\bfX)^{-1}\circ\imath\circ
\big(I_\bfE+\calB\circ\jmath\circ(A-z_j I_\bfX)^{-1}\circ\imath\big),
\qquad j\in\N,
\]
we infer:
\begin{eqnarray}\label{twi-2}
\norm{\jmath\circ (A-z_j I_\bfX)^{-1}\imath(\phi)}_{\bfF}
\leq
C\norm{
\big(I_\bfE+\calB\circ\jmath\circ(A-z_j I_\bfX)^{-1}\circ\imath\big)
\phi}_\bfE\to 0
\quad \mbox{as}\quad j\to\infty;
\end{eqnarray}
we took into account that
$(I_\bfE+\calB\circ\jmath\circ(A-z_j I_\bfX)^{-1}\circ\imath)\phi
\to
(I_\bfE+\calB(A-z_0 I_\bfX)^{-1}_{\bfE,\bfF,\varOmega})\phi=0$
in $\bfE$,
as it follows from
Lemma~\ref{lemma-br}.
Since $\jmath\circ (A-z_j I_\bfX)^{-1}\imath(\phi)
\to
(A-z_0 I_\bfX)^{-1}_{\bfE,\bfF,\varOmega}\phi$ weakly in $\bfE$,
we conclude from \eqref{twi-2}
that,
as the matter of fact, the convergence is strong in $\bfE$ and
$(A-z_0 I_\bfX)^{-1}_{\bfE,\bfF,\varOmega}\phi=0$,
hence
$\phi=-\calB(A-z_0 I_\bfX)^{-1}_{\bfE,\bfF,\varOmega}\phi=0$.
This completes the proof of Theorem~\ref{theorem-lap}.
\end{proof}

\begin{corollary}\label{Cor:DifferentConvergence2}
Let $\bfE$ and $\bfF$ be Banach spaces with
continuous embeddings
$\bfE\mathop{\longhookrightarrow}\limits\sp{\imath}\bfX
\mathop{\longhookrightarrow}\limits\sp{\jmath}\bfF$,
let $A\in\scrC(\bfX)$ satisfy Assumption~\ref{ass-virtual},
let $\varOmega\subset\C\setminus\sigma(A)$,
and let
$z_0\in\sigma\sb{\mathrm{ess}}(A)\cap z_0\in\p\varOmega$.
Let
$\calB_1,\,\calB_2\in\scrB_{00}(\bfF,\bfE)$
and denote $B_\nu=\imath\circ\calB_\nu\circ\jmath:\,\bfX\to\bfX$,
$\nu=1,\,2$.
If
$\jmath\circ(A+B_1-z I_\bfX)^{-1}\circ\imath$ has a limit
as $z\to z_0$, $z\in\varOmega$,
in the weak (or strong or uniform)
operator topology of $\scrB(\bfE,\bfF)$
and
$\jmath\circ(A+B_2-z I_\bfX)^{-1}\circ\imath$ has a limit
as $z\to z_0$, $z\in\varOmega$,
in the weak or weak$^*$
operator topology of $\scrB(\bfE,\bfF)$,
then
$\jmath\circ(A+B_2-z I_\bfX)^{-1}\circ\imath$
also has a limit
as $z\to z_0$, $z\in\varOmega$,
in the weak (or strong or uniform, respectively)
operator topology of $\scrB(\bfE,\bfF)$.
\end{corollary}

\subsection{Set of relatively compact regularizations}
\label{sect-relatively-compact}

Let us show that virtual levels
could be characterized as points $z_0$
such that the limit in \eqref{lim}
in the weak or weak$^*$ operator topology
exists for $\hatA$-compact perturbations.

\begin{definition}[The set of
relatively compact regularizing operators]
\label{definition-q}
Let $\bfE$ and $\bfF$ be Banach spaces with
continuous embeddings
$\bfE\mathop{\longhookrightarrow}\limits\sp{\imath}\bfX
\mathop{\longhookrightarrow}\limits\sp{\jmath}\bfF$
and let $A\in\scrC(\bfX)$ satisfy Assumption~\ref{ass-virtual}.
$\scrQ_{\bfE,\bfF,\varOmega}(A-z_0 I_\bfX)$
is the set of $\hatA$-compact operators
$\calB:\,\bfF\to\bfE$
(in the sense of Definition~\ref{def-compact})
such that
\begin{enumerate}
\item $\varOmega\subset\C\setminus\sigma\sb{\mathrm{ess}}(A+B)$,
where $B=\imath\circ\calB\circ\jmath:\,\bfX\to\bfX$;
\item the limit of
$\jmath\circ(A+B-z I_\bfX)^{-1}\circ\imath$,
as $z\to z_0$, $z\in\varOmega$, exists
in the weak or weak$^*$ operator topology of $\scrB(\bfE,\bfF)$.
\end{enumerate}
\end{definition}

\begin{lemma}\label{lemma-b}
Let $\calB\in\scrQ_{\bfE,\bfF,\varOmega}(A-z_0 I_\bfX)$
and let
$P_1\in\scrB_{00}(\bfE)$
be the Riesz projection
onto the generalized eigenspace
corresponding to eigenvalue $\lambda=1$
of the operator
\begin{eqnarray}\label{def-k}
K=\calB(A+B-z_0 I_\bfX)^{-1}_{\bfE,\bfF,\varOmega}\in\scrB_{0}(\bfE),
\qquad
B=\imath\circ\calB\circ\jmath:\,\bfX\to\bfX.
\end{eqnarray}
Then
\[
P_1\circ\calB\in\scrQ_{\bfE,\bfF,\varOmega}(A-z_0 I_\bfX).
\]
\end{lemma}

\begin{proof}
Let $P_1\in\scrB_{00}(\bfE)$ be the Riesz projection
onto
$\frakL(I_\bfE-K)$, the generalized eigenspace
of eigenvalue $\lambda=1$ of the operator $K$.
Then $(I_\bfE-P_1)K$ does not have eigenvalue $\lambda=1$.
By Lemma~\ref{lemma-jk}~\itref{lemma-jk-jk1},
there is the identity
\begin{eqnarray}\label{jk}
(\hatA+\jmath\circ\imath\circ\calB-z_0 I_\bfF)
(A+B-z_0 I_\bfX)^{-1}_{\bfE,\bfF,\varOmega}
=\jmath\circ\imath:\;\bfE\to\bfF;
\end{eqnarray}
hence, substituting
$\calB=P_1\circ\calB+(I_\bfE-P_1)\circ\calB$,
we arrive at
\begin{eqnarray}\label{sdf-bis}
(\hatA+\jmath\circ\imath\circ P_1\circ\calB-z_0 I_\bfF)
(A+B-z_0 I_\bfX)^{-1}_{\bfE,\bfF,\varOmega}=\jmath\circ\imath\circ
\big(I_\bfE-(I_\bfE- P_1)K\big).
\end{eqnarray}
Since the right-hand side of \eqref{sdf-bis}
is injective,
we conclude that
$(\hatA+\jmath\circ\imath\circ P_1\circ\calB-z_0 I_\bfF)$
is injective
as a map
from
$\range\big((A+B-z_0 I_\bfX)^{-1}_{\bfE,\bfF,\varOmega}\big)$ to $\bfF$.
Thus, by Theorem~\ref{theorem-lap}~\itref{theorem-lap-2}
(equivalence of~\itref{theorem-lap-2a}
and~\itref{theorem-lap-2d}),
there is the inclusion
$
P_1\circ\calB\in\scrQ_{\bfE,\bfF,\varOmega}(A-z_0 I_\bfX)
$.
\end{proof}

\begin{lemma}\label{lemma-q}
Assume that
$\calB\in\scrQ_{\bfE,\bfF,\varOmega}(A-z_0 I_\bfX)$
and
$\rank\calB<\infty$.
For any projection
$Q\in\scrB_{00}(\bfF)$
onto
$(A+B-z_0 I)^{-1}_{\bfE,\bfF,\varOmega}\ran(\calB)\subset\bfF$, one has:
\[
\calB\circ Q\in\scrQ_{\bfE,\bfF,\varOmega}(A-z_0 I_\bfX)
\cap\scrB_{00}(\bfF,\bfE).
\]
\end{lemma}

\begin{proof}
We denote
$B:=\imath\circ\calB\circ\jmath$,
$\hatB:=\jmath\circ\imath\circ\calB$,
and note that, by Lemma~\ref{lemma-jk}~\itref{lemma-jk-jk1}
(applied with $A+B$ in place of $A$),
there is the inclusion
\begin{eqnarray}\label{a-b-a}
(A+B-z_0 I_\bfX)^{-1}_{\bfE,\bfF,\varOmega}\ran(\calB)
\subset\dom(\hatA+\hatB)
=\doma;
\end{eqnarray}
let $Q\in\scrB_{00}(\bfF)$ be any projection onto
this finite-dimensional space.
Substituting
$\calB=\calB Q+\calB(I_\bfF-Q)$
into \eqref{jk},
we arrive at
\begin{eqnarray}\label{sdf-bis-2}
&&
(\hatA+\jmath\circ\imath\circ\calB\circ Q-z_0 I_\bfF)
(A+B-z_0 I_\bfX)^{-1}_{\bfE,\bfF,\varOmega}
\nonumber
\\
&&
= \jmath\circ\imath\circ
\big(I_\bfE-\calB(I_\bfF-Q)(A+B-z_0 I_{\bfX})^{-1}_{\bfE,\bfF,\varOmega}\big).
\end{eqnarray}
If $\phi\in\bfE$ is in the kernel of the right-hand side
of \eqref{sdf-bis-2},
thus satisfying the relation
\[
\phi=\calB(I_\bfF-Q)(A+B-z_0 I_{\bfX})^{-1}_{\bfE,\bfF,\varOmega}\phi,
\]
then we see that $\phi\in\ran(\calB)$,
hence
$(I_\bfF-Q)(A+B-z_0 I_{\bfX})^{-1}_{\bfE,\bfF,\varOmega}
\phi=0$
(due to the definition of $Q$), implying that $\phi=0$.
Thus, the right-hand side of \eqref{sdf-bis-2}
is injective,
and
we conclude that
$(\hatA+\jmath\circ\imath\circ\calB\circ Q-z_0 I_\bfF)$
is injective
as a map
from
$\range\big((A+B-z_0 I_\bfX)^{-1}_{\bfE,\bfF,\varOmega}\big)$ to $\bfF$.
By Theorem~\ref{theorem-lap}~\itref{theorem-lap-2}
(equivalence of~\itref{theorem-lap-2a}
and~\itref{theorem-lap-2d}),
there is the inclusion
$\calB\circ Q
\in\scrQ_{\bfE,\bfF,\varOmega}(A-z_0 I_\bfX)$.
Finally, we notice that
$\Psi\in\ran(Q)$ implies that
there is $\Psi_0\in\bfF$ such that
\[
\Psi
=(A+B-z_0 I_\bfX)^{-1}_{\bfE,\bfF,\varOmega}\calB\Psi_0,
\]
and by \eqref{a-b-a} one has
$\ran(Q)\subset\doma\subset\dom(\calB)$.
This implies that
$\calB\circ Q\in\scrB_{00}(\bfF,\bfE)$.
\end{proof}

The set $\scrQ_{\bfE,\bfF,\varOmega}(A-z_0 I_\bfX)$
is dense, in the following sense:
\begin{lemma}\label{Lem:Density}
Let $\bfE$ and $\bfF$ be Banach spaces with
continuous embeddings
$\bfE\mathop{\longhookrightarrow}\limits\sp{\imath}\bfX
\mathop{\longhookrightarrow}\limits\sp{\jmath}\bfF$
and let $A\in\scrC(\bfX)$ satisfy Assumption~\ref{ass-virtual}.
Let $\varOmega\subset\C\setminus\sigma(A)$
and assume that
$z_0\in\p\varOmega\cap\sigma\sb{\mathrm{ess}}(A)$
is a virtual level of rank $r\in\N$
relative to $(\bfE,\bfF,\varOmega)$.
If $\calB\in\scrQ_{\bfE,\bfF,\varOmega}(A-z_0 I_\bfX)$,
then there is $\epsilon>0$ such that
for all $\zeta\in\mathbb{D}_\epsilon\setminus\{0\}$
the point $z_0$ is a regular point of the essential spectrum
of $A+\zeta B$,
with $B=\imath\circ\calB\circ\jmath$,
relative to $(\bfE,\bfF,\varOmega)$.
\end{lemma}
\begin{proof}
Let $\calB\in\scrQ_{\bfE,\bfF,\varOmega}(A-z_0 I_\bfX)$.
Denote
$B=\imath\circ\calB\circ\jmath:\,\bfX\to\bfX$.
By Theorem~\ref{theorem-lap}~\itref{theorem-lap-2b}
(with $A+B$ in place of $A$ and $-B$ in place of $B$),
$
1\in\sigma\big(\calB(A+B-z_0 I_\bfX)^{-1}_{\bfE,\bfF,\varOmega}\big)
$.
Since $\calB(A+B-z_0 I_\bfX)^{-1}_{\bfE,\bfF,\varOmega}$ is compact,
there is $\delta\in(0,1)$ such that
\begin{eqnarray}\label{one-in-1}
\mathbb{D}_\delta(1)\cap
\sigma\big(\calB(A+B-z_0 I_\bfX)^{-1}_{\bfE,\bfF,\varOmega}\big)
=\{1\}.
\end{eqnarray}
Assume that $\zeta\in\mathbb{D}\setminus\{0\}$
is such that
$z_0$ is a virtual level of $A+\zeta B$
(of some finite rank $s\in\N$)
relative to $(\bfE,\bfF,\varOmega)$
(if there is no such $\zeta$,
then we arrive at the statement of the lemma
with $\epsilon=1$).
By Theorem~\ref{theorem-lap}~\itref{theorem-lap-2a}
(with $A+B$ in place of $A$
and
$(\zeta-1) B$ in place of $B$),
there is a nontrivial solution to
\begin{eqnarray}\label{into}
(\hatA+\zeta \hatB-z_0 I_\bfX)\Psi=0,
\qquad
\Psi\in\ran\big((A+B-z_0 I_\bfX)^{-1}_{\bfE,\bfF,\varOmega}\big),
\end{eqnarray}
where
$\hatB=\jmath\circ\imath\circ\calB:\,
\bfF\to\bfF$.
Let $\phi\in\bfE\setminus\{0\}$ be such that
$\Psi=(A+B-z_0 I_\bfX)^{-1}_{\bfE,\bfF,\varOmega}\phi$;
substituting it into \eqref{into}
and applying Lemma~\ref{lemma-jk}~\itref{lemma-jk-jk1},
we compute:
\[
0=(\hatA+\zeta \hatB-z_0 I_\bfX)
(A+B-z_0 I_\bfX)^{-1}_{\bfE,\bfF,\varOmega}\phi
=\big(I_\bfE
-(1-\zeta)\calB(A+B-z_0 I_\bfX)^{-1}_{\bfE,\bfF,\varOmega}
\big)\phi.
\]
Due to \eqref{one-in-1},
$(1-\zeta)^{-1}\not\in\mathbb{D}_\delta(1)$.
It follows that if $z_0$ is a virtual level
of $A+\zeta B$ relative to $(\bfE,\bfF,\varOmega)$
and $\zeta\ne 0$,
then necessarily
$\abs{\zeta}\ge\epsilon
:=1-(1+\delta)^{-1}>0$.
\end{proof}

The set $\scrQ_{\bfE,\bfF,\varOmega}(A-z_0 I_\bfX)$
is open, in the following sense:

\begin{lemma}\label{lemma-relatively-open}
Let $\bfE$ and $\bfF$ be Banach spaces with
continuous embeddings
$\bfE\mathop{\longhookrightarrow}\limits\sp{\imath}\bfX
\mathop{\longhookrightarrow}\limits\sp{\jmath}\bfF$
and let $A\in\scrC(\bfX)$ satisfy Assumption~\ref{ass-virtual}.
Assume that $\calB:\,\bfF\to\bfE$ is $\hatA$-compact.

\begin{enumerate}
\item
\label{lemma-relatively-open-1}
Let $z\in\C\setminus\sigma(A)$
and assume that
\begin{eqnarray}
\label{smaller-than-one-1}
r(\calB\circ\jmath\circ(A-z I_\bfX)^{-1}\circ\imath)<1,
\end{eqnarray}
where $r(T)$ is the spectral radius
of $T\in\scrB(\bfE)$. Then
\begin{equation}\label{implies}
z\in\C\setminus\sigma(A+B),
\qquad
\mbox{with $B:=\imath\circ\calB\circ\jmath:\,\bfX\to\bfX$.}
\end{equation}
\item
\label{lemma-relatively-open-2}
Let $\varOmega\subset\C\setminus\sigma(A)$,
and assume that
$z_0\in\sigma\sb{\mathrm{ess}}(A)\cap\p\varOmega$
is a regular point of $\sigma\sb{\mathrm{ess}}(A)$
relative to $(\bfE,\bfF,\varOmega)$, so that
there exists a limit of
\begin{eqnarray}\label{res}
\jmath\circ(A-z I_\bfX)^{-1}\circ\imath:\;\bfE\to\bfF,
\qquad
z\in\varOmega,
\end{eqnarray}
as $z\to z_0$,
denoted by
$(A-z_0 I_\bfX)^{-1}_{\bfE,\bfF,\varOmega}$,
in the weak$^*$ or weak or strong or uniform
operator topology of $\scrB(\bfE,\bfF)$.

Assume that
\begin{equation}
\label{smaller-than-one-2}
r(\calB(A-z_0 I_\bfX)^{-1}_{\bfE,\bfF,\varOmega})<1.
\end{equation}
Then
there is $\delta>0$ such that
\begin{equation}\label{implies-2}
\varOmega\cap\mathbb{D}_\delta(z_0)\subset\C\setminus\sigma(A+B)
\end{equation}
and there exists a limit of
\[
\jmath\circ(A+B-z I_\bfX)^{-1}\circ\imath:\;\bfE\to\bfF
\]
as $z\to z_0,\,z\in\varOmega\cap\mathbb{D}_\delta(z_0)$,
in the weak$^*$ or weak or strong or uniform
operator topology,
respectively,
which is equal to
\[
(A-z_0 I_\bfX)^{-1}_{\bfE,\bfF,\varOmega}
\big(I_{\bfE}+\calB(A-z_0I_\bfX)^{-1}_{\bfE,\bfF,\varOmega}\big)^{-1}:
\;\bfE\to\bfF.
\]
\end{enumerate}
\end{lemma}

Let us mention
that conditions of the form
\eqref{smaller-than-one-1},
\eqref{smaller-than-one-2}
could be verified in particular situations;
see e.g. \cite[Lemma 12]{erdogan2008strichartz}.

\begin{remark}\label{Rem:OnShiftPerturbations}
The relation \eqref{implies-2}
can be interpreted as a certain stability of the Browder
spectrum $\sigma\sb{\mathrm{ess},5}$
away from virtual levels:
if the operator satisfies
the limiting absorption principle at
some point $z_0$,
then an additional component
of $\sigma\sb{\mathrm{ess},5}$
bordering $z_0$
cannot be created
by relatively compact perturbations which are arbitrarily small
(in the operator norms corresponding to the LAP).
We recall here the standard example
(see Remark~\ref{Rem:LimitCase})
of the left shift $L$ on $\ell^2(\Z)$,
when
$\sigma(L)=\p\mathbb{D}$,
while
$\sigma(L-T)=\overline{\mathbb{D}}$
with $T:\,\e_1\mapsto\e_0$,
$\e_j\mapsto 0$ for $j\ne 1$.
In this example,
the additional component
$\mathbb{D}$
of the Browder spectrum
$\sigma_{\mathrm{ess},5}$
is created by the perturbation $-T$
which satisfies
$\lim\sb{|z|\to 1,\,|z|<1}\norm{T(L-z I)^{-1}}_{\ell^1\shortto\ell^1}=1$
(cf. Example~\ref{example-shift-2}),
hence
\eqref{smaller-than-one-2} is not satisfied.
\end{remark}

\begin{proof}
For the inclusion \eqref{implies},
it suffices to show that
\[
I_{\dom(A)}+\rho(z)\circ B\circ\updelta_A,
\]
which is the second factor in the right-hand side of
\eqref{b-a-z-1}
considered as an operator in $\dom(A)$
(with $\rho(z)\in\scrB(\bfX,\dom(A))$ defined
in Lemma~\ref{lemma-uniform}~\itref{lemma-uniform-2}),
is invertible,
with bounded inverse.

By Gelfand's formula,
\[
r\big(\calB\circ\jmath\circ(A-z I_\bfX)^{-1}\circ\imath\big)
=\lim\sb{n\to\infty}
\norm{(\calB
\circ\jmath\circ(A-z I_\bfX)^{-1}\circ\imath)^n
}_{\scrB(\bfE)}^{1/n};
\]
due to the assumption \eqref{smaller-than-one-1},
there is $N\in\N$ such that
$
\norm{(\calB
\circ\jmath\circ(A-z I_\bfX)^{-1}\circ\imath)^N}_{\scrB(\bfE)}<1$.
For any $j\in\N$,
there is the identity
\begin{eqnarray}\label{to-n}
&&
\sum\sb{n=0}^j
\big(-\hat\jmath\circ\rho(z)
\circ\imath\circ\calB\circ\updelta_{\hatA}\big)^n
\nonumber
\\
&&
=I_{\doma}-
\hat\jmath\circ\rho(z)\circ\imath\circ
\Big(\sum\sb{n=1}^{j}
\big(-\calB\circ
\jmath\circ(A-z I_\bfX)^{-1}
\circ\imath\big)^{n-1}
\Big)\circ\calB\circ\updelta_{\hatA}.
\end{eqnarray}
Writing the series
\begin{eqnarray}\label{first-N-c-0-new}
&&\sum\sb{n\in\N_0}
\big(-\calB\circ\jmath\circ(A-z I_\bfX)^{-1}\circ\imath\big)^n
\end{eqnarray}
as $\sum_{k=0}^{N-1}
\big(-\calB\circ\jmath\circ(A-z I_\bfX)^{-1}\circ\imath\big)^k
\circ
\sum_{m\in\N_0}
\big(-\calB\circ\jmath\circ(A-z I_\bfX)^{-1}\circ\imath\big)^{N m}
$,
we see that it is absolutely convergent.
Passing to the limit $j\to\infty$ in \eqref{to-n}
and using the absolute convergence in \eqref{first-N-c-0-new}
shows that the operator
$I_{\doma}+\hat\jmath\circ\rho(z)\circ\imath\circ\calB\circ\updelta_{\hatA}$
has a bounded inverse in $\doma$,
given by
\[
\big(I_{\doma}+\hat\jmath\circ\rho(z)\circ\imath\circ\calB\circ\updelta_{\hatA}\big)^{-1}
=
\sum\sb{n\in\N_0}
\big(-\hat\jmath\circ\rho(z)
\circ\imath\circ\calB\circ\updelta_{\hatA}\big)^n,
\]
for any $z\in\varOmega\cap\mathbb{D}_\delta(z_0)$.
Now the identity
\[
\big(I_{\doma}
+\hat\jmath\circ\rho(z)\circ\imath\circ\calB\circ\updelta_{\hatA}\big)
\circ\hat\jmath=\hat\jmath\circ\big(I_{\dom(A)}+\rho(z)\circ B
\circ\updelta_A\big)
\]
shows that
$
I_{\dom(A)}+\rho(z)\circ B\circ\updelta_A
$
is injective.
Since
$B\circ\updelta_A:\,\dom(A)\to\bfX$
is compact,
the operator
$
I_{\dom(A)}+\rho(z)\circ B\circ\updelta_A
$
is invertible
in $\dom(A)$ as an injective Fredholm operator of index zero.
This implies the inclusion \eqref{implies},
concluding
Part~\itref{lemma-relatively-open-1}.

Let us prove
Part~\itref{lemma-relatively-open-2}.
Applying
Lemma~\ref{lemma-kz}~\itref{lemma-kz-2}
to the assumption \eqref{smaller-than-one-2},
we conclude that
there is $\delta>0$ such that
$r(\calB
\circ\jmath\circ(A-z I_\bfX)^{-1}\circ\imath)<1$
for all $z\in\varOmega\cap\mathbb{D}_\delta(z_0)$.
Therefore, by Part~\itref{lemma-relatively-open-1},
$\varOmega\cap\mathbb{D}_\delta(z_0)\subset\C\setminus\sigma(A+B)$,
and for each $z\in\varOmega\cap\mathbb{D}_\delta(z_0)$
we can write:
\begin{eqnarray*}
\jmath\circ(A+B-z I_\bfX)^{-1}\circ\imath
&=&
\jmath\circ(A-z I_\bfX)^{-1}\circ\imath\circ
\big(I_{\bfE}+\calB\circ\imath(A-z I_\bfX)^{-1}
\circ\imath\big).
\end{eqnarray*}
As $z\to z_0$, $z\in\varOmega$,
by our assumptions, there is a convergence
$
\jmath\circ(A-z I_\bfX)^{-1}\circ\imath
\to(A-z_0 I_\bfX)^{-1}_{\bfE,\bfF,\varOmega}
$
in the weak$^*$ or weak or strong or uniform
operator topology
of $\scrB(\bfE,\bfF)$,
while
\[
\big(I_{\bfE}+\calB\circ\jmath\circ(A-z I_\bfX)^{-1}\circ\imath
\big)^{-1}
\to
\big(I_{\bfE}+\calB(A-z_0 I_\bfX)^{-1}_{\bfE,\bfF,\varOmega}
\big)^{-1}
\]
in the strong operator topology
of $\scrB(\bfE)$
by Lemma~\ref{lemma-kz}~\itref{lemma-kz-1}
(in the uniform operator topology of $\scrB(\bfE)$
if the convergence in \eqref{res} 
is in the uniform operator topology of $\scrB(\bfE,\bfF)$;
see Lemma~\ref{lemma-br-uniform}).
Therefore, the composition
\[
\jmath\circ(A-z I_\bfX)^{-1}\circ\imath
\circ
\big(I_{\bfE}+\calB\circ\jmath\circ(A-z I_\bfX)^{-1}\circ\imath
\big)^{-1}
\]
converges
in the weak$^*$ or weak or strong or uniform,
respectively,
operator topology of $\scrB(\bfE,\bfF)$
to
$
(A-z_0 I_\bfX)^{-1}_{\bfE,\bfF,\varOmega}
\big(I_{\bfE}+\calB(A-z_0I_\bfX)^{-1}_{\bfE,\bfF,\varOmega}\big)^{-1}
$,
concluding the proof.
\end{proof}

\subsection{Space of virtual states}

The following theorem
introduces the space of virtual states.

\begin{theorem}[The space of virtual states]
\label{theorem-m}
Let $\bfE$ and $\bfF$ be Banach spaces with
continuous embedding
$\bfE\mathop{\longhookrightarrow}\limits\sp{\imath}\bfX
\mathop{\longhookrightarrow}\limits\sp{\jmath}\bfF$
and let $A\in\scrC(\bfX)$
satisfy Assumption~\ref{ass-virtual}.
Let $\varOmega\subset\C\setminus\sigma(A)$
and let
$z_0\in\sigma\sb{\mathrm{ess}}(A)\cap\p\varOmega$.
Assume that the set
$\scrQ_{\bfE,\bfF,\varOmega}(A-z_0 I_\bfX)$
\textup(see Definition~\ref{definition-q}\textup)
is nonempty.
For $\calB\in\scrQ_{\bfE,\bfF,\varOmega}(A-z_0 I_\bfX)$,
define
\emph{the space of virtual states} by
\begin{eqnarray}\label{def-m-frak}
\frakM_{\bfE,\bfF,\varOmega}(A-z_0 I_\bfX)
=\big\{
\Psi\in
\range\big((A+B-z_0 I_\bfX)^{-1}_{\bfE,\bfF,\varOmega}\big)
\sothat
(\hatA-z_0 I_\bfF)\Psi=0
\big\}
\subset\bfF,
\end{eqnarray}
where
$B=\imath\circ\calB\circ\jmath:\,\bfX\to\bfX$
and
$(A+B-z_0 I_\bfX)^{-1}_{\bfE,\bfF,\varOmega}:\,\bfE\to\bfF$.
Then:
\begin{enumerate}
\item
\label{theorem-m-1}
$\range\big((A+B-z_0 I_\bfX)^{-1}_{\bfE,\bfF,\varOmega}\big)$
and
$\frakM_{\bfE,\bfF,\varOmega}(A-z_0 I_\bfX)$
do not depend on the choice of
$\calB\in\scrQ_{\bfE,\bfF,\varOmega}(A-z_0 I_\bfX)$;
\item
\label{theorem-m-2}
$\jmath\big(\imath(\bfE)\cap\ker(A-z_0 I_\bfX)\big)
\subset
\frakM_{\bfE,\bfF,\varOmega}(A-z_0 I_\bfX)$;
\item
\label{theorem-m-3}
For any $\calB\in\scrQ_{\bfE,\bfF,\varOmega}(A-z_0 I_\bfX)$
one has
\[
\dim\frakM_{\bfE,\bfF,\varOmega}(A-z_0 I_\bfX)
=
r
\le \dim\frakL\big(I_\bfE-\calB(A+B-z_0 I_\bfX)^{-1}_{\bfE,\bfF,\varOmega}\big),
\]
where $r:=\min\{\rank \calB\sothat
\calB\in\scrQ_{\bfE,\bfF,\varOmega}(A-z_0 I_\bfX)\}$
is finite;
\item
\label{theorem-m-4}
If
$\calB\in\scrQ_{\bfE,\bfF,\varOmega}(A-z_0 I_\bfX)$
is of rank $r$,
then the null space and the generalized null spaces
of $I_\bfE-\calB(A+B-z_0 I_\bfX)^{-1}_{\bfE,\bfF,\varOmega}$
coincide:
\[
\ker\big(I_\bfE-\calB(A+B-z_0 I_\bfX)^{-1}_{\bfE,\bfF,\varOmega}\big)
=\frakL\big(I_\bfE-\calB(A+B-z_0 I_\bfX)^{-1}_{\bfE,\bfF,\varOmega}\big).
\]
\end{enumerate}
\end{theorem}

Above,
the root lineal corresponding to the zero eigenvalue
of $A\in\scrC(\bfX)$
is defined by
\[
\frakL(A)=
\{
\psi\in\dom(A)\sothat
A^j\psi\in\dom(A)\ \forall j\in\N,
\ \exists k\in\N\ \mbox{so that}\ A^k\psi=0
\}.
\]
We recall that a \emph{lineal} is a linear manifold which is not necessarily closed.
In the case when $A=\lambda I_\bfX-K$ with $\lambda\ne 0$ and $K\in\scrB_0(\bfX)$,
the root lineal $\frakL(\lambda I_\bfX-K)$ is a generalized eigenspace corresponding
to eigenvalue $\lambda$.

\begin{remark}\label{remark-sommerfeld}
The inclusion
$\Psi\in\range\big((A+B-z_0 I_\bfX)^{-1}_{\bfE,\bfF,\varOmega}\big)$
in \eqref{def-m-frak}
is related
to the Sommerfeld radiation condition.
For example,
if the solution to $(-\Delta+V-z_0 I)\Psi=0$, with $z_0>0$,
satisfies
$\Psi\in\range\big(R_V^{(3)}(z_0\pm\jj 0)\big)$,
hence $\Psi\in\range\big(R_0^{(3)}(z_0\pm\jj 0)\big)$,
then
$\Psi(x)=\frac{e^{\pm\jj z_0^{1/2}|x|}}{4\pi|x|}\ast\phi$,
$\phi\in\bfE$
(where one can take $\bfE=L^2_s(\R^3)$, $s>1/2$;
cf. Theorem~\ref{theorem-3d}),
is either radiating (for the ``$+$'' sign)
or incoming (for the ``$-$'' sign).
\end{remark}

\begin{proof}
The independence of
$\frakM_{\bfE,\bfF,\varOmega}(A-z_0 I_\bfX)$
on the choice of
$\calB\in\scrQ_{\bfE,\bfF,\varOmega}(A-z_0 I_\bfX)$
is due to the relation
\[
\ran\big((A+B_1-z_0 I_\bfX)^{-1}_{\bfE,\bfF,\varOmega}\big)
=
\ran\big((A+B_2-z_0 I_\bfX)^{-1}_{\bfE,\bfF,\varOmega}\big),
\]
for any $\calB_1,\,\calB_2\in\scrQ_{\bfE,\bfF,\varOmega}(A-z_0 I_\bfX)$,
$\calB_i:\,\bfF\to\bfE$,
$B_i:=\imath\circ\calB_i\circ\jmath:\,\bfX\to\bfX$,
$1\le i\le 2$,
which follows from the relation \eqref{l-u-c-zero-new}
(applied with $A+B_1$ instead of $A$
and $B_2-B_1$ instead of $B$).
This completes Part~\itref{theorem-m-1}.

\smallskip

For Part~\itref{theorem-m-2},
let $\phi\in\bfE$
be
such that
$\imath(\phi)\in\ker(A-z_0 I_\bfX)$.
Let $\calB\in\scrQ_{\bfE,\bfF,\varOmega}(A-z_0 I_\bfX)$.
Since
$(A+B-z_0 I_\bfX)\imath(\phi)
=B\imath(\phi)\in\imath(\bfE)$,
one has
\[
(\hatA+\hatB-z_0 I_\bfF)\jmath\circ\imath(\phi)
=
\jmath\circ(A+B-z_0 I_\bfX)\imath(\phi)
=
\jmath\circ B\circ\imath(\phi)
=\jmath\circ\imath\circ\calB\phi
\in\jmath\circ\imath(\bfE),
\]
hence
$
(A+B-z_0 I_\bfX)\imath(\phi)
=
B\circ\imath(\phi)
=\imath\circ\calB\phi
\in\imath(\bfE)$.
By
Lemma~\ref{lemma-jk}~\itref{lemma-jk-jk2},
$\jmath\circ\imath(\phi)
\in\ran\big((A+B-z_0 I_\bfX)^{-1}_{\bfE,\bfF,\varOmega})$,
and thus
$\jmath\circ\imath(\phi)\in\frakM_{\bfE,\bfF,\varOmega}(A-z_0 I_\bfX)$.

\smallskip

Let us prove Part~\itref{theorem-m-3}.
We recall that, by Lemma~\ref{lemma-b},
the value of
$r:=\min\{\rank\calB\sothat
\calB\in\scrQ_{\bfE,\bfF,\varOmega}(A-z_0 I_\bfX)
\}$
is finite.
Fix $\calB\in\scrQ_{\bfE,\bfF,\varOmega}(A-z_0 I_\bfX)$
and
recall the notations from \eqref{def-k}:
\begin{eqnarray}\label{def-k-bis}
K=\calB(A+B-z_0 I_\bfX)^{-1}_{\bfE,\bfF,\varOmega}\in\scrB_{0}(\bfE),
\qquad
B=\imath\circ\calB\circ\jmath:\,\bfX\to\bfX.
\end{eqnarray}
One has:
\begin{eqnarray}\label{n-le-r}
\dim\frakM_{\bfE,\bfF,\varOmega}(A-z_0 I_\bfX)
=\dim\ker(I_\bfE-K)
\le\dim\frakL(I_\bfE-K)
\le\rank K
\le\rank\calB,
\end{eqnarray}
with the first equality due to Theorem~\ref{theorem-lap}~\itref{theorem-lap-1}.

\begin{lemma}\label{lemma-nu}
If $\rank\calB=r$, then
\[
\dim\frakM_{\bfE,\bfF,\varOmega}(A-z_0 I_\bfX)
=\dim\ker(I_\bfE-K)
\le\dim\frakL(I_\bfE-K)
=\rank K
=\rank\calB,
\]
with $K$ from \eqref{def-k-bis}.
\end{lemma}

\begin{proof}
Due to the inequalities \eqref{n-le-r},
$\dim\frakL(I_\bfE-K)\le\rank\calB=r$;
it is enough to show that for any
$\calB\in\scrQ_{\bfE,\bfF,\varOmega}(A-z_0 I_\bfX)
$
there is the inequality
$\dim\frakL(I_\bfE-K)\ge r$.
Let $P_1\in\scrB_{00}(\bfE)$ be the Riesz projection onto
$\frakL(I_\bfE-K)$, the generalized eigenspace
of eigenvalue $\lambda=1$ of the operator $K$.
By Lemma~\ref{lemma-b},
$P_1\circ\calB\in\scrQ_{\bfE,\bfF,\varOmega}(A-z_0 I_\bfX)$.
Due to the inequality
\[
\rank(P_1\circ\calB)\le\rank P_1=\dim\frakL(I_\bfE-K),
\]
we conclude that $r\le\rank(P_1\circ\calB)\le\dim\frakL(I_\bfE-K)$,
as needed.
\end{proof}

Let us show that
$\calB\in\scrQ_{\bfE,\bfF,\varOmega}(A-z_0 I_\bfX)$
can be chosen so that
$\rank\calB=r$
and so that
$I_\bfE-K$ has no Jordan block.
By Part~\itref{theorem-m-1},
$m:=\dim\ker(I_\bfE-K)=\dim\frakM_{\bfE,\bfF,\varOmega}(A-z_0 I_\bfX)$
does not depend on the choice
of $\calB\in\scrQ_{\bfE,\bfF,\varOmega}(A-z_0 I_\bfX)$.
Let $\big\{\Psi_i\in\bfF\big\}_{1\le i\le m}$
be a basis in $\frakM_{\bfE,\bfF,\varOmega}(A-z_0 I_\bfX)$.
Then, by Theorem~\ref{theorem-lap}~\itref{theorem-lap-1},
\begin{eqnarray}\label{def-phi-i}
\phi_i=\calB\Psi_i\in\bfE,
\qquad
1\le i\le m,
\end{eqnarray}
is a basis in $\ker(I_\bfE-K)$.

\begin{lemma}\label{lemma-k-ast}
For each
$\calB\in\scrQ_{\bfE,\bfF,\varOmega}(A-z_0 I_\bfX)$,
the space
$\ker(I_{\bfE^*}-K^*)\subset\bfE^*$
does not depend on the choice of $\calB$
and is of dimension $m=\dim\frakM_{\bfE,\bfF,\varOmega}(A-z_0 I_\bfX)$.
\end{lemma}

\begin{proof}
Assume that $\calB_1,\,\calB_2
\in\scrQ_{\bfE,\bfF,\varOmega}(A-z_0 I_\bfX)$.
The operators $(A+B_\nu-z_0 I_\bfX)^{-1}_{\bfE,\bfF,\varOmega}$,
$\nu=1,\,2,$
with $B_\nu=\imath\circ\calB_\nu\circ\jmath$,
are injective by Lemma~\ref{lemma-jk}~\itref{lemma-jk-jk1} and have same range by  Theorem~\ref{theorem-m}~\itref{theorem-m-1}. Hence there exists $X(B_1,B_2)$,
an invertible operator in $\bfE$
(not necessarily bounded),
such that
\[
(A+B_2-z_0 I_\bfX)^{-1}_{\bfE,\bfF,\varOmega}=(A+B_1-z_0 I_\bfX)^{-1}_{\bfE,\bfF,\varOmega}X(B_1,B_2).
\]
Multiplying by $(\hatA-z_0 I_\bfF)$
and using Lemma~\ref{lemma-jk}~\itref{lemma-jk-jk1}
(with $\hatA+\jmath\circ\imath\circ\calB_\nu$
in place of $\hatA$),
we get
\[
I_{\bfE}-\calB_2(A+B_2-z_0 I_\bfX)^{-1}_{\bfE,\bfF,\varOmega}
=
\big(
I_{\bfE}-\calB_1(A+B_1-z_0 I_\bfX)^{-1}_{\bfE,\bfF,\varOmega}
\big)X(B_1,B_2).
\]
We deduce that
$I_\bfE-\calB_2(A+B_2-z_0 I_\bfX)^{-1}_{\bfE,\bfF,\varOmega}$
and
$I_\bfE-\calB_1(A+B_1-z_0 I_\bfX)^{-1}_{\bfE,\bfF,\varOmega}$
have the same range and thus the same annihilator in $\bfE^*$. Thus
their adjoints have the same kernel.

The value of dimension follows from
$
\dim\ker(I_{\bfE^*}-K^*)=\dim\ker(I_\bfE-K)=m$,
due to $K$ being compact;
see e.g. \cite[Theorem V.7.14]{taylor1980introduction}.
\end{proof}

Let $\big\{\xi_i\in\bfE^*\big\}_{1\le i\le m}$
be a basis in $\ker(I_{\bfE^*}-K^*)$;
by Lemma~\ref{lemma-k-ast},
it does not depend on $\calB\in\scrQ_{\bfE,\bfF,\varOmega}(A-z_0 I_\bfX)$.
To investigate the presence of Jordan blocks,
we need to study the space of solutions to the problem
\begin{eqnarray}\label{dim-space}
(I_\bfE-K)\phi\in\ker(I_\bfE-K),
\qquad
\phi\in\bfE.
\end{eqnarray}
The dimension of the space of solutions
to \eqref{dim-space} is given by
\begin{eqnarray}\label{i-k}
\dim\big(
\ran(I_\bfE-K)\cap\ker(I_\bfE-K)
\big)
=m-\rank\langle\xi_i,\phi_j\rangle_\bfE
=m-\rank\langle\xi_i,\calB\Psi_j\rangle_\bfE,
\end{eqnarray}
with $\big\{\phi_j\big\}_{1\le j\le m}$ (see \eqref{def-phi-i})
a basis in $\ker(I_\bfE-K)$.
To complete the proof
of Part~\itref{theorem-m-3},
it suffices to vary $\calB$
(such variations are allowed by Lemma~\ref{lemma-relatively-open}),
without changing its rank $r\ge m$, so that
rank of the matrix $\langle\xi_i,\calB\Psi_j\rangle_\bfE$
equals $m$,
and then one can see from \eqref{i-k} that
$I_\bfE-K$ does not have Jordan blocks
corresponding to zero eigenvalue.
It follows that the null space and the generalized null space of
$I_\bfE-K$ coincide.
By Lemma~\ref{lemma-nu},
this implies that
$\dim\frakM_{\bfE,\bfF,\varOmega}(A-z_0 I_\bfX)=r$.

Let us prove Part~\itref{theorem-m-4}.
By Part~\itref{theorem-m-3},
one has $\dim\frakM_{\bfE,\bfF,\varOmega}(A-z_0 I_\bfX)=r$;
hence, if $\calB$ is of rank $r$,
the inequalities \eqref{n-le-r}
turn into equalities,
leading to $\ker(I_\bfE-K)=\frakL(I_\bfE-K)$.
\end{proof}

\subsection{Bifurcations from the essential spectrum}

Now we elaborate on the relation of virtual levels
and bifurcations from the essential spectrum.

\begin{theorem}[Virtual levels and bifurcations of eigenvalues
from the essential spectrum]
\label{theorem-b}
Let $\bfE$ and $\bfF$ be Banach spaces with
continuous embedding
$\bfE\mathop{\longhookrightarrow}\limits\sp{\imath}\bfX
\mathop{\longhookrightarrow}\limits\sp{\jmath}\bfF$
and let $A\in\scrC(\bfX)$ satisfy Assumption~\ref{ass-virtual}.
Let
$\varOmega\subset\C\setminus\sigma(A)$
and let $z_0\in\sigma\sb{\mathrm{ess}}(A)\cap\p\varOmega$.
\begin{enumerate}
\item
\label{theorem-b-1}
Assume that
there is a sequence of
$\hatA$-bounded
perturbations
$\chB_j:\,\bfF\to\bfE$,
with
\begin{eqnarray}\label{v-a-zero}
\lim\sb{j\to\infty}
\norm{\chB_j}\sb{\doma\shortto\bfE}=0,
\end{eqnarray}
with the corresponding sequence
\[
z_j\in\sigma(A+\imath\circ\chB_j\circ\jmath)\cap\varOmega,
\qquad
j\in\N,
\qquad
z_j\to z_0.
\]
If
either
$\imath\circ\chB_j\circ\jmath$, $j\in\N$,
are $\hatA$-compact,
or
$z_j\in\sigma\sb{\mathrm{p}}(A+\imath\circ\chB_j\circ\jmath)\cap\varOmega$,
$j\in\N$,
then 
$\limsup\limits\sb{j\to\infty}
\norm{\jmath\circ(A-z_j I_\bfX)^{-1}\circ\imath}_{\bfE\to\bfF}
=\infty$.

\item
\label{theorem-b-2}
Assume that
$\varOmega$ is contained in one component of
$\C\setminus\sigma_{\mathrm{ess},1}(A)$
(for example, this is the case if $\varOmega$ is connected).
Assume that $z_0$ is a virtual level of $A$
relative to $(\bfE,\bfF,\varOmega)$.
There is
$\delta>0$
such that for any sequence
$z_j\in\varOmega\cap\mathbb{D}_{\delta}(z_0)$, $j\in\N$,
$z_j\to z_0$,
there is
$\chB\in\scrB_{00}(\bfF,\bfE)$
and a sequence $\zeta_j\in\C$, $\zeta_j\to 0$,
such that
\begin{eqnarray}\label{z-j-s}
z_j\in\sigma\sb{\mathrm{d}}(A+\zeta_j B),
\qquad
j\in\N,
\qquad
\mbox{where
$B=\imath\circ\chB\circ\jmath\in\scrB_{00}(\bfX)$.}
\end{eqnarray}
\end{enumerate}
\end{theorem}

\begin{proof}
Let us prove Part~\itref{theorem-b-1}.
Assume that, contrary to the statement of the theorem,
$\sup\sb{j\in\Z}
\norm{\jmath\circ(A-z_j I_\bfX)^{-1}\circ\imath}\sb{\bfE\shortto
\bfF}<\infty$;
then,
as one can see from \eqref{if-you-wish},
we also have
\begin{eqnarray}\label{wah}
\sup\sb{j\in\Z}
\norm{\jmath\circ(A-z_j I_\bfX)^{-1}\circ\imath}\sb{\bfE\shortto
\doma}<\infty.
\end{eqnarray}

We will first consider the case when
$\chB_j$, $j\in\N$, are $\hatA$-compact.
Lemma~\ref{lemma-relatively-open}~\itref{lemma-relatively-open-1}
(where we take $B=\imath\circ\chB_j\circ\jmath$)
shows that
$z_j\in\C\setminus\sigma(A+\imath\circ\chB_j\circ\jmath)$ for all $j$ sufficiently large,
in contradiction to the inclusions $z_j\in\sigma(A+\imath\circ\chB_j\circ\jmath)$.
This contradiction completes the proof in the case
when $\chB_j$ are $\hatA$-compact.

Now we consider the case when
$z_j\in\sigma\sb{\mathrm{p}}(A+\imath\circ\chB_j\circ\jmath)\cap\varOmega$,
while $\chB_j$ are not assumed to be
$\hatA$-compact.
Let $\psi_j\in\bfX\setminus\{0\}$, $j\in\N$,
be eigenfunctions
of $A+\imath\circ\chB_j\circ\jmath$ corresponding to $z_j$, so that
$z_j\psi_j=(A+\imath\circ\chB_j\circ\jmath)\psi_j$,
hence
\[
\jmath(\psi_j)
=-\jmath\circ(A-z_j I_\bfX)^{-1}\circ\imath\circ\chB_j\circ\jmath(\psi_j),
\qquad
j\in\N.
\]
Due to \eqref{wah}, the above relation leads to a contradiction
with \eqref{v-a-zero}.
This completes the proof of Part~\itref{theorem-b-1}.

\smallskip

Let us prove Part~\itref{theorem-b-2}.
By assumption, the operator family
$\big\{\jmath\circ(A-z I_\bfX)^{-1}\circ\imath
\big\}_{z\in\varOmega}$
does not have a limit
as $z\to z_0$, $z\in\varOmega$,
in the weak and weak$^*$
operator topology of $\scrB(\bfE,\bfF)$,
while there is $\calB\in\scrB_{00}(\bfF,\bfE)$
and $\delta>0$ such that
\[
\varOmega\cap\mathbb{D}_\delta(z_0)\subset\C\setminus\sigma(A+B),
\]
with $B=\imath\circ\calB\circ\jmath$,
and the operator family
$
\big\{\jmath\circ(A+B-z I_\bfX)^{-1}\circ\imath
\big\}_{z\in\varOmega\cap\mathbb{D}_\delta(z_0)}
$
does have a limit
$(A+B-z_0 I_\bfX)^{-1}_{\bfE,\bfF,\varOmega}$
in the weak or weak$^*$
operator topology of $\scrB(\bfE,\bfF)$
as $z\to z_0$, $z\in\varOmega$.
By Theorem~\ref{theorem-lap}~\itref{theorem-lap-2}
(applied with $A+B$ in place of $A$ and $-B$ in place of $B$),
one has
$1\in\sigma\big((A+B-z_0 I_\bfX)^{-1}_{\bfE,\bfF,\varOmega}\calB\big)$.
By Theorem~\ref{theorem-lap}~\itref{theorem-lap-1},
$\lambda=1$
is an isolated
eigenvalue of $(A+B-z_0 I_\bfX)^{-1}_{\bfE,\bfF,\varOmega} \calB$
of finite multiplicity
(since $\calB\in\scrB_{00}(\bfF,\bfE)$).

\begin{lemma}\label{lemma-lambda}
For any sequence
$z_j\in\varOmega\cap\mathbb{D}_\delta(z_0)$, $z_j\to z_0$, $j\in\N$,
there is a sequence of eigenvalues
$\lambda_j\in\sigma\sb{\mathrm{p}}
\big(\jmath\circ(A+B-z_j I_\bfX)^{-1}\circ\imath\circ\calB\big)$
such that
$\lambda_j\to 1$ as $j\to\infty$.
\end{lemma}
\begin{proof}
Let $P_0\in\scrB_{00}(\bfF)$ be any projection onto
the kernel of $\calB$.
Denote $P_1=I_\bfF-P_0$, so that $\rank P_1=\rank\calB<\infty$
and $\calB P_1=\calB$.
Since $P_1$ and $\calB$ are both of finite rank,
$P_1\circ\jmath\circ(A+B-z_j I_\bfX)^{-1}\circ\imath\circ\calB$
converges to $P_1(A+B-z_0 I_\bfX)^{-1}_{\bfE,\bfF,\varOmega}\calB$
in the uniform operator topology of $\scrB(\bfF)$
(see Lemma~\ref{lemma-rc})
and
the eigenvalues of $P_1\circ\jmath\circ(A+B-z_j I_\bfX)^{-1}\circ\imath\circ\calB$
converge to eigenvalues of
$P_1(A+B-z_0 I_\bfX)^{-1}_{\bfE,\bfF,\varOmega}\calB$.
Therefore, for any sequence
$z_j\in\varOmega\cap\mathbb{D}_\delta(z_0)$, $z_j\to z_0$, $j\in\N$,
there is a sequence of eigenvalues
\[
\lambda_j\in
\sigma\sb{\mathrm{p}}
\big(P_1\circ\jmath\circ(A+B-z_j I_\bfX)^{-1}\circ\imath\circ\calB\big)\setminus\{0\}
=
\sigma\sb{\mathrm{p}}
\big(\jmath\circ(A+B-z_j I_\bfX)^{-1}\circ\imath\circ\calB\big)
\setminus\{0\}
\]
such that
\[
\lambda_j\to 1
\in
\sigma\sb{\mathrm{p}}
\big(P_1(A+B-z_0 I_\bfX)^{-1}_{\bfE,\bfF,\varOmega}\calB\big)
\setminus\{0\}
=
\sigma\sb{\mathrm{p}}
\big((A+B-z_0 I_\bfX)^{-1}_{\bfE,\bfF,\varOmega}\calB\big)
\setminus\{0\}
\]
as $j\to\infty$.
The coincidence of point spectra
away from zero
in the above relations
follows from the relation
$\calB P_1=\calB$
and from
$\sigma\sb{\mathrm{p}}(ST)\setminus\{0\}
=\sigma\sb{\mathrm{p}}(TS)\setminus\{0\}$
for any $S,\,T\in\scrB(\bfE)$.
\end{proof}

For a given sequence
of values $z_j\in\varOmega\cap\mathbb{D}_\delta(z_0)$,
$z_j\to z_0$,
we use Lemma~\ref{lemma-lambda}
which provides the sequence
$\lambda_j\in\C$, $\lambda_j\to 1$.
Discarding finitely many values of $j$ if necessary,
we may assume that $\lambda_j\ne 0$ for all $j\in\N$.
  From the relation
\[
\jmath\circ
(A+B-z_j I_\bfX)^{-1}\circ\imath\circ\calB\Psi_j=\lambda_j\Psi_j,
\qquad
\Psi_j\in\bfF,
\quad\Psi_j\ne 0,
\]
where $z_j\in\varOmega\cap\mathbb{D}_\delta(z_0)\subset\C\setminus\sigma(A+B)$,
we conclude that
$\psi_j:=(A+B-z_j I_\bfX)^{-1}\circ\imath\circ\calB\Psi_j\in\bfX$
satisfies
\[
\lambda_j^{-1}\imath\circ\calB\circ\jmath(\psi_j)
=\lambda_j^{-1}B\psi_j
=(A+B-z_j I_\bfX)\psi_j,
\qquad
j\in\N,
\]
hence $z_j\in\sigma\sb{\mathrm{p}}\big(A+(1-\lambda_j^{-1})B\big)$.
Therefore, in \eqref{z-j-s}, we can take
\[
\zeta_j=1-\lambda_j^{-1},
\qquad
j\in\N;
\qquad
\lim\sb{j\to\infty}\zeta_j=0.
\]
Finally, let us show that,
as long as
$\lambda_j\in\mathbb{D}_\epsilon(1)$
with $\epsilon>0$ small enough,
the values $z_j$ are from the discrete spectrum,
so that
\begin{eqnarray}\label{z-j-in-d}
z_j\in\sigma\sb{\mathrm{d}}\big(A+(1-\lambda_j^{-1})B\big),
\qquad
j\in\N.
\end{eqnarray}
Since $B$ is of finite rank,
by the Weyl theorem on the essential spectrum \cite[Theorem IX.2.1]{edmunds2018spectral}
one has
\begin{eqnarray}\label{1-4}
\sigma\sb{\mathrm{ess},k}(A+\zeta B)
=\sigma\sb{\mathrm{ess},k}(A),
\qquad
\forall \zeta\in\C,\quad 1\le k\le 4
\end{eqnarray}
(see Remark~\ref{remark-ess} for more details).
We recall that
$\sigma\sb{\mathrm{ess},5}(A)$ is defined
as the union of
$\sigma\sb{\mathrm{ess},1}(A)$
and the connected (open) components of
$\C\setminus \sigma\sb{\mathrm{ess},1}(A)$
which do not intersect the resolvent set of $A$,
while \eqref{1-4} yields
\[
\varOmega\cap\sigma\sb{\mathrm{ess},1}(A+\zeta B)
=
\varOmega\cap\sigma\sb{\mathrm{ess},1}(A)
\subset
\varOmega\cap\sigma(A)=\emptyset,
\qquad
\zeta\in\C;
\]
therefore, due to the assumption
in Part~\itref{theorem-b-2}
that $\varOmega$ is contained in one component of
$\C\setminus\sigma\sb{\mathrm{ess},1}(A)$,
there are two possibilities for a particular $\zeta\in\C$:
\begin{eqnarray}\label{eo}
\mbox{
\it
either
\quad
$\varOmega
\subset\sigma\sb{\mathrm{ess},5}(A+\zeta B)$,
}
\quad
\mbox{\it
or
\quad
$\varOmega
\cap\sigma\sb{\mathrm{ess},5}(A+\zeta B)=\emptyset$
}
\end{eqnarray}
(cf. Remark~\ref{remark-ess}).
If $\abs{\zeta}$ is sufficiently small,
then only the latter alternative is possible.
Indeed,
since e.g. $z_1\in\varOmega\subset\C\setminus\sigma(A)$,
there is $r>0$ such that
$z_1\not\in\sigma(A+\zeta B)$
for $\abs{\zeta}<r$,
and then, due to the dichotomy \eqref{eo},
one has
$\varOmega\cap\sigma\sb{\mathrm{ess},5}(A+\zeta B)=\emptyset$.
Therefore,
there is $\epsilon\in(0,1)$ small enough
such that
if  $\lambda_j\in\mathbb{D}_\epsilon(1)$,
then
$z_j\not\in\sigma\sb{\mathrm{ess},5}(A+(1-\lambda_j^{-1})B)$;
thus, the inclusion \eqref{z-j-in-d} holds.
This completes the proof of Part~\itref{theorem-b-2}.
\end{proof}

\begin{example}
[Virtual levels of $-\Delta+V(x)$ at $z_0\ge 0$, $d=3$]
\label{example-3d-bifurcations}
For $x\in\R^3$ and $\zeta\in\overline{\C\sb{+}}$, define
\[
\psi(x,\zeta)=
\begin{cases}
\fra{e^{\jj \zeta\abs{x} }}{\abs{x}},&\abs{x}\ge 1,
\\
(\fra{(3-\abs{x}^2)}{2})e^{\jj \zeta(1+\abs{x}^2)/2},&0\le \abs{x}<1;
\end{cases}
\qquad
\psi(\cdot,\zeta)\in C^2(\R^3),
\]
so that
$-\Delta\psi=\zeta^2 \psi$
for $x\in\R^3\setminus\mathbb{B}^3_1$.
For each $\zeta\in\overline{\C\sb{+}}$,
define the potential
$V(x,\zeta)$
by the relation
\[
-\p_r^2\psi-2r^{-1}\p_r\psi+V\psi=\zeta^2\psi,
\qquad
x\in\R^3,
\quad
r=\abs{x}.
\]
Then, for each $\zeta\in\overline{\C\sb{+}}$,
the potential
$V(\cdot,\zeta)\in L^\infty(\R^3)$
is spherically symmetric,
piecewise smooth,
with
$\supp V\subset\mathbb{B}^3_1$. 
For
$\zeta\in\C\sb{+}$,
one has
$z=\zeta^2\in\sigma\sb{\mathrm{p}}(-\Delta+V(\zeta))$.
Since $V$ is bounded and compactly supported, we note that,
by Weyl's theorem on the essential spectrum
\cite[Theorem IX.2.1]{edmunds2018spectral},
$\sigma\sb{\mathrm{ess},i}(-\Delta+V(\zeta))
=\sigma\sb{\mathrm{ess},i}(-\Delta)
=\overline{\R_{+}}$,
$1\le i\le 4$.
This implies that
there can be no components of
$\C\setminus\sigma\sb{\mathrm{ess},1}(-\Delta+V(\zeta))$
without intersection with the resolvent set,
hence
$\sigma\sb{\mathrm{ess},i}(-\Delta+V(\zeta))
=\overline{\R_{+}}$
for all $1\le i\le 5$.
(see Remark~\ref{remark-ess}).
Therefore,
$z=\zeta^2\in\sigma\sb{\mathrm{d}}(-\Delta+V(\zeta))$.
Thus, for each $\zeta_0\ge 0$,
there is an eigenvalue family bifurcating from
$z_0=\zeta_0^2\in\sigma\sb{\mathrm{ess}}(-\Delta+V(\cdot,\zeta_0))
=\overline{\R_{+}}$ into $\C_{+}$.
By Theorem~\ref{theorem-b}~\itref{theorem-b-1},
$z_0=\zeta_0^2$
is a virtual level of
$-\Delta+V(x,\zeta_0)$ relative to $\C_{+}$.
\end{example}

\subsection{Dependence on the choice of regularizing spaces}
\label{sect-choice}

The concept of virtual levels does not depend on the choice of
``regularizing'' spaces $\bfE_\nu$ and $\bfF_\nu$, $\nu=1,\,2$,
if certain density assumptions are satisfied
(this result is similar to \cite[Proposition 4.1]{agmon1998perturbation}
in the context of resonances,
where one assumes that the embeddings
$\bfE_\nu\hookrightarrow\bfX\hookrightarrow\bfF_\nu$, $\nu=1,\,2$,
are dense, and that
$\bfE_1\cap\bfE_2$ is dense both in $\bfE_1$ and $\bfE_2$).

Let
\begin{eqnarray}\label{def-ij}
\imath_\nu:\;\bfE_\nu\hookrightarrow\bfX,
\qquad
\jmath_\nu:\;\bfX\hookrightarrow\bfF_\nu,
\qquad
\nu=1,\,2,
\end{eqnarray}
be continuous embeddings of Banach spaces.
The set
$\bfE=\bfE_1\cap\bfE_2$ is defined as the following Banach space:
\begin{eqnarray}
&
\bfE=\bfE_1\cap\bfE_2
=\big\{\psi\in\bfX\sothat
\exists(\phi_1,\phi_2)\in\bfE_1\times\bfE_2,
\ \imath_1(\phi_1)=\imath_2(\phi_2)=\psi
\big\},
\label{x-phi}
\\[1ex]
&
\nonumber
\mbox{with the norm}
\quad
\norm{\psi}_{\bfE}
=\max(\norm{\phi_1}_{\bfE_1},\norm{\phi_2}_{\bfE_2})
.
\end{eqnarray}
There are continuous embeddings
$\upalpha_\nu:\;
\bfE\hookrightarrow\bfE_\nu$,
$\psi\mapsto\phi_\nu$,
$\nu=1,\,2$,
with $\psi$, $\phi_1$, and $\phi_2$ from \eqref{x-phi}.

Let us assume that the Banach spaces
$\bfF_1$ and $\bfF_2$
are identified as subspaces
of a Hausdorff vector space $\bfG$
(hence they form a \emph{compatible couple});
we assume that the embeddings $\jmath_\nu$
from \eqref{def-ij}
are compatible with this identification, so that
for any $\psi\in\bfX$ one has $\jmath_1(\psi)=\jmath_2(\psi)$ in $\bfG$.
The Banach space $\bfF=\bfF_1+\bfF_2$ is defined,
with an abuse of notation, by
\begin{eqnarray*}
&
\bfF=\bfF_1+\bfF_2=
\big\{\psi\in\bfG\sothat\exists
(\psi_1,\psi_2)\in\bfF_1\times\bfF_2,
\ \psi=\psi_1+\psi_2\big\},
\\[1ex]
&
\mbox{with the norm}
\quad
\norm{\psi}_{\bfF}
=\inf\limits\sb{\psi=\psi_1+\psi_2;\;\psi_1\in\bfF_1,\,\psi_2\in\bfF_2}
\big(
\norm{\psi_1}_{\bfF_1}+\norm{\psi_2}_{\bfF_2}
\big).
\end{eqnarray*}
There are continuous embeddings
$\upbeta_\nu:\,\bfF_\nu\hookrightarrow\bfF$,
$\nu=1,\,2$.
See Figure~\ref{fig-abij}.

\begin{figure}[ht]
\begin{picture}(0,110)(-220,-50)
\linethickness{0.33pt}
\put(-110,0){$\bfE=\bfE_1\cap\bfE_2$}
\put(-55,20) {$\longhooknearrow$}
\put(-55,-20){$\longhooksearrow$}
\put(-60, 26){$\upalpha_1$}
\put(-60,-24) {$\upalpha_2$}
\put(-37,0){$\mathop{\longhookrightarrow}\limits\sp{\textstyle\imath}$}
\put(-35,35){$\bfE_1$}
\put(-35,-35){$\bfE_2$}
\put(-18,18) {$\longhooksearrow$}
\put(-18,-18){$\longhooknearrow$}
\put(-10, 26){$\imath_1$}
\put(-10,-24) {$\imath_2$}
\put(5,0){$\bfX$}
\put(20,20) {$\longhooknearrow$}
\put(20,-20){$\longhooksearrow$}
\put(17, 26){$\jmath_1$}
\put(19,-24) {$\jmath_2$}
\put(40,35){$\bfF_1$}
\put(40,-35){$\bfF_2$}
\put(55,20){$\longhooksearrow$}
\put(55,-20) {$\longhooknearrow$}
\put(64, 26){$\upbeta_1$}
\put(64,-26) {$\upbeta_2$}
\put(70,0){$\bfF=\bfF_1+\bfF_2$}
\put(33,0){$\mathop{\longhookrightarrow}\limits\sp{\textstyle\jmath}$}
\end{picture}
\caption{Embeddings $\bfE_\nu\hookrightarrow\bfX\hookrightarrow\bfF_\nu$
and
$\bfE=\bfE_1\cap\bfE_2
\hookrightarrow\bfX\hookrightarrow
\bfF=\bfF_1+\bfF_2$.}
\label{fig-abij}
\end{figure}
We denote the resulting continuous embeddings
$\bfE=\bfE_1\cap\bfE_2\hookrightarrow\bfX$
and
$\bfX\hookrightarrow\bfF=\bfF_1+\bfF_2$
by
\[
\imath=\imath_1\circ\upalpha_1=\imath_2\circ\upalpha_2
:\;
\bfE\hookrightarrow\bfX,
\qquad
\jmath=\upbeta_1\circ\jmath_1=\upbeta_2\circ\jmath_2
:\;
\bfX\hookrightarrow\bfF.
\]

\begin{theorem}[Independence on the choice of regularizing spaces]
\label{theorem-a}
Let $\bfE_\nu$ and $\bfF_\nu$, $\nu=1,\,2$,
$\bfE=\bfE_1\cap\bfE_2$, and $\bfF=\bfF_1+\bfF_2$ be as above.
Let
$A\in\scrC(\bfX)$
and let $z_0\in\sigma\sb{\mathrm{ess}}(A)\cap\p\varOmega$,
with $\varOmega\subset\C\setminus\sigma(A)$.
\begin{enumerate}
\item
\label{theorem-a-1}
\begin{enumerate}
\item
\label{theorem-a-1a}
Assume that the continuous mappings
\[
\upalpha_2:\;\bfE\hookrightarrow\bfE_2,
\qquad
\upbeta_2^*:\;\bfF^*\to\bfF_2^*
\]
have dense ranges.
Assume that the operator family
$\big\{\jmath_1\circ(A-z I_\bfX)^{-1}\circ\imath_1\big\}_{z\in\varOmega}$
has a limit
as $z\to z_0$, $z\in\varOmega$, in the weak operator topology
of $\scrB(\bfE_1,\bfF_1)$.
If there is $\delta>0$ such that the operator family
$\big\{\jmath_2\circ(A-z I_\bfX)^{-1}\circ\imath_2
\big\}_{z\in\varOmega\cap\mathbb{D}_\delta(z_0)}$
is uniformly bounded
in $\scrB(\bfE_2,\bfF_2)$,
then this family has a limit
as $z\to z_0$
in the weak operator topology of $\scrB(\bfE_2,\bfF_2)$.

\item
\label{theorem-a-1b}
Assume that
$\bfF_1$,
$\bfF_2$,
and
$\bfF$
have pre-duals,
denoted by
$(\bfF_1)_*$,
$(\bfF_2)_*$,
and
$\bfF_*$,
respectively,
and that the continuous
mappings
\begin{eqnarray}\label{starr-1}
\upalpha_2:\;
\bfE\hookrightarrow\bfE_2,
\qquad
\upbeta_2^*\circ J_{\bfF_*}:\;
\bfF_*\to\bfF_2^*
\end{eqnarray}
have dense ranges.
Assume that the operator family
$\big\{\jmath_1\circ(A-z I_\bfX)^{-1}\circ\imath_1\big\}_{z\in\varOmega}$
has a limit
as $z\to z_0$, $z\in\varOmega$, in the weak$^*$ operator topology
of $\scrB(\bfE_1,\bfF_1)$.
If there is $\delta>0$ such that the operator family
$\big\{\jmath_2\circ(A-z I_\bfX)^{-1}\circ\imath_2
\big\}_{z\in\varOmega\cap\mathbb{D}_\delta(z_0)}$
is uniformly bounded
in $\scrB(\bfE_2,\bfF_2)$,
then this family has a limit
as $z\to z_0$
in the weak$^*$ operator topology of $\scrB(\bfE_2,\bfF_2)$.
\end{enumerate}
\item
\label{theorem-a-2}
Let $A\in\scrC(\bfX)$ satisfy Assumption~\ref{ass-virtual}
(with respect to the weak operator topology),
with $\hatA$ a closed extension
of $A$ onto $\bfF$.
\begin{enumerate}
\item
\label{theorem-a-2a}
Assume that the continuous mappings
\begin{eqnarray}\label{two-2}
\upalpha_\nu:\;
\bfE\hookrightarrow\bfE_\nu,
\qquad
\upbeta_\nu^*:\;\bfF^*
\to\bfF_\nu^*,
\qquad
\nu=1,\,2,
\end{eqnarray}
have dense ranges.
If the sets
$\scrQ_{\bfE_\nu,\bfF_\nu,\varOmega}(A-z_0 I_\bfX)$
(with respect to weak convergence
in Definition~\ref{definition-q}), with $\nu=1,\,2$,
are nonempty, then
\begin{eqnarray}\label{din-din}
\upbeta_1\big(\frakM_{\bfE_1,\bfF_1,\varOmega}(A-z_0 I_\bfX)\big)
=
\upbeta_2\big(\frakM_{\bfE_2,\bfF_2,\varOmega}(A-z_0 I_\bfX)\big).
\end{eqnarray}

\item
\label{theorem-a-2b}
Assume that $\bfF_1$ and $\bfF_2$
have pre-duals $(\bfF_1)_*$ and $(\bfF_2)_*$
and that there is a Hausdorff vector space $\bfG_0$
and continuous embeddings
$\upchi_\nu:\,(\bfF_\nu)_*\hookrightarrow\bfG_0$,
$\nu=1,\,2$.
Assume that $(\bfF_1)_*$ and $(\bfF_2)_*$
are mutually dense, in the sense that
the subsets
\begin{eqnarray}\label{chi-chi}
\upchi_\nu^{-1}
\Big(
\upchi_1\big((\bfF_1)_*\big)
\cap
\upchi_2\big((\bfF_2)_*\big)
\Big)
\subset
(\bfF_\nu)_*,
\qquad
\nu=1,\,2,
\end{eqnarray}
are dense in
$(\bfF_1)_*$ and $(\bfF_2)_*$,
respectively.\footnote{
That is, $(\bfF_1)_*$ and $(\bfF_2)_*$ form
a \emph{conjugate couple} in the sense of
\cite[Definition 8.II]{aronszajn1965interpolation}.
}
Then $\bfF$ has a pre-dual;
we assume that the closed extension $\hatA\in\scrC(\bfF)$
of $A$ onto $\bfF$
is weak$^*$-closed.
Assume that the continuous embeddings
\begin{eqnarray}\label{starr-2}
\upalpha_\nu:\;
\bfE\hookrightarrow\bfE_\nu,
\qquad
\nu=1,\,2,
\end{eqnarray}
are dense.
If the sets
$\scrQ_{\bfE_\nu,\bfF_\nu,\varOmega}(A-z_0 I_\bfX)$
(with respect
to weak$^*$ convergence
in Definition~\ref{definition-q}),
$\nu=1,\,2$
are nonempty,
then the relation \eqref{din-din} is satisfied.
\end{enumerate}
\end{enumerate}
\end{theorem}

We note that the requirement on the uniform boundedness
of $\jmath_2\circ(A-z I_\bfX)^{-1}\circ\imath_2:\,\bfE_2\to\bfF_2$
for $z\in\varOmega\cap\mathbb{D}_\delta(z_0)$
with some $\delta>0$ cannot be dropped
(at least in the case when
the embeddings
$\bfE_1\hookrightarrow \bfX\hookrightarrow\bfF_1$ are dense)
as one can see in the situation when $\bfE_2=\bfX=\bfF_2$.

\begin{remark}
By Theorem~\ref{theorem-m}~\itref{theorem-m-3},
the equality \eqref{din-din}
leads to the equality
of the ranks of the virtual level $z_0$ of $A$
relative to
$(\bfE_1,\bfF_1,\varOmega)$
and
$(\bfE_2,\bfF_2,\varOmega)$:
\[
\min\big\{\rank \calB\,;\,\calB\in\scrQ_{\bfE_1,\bfF_1,\varOmega}(A-z_0 I_\bfX)
\big\}
=
\min\big\{\rank \calB\,;\,\calB\in\scrQ_{\bfE_2,\bfF_2,\varOmega}(A-z_0 I_\bfX)
\big\}.
\]
\end{remark}

\begin{proof}
We recall the notation
$
R(z)=(A-z I_\bfX)^{-1}:\,\bfX\to\bfX$,
$z\in\C\setminus\sigma(A)$;
see \eqref{def-r-z0}.
Let us show that the convergence of
$\jmath_2\circ R(z)\circ\imath_2$
holds in the weak operator topology.
Fix $\phi\in\bfE_2$, $\norm{\phi}_{\bfE_2}=1$,
and $\eta\in\bfF_2^*$, $\norm{\eta}_{\bfF_2^*}=1$.
We need to show that
the function
$f(z):=\langle\eta,R(z)\phi\rangle_{\bfF_2}$, $z\in\varOmega$
(with $\langle\,,\,\rangle_{\bfF_2}$ the pairing
$\bfF_2^*\times\bfF_2\to\C$)
has a limit as $z\to z_0$.
Let us denote
\[
C:=
\sup\sb{z\in\varOmega\cap\mathbb{D}_\delta(z_0)}
\norm{\jmath_1\circ R(z)\circ\imath_1}_{\bfE_1\shortto\bfF_1}
+
\sup\sb{z\in\varOmega\cap\mathbb{D}_\delta(z_0)}
\norm{\jmath_2\circ R(z)\circ\imath_2}_{\bfE_2\shortto\bfF_2}
<\infty.
\]
Let $\varepsilon\in(0,1)$.
Due to the density assumptions,
there are $v\in\bfE$
and
$h\in\bfF^*$
such that
\begin{eqnarray}\label{close}
\norm{\phi-\upalpha_2(v)}_{\bfE_2}<\varepsilon/(1+8C),
\qquad
\norm{\eta-\upbeta_2^*(h)}_{\bfF_2^*}<\varepsilon/(1+8C).
\end{eqnarray}
For
$z,\,z'\in\varOmega\cap\mathbb{D}_\delta(z_0)$,
one has:
\begin{eqnarray}\label{t-l-t}
&&
\nonumber
f(z)-f(z')
=
\big\langle\eta,\jmath_2\circ(R(z)-R(z'))\imath_2(\phi)
\big\rangle_{\bfF_2}
\\
&&
=
\big\langle\eta-\upbeta_2^*(h),
\jmath_2\circ(R(z)-R(z'))\imath_2(\phi)\big\rangle_{\bfF_2}
+
\big\langle\upbeta_2^*(h),
\jmath_2\circ(R(z)-R(z'))\imath_2(\phi-\upalpha_2(v))\big\rangle_{\bfF_2}
\nonumber
\\
&&
\qquad
+
\big\langle\upbeta_2^*(h),
\jmath_2\circ(R(z)-R(z'))\imath_2\circ\upalpha_2(v)\big\rangle_{\bfF_2}.
\end{eqnarray}
By our assumptions,
the absolute value of the first term
in the right-hand side is bounded by $\varepsilon/4$.
Taking into account that
$\norm{\upbeta_2^*(h)}_{\bfF_2^*}<1+\varepsilon<2$
in view of \eqref{close},
the absolute value of the second term
in the right-hand side of \eqref{t-l-t}
is bounded by $\varepsilon/2$.
The last term in the right-hand side of
\eqref{t-l-t}, which can be rewritten as
\[
\big\langle h,
\jmath\circ(R(z)-R(z'))\imath(v)\big\rangle_{\bfF}
=
\big\langle\upbeta_1^*(h),
\jmath_1\circ(R(z)-R(z'))\imath_1\circ\upalpha_1(v)\big\rangle_{\bfF_1},
\]
is bounded in the absolute value by $\varepsilon/4$
as soon as both $z$ and $z'$ are close enough to $z_0$,
due to the existence of
the limit of
$
\jmath_1\circ R(z)\circ\imath_1:\,\bfE_1\to\bfF_1$
in the weak operator topology,
and then $\abs{f(z)-f(z')}<\varepsilon$.
This proves the existence of
$\lim_{z\to z_0,\,z\in\varOmega}f(z)$,
showing that the convergence of $\jmath_2\circ R(z)\circ\imath_2$
holds in the weak operator topology.
This completes the proof of
Part~\itref{theorem-a-1a}.

\smallskip

Let us now prove Part~\itref{theorem-a-1b},
considering the case when both $\bfF_1$, $\bfF_2$, and $\bfF$
have pre-duals
and the continuous mappings \eqref{starr-1}
have dense ranges.
Fix $\phi\in\bfE_2$, $\norm{\phi}_{\bfE_2}=1$,
and $x\in(\bfF_2)_*$, $\norm{x}_{(\bfF_2)_*}=1$.
We need to show that
the function
$f(z):=\langle J_{(\bfF_2)_*}x,R(z)\phi\rangle_{\bfF_2}$, $z\in\varOmega$
has a limit as $z\to z_0$.
Let $\varepsilon\in(0,1)$.
Due to the density assumptions,
there are $v\in\bfE$
and
$y\in\bfF_*$
such that
\begin{eqnarray}\label{close-ws}
\norm{\phi-\upalpha_2(v)}_{\bfE_2}<\varepsilon/(1+8C),
\qquad
\norm{J_{(\bfF_2)_*}x-
\upbeta_2^*\circ J_{\bfF_*} y}_{\bfF_2^*}<\varepsilon/(1+8C).
\end{eqnarray}
Then, in place of \eqref{t-l-t},
one has:
\begin{eqnarray}\label{t-l-t-ws}
&&
f(z)-f(z')
=
\big\langle J_{(\bfF_2)_*}x,\jmath_2\circ(R(z)-R(z'))\imath_2(\phi)
\big\rangle_{\bfF_2}
\nonumber
\\
&&
=
\big\langle J_{(\bfF_2)_*}x- 
\upbeta_2^*\circ J_{\bfF_*} y
,
\jmath_2\circ(R(z)-R(z'))\imath_2(\phi)\big\rangle_{\bfF_2}
\nonumber
\\
&&
\qquad
\nonumber
+
\big\langle
\upbeta_2^* \circ J_{\bfF_*} y,
\jmath_2\circ(R(z)-R(z'))\imath_2(\phi-\upalpha_2(v))\big\rangle_{\bfF_2}
\nonumber
\\
&&
\qquad
+
\big\langle
\upbeta_2^*\circ J_{\bfF_*} y,
\jmath_2\circ(R(z)-R(z'))\imath_2\circ\upalpha_2(v)\big\rangle_{\bfF_2}.
\end{eqnarray}
Proceeding like after \eqref{t-l-t},
the absolute value of the first term
in the right-hand side is bounded by $\varepsilon/4$.
Since
$
\norm{\upbeta_2^* J_{\bfF_*} y}_{(\bfF_2)_*}
<1+\varepsilon<2$
in view of \eqref{close-ws},
the absolute value of the second term
in the right-hand side of \eqref{t-l-t-ws}
is bounded by $\varepsilon/2$.
The last term in the right-hand side of
\eqref{t-l-t-ws}
is bounded in the absolute value by $\varepsilon/4$
as soon as both $z$ and $z'$ are close enough to $z_0$,
due to the existence of
the limit of
$\jmath_1\circ R(z)\circ\imath_1:\,\bfE_1\to\bfF_1$
in the weak$^*$ operator topology,
and then $\abs{f(z)-f(z')}<\varepsilon$.
This proves the existence of
$\lim_{z\to z_0,\,z\in\varOmega}f(z)$,
showing that the convergence of
$\jmath_2\circ R(z)\circ\imath_2$
holds in the weak$^*$ operator topology.

Let us prove Part~\itref{theorem-a-2a}.
Let $A\in\scrC(\bfX)$ satisfy Assumption~\ref{ass-virtual}
(with respect to the weak operator topology).

\begin{lemma}\label{lemma-aff1}
Define
$A_{\bfF_\nu\shortto\bfF_\nu}:\;\bfF_\nu\to\bfF_\nu$,
$\nu=1,\,2$,
by
\begin{eqnarray}\label{def-a-nu}
\dom(A_{\bfF_\nu\shortto\bfF_\nu})
=\big\{
\Psi\in\bfF_\nu\sothat
\exists\psi\in\dom(A),\ \Psi=\jmath_\nu(\psi)
\big\},
\quad
A_{\bfF_\nu\shortto\bfF_\nu}:\;
\Psi\mapsto\jmath_\nu(A\psi).
\end{eqnarray}

The operators
$A_{\bfF_\nu\shortto\bfF_\nu}:\,\bfF_\nu\to\bfF_\nu$,
$\nu=1,\,2$,
have closed extensions
$\hatA_\nu\in\scrC(\bfF_\nu)$.

Moreover, for $\nu=1,\,2$,
one has
$\upbeta_\nu(\dom(\hatA_\nu))\subset \doma$ and
\[
\,\upbeta_\nu(\hatA_\nu\phi_\nu)=\hatA\upbeta_\nu(\phi_\nu),
\qquad
\forall\phi_\nu\in \dom(\hatA_\nu).
\]
\end{lemma}

\begin{proof}
According to the definition \eqref{def-a-nu}
of $A_{\bfF_\nu\shortto\bfF_\nu}$, one has
\begin{eqnarray}\label{bb}
(\upbeta_\nu\times\upbeta_\nu)
\left(
 \mathcal{G}\big(A_{\bfF_\nu\shortto\bfF_\nu}-z_0 I_{\bfF_\nu}\big)
 \right)
 \subset
 \mathcal{G}\big(\hatA-z_0 I_\bfF\big),
\qquad
\nu=1,\,2;
\end{eqnarray}
since the graph $\mathcal{G}\big(\hatA-z_0 I_\bfF\big)$
is a closed subspace of $\bfF\times\bfF$,
we conclude from \eqref{bb} that
operators
$A_{\bfF_\nu\shortto\bfF_\nu}-z_0 I_{\bfF_\nu}$,
$\nu=1,\,2$,
is closable in $\bfF_\nu$.
Moreover,
for $\nu=1,\,2$,
if $\hatA_\nu$ is a closed extension
of $A_{\bfF_\nu\shortto\bfF_\nu}$,
then
$(\upbeta_\nu\times\upbeta_\nu)
\left(
 \mathcal{G}\big(A_\nu-z_0 I_{\bfF_\nu}\big)
 \right)
 \subset
 \mathcal{G}\big(\hatA-z_0 I_\bfF\big)$,
so that for any $\phi_\nu\in \dom(\hatA_\nu)$
one has
$\upbeta_\nu(\phi_\nu)\in \doma$ and
$\upbeta_\nu(\hatA_\nu\phi_\nu)=\hatA\upbeta_\nu(\phi_\nu)$.
\end{proof}

Let $\nu=1$ or $2$.
By Lemma~\ref{lemma-aff1},
there is a closed
extension of $A_{\bfF_\nu\shortto\bfF\nu}$
which we denote by $\hatA_\nu\in\scrC(\bfF_\nu)$.
Let $\calB_\nu\in\scrQ_{\bfE_\nu,\bfF_\nu,\varOmega}(A-z_0 I_\bfX)\cap\scrB_{00}(\bfF_\nu,\bfF_\nu)$;
such an operator exists
due to Lemmata~\ref{lemma-b} and~\ref{lemma-q}.
Let $n=\rank\calB_\nu$
and let
$\phi_j\in\bfE_\nu$, $\eta_j\in\bfF_\nu^*$,
$1\le j\le n$,
be such that
$\calB_\nu=\sum_{j=1}^{n}\phi_j
\overline{\langle\eta_j,\,\cdot\,\rangle_{\bfF_\nu}}$
(the complex conjugation is needed
since the coupling with the dual space
is defined
to be $\C$-linear in the first component).
Since the
continuous
mappings
$\upalpha_\nu:\,\bfE\hookrightarrow\bfE_\nu$
and
$\upbeta_\nu^*:\,\bfF^*
\to\bfF_\nu^*$
have dense ranges by our assumptions,
for any $\varepsilon\in(0,1)$
there are $v_j\in\bfE$ and $h_j\in\bfF^*$,
$1\le j\le n$,
such that
$\calC=\sum_{j=1}^n v_j
\overline{\langle h_j,\,\cdot\,\rangle_{\bfF}}
\in\scrB_{00}(\bfF,\bfE)$
satisfies
\[
\norm{\calB_\nu-\upalpha_\nu\circ\calC\circ\upbeta_\nu}_{\scrB(\bfF_\nu,\bfE_\nu)}
=
\Norm{\sum_{j=1}^{n}\phi_j
\overline{\langle\eta_j,\,\cdot\,\rangle_{\bfF_\nu}}
-
\sum_{j=1}^{n}
\upalpha_\nu(v_j)
\overline{\langle
 \upbeta_\nu^*(h_j),\,\cdot\,\rangle_{\bfF_\nu}}
}_{\scrB(\bfF_\nu,\bfE_\nu)}
<\varepsilon.
\]
We choose
$\varepsilon=\big(1+\norm{(A+B_\nu-z_0 I_\bfX)^{-1}_{\bfE_\nu,\bfF_\nu,\varOmega}}\big)^{-1}$,
with $B_\nu=\imath_\nu\circ\calB_\nu\circ\jmath_\nu\in\scrB_{00}(\bfX)$;
then,
by Lemma~\ref{lemma-relatively-open}~\itref{lemma-relatively-open-2},
$\upalpha_\nu\circ\calC\circ\upbeta_\nu\in
\scrQ_{\bfE_\nu,\bfF_\nu,\varOmega}(A-z_0 I_\bfX)$.
Denote
$C=\imath\circ\calC\circ\jmath\in\scrB_{00}(\bfX)$.
Since the limits
\begin{eqnarray}
\label{trace-is-trace}
&&
\lim\sb{z\to z_0,\,z\in\varOmega}
\jmath\circ(A+C-z_0 I_\bfX)^{-1}\circ\imath
\\
\nonumber
&&
=
\upbeta_\nu\circ
\lim\sb{z\to z_0,\,z\in\varOmega}
\jmath_\nu\circ(A+C-z_0 I_\bfX)^{-1}
\circ\imath_\nu\circ\upalpha_\nu
=
\upbeta_\nu\circ
(A+C-z_0 I_\bfX)^{-1}_{\bfE_\nu,\bfF_\nu,\varOmega}
\circ\upalpha_\nu
\end{eqnarray}
exist
in the weak operator topology,
one has
$\calC\in\scrQ_{\bfE,\bfF,\varOmega}(A-z_0 I_\bfX)$.

Let us show that
there is the inclusion
\begin{eqnarray}\label{inclusions-old}
\frakM_{\bfE,\bfF,\varOmega}(A-z_0 I_\bfX)
\subset
\upbeta_\nu\big(\frakM_{\bfE_\nu,\bfF_\nu,\varOmega}(A-z_0 I_\bfX)\big),
\end{eqnarray}
where
$\frakM_{\bfE,\bfF,\varOmega}(A-z_0 I_\bfX)
=
\big\{
\Theta\in
\range\big((A+C-z_0 I_\bfX)^{-1}_{\bfE,\bfF,\varOmega}\big):\,
(\hatA-z_0 I_{\bfF})\Theta=0
\big\}$.
If $\Psi\in\frakM_{\bfE,\bfF,\varOmega}(A-z_0 I_\bfX)$,
then there is $\phi\in\bfE$
such that
$\Psi=(A+C-z_0 I_\bfX)^{-1}_{\bfE,\bfF,\varOmega}\phi$;
by \eqref{trace-is-trace},
$
\Psi=\upbeta_\nu(\Psi_\nu)$,
with
$\Psi_\nu
:=(A+C-z_0 I_\bfX)^{-1}_{\bfE_\nu,\bfF_\nu,\varOmega}
\upalpha_\nu(\phi)$.
By Lemma~\ref{lemma-jk}~\itref{lemma-jk-jk1},
one has
$\Psi\in
\dom(\hatA+\jmath\circ\imath\circ\calC)
=\dom(\hatA)$,
$\Psi_\nu\in
\dom(\hatA_\nu+\jmath_\nu\circ\imath\circ\calC\circ\upbeta_\nu)
=\dom(\hatA_\nu)$,
and
\begin{eqnarray}
\label{is-zero}
(\hatA-z_0 I_\bfF)\Psi
&=&
\jmath\circ\imath(\phi)
-
\jmath\circ\imath\circ\calC\Psi,
\\[1ex]
\label{is-zero-nu}
(\hatA_\nu-z_0 I_\bfF)\Psi_\nu
&=&
\jmath_\nu\circ\imath_\nu
\circ\upalpha_\nu(\phi)
-
\jmath_\nu\circ\imath\circ
\calC\circ\upbeta_\nu(\Psi_\nu).
\end{eqnarray}
By \eqref{is-zero} and
\eqref{is-zero-nu},
\[
\upbeta_\nu\big(
(\hatA_\nu-z_0 I_\bfF)\Psi_\nu
\big)
=
\upbeta_\nu\big(
\jmath_\nu\circ\imath_\nu
\circ\upalpha_\nu(\phi)
-
\jmath_\nu\circ\imath\circ
\calC\circ\upbeta_\nu(\Psi_\nu)
\big)
=
\jmath\circ\imath(\phi)
-
\jmath\circ\imath\circ\calC\Psi
=0;
\]
we took into account that
$\Psi\in\frakM_{\bfE,\bfF,\varOmega}(A-z_0 I_\bfX)$.
We conclude that
$(\hatA_\nu-z_0 I_\bfF)\Psi_\nu=0$,
hence
$\Psi_\nu\in\frakM_{\bfE_\nu,\bfF_\nu,\varOmega}(A-z_0 I_\bfX)$,
hence
$\Psi=\upbeta_\nu(\Psi_\nu)
\in\upbeta_\nu\big(\frakM_{\bfE_\nu,\bfF_\nu,\varOmega}(A-z_0 I_\bfX)\big)$,
proving
\eqref{inclusions-old}.

Now let us prove the reverse inclusion.
Let us assume that
$\Psi_\nu\in\frakM_{\bfE_\nu,\bfF_\nu,\varOmega}(A-z_0 I_\bfX)\subset\bfF_\nu$,
so that
$(\hatA-z_0 I_{\bfF})\upbeta_\nu(\Psi_\nu)=0$
and there is
$\phi_\nu\in\bfE_\nu$
such that
$
\Psi_\nu=(A+C-z_0 I_\bfX)^{-1}_{\bfE_\nu,\bfF_\nu,\varOmega}\phi_\nu\in\bfF_\nu
$,
with
$C=\imath\circ\calC\circ\jmath\in\scrB_{00}(\bfX)$,
$\calC\in\scrQ_{\bfE,\bfF,\varOmega}(A-z_0 I_\bfX)\cap\scrB_{00}(\bfF,\bfE)$
the same as above.
By Lemma~\ref{lemma-jk}~\itref{lemma-jk-jk1},
\[
\jmath_\nu\circ\imath_\nu(\phi_\nu)=(\hatA
+\jmath_\nu\circ\imath\circ\calC\circ\upbeta_\nu-z_0 I_{\bfF_\nu})\Psi_\nu
=\jmath_\nu\circ\imath\circ\calC\circ\upbeta_\nu
(\Psi_\nu)\in\jmath_\nu\circ\imath(\bfE),
\]
so
$\phi_\nu=\upalpha_\nu(\phi)$
for some $\phi\in\bfE$,
and then, by \eqref{trace-is-trace},
\[
\upbeta_\nu(\Psi_\nu)=
\upbeta_\nu\circ(A+C-z_0 I_\bfX)^{-1}_{\bfE_\nu,\bfF_\nu,\varOmega}
\circ
\upalpha_\nu(\phi)
=
(A+C-z_0 I_\bfX)^{-1}_{\bfE,\bfF,\varOmega}(\phi).
\]
Thus,
$\upbeta_\nu(\Psi_\nu)
\in\range\big((A+C-z_0 I_\bfX)^{-1}_{\bfE,\bfF,\varOmega}\big)$
belongs to
$\frakM_{\bfE,\bfF,\varOmega}(A-z_0 I_\bfX)$,
implying that
\begin{eqnarray}\label{tai-old}
\upbeta_\nu\big(\frakM_{\bfE_\nu,\bfF_\nu,\varOmega}(A-z_0 I_\bfX)\big)
\subset
\frakM_{\bfE,\bfF,\varOmega}(A-z_0 I_\bfX).
\end{eqnarray}
Due to the inclusion \eqref{tai-old},
the inclusion \eqref{inclusions-old} is actually equality;
it follows that
\[
\upbeta_1\big(\frakM_{\bfE_1,\bfF_1,\varOmega}(A-z_0 I_\bfX)\big)
=
\frakM_{\bfE,\bfF,\varOmega}(A-z_0 I_\bfX)
=
\upbeta_2\big(\frakM_{\bfE_2,\bfF_2,\varOmega}(A-z_0 I_\bfX)\big).
\]

Let us prove Part~\itref{theorem-a-2b},
the case when
$\bfF_1$ and $\bfF_2$ have pre-duals
continuously embedded into some
Hausdorff vector space $\bfG_0$
which are mutually dense.
We identify the subset
$
\upchi_1\big((\bfF_1)_*\big)\cap\upchi_2\big((\bfF_2)_*\big)
\subset\bfG_0
$
with $(\bfF_1)_*\cap(\bfF_2)_*$,
with the norm
(cf. \eqref{x-phi})
\[
\norm{y}_{(\bfF_1)_*\cap(\bfF_2)_*}
=\max\big(
\norm{\chi_1^{-1}(y)}_{(\bfF_1)_*},\norm{\chi_2^{-1}(y)}_{(\bfF_2)_*}
\big),
\]
and denote the continuous embeddings
corresponding to \eqref{chi-chi} by
\begin{eqnarray}\label{def-gamma-nu}
\upgamma_\nu
=
\upchi_\nu^{-1}\at
{\upchi_1((\bfF_1)_*)\cap\upchi_2((\bfF_2)_*)}
:\;
(\bfF_1)_*\cap(\bfF_2)_*
\hookrightarrow
(\bfF_\nu)_*,
\qquad
\nu=1,\,2.
\end{eqnarray}
Since the ranges of these continuous embeddings are dense,
their adjoints are also continuous embeddings:
\[
\upbeta_\nu:=\upgamma_\nu^*:\,
\bfF_\nu\hookrightarrow
((\bfF_1)_*\cap(\bfF_2)_*)^*
=\bfF_1+\bfF_2=\bfF.
\]
It follows that $\bfF$ has a pre-dual $\bfF_*$
which can be canonically identified with
$(\bfF_1)_*\cap(\bfF_2)_*$.

\begin{lemma}\label{lemma-aff2}
There are closed extensions $\hatA_\nu\in\scrC(\bfF_\nu)$
of $A_{\bfF_\nu\shortto\bfF_\nu}:\,\bfF_\nu\to\bfF_\nu$, $\nu=1,\,2$,
which are weak$^*$-closed.
Moreover, for $\nu=1,\,2$,
one has
$\upbeta_\nu(\dom(\hatA_\nu))\subset \doma$ and
\[
\upbeta_\nu(\hatA_\nu\phi_\nu)=\hatA\upbeta_\nu(\phi_\nu),
\qquad
\forall\phi_\nu\in \dom(\hatA_\nu).
\]
\end{lemma}

\begin{proof}
The existence of closed extensions
was proved in Lemma~\ref{lemma-aff1}.
If additionally both $\bfF_1$ and $\bfF_2$ have pre-duals
which are mutually dense,
denoted by $(\bfF_1)_*$ and $(\bfF_2)_*$, respectively,
with dense continuous mappings
$\upgamma_\nu:\,\bfF_*\to(\bfF_\nu)_*$,
where $\bfF_*=(\bfF_1)_*\cap(\bfF_2)_*$,
then
$\upbeta_\nu=\upgamma_\nu^*:\,\bfF_\nu\hookrightarrow\bfF$
are weak$^*$-continuous.
Since the graph $\mathcal{G}\big(\hatA\big)$
is assumed weak$^*$-closed,
we conclude from \eqref{bb} that the operators
$A_{\bfF_\nu\shortto\bfF_\nu}-z_0 I_{\bfF_\nu}:\,\bfF_\nu\to\bfF_\nu$
are weak$^*$-closable.

The last assertion
of the lemma
is proved as in Lemma~\ref{lemma-aff1}.
\end{proof}

By Lemma~\ref{lemma-aff2},
there are closed and weak$^*$-closed
extensions of $A_{\bfF_\nu\shortto\bfF\nu}$,
which we denote
by $\hatA_\nu\in\scrC(\bfF_\nu)$, $\nu=1,\,2$.
Let $\calB_\nu\in\scrQ_{\bfE_\nu,\bfF_\nu,\varOmega}(A-z_0 I_\bfX)\cap\scrB_{00}(\bfF_\nu,\bfF_\nu)$;
such an operator exists
due to Lemmata~\ref{lemma-b} and~\ref{lemma-q}.
Let $n=\rank\calB_\nu$
and let
$\phi_j\in\bfE_\nu$, $y_j\in(\bfF_\nu)_*$,
$1\le j\le n$,
be such that
$\calB_\nu=\sum_{j=1}^{n}\phi_j
\overline{\langle J_{(\bfF_\mu)_*}y_j,\,\cdot\,\rangle_{\bfF_\nu}}$
Since the
continuous
maps
$\upalpha_\nu:\,\bfE\hookrightarrow\bfE_\nu$
and
$\upgamma_\nu:\,\bfF_*
\to(\bfF_\nu)_*$
have dense ranges
by our assumptions,
for any $\varepsilon\in(0,1)$
there are $v_j\in\bfE$ and $g_j\in\bfF_*$, $1\le j\le n$,
such that
$\calC=\sum_{j=1}^n v_j
\overline{\langle J_{(\bfF_\nu)_*}g_j,\,\cdot\,\rangle_{\bfF}}
\in\scrB_{00}(\bfF,\bfE)$
satisfies
\[
\norm{\calB_\nu-\upalpha_\nu\circ\calC\circ\upgamma_\nu^*
}_{\scrB(\bfF_\nu,\bfE_\nu)}
=
\Norm{\sum_{j=1}^{n}\phi_j
\overline{\langle J_{(\bfF_\nu)_*}y_j,\,\cdot\,\rangle_{\bfF_\nu}}
-
\sum_{j=1}^{n}
\upalpha_\nu(v_j)
\overline{\langle
\upgamma_\nu(g_j),\,\cdot\,\rangle_{\bfF_\nu}}
}_{\scrB(\bfF_\nu,\bfE_\nu)}
<\varepsilon.
\]
The rest of the proof repeats the proof of
Part~\itref{theorem-a-2a},
with the convergence in weak$^*$ operator topology
in place of weak operator topology
in \eqref{trace-is-trace}.
\end{proof}

According to
\eqref{free-resolvent-1d}
and Example~\ref{example-laplace-2d},
the point
$z_0=0$ is not a regular point of the essential spectrum
of the Laplace operator in $\R^d$ for $d\le 2$
for e.g. $\bfE=L^2_s(\R^d)$ and $\bfF=L^2_{-s'}(\R^d)$
with arbitrarily large $s,\,s'\ge 0$.
Moreover, one can show that
if $s+s'>2$ and $s,\,s'>d/2$,
then the point $z_0=0$ is a virtual level of rank one.

Let us show that one can choose particular $\bfE$ and $\bfF$
so that
$z_0=0$ becomes a regular point of the essential spectrum
of the Laplace operator in $L^2(\R^d)$, $d\le 2$,
relative to $\big(\bfE,\bfF,\cnor\big)$.
The construction in the following example
is based on the idea which we learned from
Prof. Roman V. Romanov.

\begin{example}\label{example-roman}
While $z_0=0$ is 
a virtual level of rank $1$ of $-\Delta$ in $L^2(\R^2)$
relative to
$\big(L^2_s(\R^2),L^2_{-s'}(\R^2),\varOmega\big)$,
with
$\varOmega=\C\setminus\overline{\mathbb{D}}$
and $s,\,s'>1$
(this follows from
Example~\ref{example-laplace-2d}
and Theorem~\ref{theorem-2d} below),
it is a regular point of the essential spectrum of $-\Delta$
relative to
$\big((I-P)L^2_{s}(\R^2),L^2_{-s'}(\R^2),\varOmega\big)$,
with $P=\varPhi\otimes 1\in\scrB_{00}\big(L^2_s(\R^2)\big)$,
$\varPhi\in L^2_{s}(\R^2)$,
$\int_{\R^2}\varPhi(x)\,dx=1$
(so that $P$ is a projector in $L^2_s(\R^2)$);
this follows Lemma~\ref{lemma-2d-angular-vnew} below.
We note that both
$L^2_s(\R^2)$ and
$(I-P)L^2_{s}(\R^2)$ are densely embedded into $L^2(\R^2)$,
while $L^2_s(\R^2)\cap (I-P)L^2_{s}(\R^2)
=(I-P)L^2_{s}(\R^2)$
is not dense in $L^2_s(\R^2)$,
violating the density assumptions
of Theorem~\ref{theorem-a}~\itref{theorem-a-2}.
\end{example}

The following example shows that
if $A\in\scrC(\bfX)$
has a virtual level at $z_0\in\sigma\sb{\mathrm{ess}}(A)$
relative to $\big(\bfE,\bfF,\varOmega\big)$,
$z\in\p\varOmega$,
$\bfE
\mathop{\hookrightarrow}\limits\sp\imath
\bfX
\mathop{\hookrightarrow}\limits\sp\jmath
\bfF$,
then there is not necessarily a projection
$P\in\scrB_{00}(\bfE)$
such that $\jmath\circ(A-z I)^{-1}\circ\imath\circ(I_\bfE-P)$
would have a limit as
$z\to z_0$, $z\in\varOmega$,
even if $\bfF$ is substituted by some bigger space
$\bfF\hookrightarrow\bfG$.

\begin{example}
\label{example-2}
Let $L$ be the left shift in
$\ell^2(\N)$
and let
$\imath:\,\ell^1(\N)\hookrightarrow\ell^2(\N)$
and
$\jmath:\,\ell^2(\N)\hookrightarrow\ell^\infty(\N)$
be the corresponding
continuous embeddings.
Let
$
K=\sum\sb{i\in\N}\e_1\langle\,\cdot\,,\e_i\rangle_{c_0}
\in\scrB_{00}\big(\ell^1(\N)\big)
$,
where we consider $\e_i$ as elements of $c_0(\N)$,
$\sigma(
K)=\{0,1\}$,
$K\e_1=\e_1$.
We claim that, given any
$\bfE_0$ which is a closed vector subspace of $\ell^1(\N)$
of finite codimension,
the resolvent $(A-z I)^{-1}$ has no limit
in the weak$^*$ operator topology
of $\scrB\big(\bfE_0,\ell^\infty(\N)\big)$
as $z\to z_0=1$, $\abs{z}>1$.
 In the matrix form (cf. Example~\ref{example-shift}),
 \[
 K
 =
 \begin{bmatrix}
 1&1&1&\dots
 \\
 0&0&0&\dots
 \\
 \vdots&\vdots&\vdots&\ddots
 \end{bmatrix}
 \qquad
 I-L
 =
 \begin{bmatrix}
 1&-1&0&\dots
 \\
 0&1&-1&\dots
 \\
 \vdots&\vdots&\vdots&\ddots
 \end{bmatrix}
 \qquad
 K(I-L)
 =
 \begin{bmatrix}
 1&0&0&\dots
 \\
 0&0&0&\dots
 \\
 \vdots&\vdots&\vdots&\ddots
 \end{bmatrix}.
 \]
We consider the mapping
$A=L+K(I-L):\,\ell^2(\N)\to\ell^2(\N)$.
For $z\in\C$, $\abs{z}>1$, we have:
\[
(A-z I)^{-1}
=
(L-z I)^{-1}(I+K(I-L)(L-z I)^{-1})^{-1}.
\]
Using the expressions
for $(L-z I )^{-1}$
from Example~\ref{example-shift},
we derive:
\[
M:=
K(I-L)(L-z I)^{-1}
=
-
\begin{bmatrix}
z^{-1}&z^{-2}&\dots
\\
0&0&\dots
\\
\vdots&\vdots&\ddots
\end{bmatrix},
\quad
M^n
=(-1)^n z^{1-n}
\begin{bmatrix}
z^{-1}&z^{-2}&\dots
\\
0&0&\dots
\\
\vdots&\vdots&\ddots
\end{bmatrix},
\]
\[
\big(I+K(I-L)(L-z I)^{-1}\big)^{-1}
=I+\sum_{n\in\N}
(-M)^n
=
I+
\frac{1}{1-z^{-1}}
\begin{bmatrix}
z^{-1}&z^{-2}&\dots
\\
0&0&\dots
\\
\vdots&\vdots&\ddots
\end{bmatrix},
\]
\[
(A-z I)^{-1}
=
(L-z I)^{-1}
\big(I+K(I-L)(L-z I)^{-1}\big)^{-1}
=
(L-z I)^{-1}
-
\frac{1}{z-1}
\begin{bmatrix}
z^{-1}&z^{-2}&\dots
\\
0&0&\dots
\\
\vdots&\vdots&\ddots
\end{bmatrix}.
\]
Thus,
given $\bm{u}\in\ell^1(\N)$,
the vector
$(A-z I)^{-1}\bm{u}$, $\abs{z}>1$, is given by
\begin{eqnarray}\label{lzi}
(L-z I)^{-1}\bm{u}
-
\frac{1}{z-1}
\begin{bmatrix}
z^{-1}&z^{-2}&\dots
\\
0&0&\dots
\\
\vdots&\vdots&\ddots
\end{bmatrix}
\begin{bmatrix}u_1\\u_2\\\vdots
\end{bmatrix}
=
-\begin{bmatrix}
\sum\sb{j\ge 1}u_j z^{-j}
\\
z\sum\sb{j\ge 2}u_j z^{-j}
\\\vdots
\end{bmatrix}
-
\sum_{j\ge 1}\frac{u_j z^{-j}}{z-1}
\begin{bmatrix}
1
\\
0
\\
\vdots
\end{bmatrix}.
\quad
\end{eqnarray}
For the last term in the right-hand side to
remain finite in the limit $z\to z_0=1$,
one needs
$\sum_{j\in\N}u_j=0$.
Then
\begin{eqnarray}\label{isf}
\sum_{j\in\N}\frac{u_j z^{-j}}{z-1}
=\sum_{j\in\N}\frac{u_j(z^{-j}-1)}{z-1};
\end{eqnarray}
this expression converges as $z\to 1+$
if and only if
$\sum_{j\in\N} j u_i $ is convergent.
Let
$\bfE_1
=\{\bm{u}\in\ell^1(\N)\sothat\sum\sb{j\in\N}u_j=0\}
=(1,1,1,\dots)^\perp
$,
or, more generally,
$\bfE_1=\cap_{k=1}^n\bm{r}_k\sp\perp\subset\ell^1(\N)$,
for any finite number of vectors $\bm{r}_k\in\ell^\infty(\N)$,
$1\le k\le n$,
with $\bm{r}_1=(1,1,1,\dots)$.
There is $N\ge n$ and
a vector $\bm{u}\in\bfE_1$ with components
$u_j=j^{-3/2}$ for
$j\ge N+1$
and with the first $N$
components $u_j$, $1\le j\le N$,
such that
$\langle\bm{r}_k,\bm{u}\rangle_{\ell^1}=0$
for all $k\ge 1$, $k\le n$.
Then
$\sum\sb{j=1}^{J}j u_j\to +\infty$ as $J\to\infty$, hence
the absolute value of \eqref{isf}
tends to infinity
and then the same is true for
$\langle \e_1,
\jmath\circ(A-z I)^{-1}\imath(\bm{u})\rangle_{\ell^\infty}$
as $z\to z_0=1$, $\abs{z}>1$.
Since 
$\jmath\circ(A-z I)^{-1}\imath$ is not uniformly bounded,
there is no convergence
neither in the weak nor in the weak$^*$ operator topology
of $\scrB(\bfE_1,\ell^\infty(\N))$
as $z\to z_0=1$, $\abs{z}>1$.
Moreover,
since in \eqref{lzi}
only the last term in the right-hand side
is unbounded in the limit $z\to z_0=1$, $\abs{z}>1$,
if there is a continuous embedding
$\upchi:\,\ell^\infty(\N)\hookrightarrow\bfG$
of $\ell^\infty(\N)$ into some Banach space $\bfG$,
then
$\upchi\circ\jmath\circ(A-z I)^{-1}\circ\imath:\,
\bfE_1\to\bfG$
is not bounded uniformly in $\abs{z}>1$
and hence has no limit in the weak
or weak$^*$ (if $\bfG$ has a pre-dual)
operator topology of $\scrB(\bfE_1,\bfG)$
as $z\to z_0=1$, $\abs{z}>1$.
\end{example}

\subsection{Derivation of the LAP properties}
\label{sect-factorization}

Now we will consider the resolvent
as an operator $\tilde\calR(z)$
from $\tilde\bfE$ to $\tilde\bfF$
in a more general situation,
when there are
continuous embeddings
$\bfE\mathop{\longhookrightarrow}\limits\sp{\imath}
\bfX\mathop{\longhookrightarrow}\limits\sp{\jmath}
\bfF$,
$\bfE\mathop{\longhookrightarrow}\limits\sp{\upvarepsilon}
\tilde\bfE$,
$\tilde\bfF
\mathop{\longhookrightarrow}\limits\sp{\upvarphi}\bfF$,
as on Figure~\ref{fig-relation-2}.

\begin{figure}[ht]
\begin{picture}(0,100)(-200,-20)
\put(38,-21){$\scriptstyle\calR(z)$}
\qbezier(45,-13)(20,-13)(3,-5)
\qbezier(45,-13)(70,-13)(85,-4)
\put(83,-5){\vector(2,1){5}}
\qbezier(45,20)(20,20)(3,12)
\qbezier(45,20)(70,20)(85,11)
\put(6,13.5){\vector(-2,-1){5}}
\put(44,23){$\scriptstyle\calB$}
\put(2, 64){\vector(1,0){85}}
\put(40,69){$\scriptstyle\tilde\calR(z)$}
\put(-15,17){\vector(0,1){35}}
\put( -23,30){$\upvarepsilon$}
\put( 101,33){$\upvarphi$}
\put( 99,52){\vector(0,-1){35}}
\put(-20,60){$\tilde\bfE\qquad\qquad\qquad\qquad\quad\tilde\bfF$}
\put(17,6){$\scriptstyle\imath$}
\put(69,6){$\scriptstyle\jmath$}
\put(-20,0){$\bfE
\quad\ \longhookrightarrow\ \ \bfX
\ \ \longhookrightarrow\ \ \,\bfF $}
\end{picture}
\caption{Resolvent considered as a mapping
$\tilde\calR:\,\tilde\bfE\to\tilde\bfF$.}
\label{fig-relation-2}
\end{figure}

We will show that if the resolvent
estimates could be improved to some spaces
$\tilde\bfE\to\tilde\bfF$, such that
$\bfE\hookrightarrow\tilde\bfE$,
$\tilde\bfF\hookrightarrow\bfF$
while $\tilde\bfE\not\hookrightarrow\bfX$
and
$\bfX\not\hookrightarrow\tilde\bfF$,
then the same improved estimates
will take place under certain relatively compact perturbations,
unless the corresponding point
becomes a virtual level.
Another important point is that
the trace of the resolvent at
one point of the essential spectrum
could have better regularity properties
than at the nearby points;
see
Example~\ref{example-d5} below.
It turns out that this improved
regularity also persists under
certain relatively compact perturbations
(unless the corresponding point
becomes a virtual level).

\begin{theorem}
\label{theorem-factorization}
Let $\bfE$ and $\bfF$ be Banach spaces with
continuous embeddings
$\bfE\mathop{\longhookrightarrow}\limits\sp{\imath}\bfX
\mathop{\longhookrightarrow}\limits\sp{\jmath}\bfF$.
let $A\in\scrC(\bfX)$ satisfy Assumption~\ref{ass-virtual},
let $\varOmega\subset\C\setminus\sigma(A)$,
and let
$z_0\in\sigma\sb{\mathrm{ess}}(A)\cap\p\varOmega$.
Assume that the mapping
$\jmath\circ(A-z I_\bfX)^{-1}\circ\imath:\,\bfE\to\bfF$,
$z\in\varOmega$,
has a limit as $z\to z_0$, $z\in\varOmega$,
in the weak or weak$^*$
operator topology of $\scrB(\bfE,\bfF)$,
denoted by
$(A-z_0 I_\bfX)^{-1}_{\bfE,\bfF,\varOmega}:\,\bfE\to\bfF$.
Assume further that
$\calB:\,\bfF\to\bfE$ is $\hatA$-compact
and that any -- and hence all --
of the statements of
Theorem~\ref{theorem-lap}~\itref{theorem-lap-2}
are satisfied.
Let there be Banach spaces $\tilde\bfE,\,\tilde\bfF$
with continuous embeddings
(see Figure~\ref{fig-relation-2})
\[
\bfE
\mathop{\longhookrightarrow}\limits\sp{\upvarepsilon}
\tilde\bfE,
\qquad
\tilde\bfF
\mathop{\longhookrightarrow}\limits\sp{\upvarphi}
\bfF.
\]
\begin{enumerate}
\item
\label{theorem-factorization-1}
Assume that there is a factorization
\begin{eqnarray}\label{factorization-2a}
(A-z_0 I_\bfX)^{-1}_{\bfE,\bfF,\varOmega}
=
\upvarphi\circ\tilde\calR\circ\upvarepsilon,
\quad
\mbox{with}
\quad
\tilde\calR\in\scrB\big(\tilde\bfE,\tilde\bfF\big),
\end{eqnarray}
and assume that
\begin{eqnarray}\label{ebfr}
\upvarepsilon\circ\calB\circ\upvarphi\circ\tilde\calR
\in\scrB_0(\tilde\bfE).
\end{eqnarray}
Then:
\begin{enumerate}
\item\label{theorem-factorization-1a}
The operator
$(A+B-z_0 I_\bfX)^{-1}_{\bfE,\bfF,\varOmega}
\in\scrB(\bfE,\bfF)$,
with $B=\imath\circ\calB\circ\jmath$,
can be
factored into
\begin{eqnarray}\label{factorization}
(A+B-z_0 I_\bfX)^{-1}_{\bfE,\bfF,\varOmega}
=
\upvarphi\circ\tilde\calR_{B}\circ\upvarepsilon,
\quad
\mbox{with}
\quad
\tilde\calR_{B}
\in\scrB\big(\tilde\bfE,\tilde\bfF\big);
\end{eqnarray}
\item\label{theorem-factorization-1b}

For any $\lambda\in\C\setminus\{0\}$,
there are the
following relations:
\begin{eqnarray*}
\upvarepsilon\big(
\ker\big(
\lambda I_{\bfE}-\calB(A-z_0 I_\bfX)^{-1}_{\bfE,\bfF,\varOmega}\big)
\big)
&=&
\ker\big(\lambda I_{\tilde\bfE}-\upvarepsilon\circ\calB\circ\upvarphi\circ\tilde\calR\big),
\\[1ex]
\upvarphi\big(
\ker\big(\lambda I_{\tilde\bfF}
-\tilde\calR\circ\upvarepsilon\circ\calB\circ\upvarphi\big)
\big)
&=&
\ker\big(\lambda I_{\bfF}
-(A-z_0 I_\bfX)^{-1}_{\bfE,\bfF,\varOmega}\calB\big)
.
\qquad
\end{eqnarray*}
\end{enumerate}

\item
\label{theorem-factorization-2}
If there is a factorization
\begin{eqnarray}\label{factorization-2b}
\jmath\circ(A-z I_\bfX)^{-1}\circ\imath
=
\upvarphi\circ\tilde\calR(z)
\circ\upvarepsilon,
\qquad
z\in\varOmega,
\end{eqnarray}
such that the operator family
$\big\{\tilde\calR(z)\big\}_{z\in\varOmega}$
is convergent
in the weak$^*$ or weak or strong operator topology
of $\scrB\big(\tilde\bfE,\tilde\bfF\big)$
as $z\to z_0$
to
$\tilde\calR(z_0)\in\scrB\big(\tilde\bfE,\tilde\bfF\big)$
and such that
the operator family
$\big\{\upvarepsilon\circ\calB\circ\upvarphi\circ
\tilde\calR(z)\in\scrB_0(\tilde\bfE)\}_{z\in\varOmega}$
is collectively compact
(this assumption is automatically satisfied
if $\tilde\bfE=\bfE$),
then there is $\delta>0$ such that
for $z\in\varOmega\cap\mathbb{D}_\delta(z_0)$
there is a factorization
\begin{eqnarray}\label{factorization-0}
\jmath\circ(A+B-z I_\bfX)^{-1}\circ\imath
=
\upvarphi\circ\tilde\calR_B(z)
\circ\upvarepsilon,
\end{eqnarray}
with $B=\imath\circ\calB\circ\jmath$,
with the operator family
$\big\{\tilde\calR_B(z)\}_{z\in\varOmega\cap\mathbb{D}_\delta(z_0)}$
convergent
in the weak$^*$ or weak or strong operator topology
of $\scrB\big(\tilde\bfE,\tilde\bfF\big)$,
respectively,
to
$
\tilde\calR(z_0)\circ\big(I_{\tilde\bfE}
+\upvarepsilon\circ
\calB\circ\upvarphi\circ\tilde\calR(z_0)\big)^{-1}
$
as $z\to z_0$.

If, moreover,
\begin{eqnarray}\label{factorization-3}
\upvarepsilon\circ\calB\circ\upvarphi\circ\tilde\calR(z)\to\upvarepsilon\circ\calB\circ\upvarphi\circ\tilde\calR(z_0)
\qquad
\mbox{ as \ $z\to z_0$, \ $z\in\varOmega$,}
\end{eqnarray}
in the uniform operator topology
of $\scrB\big(\tilde\bfE\big)$,
then
the convergence $\tilde\calR_B(z)\to\tilde\calR_B(z_0)$ as $z\to z_0$, $z\in\varOmega$,
is also in the uniform operator topology of $\scrB\big(\tilde\bfE,\tilde\bfF\big)$.
\end{enumerate}
\end{theorem}

Let us note that
if
the embedding $\upvarepsilon:\,\bfE\hookrightarrow\tilde\bfE$
is assumed to be dense,
then the mappings
$\tilde\calR_{B}$ in \eqref{factorization}
and
$\tilde\calR_B(z)$ in \eqref{factorization-0}
are
defined uniquely.

We notice that in Example~\ref{Ex:FreeDim3},
according to
\cite[p.\,589]{jensen1979spectral},
the null spaces $\mathfrak{M}$
(kernel of $I+R_0(0)V$ in $H^{1,-s}$)
and $\mathfrak{N}$
(kernel of $I+V R_0(0)$ in $H^{-1,s}$)
do not depend on $s$ (in a certain interval);
this can be considered
as a particular case of
Theorem~\ref{theorem-factorization}~\itref{theorem-factorization-1b}.

\begin{proof}
Let us prove Part~\itref{theorem-factorization-1a}.
We start with the relation \eqref{l-u-c-zero-new}:
\begin{eqnarray}\label{l-u-c-zero-new-1}
(A+B-z_0 I_\bfX)^{-1}_{\bfE,\bfF,\varOmega}
&=&
\calR(z_0)
\big(I_\bfE+\calB\calR(z_0)\big)^{-1}
\nonumber
\\
&=&
\upvarphi\circ\tilde\calR
\circ\upvarepsilon
\circ
\big(I_\bfE+\calB\circ\upvarphi\circ\tilde\calR\circ\upvarepsilon\big)^{-1}
:\;\bfE\to\bfF.
\end{eqnarray}
Using the identity
$
\big(I_{\tilde\bfE}+\upvarepsilon
\circ\calB\circ\upvarphi\circ\tilde\calR\big)
\circ\upvarepsilon
=
\upvarepsilon\circ
\big(I_{\bfE}+\calB\circ\upvarphi\circ\tilde\calR\circ\upvarepsilon\big)
$,
we arrive at
\[
\upvarepsilon\circ
\big(I_{\bfE}+\calB
\circ\upvarphi\circ\tilde\calR\circ\upvarepsilon
\big)^{-1}
=
\big(I_{\tilde\bfE}
+
\upvarepsilon\circ\calB\circ\upvarphi\circ
\tilde\calR\big)^{-1}
\circ\upvarepsilon
:\;\bfE\to\tilde\bfE.
\]
(Let us mention that,
as follows from our assumptions,
the operator
$\upvarepsilon\circ\calB\circ\upvarphi\circ\tilde\calR$
is compact and does not have eigenvalue $-1$.)
Therefore, \eqref{l-u-c-zero-new-1} can be rewritten as
\[
(A+B-z_0 I_\bfX)^{-1}_{\bfE,\bfF,\varOmega}
=
\upvarphi\circ\tilde\calR\circ
\big(I_{\tilde\bfE}
+\upvarepsilon\circ
\calB\circ\upvarphi\circ\tilde\calR\big)^{-1}
\circ\upvarepsilon.
\]
Comparing to \eqref{factorization-2a},
we conclude that we can choose
$
\tilde\calR_B(z)=
\tilde\calR\circ\big(I_{\tilde\bfE}
+\upvarepsilon\circ\calB\circ\upvarphi\circ\tilde\calR\big)^{-1}
$.

Let us now prove Part~\itref{theorem-factorization-1b}.
Recall that
\[
 \ker\big(\lambda I_{\bfE}-\calB(A-z_0 I_\bfX)^{-1}_{\bfE,\bfF,\varOmega}\big)\cong
 \ker\big(\lambda I_{\bfF}-(A-z_0 I_\bfX)^{-1}_{\bfE,\bfF,\varOmega}\calB\big),
\]
with both spaces finite-dimensional
(see Theorem~\ref{theorem-lap}~\itref{theorem-lap-1}).
By the same argument,
\[
 \ker\big(\lambda I_{\tilde\bfE}-\upvarepsilon\circ\calB\circ\upvarphi\circ\tilde\calR\big)
 \cong
 \ker\big(\lambda I_{\tilde\bfF}
-
\tilde\calR\circ\upvarepsilon\circ\calB\circ\upvarphi
\big),
\]
also with both spaces finite-dimensional.

We have the identities
\[
\upvarepsilon\circ\left(\lambda I_{\bfE}-\calB(A-z_0 I_\bfX)^{-1}_{\bfE,\bfF,\varOmega}\right)=
\big(\lambda I_{\tilde\bfE}-\upvarepsilon\circ\calB\circ\upvarphi\circ\tilde\calR\big)\circ\upvarepsilon
\]
and
\[
\left(\lambda I_{\bfF}-(A-z_0 I_\bfX)^{-1}_{\bfE,\bfF,\varOmega}\calB\right)\circ\upvarphi
 =\upvarphi\circ\big(\lambda I_{\tilde\bfF}-\tilde\calR\circ\upvarepsilon\circ\calB\circ\upvarphi\big);
\]
we deduce that
\begin{eqnarray}\label{varepsilon-zero}
\upvarepsilon
\left(
\ker\big(\lambda I_{\bfE}-\calB(A-z_0 I_\bfX)^{-1}_{\bfE,\bfF,\varOmega}\big)\right)
\subset
 \ker\big(\lambda I_{\tilde\bfE}-\upvarepsilon\circ\calB\circ\upvarphi\circ\tilde\calR\big)
\end{eqnarray}
and
\begin{eqnarray}\label{varphi-zero}
\upvarphi\big(
  \ker\big(\lambda I_{\tilde\bfF}
-
\tilde\calR\circ\upvarepsilon\circ\calB\circ\upvarphi
\big)
\big)
\subset \ker\big(\lambda I_{\bfF}-(A-z_0 I_\bfX)^{-1}_{\bfE,\bfF,\varOmega}\calB\big).
\end{eqnarray}
To prove that the inclusion
\eqref{varepsilon-zero}
is the equality, assume that
$\phi_0\in\ker(\lambda I_{\tilde\bfE}-
\upvarepsilon\circ\calB\circ\upvarphi\circ\tilde\calR)$.
Define $\phi=\calB\circ\upvarphi\circ\tilde\calR\phi_0\in\bfE$,
so that
$\upvarepsilon(\phi)=\lambda\phi_0$.
One can see that
$
\upvarepsilon\big(
(\lambda I_\bfE-\calB\circ\upvarphi\circ\tilde\calR\circ\upvarepsilon)\phi
\big)
=
(\lambda I_{\tilde\bfE}-\calB\circ\upvarphi\circ\tilde\calR)\upvarepsilon(\phi)
=0$,
hence $\phi\in\ker(\lambda I_\bfE-\calB\circ\upvarphi\circ
\tilde\calR\circ\upvarepsilon)$,
so
$\phi_0$ belongs to both sides of \eqref{varepsilon-zero}
(we recall that $\lambda\ne 0$).
Similarly, to prove that the inclusion
\eqref{varphi-zero}
is the equality, assume that
$\Psi\in\ker(\lambda I_{\bfF}-
\upvarphi\circ\tilde\calR\circ\upvarepsilon\circ\calB)$.
Define $\Psi_0=\tilde\calR\circ\calB\circ\upvarphi(\Psi)\in\tilde\bfF$,
so that
$\upvarphi(\Psi_0)=\lambda\Psi$.
One can see that
$
\upvarphi\big(
(\lambda I_{\tilde\bfF}
-\tilde\calR\circ\upvarepsilon\circ\calB\circ\upvarphi)
\Psi_0
\big)
=
(\lambda I_{\bfF}
-
\upvarphi\circ\tilde\calR\circ\upvarepsilon\circ\calB)\upvarphi(\Psi_0)
=0$,
hence
$\Psi_0\in\ker(\lambda I_{\tilde\bfF}
-\tilde\calR\circ\upvarepsilon\circ\calB\circ\upvarphi)$,
so
$\Psi$ belongs to both sides of \eqref{varphi-zero}.
This completes the proof of
Part~\itref{theorem-factorization-1}.

Let us prove Part~\itref{theorem-factorization-2}.
The factorization
$
\calR(z)
=\upvarphi\circ\tilde\calR(z)
\circ\upvarepsilon:\,\bfE\to\bfF$,
$z\in\varOmega$
(with $\calR(z)
=\jmath\circ(A-z I_\bfX)^{-1}\circ\imath:\,\bfE\to\bfF$;
see \eqref{def-r-z0})
leads to the identity
\begin{eqnarray}\label{w-d-new}
\big(I_{\tilde\bfE}+\upvarepsilon
\circ\calB\circ\upvarphi\circ\tilde\calR(z)\big)
\circ\upvarepsilon
=
\upvarepsilon\circ\big(I_{\bfE}+\calB\calR(z)\big)
:\;\bfE\to\tilde\bfE
,
\qquad z\in\varOmega.
\end{eqnarray}
By Lemma~\ref{lemma-br}
(with $\upvarepsilon\circ\calB\circ\upvarphi$ and $\tilde\calR(z)$
in place of $\calB$ and $\calR(z)$;
we note that by our assumptions the operator family
$\big\{\upvarepsilon\circ\calB\circ\upvarphi\circ\tilde\calR(z)
\big\}_{z\in\varOmega}$
is collectively compact),
there is the convergence
$\upvarepsilon\circ\calB\circ\upvarphi\circ\tilde\calR(z)
\to
\upvarepsilon\circ\calB\circ\upvarphi\circ\tilde\calR(z_0)
\in\scrB_0(\tilde\bfE)$
as $z\to z_0$, $z\in\varOmega$,
in the strong operator topology of $\scrB(\tilde\bfE)$.
The equality
of the nonzero spectra for compact operators
$
\calB\calR(z_0)=\calB\circ\upvarphi\circ
\tilde\calR(z_0)\circ\upvarepsilon$
and $\upvarepsilon\circ
\calB\circ\upvarphi\circ\tilde\calR(z_0)$
leads to
$-1\not\in
\sigma(\calB\calR(z_0))
=
\sigma(
\upvarepsilon\circ\calB\circ\upvarphi\circ\tilde\calR(z_0))
$
(cf. Theorem~\ref{theorem-lap}~\itref{theorem-lap-2}).
Therefore,
by Lemma~\ref{lemma-kz}~\itref{lemma-kz-1},
if $\delta>0$ is sufficiently small,
the first factor in the left-hand side
of \eqref{w-d-new}
is invertible
for $z\in\varOmega\cap\mathbb{D}_\delta(z_0)$,
and we arrive at
\begin{eqnarray}\label{j-a-i}
\upvarepsilon\circ
\big(I_{\bfE}+\calB\calR(z)\big)^{-1}
=
\big(I_{\tilde\bfE}
+
\upvarepsilon\circ\calB\circ\upvarphi\circ
\tilde\calR(z)\big)^{-1}
\circ\upvarepsilon
:\;\bfE\to\tilde\bfE,
\quad
z\in\varOmega\cap\mathbb{D}_\delta(z_0).
\end{eqnarray}
For $z\in\varOmega\cap\mathbb{D}_\delta(z_0)$,
we have:
\begin{eqnarray}\label{asdf-new}
&&
\jmath\circ(A+B-z I_\bfX)^{-1}\circ\imath
=\jmath\circ(A-z I_\bfX)^{-1}\circ
(I_{\bfX}+B(A-z I_\bfX)^{-1})^{-1}\circ\imath
\nonumber
\\
&&
=\jmath\circ(A-z I_\bfX)^{-1}
\circ\imath\circ(I_{\bfE}+\calB\calR(z))^{-1}
=\upvarphi\circ\tilde\calR(z)
\circ\upvarepsilon\circ(I_{\bfE}+\calB\calR(z))^{-1}.
\end{eqnarray}
Using \eqref{j-a-i},
we rewrite \eqref{asdf-new} as
\[
\jmath\circ(A+B-z I_\bfX)^{-1}\circ\imath
=
\upvarphi\circ\tilde\calR(z)\circ
\big(I_{\tilde\bfE}
+\upvarepsilon\circ
\calB\circ\upvarphi\circ\tilde\calR(z)\big)^{-1}
\circ\upvarepsilon,
\qquad
z\in\varOmega\cap\mathbb{D}_\delta(z_0).
\]
Comparing to \eqref{factorization-0},
we conclude that we can choose
$
\tilde\calR_B(z)=
\tilde\calR(z)
\circ
\big(I_{\tilde\bfE}
+\upvarepsilon\circ
\calB\circ\upvarphi\circ\tilde\calR(z)\big)^{-1}
$.

Due to the convergence
$
\big(I_{\tilde\bfE}
+\upvarepsilon\circ
\calB\circ\upvarphi\circ\tilde\calR(z)
\big)^{-1}
\to
\big(I_{\tilde\bfE}
+\upvarepsilon\circ
\calB\circ\upvarphi\circ\tilde\calR(z_0)
\big)^{-1}$
in the strong operator topology of $\scrB(\tilde\bfE)$
as $z\to z_0$, $z\in\varOmega\cap\mathbb{D}_\delta(z_0)$,
by Lemma~\ref{lemma-kz}~\itref{lemma-kz-1},
the convergence of
\begin{eqnarray}\label{one-can-see}
\tilde\calR_B(z)=
\tilde\calR(z)
\circ\big(I_{\tilde\bfE}
+\upvarepsilon\circ
\calB\circ\upvarphi\circ\tilde\calR(z)\big)^{-1}
\in\scrB\big(\tilde\bfE,\tilde\bfF\big),
\qquad
z\in\varOmega\cap\mathbb{D}_\delta(z_0),
\end{eqnarray}
as $z\to z_0$
to
$
\tilde\calR(z_0)\circ\big(I_{\tilde\bfE}
+\upvarepsilon\circ
\calB\circ\upvarphi\circ\tilde\calR(z_0)\big)^{-1}
$
is in the weak$^*$ or weak or strong operator topology
of $\scrB\big(\tilde\bfE,\tilde\bfF\big)$,
respectively,
when so is the convergence $\tilde\calR(z)\to\tilde\calR(z_0)$.

Now let us assume that
$\tilde\calR(z)\to\tilde\calR(z_0)$ as $z\to z_0$, $z\in\varOmega$,
in the uniform operator topology of $\scrB\big(\tilde\bfE,\tilde\bfF\big)$.
We can
see from
\eqref{one-can-see}
that
if $\upvarepsilon\circ\calB\circ\upvarphi\circ
\tilde\calR(z)\to\upvarepsilon\circ\calB\circ\upvarphi\circ
\tilde\calR(z_0)$, as $z\to z_0$, $z\in\varOmega$,
in the uniform operator topology of $\scrB(\tilde\bfE)$,
then $\tilde\calR_B(z)\to\tilde\calR_B(z_0)$ in the uniform operator topology of $\scrB(\tilde\bfE,\tilde\bfF)$.
\end{proof}

Regarding Theorem~\ref{theorem-factorization}~\itref{theorem-factorization-1},
let us mention that it is possible that
$(A+B-z_0 I_\bfX)^{-1}_{\bfE,\bfF,\varOmega}:\,\bfE\to\bfF$
extends to a continuous linear mapping
$\tilde\bfE\to\tilde\bfF$,
with
$\bfE\longhookrightarrow\tilde\bfE
\mathop{\longhookrightarrow}\limits\sp{\tilde\imath}
\bfX\mathop{\longhookrightarrow}\limits\sp{\tilde\jmath}
\tilde\bfF
\mathop{\longhookrightarrow}\limits\sp{\upbeta}
\bfF$,
while
the convergence of $\tilde\jmath\circ(A+B-z I_\bfX)^{-1}\circ\tilde\imath$
does not hold in the topology of $\tilde\bfE\to\tilde\bfF$.
Moreover,
in view
of Theorem~\ref{theorem-a}~\itref{theorem-a-1},
if the
continuous embeddings $\upvarepsilon$ and $\upvarphi$ are dense
and
$\upbeta^*(\bfF^*)$
is dense in
$\tilde\bfF^*$,
 then
 $\tilde\jmath\circ(A+B-z I_\bfX)^{-1}\circ\tilde\imath$
is not bounded uniformly on
$\varOmega\cap\mathbb{D}_\delta(z_0)$ for any $\delta>0$
in the norm of mappings $\tilde\bfE\to\tilde\bfF$.

\begin{example}
\label{example-d5}
The resolvent of the free
Laplace operator in $\R^d$, $d\ge 5$,
converges in the weak operator topology of
$\scrB\big(L^2_s(\R^d),L^2_{-s'}(\R^d)\big)$, $s+s'\ge 2$,
only as long as $s,\,s'>1/2$,
while the limit operator
extends to continuous linear mappings
$L^2_2(\R^d)\to L^2(\R^d)$
and
$L^2(\R^d)\to L^2_{-2}(\R^d)$.
See Theorem~\ref{theorem-3d} below
for more details.
\end{example}

We also notice that
Theorem~\ref{theorem-factorization}~\itref{theorem-factorization-1b}
shows that the virtual states have improved
regularity, belonging to the images of
$\range(\upvarphi)\subset\bfF$.

Here is a standard example
of the application
of relatively compact perturbations
in the context of differential operators.

\begin{example}
\label{example-laplace-1d}
Consider $A=-\Delta$ in $\bfX=L^2(\R)$, $\dom(A)=H^2(\R)$.
We note that
for any $s,\,s'\ge 0$
the resolvent
$R_0^{(1)}(z)=(A-z I_\bfX)^{-1}$, $z\in\cnor$,
with the integral kernel
$R_0^{(1)}(x,y;z)=\frac{e^{-\sqrt{-z}\abs{x-y}}}{2\sqrt{-z}}$,
$\Re\sqrt{-z}>0$,
does not extend to a linear mapping
$L^2_s(\R)\to L^2_{-s'}(\R)$
which would be bounded uniformly for
$z\in\mathbb{D}_\delta\setminus\overline{\R_{+}}$
with some $\delta>0$.
At the same time,
if $V\in C\sb{\mathrm{comp}}([-a,a],\C)$
is any potential
such that the solution
$\theta_{+}(x)$ to $(-\Delta+V(x))u=0$,
$u\at{x\ge a}=1$,
remains unbounded for $x\le 0$
(for example, one can take
$V$ nonnegative and not identically zero),
so that it is linearly independent
with $\theta_{-}(x)$
(solution which equals one for $x<-a$),
then, by Lemma~\ref{lemma-1d} below,
for any $s,\,s'>1/2$, $s+s'\ge 2$,
the resolvent
$R_V(z)=(A+V-z I_\bfX)^{-1}$
extends to a bounded linear mapping
$L^2_s(\R)\to L^2_{-s'}(\R)$
for all $z\in\mathbb{D}_\delta\setminus\overline{\R_{+}}$
with $\delta>0$ sufficiently small
and has a limit
in the strong operator topology
as $z\to z_0=0$, $z\not\in\overline{\R_{+}}$;
thus, $z_0=0$ is a regular point of $A+V$
relative to $\cnor$.
Since the operator of multiplication by $V(x)$
is $A$-compact,
$z_0=0$ is a virtual level of $A=-\Delta$ in $L^2(\R)$
(relative to $\big(L^2_s(\R),L^2_{-s'},\cnor\big)$).
\end{example}

Now we present a construction
of the resolvent of regularized operator
which we will use to derive resolvent estimates
for Schr\"odinger operators in $\R^2$ (see Section~\ref{sect-2d} below).
We formulate this result in the framework of
continuous embeddings
$\bfE\hookrightarrow\tilde\bfE$
and
$\tilde\bfF\hookrightarrow\bfF$
as on Figure~\ref{fig-relation-2}.

\begin{theorem}\label{theorem-construction}
Let $\bfE$, $\bfF$ be Banach spaces with
continuous embeddings
$\bfE\mathop{\longhookrightarrow}\limits\sp{\imath}\bfX
\mathop{\longhookrightarrow}\limits\sp{\jmath}\bfF$,
let $A\in\scrC(\bfX)$ satisfy Assumption~\ref{ass-virtual}.
let $\varOmega\subset\C\setminus\sigma(A)$,
and let $z_0\in\sigma\sb{\mathrm{ess}}(A)\cap\p\varOmega$.
Let there be Banach spaces $\tilde\bfE,\,\tilde\bfF$
with continuous embeddings
(see Figure~\ref{fig-relation-2})
\[
\bfE\mathop{\longhookrightarrow}\limits\sp{\upvarepsilon}
\tilde\bfE,
\qquad
\tilde\bfF
\mathop{\longhookrightarrow}\limits\sp{\upvarphi}\bfF.
\]
Assume that there are projectors
$P\in\scrB_{00}(\bfE)$ and $\tilde P\in\scrB_{00}(\tilde\bfE)$,
$\rank P=\rank\tilde P=1$, such that
\begin{eqnarray}\label{p-p}
\upvarepsilon\circ P=\tilde P\circ\upvarepsilon:\,
\bfE\to\tilde\bfE.
\end{eqnarray}
Assume that
there is $\tilde\calR(z)\in\scrB\big(\tilde\bfE,\tilde\bfF\big)$
such that
$\jmath\circ(A-z I_\bfX)^{-1}\circ\imath
=\upvarphi\circ\tilde\calR(z)\circ\upvarepsilon$
and that
$\tilde\calS(z):=\tilde\calR(z)\circ(I_{\tilde\bfE}-\tilde{P})
:\,\tilde\bfE\to\tilde\bfF$
converges to
$\tilde\calS(z_0)\in\scrB\big(\tilde\bfE,\tilde\bfF\big)$,
\begin{eqnarray}\label{or1}
\tilde\calS(z)
\to
\tilde\calS(z_0),
\qquad
z\to z_0,
\quad
z\in\varOmega,
\end{eqnarray}
in the weak$^*$ or weak or strong or uniform
operator topology
of $\scrB(\tilde\bfF,\tilde\bfE)$
as $z\to z_0$, $z\in\varOmega$.

Assume that there are $\varTheta(z)\in\bfE$
and $\tilde\varPsi(z)\in\tilde\bfF$
defined for $z\in\varOmega$
such that
\begin{eqnarray}\label{a-psi}
\upvarphi(\tilde\varPsi(z))
=\jmath\circ(A-z I_\bfX)^{-1}
\imath(\varTheta(z)),
\qquad
\lim\sb{z\to z_0,\,z\in\varOmega}
\norm{\upvarepsilon(\varTheta(z))}\sb{\tilde\bfE}=0,
\end{eqnarray}
and
\begin{eqnarray}\label{or2}
\tilde\varPsi(z)
\to
\tilde\varPsi_0\in\tilde\bfF\setminus\{0\},
\qquad
z\to z_0,\quad z\in\varOmega,
\end{eqnarray}
where the convergence is in $\tilde\bfF$
or in the weak or weak$^*$ topology of $\tilde\bfF$.

Then:

\begin{enumerate}
\item
\label{theorem-construction-1}
There is
$B\in\scrB_{00}(\bfX)$, $\rank B=1$,
and $\delta>0$
such that
$\sigma(A+B)\cap\varOmega\cap\mathbb{D}_\delta(z_0)
=\emptyset$
and there is the operator family
$\tilde\calR_B(z)
\in\scrB\big(\overline{\upvarepsilon(\bfE)},\tilde\bfF\big)$
defined for $z\in\varOmega\cap\mathbb{D}_\delta(z_0)$
such that
$\jmath\circ(A+B-z I_\bfX)^{-1}\circ\imath
=\upvarphi\circ\tilde\calR_B(z)\circ\upvarepsilon
$, $z\in\varOmega\cap\mathbb{D}_\delta(z_0)$
and which converges
as $z\to z_0$, $z\in\varOmega\cap\mathbb{D}_\delta(z_0)$,
in the following operator topologies of
$\scrB\big(\tilde\bfE,\tilde\bfF\big)$:

\begin{enumerate}
\item
\label{theorem-construction-1a}
Weak$^*$ operator topology
if the convergence in \eqref{or1} is in the
weak$^*$ operator topology of $\scrB(\tilde\bfE,\tilde\bfF)$
and the convergence in \eqref{or2}
is in the weak$^*$ topology of $\tilde\bfF$;
\item
\label{theorem-construction-1b}
Weak operator topology
if the convergence in \eqref{or1} is in the
weak operator topology of $\scrB(\tilde\bfE,\tilde\bfF)$
and the convergence in \eqref{or2}
is in the weak topology of $\tilde\bfF$;
\item
\label{theorem-construction-1c}
Strong or uniform operator topology
of $\scrB\big(\tilde\bfE,\tilde\bfF\big)$
if the convergence in \eqref{or1}
is in the strong or uniform operator topology,
respectively,
and the convergence in \eqref{or2} holds in $\tilde\bfF$.
\end{enumerate}

\item
\label{theorem-construction-2}

\begin{enumerate}
\item
\label{theorem-construction-2a}
If
the convergence in \eqref{or2} holds weakly in $\tilde\bfF$
and for any $\tilde h\in\tilde\bfF^*$
the family $\tilde\calS(z)^*\tilde h$,
$z\in\varOmega$, converges in $\tilde\bfE^*$
as $z\to z_0$,
then the operator family
$\big\{\tilde\calR_B(z)^*
\big\}_{z\in\varOmega\cap\mathbb{D}_\delta(z_0)}$
converges in the strong operator topology of
$\scrB(\bfF^*,\bfE^*)$
as $z\to z_0$;
\item
\label{theorem-construction-2b}
If
$\tilde\bfF$ has a pre-dual
and
for any $\tilde h\in J_{\tilde\bfF_*}\tilde\bfF_*
\subset\tilde\bfF^*$
the family $\tilde\calS(z)^*\tilde h$,
$z\in\varOmega$, converges in $\tilde\bfE^*$ as $z\to z_0$,
then the operator family
$\big\{\tilde\calR_B(z)^* J_{\tilde\bfF_*}
\big\}_{z\in\varOmega\cap\mathbb{D}_\delta(z_0)}$
converges in the strong operator topology of
$\scrB(\bfF_*,\bfE^*)$
as $z\to z_0$;
\item
\label{theorem-construction-2c}
If the convergence in \eqref{or1}
holds in the uniform operator topology
of $\scrB(\tilde\bfE,\tilde\bfF)$
and the convergence in \eqref{or2} holds in $\tilde\bfF$,
then the operator family
$\big\{\tilde\calR_B(z)^*
\big\}_{z\in\varOmega\cap\mathbb{D}_\delta(z_0)}$
converges in the uniform operator topology of
$\scrB(\bfF^*,\bfE^*)$
as $z\to z_0$.
\end{enumerate}
\end{enumerate}
\end{theorem}

\begin{proof}
Due to \eqref{p-p}, there are
$\varPhi\in\bfE$ and $\tilde\xi\in\tilde\bfE^*$,
$\langle\upvarepsilon^*(\tilde\xi),\varPhi\rangle_{\bfE}=1$,
such that
$P
=\varPhi\otimes\upvarepsilon^*(\tilde\xi)\in\scrB_{00}(\bfE)$
and
$\tilde P
=\upvarepsilon(\varPhi)\otimes\tilde\xi\in\scrB_{00}(\tilde\bfE)$,
\[
P:\;\phi\mapsto
\varPhi
\overline{\langle\upvarepsilon^*(\tilde\xi),\phi\rangle_{\bfE}},
\qquad
\phi\in\bfE,
\qquad
\tilde P:\;\tilde\phi\mapsto
\upvarepsilon(\varPhi)
\overline{\langle\tilde\xi,\tilde\phi\rangle_{\bfE}},
\qquad
\tilde\phi
\in\tilde\bfE.
\]
Below, we will use the notations
$\tilde\varPhi=\upvarepsilon(\varPhi)\in\tilde\bfE$,
$\xi=\upvarepsilon^*(\tilde\xi)\in\bfE^*$;
one has
$\langle\xi,\varPhi\rangle_{\bfE}
=\langle\tilde\xi,\tilde\varPhi\rangle_{\tilde\bfE}=1$.

We claim that,
without loss of generality,
we may assume that
\begin{eqnarray}\label{a-psi-bis}
\upvarphi(\tilde\varPsi(z))
=
a(z)
\jmath\circ
R(z)
\imath(\varPhi),
\qquad
z\in\varOmega,
\end{eqnarray}
where
$R(z)=(A-z I_\bfX)^{-1}\in\scrB(\bfX)$,
$a(z)\in\C$, $z\in\varOmega$,
$a(z)\to 0$ as $z\to z_0$,
and with
\begin{eqnarray}\label{or2-1}
\tilde\varPsi(z)\to\tilde\varPsi_0\in\tilde\bfF
\qquad
\mbox{as $z\to z_0$, $z\in\varOmega$,}
\end{eqnarray}
where the convergence is in $\tilde\bfF$
or in the weak or weak$^*$ topology of $\tilde\bfF$
(in the same topology as $\tilde\varPsi(z)\to\tilde\varPsi_0$).
Indeed, by \eqref{a-psi},
\begin{eqnarray*}
\upvarphi(\tilde\varPsi(z))
&=&\jmath\circ R(z)\imath(P\varTheta(z))
+\jmath\circ R(z)\imath((I_\bfE-P)\varTheta(z))
\\
&=&\jmath\circ R(z)\imath(P\varTheta(z))
+
\upvarphi\circ\tilde\calS(z)\circ(I_{\tilde\bfE}-\tilde{P})\circ\upvarepsilon(\varTheta(z)),
\qquad
z\in\varOmega,
\end{eqnarray*}
hence we can define
\[
\tilde\varPsi_1(z)
=\tilde\varPsi(z)
-
\tilde\calS(z)\circ(I_{\tilde\bfE}-\tilde{P})\circ\upvarepsilon(\varTheta(z)),
\qquad
z\in\varOmega.
\]
Due to the assumptions of the theorem,
$\tilde\calS(z)\circ(I_{\tilde\bfE}-\tilde{P})
\in\scrB(\tilde\bfE,\tilde\bfF)$
is bounded uniformly in $z\in\varOmega$
while
$\norm{\upvarepsilon(\varTheta(z))}_{\tilde\bfE}\to 0$ as $z\to z_0$,
hence one has
$\tilde\varPsi_1(z)\to\tilde\varPsi_0$
as $z\to z_0$, $z\in\varOmega$,
in $\bfF$ or in the weak or weak$^*$ topology of $\bfF$
(in the same topology as $\tilde\varPsi(z)\to\tilde\varPsi_0$).
Therefore, we may substitute
$\tilde\varPsi(z)$ and $\varTheta(z)$
by $\tilde\varPsi_1(z)$ and
$P\varTheta(z)
=\varPhi\overline{\langle\xi,\varTheta(z)\rangle_{\bfE}}$,
respectively,
arriving at \eqref{a-psi-bis} and \eqref{or2-1}
with $a(z)
=\overline{\langle\xi,\varTheta(z)\rangle_{\bfE}}
=\overline{\langle\tilde\xi,\tilde\varTheta(z)\rangle_{\tilde\bfE}}
\to 0$
as $z\to z_0$.

Let $\eta\in\bfF^*$
($\eta\in\range(J_{\bfF_*})$ if $\bfF$ has a pre-dual)
be such that
$c_0:=\langle\eta,\upvarphi(\tilde\varPsi_0)\rangle_{\bfF}\ne 0$;
there is $\delta>0$ such that
\begin{eqnarray}\label{ge-half}
\abs{\langle\eta,\upvarphi(\tilde\varPsi(z))\rangle_{\bfF}}>
\abs{c_0}/2
\qquad
\mbox{and}
\qquad
\abs{a(z)}<\abs{c_0}/2
\qquad
\forall z\in\varOmega\cap\mathbb{D}_\delta(z_0).
\end{eqnarray}
Then, by \eqref{a-psi-bis} and \eqref{ge-half}, we have:
\begin{eqnarray}\label{inf00}
\abs{\langle\eta,
\jmath\circ R(z)\imath(\varPhi)\rangle_{\bfF}}
=
\abs{
\langle\eta,
\upvarphi(\tilde\varPsi(z))\rangle_{\bfF}
/a(z)
}
>1,
\qquad
\forall z\in\varOmega\cap\mathbb{D}_\delta(z_0).
\end{eqnarray}
Define
$\calB=\varPhi\otimes\eta\in\scrB_{00}(\bfF,\bfE)$,
\begin{eqnarray}\label{def-b}
\calB:\;\psi\mapsto\varPhi\overline{\langle\eta,\psi\rangle_\bfF};
\qquad
B:=\imath\circ\calB\circ\jmath\in\scrB_{00}(\bfF),
\qquad
\hatB:=\jmath\circ\imath\circ\calB\in\scrB_{00}(\bfF).
\end{eqnarray}
It follows from \eqref{a-psi-bis}
that
\begin{eqnarray}\label{a-psi-bis-1}
(\hatA+\hatB-z)\upvarphi(\tilde\varPsi(z))
=
(c_0+a(z))
\jmath\circ\imath(\varPhi),
\qquad
z\in\varOmega,
\end{eqnarray}
with $\hatA\in\scrC(\bfF)$ a closed
extension of $A$ onto $\bfF$
from Assumption~\ref{ass-virtual}.
Let us show that
\begin{eqnarray}\label{no-spectrum}
\sigma(A+B)
\cap\varOmega\cap\mathbb{D}_\delta(z_0)=\emptyset.
\end{eqnarray}
Let $z\in\varOmega\cap\mathbb{D}_\delta$
and assume that $u\in\ker(I_\bfX+R(z)B)$.
Then
\[
0=
\langle\jmath^*(\eta),(I_\bfX+R(z)B)u\rangle_\bfX
=
\big(1+\langle\jmath^*(\eta),R(z)\imath(\varPhi)
\rangle_\bfX
\big)
\langle\jmath^*(\eta),u\rangle_{\bfX}.
\]
In view of \eqref{inf00},
the above leads to $\langle\jmath^*(\eta),u\rangle_\bfX=0$,
and then $B u=\imath(\varPhi)
\overline{\langle\jmath^*(\eta),u\rangle_{\bfX}}=0$,
so the relation $(I_\bfX+R(z)B)u=0$
results in $u=0$.
Since $R(z)\in\scrB(\bfX)$
and
$B\in\scrB_{00}(\bfX)$,
one also has $R(z)B\in\scrB_{00}(\bfX)$;
we conclude from $\ker(I_\bfX+R(z)B)=\{0\}$
that $I_\bfX+R(z)B$ has a bounded inverse,
and the identity
$A+B-z I_\bfX=(A-z I_\bfX)(I_\bfX+R(z)B)$
leads to \eqref{no-spectrum}.

Consider the equation
\begin{eqnarray}\label{auf}
(\hatA+\hatB-z I_\bfF)\upvarphi(\tilde{u})=\jmath\circ\imath(\phi),
\qquad
\phi\in\bfE.
\end{eqnarray}
Denote
$\tilde\phi=\upvarepsilon(\phi)\in\tilde\bfE$
and let
\begin{eqnarray}\label{g-f-k}
k(\tilde\phi)
=
\overline{\langle\tilde\xi,\tilde\phi\rangle_{\tilde\bfE}},
\qquad
\tilde{g}(\tilde\phi)=(I_{\tilde\bfE}-\tilde{P})\tilde{\phi}
=\tilde{\phi}-k(\tilde\phi)\tilde\varPhi
=\upvarepsilon(\phi-k(\tilde\phi)\varPhi).
\end{eqnarray}
A solution to
$
(\hatA-z I_\bfF)\upvarphi(\tilde{v})=\jmath\circ\imath(
\phi-k(\tilde\phi)\varPhi)
$,
$\tilde{v}\in\tilde\bfF$
is given by $\tilde{v}=\tilde\calS(z)\tilde{g}(\tilde\phi)$;
using the definition \eqref{def-b}, we have:
\begin{eqnarray}\label{al-1}
(\hatA+\hatB-z I_\bfF)\upvarphi(\tilde\calS(z)\tilde{g}(\tilde\phi))
=\jmath\circ\imath
\Big(
\phi-k(\tilde\phi)\varPhi
+
\varPhi
\overline{\langle\tilde\eta,\tilde\calS(z)\tilde{g}(\tilde\phi)\rangle_{\tilde\bfF}}
\Big).
\end{eqnarray}
By \eqref{a-psi-bis-1} and \eqref{al-1},
for $z\in\varOmega\cap\mathbb{D}_\delta(z_0)$,
there is a solution $\tilde{u}=\tilde{u}(\phi)$ to \eqref{auf}
which is explicitly given by
$\tilde\calR_B(z)\tilde\phi$,
with
$\tilde\calR_B(z)\in\scrB(\tilde\bfE,\tilde\bfF)$
defined by
\begin{eqnarray}\label{u-is-psi}
\tilde\phi
\mapsto
\tilde\calR_B(z)\tilde\phi
:=
\tilde\calS(z)\tilde{g}(\tilde\phi)
+
\frac{\overline{\langle\tilde\eta,
\tilde\calS(z)\tilde{g}(\tilde\phi)\rangle_{\tilde\bfF}}-k(\tilde\phi)}
{c_0+a(z)}
\tilde\varPsi(z),
\qquad
z\in\varOmega\cap\mathbb{D}_\delta(z_0).
\end{eqnarray}
For the future use, we notice that its adjoint
$\tilde\calR_B(z)^*\in\scrB(\tilde\bfF^*,\tilde\bfE^*)$
is explicitly given by
\begin{eqnarray}\label{u-is-psi-dual}
\tilde h
\mapsto
\tilde\calS(z)^*\tilde{h}
-
\langle\tilde\calS(z)^*\tilde{h},\tilde\varPhi\rangle_{\tilde\bfE}
\tilde\xi
+
\frac{
\tilde\calS(z)^*\tilde\eta
-
\langle\tilde\calS(z)^*\tilde\eta,\tilde\varPhi\rangle_{\tilde\bfE}
\tilde\xi
-
\tilde\xi
}
{\bar c_0+\overline{a(z)}}
\langle\tilde h,\tilde\varPsi(z)\rangle_{\tilde\bfF}.
\end{eqnarray}
Since $A+B-z I_\bfX$ is invertible for
$z\in\varOmega\cap\mathbb{D}_\delta(z_0)$
(see \eqref{no-spectrum}),
there is the relation
\[
\jmath\circ(A+B-z I_\bfX)^{-1}\circ\imath(\phi)
=\upvarphi(\tilde u)
=\upvarphi\circ\tilde\calR_B(z)\circ\upvarepsilon(\phi),
\qquad
z\in\varOmega\cap\mathbb{D}_\delta(z_0),
\]
and it follows that
$\tilde\calR_B(z)$ is defined uniquely on
$\range(\upvarepsilon)$.
One can see from
\eqref{g-f-k}
that there is $C>0$
(independent on $\tilde\phi$)
such that
\[
\abs{k(\tilde\phi)}
\le
C\norm{\tilde\phi}_{\tilde\bfE},
\qquad
\norm{\tilde{g}(\tilde\phi)}_{\tilde\bfE}
\le
(1+C\norm{\varPhi}_{\bfE})
\norm{\tilde\phi}_{\tilde\bfE}.
\]
It follows that
if $\tilde\calS(z)\in\scrB(\tilde\bfE,\tilde\bfF)$
converges
in the weak$^*$
(recall that $\eta\in\range(J_{\bfF_*})$ if $\bfF$ has a pre-dual)
or weak operator topology
as $z\to z_0$, $z\in\varOmega\cap\mathbb{D}_\delta(z_0)$,
and
$\tilde\varPsi(z)$
converges in the weak$^*$ or weak,
respectively,
topology of $\tilde\bfF$,
then
\[
\tilde\calR_B(z)
:\,\ran(\upvarepsilon)\subset\tilde\bfE\to\tilde\bfF,
\qquad
\upvarepsilon(\phi)\mapsto\tilde{u},
\qquad
\phi\in\bfE,
\]
with $\tilde{u}$ given by \eqref{u-is-psi},
converges as $z\to z_0$,
$z\in\varOmega\cap\mathbb{D}_\delta(z_0)$
in the likewise operator topology.
By
continuity, the operators
$\tilde\calR_B(z):\,
\upvarepsilon(\tilde\bfE)\to\tilde\bfF$,
$z\in\varOmega\cap\mathbb{D}_\delta(z_0)$,
are extended to linear mappings
$\overline{\upvarepsilon(\bfE)}\to\tilde\bfF$
bounded in the norm of $\scrB\big(\tilde\bfE,\tilde\bfF\big)$.
Abusing notations,
we denote these extensions by $\tilde\calR_B(z)$.

The improvements to the convergence of $\tilde\calR_B(z)$
in the strong or uniform operator topology
follows from the explicit form of the expression
\eqref{u-is-psi}.
This completes the proof of Part~\itref{theorem-construction-1}.
The proof of the statements
in Part~\itref{theorem-construction-2}
follows from the explicit expression \eqref{u-is-psi-dual}.
\end{proof}

\subsection{Virtual levels of the adjoint operator}
\label{sect-adjoint}

\begin{theorem}[Virtual level of the adjoint operator]
\label{theorem-adjoint}
Let $\bfE$ and $\bfF$ be Banach spaces with
dense continuous embedding
$\bfE\mathop{\longhookrightarrow}\limits\sp{\imath}\bfX
\mathop{\longhookrightarrow}\limits\sp{\jmath}\bfF$,
let $A\in\scrC(\bfX)$
with dense domain $\dom(A)$
satisfy Assumption~\ref{ass-virtual},
let $\varOmega\subset\C\setminus\sigma(A)$,
and let
$z_0\in\sigma\sb{\mathrm{ess}}(A)\cap\p\varOmega$.
\begin{enumerate}
\item
\label{theorem-adjoint-0}
The mapping
$\jmath\circ(A-z I_\bfX)^{-1}\circ\imath$
converges in the weak operator topology of
$\scrB(\bfE,\bfF)$ as $z\to z_0$, $z\in\varOmega$,
if and only if
the mapping
$\imath^*\circ(A^*-\zeta I_\bfX)^{-1}\circ\jmath^*$
converges in the weak$^*$
operator topology of
$\scrB(\bfF^*,\bfE^*)$ as $\zeta\to\bar z_0$,
$\zeta\in\varOmega^*$,
where
$
\varOmega^*
=\big\{\zeta\in\C\sothat\bar\zeta\in\varOmega\big\}.
$

\item
\label{theorem-adjoint-1}
If $z_0$ is a point of the essential spectrum of $A$
of rank $r\in\N$
relative to $(\bfE,\bfF,\varOmega)$,
and moreover
there is $\calB\in\scrB_{00}(\bfF,\bfE)$
such that
$\jmath\circ(A+B-z I_\bfX)^{-1}\circ\imath:\,\bfE\to\bfF$
with $B=\imath\circ\calB\circ\jmath$
converges in the weak operator topology of $\scrB(\bfE,\bfF)$,
then
$\bar z_0$ is a point of the essential spectrum of $A^*$
of
rank $s\le r$, $s\ge 1$
relative to $(\bfF^*,\bfE^*,\varOmega^*)$.
\item
\label{theorem-adjoint-2}
Assume that
the domain of the
operator
\[
A_{\bfE\shortto\bfE}:\,\bfE\to\bfE,
\qquad
\phi\mapsto \imath^{-1}(A\imath(\phi)),
\qquad
\dom(A_{\bfE\shortto\bfE})
=\{\phi\in\bfE\sothat A\imath(\phi)\in\imath^{-1}(\bfE)\},
\]
is dense in $\bfE$.
If
$\zeta_0$
is a point of the essential spectrum of $A^*$
of
rank $s\in\N$ relative to $(\bfF^*,\bfE^*,\varOmega^*)$,
then $z_0=\bar\zeta_0$ is a point of the essential spectrum of $A$
of rank $r=s$ relative to $(\bfE,\bfF,\varOmega)$.
\end{enumerate}
\end{theorem}

We start with the following technical lemma.

\begin{lemma}\label{lemma-aws}
If $A\in\scrC(\bfX)$ is densely defined,
then its adjoint is weak$^*$-closed.
\end{lemma}

\begin{proof}
One has
\[
\mathcal{G}\big(\hatA^*\big)=\varSigma\mathcal{G}\big(\hatA\big)^\bot,
\]
where
$\varSigma:\,\bfF\times\bfF\to\bfF\times\bfF$,
$(u,v)\mapsto(-v,u)$.
Therefore,
\[
\mathcal{G}\big(\hatA^*\big)
=\bigcap_{(u,v)\in\mathcal{G}(\hatA)}\{(-v,u)\}^\bot.
\]
Since $\{(-v,u)\}^\bot$ is the kernel of a linear
functional
on
$\bfX^*\times \bfX^*$ from
$\range(J_{\bfX\times \bfX})$,
it is
closed in the weak$^*$ topology of  $\bfX^*\times \bfX^*$.
So, $\mathcal{G}\big(\hatA^*\big)$ is closed
as an intersection of closed sets.
\end{proof}

\begin{proof}[Proof of Theorem~\ref{theorem-adjoint}]
Part~\itref{theorem-adjoint-0}
follows from the observation that both statements
are equivalent to the statement that
for each $\phi\in\bfE$ and $\eta\in\bfF^*$,
there exists a limit
\[
\lim\sb{z\to z_0,\,z\in\varOmega}
\langle\eta,\jmath\circ(A-z I_{\bfX})^{-1}\circ\imath(\phi)\rangle_{\bfF}.
\]

Let us prove Part~\itref{theorem-adjoint-1}.
We first note that since the embeddings
$\imath$ and $\jmath$ are assumed continuous and dense,
there are continuous embeddings
\[
\bfF^*
\mathop{\longhookrightarrow}\limits\sp{\jmath^*}
\bfX^*
\mathop{\longhookrightarrow}\limits\sp{\imath^*}
\bfE^*
\]
(which are not necessarily dense
unless $\bfE$ and $\bfX$ are reflexive).
According to our assumptions,
there is $\calB\in\scrB_{00}(\bfF,\bfE)$
such that
$\jmath\circ(A+B-z I_{\bfX})^{-1}\circ\imath$,
with
$B=\imath\circ\calB\circ\jmath\in\scrB_{00}(\bfX)$,
converges in the weak operator topology of $\scrB(\bfE,\bfF)$
as $z\to z_0$, $\in\varOmega$,
and by Corollary~\ref{Cor:DifferentConvergence2}
we may assume that $\rank \calB=r\in\N$.
Thus, for any $\eta\in\bfF^*$ and $\phi\in\bfE$,
there is a limit
\begin{eqnarray*}
\lim\sb{z\to z_0,\,z\in\varOmega}
\langle\eta,\jmath\circ(A+B-z I_{\bfX})^{-1}\circ\imath(\phi)\rangle_{\bfF}
\quad
=\lim\sb{z\to z_0,\,z\in\varOmega}
\langle
\imath^*\circ(A^*+B^*-\bar z I_{\bfX^*})^{-1}\circ\jmath^*(\eta),\phi\rangle_{\bfE}
\\
\quad
=\lim\sb{\zeta\to\bar z_0,\,\zeta\in\varOmega^*}
\overline{
\langle
J_\bfE(\phi),
\imath^*\circ(A^*+B^*-\zeta I_{\bfX^*})^{-1}\circ\jmath^*(\eta)\rangle_{\bfE^*}
},
\end{eqnarray*}
where
$B^*=\jmath^*\circ\calB^*\circ\imath^*\in\scrB_{00}(\bfX^*)$
and
$J_\bfE:\,\bfE\hookrightarrow\bfE^{**}$
is the canonical embedding.
It follows that
$\calB^*\in\scrB_{00}(\bfF^*,\bfE^*)$
is such that
$\imath^*\circ(A^*+B^*-\zeta I_{\bfX^*})^{-1}\circ\jmath^*$
converges in the weak$^*$ topology
of $\scrB(\bfF^*,\bfE^*)$
as $\zeta\to\bar z_0$, $\zeta\in\varOmega^*$.
Thus, $\bar z_0$ is
of rank $s\le r
=\rank\calB^*$ relative to $(\bfF^*,\bfE^*,\varOmega^*)$.
We note that
$s\ge 1$
by Part~\itref{theorem-adjoint-0}.

Let us prove Part~\itref{theorem-adjoint-2}.
We note that
since $A_{\bfE\shortto\bfE}:\,\bfE\to\bfE$
is closed
(due to $A\in\scrC(\bfX)$)
and densely defined by the assumption,
its adjoint $(A_{\bfE\shortto\bfE})^*\in\scrC(\bfE^*)$
is also closed and densely defined.
Moreover, by Lemma~\ref{lemma-aws},
$(A_{\bfE\shortto\bfE})^*$ is weak$^*$-closed.
Since
$\imath\circ A_{\bfE\shortto\bfE}\subset A\circ\imath$,
one has
\[
(A_{\bfE\shortto\bfE})^*\circ\imath^*
\supset
\imath^*\circ A^*,
\]
hence $(A_{\bfE\shortto\bfE})^*$
is a closed
and weak$^*$-closed
extension of $A^*$ onto $\bfE^*$
(cf. Assumption~\ref{ass-virtual}).

\begin{lemma}\label{lemma-preadjoint}
Let $\calC_0\in\scrB_{00}(\bfE^*,\bfF^*)$,
$\rank\calC_0=s$,
and assume that
there is a convergence of
\[
\imath^*\circ
(A^*+C_0
-\zeta I_{\bfX^*})^{-1}
\circ\jmath^*,
\qquad
C_0:=\jmath^*\circ\calC_0\circ\imath^*\in\scrB_{00}(\bfX^*),
\]
in the weak$^*$ operator topology in $\scrB(\bfE^*,\bfF^*)$
as $\zeta\to\bar z_0$, $\zeta\in\varOmega^*$.
Then one can choose
$\calC\in\scrB_{00}(\bfE^*,\bfF^*)$
such that
there is $\calB\in\scrB_{00}(\bfF,\bfE)$
such that
$B=\imath\circ\calB\circ\jmath\in\scrB_{00}(\bfX)$ satisfies
\[
C=B^*\in\scrB_{00}(\bfX^*),
\qquad
\rank B=\rank\calB=\rank\calC=\rank C=s,
\]
and such that
there is a convergence of
\[
\imath^*\circ
(A^*+C-\zeta I_{\bfX^*})^{-1}
\circ\jmath^*,
\qquad
C:=\jmath^*\circ\calC\circ\imath^*\in\scrB_{00}(\bfX^*),
\]
as $\zeta\to\bar z_0$, $\zeta\in\varOmega^*$,
in the weak$^*$
operator topology in $\scrB(\bfE^*,\bfF^*)$.
\end{lemma}

\begin{proof}
Choose a basis $\big\{\xi_i\in\bfE^*\big\}_{1\le i\le s}$
in
$(A^*+C_0-\zeta_0 I_{\bfX^*})^{-1}_{\bfF^*,\bfE^*,\varOmega^*}
\ran(\calC_0)$.
There is a pre-dual basis
$\big\{\phi_i\in\bfE\big\}_{1\le i\le s}$
such that $\langle\xi_i,\phi_j\rangle_\bfE=\delta_{i j}$,
$1\le i,\,j\le s$.
Then we define
\[
Q
=\sum_{i=1}^{s}\xi_i
\overline{\langle J_\bfE(\phi_i),\,\cdot\,\rangle_{\bfE^*}}:\;
\bfE^*\to\bfE^*,
\]
so that
\[
Q^*
=\sum_{i=1}^{s}
J_\bfE(\phi_i)
\langle\,\cdot\,,\xi_i\rangle_{\bfE^*}
=J_\bfE\circ M
:\;
\bfE^{**}\to\bfE^{**},
\qquad
M:=\sum_{i=1}^{s}
\phi_i
\langle\,\cdot\,,\xi_i\rangle_{\bfE^*}
:\;
\bfE^{**}\to\bfE.
\]
Let
$\calC=\calC_0\circ Q$,
so $\rank\calC\le\rank\calC_0=s$.
By Lemma~\ref{lemma-q},
$\calC\in\scrQ_{\bfF^*,\bfE^*,\varOmega^*}(A^*-\bar z_0 I_{\bfX^*})
\cap\scrB_{00}(\bfE^*,\bfF^*)$,
hence $\rank\calC=s$.
Since $Q$ is a projector,
one has
$\calC=\calC_0\circ Q\circ Q=\calC\circ Q$,
and then we can write
$
\calC^*
=
(\calC\circ Q)^*
=
Q^*\circ\calC^*
=J_\bfE\circ M\circ\calC^*;
$
thus,
\begin{eqnarray}\label{pre-dual}
\calC^*\circ J_\bfF
=J_\bfE\circ M\circ\calC^*\circ J_\bfF
\in\scrB_{00}(\bfF,\bfE^{**}),
\qquad
M\in\scrB_{00}(\bfE^{**},\bfE).
\end{eqnarray}
We note that the range of the mapping
$\calB\in\scrB(\bfE_*,\bfF_*)\to \calB^{*}\in\scrB(\bfF,\bfE)$
is the set $\{\calD\in \scrB(\bfF,\bfE) \sothat \exists \calD_0\in \scrB(\bfE^*,\bfF_*), \,  \calD^*=J_{\bfF_*}\circ\calD_0\}$ and the pre-adjoint of $\calD$ is
given by $\calD_0\circ J_{\bfE_*}$.
Therefore,
$\calC$ has
pre-adjoint given by
\[
\calB:=M\circ\calC^*\circ J_\bfF
\in\scrB_{00}(\bfF,\bfE),
\qquad
\calB^*=\calC,
\qquad
\rank\calB=\rank\calC=s.
\]
One can see that
$B=\imath\circ\calB\circ\jmath\in\scrB_{00}(\bfX)$
is a pre-adjoint of $C$:
\[
B^*
=
\imath^*\circ\calB^*\circ\jmath^*
=
\imath^*\circ\calC\circ\jmath^*
=C\in\scrB_{00}(\bfX^*),
\]
thus
$\rank C=\rank B=\rank \calB=s$
(above, we took into account that
the continuous embeddings $\imath$, $\jmath$
are assumed to be dense).
\end{proof}

We choose
$\calC\in\scrQ_{\bfF^*,\bfE^*,\varOmega^*}(A^*-\bar z_0 I_{\bfX^*})
\cap\scrB_{00}(\bfE^*,\bfF^*)$,
$\rank\calC=s$,
as in Lemma~\ref{lemma-preadjoint},
so that there is $\calB\in\scrB_{00}(\bfF,\bfE)$,
$\rank\calB=\rank\calC=s$,
such that
$B=\imath\circ\calB\circ\jmath$
and
$C=\jmath^*\circ\calC\circ\imath^*$
satisfy $C=B^*$.
This implies that,
for any $\phi\in\bfE$ and $\eta\in\bfF^*$,
there exists the following limit:
\begin{eqnarray*}
\lim\sb{z\to z_0,\,z\in\varOmega}
\langle
\eta,
\jmath\circ(A+B-z I_{\bfX})^{-1}\imath(\phi)
\rangle_{\bfF}
=\lim\sb{\zeta\to\bar z_0,\,\zeta\in\varOmega^*}
\langle
\imath^*\circ
(A^*+C-\zeta I_{\bfX^*})^{-1}
\jmath^*(\eta),\phi
\rangle_{\bfE}
\\
=\lim\sb{\zeta\to\bar z_0,\,\zeta\in\varOmega^*}
\overline{
\langle
J_\bfE(\phi),
\imath^*\circ
(A^*+C-\zeta I_{\bfX^*})^{-1}
\jmath^*(\eta)
\rangle_{\bfE^*}
},
\end{eqnarray*}
thus
$\calB\in\scrQ_{\bfE,\bfF,\varOmega}(A-z_0 I_\bfX)$,
which means that
$z_0$ is a virtual level of $A$
of rank $r\le \rank\calB=s$
relative to $(\bfE,\bfF,\varOmega)$,
and moreover
$\jmath\circ(A+B-z I_\bfX)^{-1}\circ\imath$
converges
in the weak operator topology of $\scrB(\bfE,\bfF)$
as $z\to z_0$, $z\in\varOmega$.
By Part~\itref{theorem-adjoint-1},
there is the inequality $s\le r$;
thus, $r=s$.
This completes the proof of Theorem~\ref{theorem-adjoint}.
\end{proof}

\section{Schr\"odinger operators in one dimension}
\label{sect-1d}

Virtual levels of Schr\"odinger operators
require certain care
since the theory is sensitive to the spatial dimension:
the Laplace operator in dimensions $d\le 2$
has a virtual level at $z_0=0$.
The framework
developed in Section~\ref{sect-main-results}
applies in a uniform way to nonselfadjoint operators in any dimension.

Let us mention that
uniform resolvent bounds for
Schr\"odinger operators in higher dimensions appeared in
\cite{kenig1987uniform},
\cite{frank2011eigenvalue},
\cite{frank2017eigenvalue}, \cite{gutierrez2004nontrivial},
\cite{bouclet2018uniform}, \cite{ren2018endpoint}, \cite{mizutani2019eigenvalue}, \cite{kwon2020sharp}.
For the classical Rollnik bound,
see, for instance,
\cite[Example~3, p.~150]{reed1978methods4}, \cite[Sect.~I.4]{simon1971quantum},
\cite[Proposition~7.1.16]{yafaev2010mathematical} (for a discrete analog in this context see \cite{tadano2019uniform}).
We also mention that
the expansions of the integral kernel of the free resolvent
$(-\Delta-z I)^{-1}$
are given, for instance, in
\cite{jensen1979spectral} ($d=3$),
\cite{jensen1980spectral} ($d\ge 5$),
\cite{jensen1984spectral} ($d=4$),
and
\cite{jensen2001unified} ($d\le 3$).
For
the Laplacian in $\R^d$, $d\ge 3$,
the $L^p\to L^{p'}$ resolvent estimates
were proved in \cite{kenig1987uniform}.

The absence of virtual levels
of Schr\"odinger and Dirac operators in higher dimensions
($d\ge 5$ for Schr\"odinger and massive Dirac;
$d\ge 3$ for massless Dirac)
with potentials having sufficient spatial decay
is well-known; see e.g.
\cite{saito2008zero,gesztesy2020absence}.

When we consider complex-valued potentials,
the
discrete spectrum
is
no longer necessarily real
and can accumulate
under relatively compact perturbations
not only at the threshold
but also to the bulk of the essential spectrum.
See
\cite{pavlov1961non,pavlov1962spectral}
(who treats Schr\"{o}dinger operators on a half-line,
with the nonselfadjointness coming from the boundary condition)
and \cite{boegli2017schroedinger}
(in the higher-dimensional case with a complex-valued potential).
For embedded eigenvalues of Schr\"odinger and Dirac operators,
see \cite{cuenin2020embedded}
and the references therein.
We will only consider the most involved case of virtual levels
at the threshold point $z_0=0$.

In this section,
we give an elementary proof of
resolvent estimates
in the one-dimensional case;
two-dimensional case is covered in Section~\ref{sect-2d}
while
dimensions three and higher
are covered in Section~\ref{sect-free}.
Related results on properties of virtual states
are in recent articles
\cite[Theorem 2.3]{barth2021absence}
(for selfadjoint
Schr\"odinger operators in dimensions $d\le 2$)
and \cite[Theorem 3.3]{gesztesy2020absence}
(for selfadjoint
Schr\"odinger operators in dimensions $d\ge 3$).

Under additional assumptions on the $\Delta$-compact operator $W$,
the properties of virtual states of $-\Delta+W$
given in
Theorems~\ref{theorem-1d-general},
\ref{theorem-2d-general},
and~\ref{theorem-3d-general}
can be improved
(cf. Theorem~\ref{theorem-factorization}~\itref{theorem-factorization-1}).
In particular, for $W$ a multiplication operator
by a function $V$ with sufficiently fast decay,
virtual states and eigenfunctions
in dimensions $d\le 7$
belong to $L^\infty(\R^d)$ \cite[Theorem 3.3]{gesztesy2020absence}.

\begin{lemma}
\label{lemma-v-l1}
Let $V\in L^1(\R)$.
Then:
\begin{enumerate}
\item
\label{lemma-v-l1-1}
The operator $H=-\p_x^2+V$,
$\dom(H)=C^\infty_{\mathrm{comp}}(\R)$,
is closable on $L^2(\R)$ and its form
closure domain is $H^1(\R)$;
\item
\label{lemma-v-l1-2}
$\sigma\sb{\mathrm{ess}}(H)=\overline{\R_{+}}$,
$\sigma\sb{\mathrm{cont}}(H)\supset\R_{+}$.
If, moreover, $V\in L^1_1(\R)$, then
$\sigma\sb{\mathrm{cont}}(H)=\overline{\R_{+}}$.
\end{enumerate}
\end{lemma}

We recall that the continuous spectrum
$\sigma\sb{\mathrm{cont}}(H)$
is defined as the set of $z\in\sigma(H)$ such that
$\ker(H-z I)=\{0\}$
and the range of $H-z I$ is dense.

\begin{remark}\label{remark-cont}
The study of
$-\p_x^2+V$ with $V\in L^1(\R)$ has a long
history;
we refer to \cite{davies2002schrodinger}
where the conclusion
$\sigma\sb{\mathrm{ess}}(H)=\overline{\R_{+}}$
is proved for complex-valued
potentials $V\in L^1(\R)+L^\infty_0(\R)$
(in the appropriate sense).
Let us note that
an elementary proof
of Part~\itref{lemma-v-l1-2}
follows from utilizing the Jost solutions
available for all $z\ne 0$ if $V\in L^1$
(see Lemma~\ref{lemma-jost}),
allowing one to conclude that
one could only have the discrete spectrum
in $\C\setminus\overline{\R}$
(This corresponds to $0$ of the Evans function
which is holomorphic on $\cnor$
and not identically vanishing)
while for $z>0$ one can
construct the Weyl sequences
out of the corresponding Jost solutions
to conclude that
$\R_{+}\subset\sigma\sb{\mathrm{ess,2}}(H)$
(see Remark~\ref{remark-ess});
it then follows that
$\sigma\sb{\mathrm{ess}}(H)=\overline{\R_{+}}$.

One can also see
the continuous spectrum of $H=-\p_x^2+V$
contains $\R_{+}$
(moreover, it is
$\overline{\R_{+}}$ if $V\in L^1_1(\R)$).
Indeed, for $z>0$, one has
$\ker(H-z I)=\{0\}$;
to prove
the density of the range of $H-z I$,
we notice that
for any $\varphi\in C^\infty_{\mathrm{comp}}(\R)$
such that
$\langle\overline\Psi_1,\varphi\rangle
=0=\langle\overline\Psi_2,\varphi\rangle$,
with $\Psi_1,\,\Psi_2$
two linearly independent
solutions to $(H-z I)\Psi=0$,
there is a compactly supported solution
to $(H-z I)u=\varphi$
given by
\[
u(x)=
\frac{1}{c}
\Big(
\Psi_1(x)\int_{-\infty}^x\Psi_2(y)\varphi(y)\,dy
+
\Psi_2(x)\int_x^{+\infty}\Psi_1(y)\varphi(y)\,dy
\Big),
\]
with
$c=W(\Psi_1,\Psi_2)=\Psi_1\Psi_2'-\Psi_1'\Psi_2$
the value of the Wronskian
(note that $u$ is compactly supported
since so are $\varphi$ and
$\int_{-\infty}^x\Psi_2(y)\varphi(y)\,dy
=-\int_x^{+\infty}\Psi_2(y)\varphi(y)\,dy$).
The set of such $\varphi$ is dense in $L^2(\R)$
(since any nontrivial linear combination of
$\Psi_1$ and $\Psi_2$
does not belong to $L^2(\R)$),
proving the density of the range of $H-z I$.
The same argument shows
that $0\in\sigma\sb{\mathrm{cont}}(-\p_x^2+V)$
if $V\in L^1_1(\R)$,
since in this case Lemma~\ref{lemma-jost}
also provides the Jost solutions for $z=0$.
Let us notice that
$z=0$ is not necessarily in the continuous spectrum
of $H=-\p_x^2+V$ if $V\in L^1(\R)$, as demonstrated
by the example
$u(x)=1/\langle x\rangle\in L^2(\R)$,
$V(x)=u''(x)/u(x)
=O(1/\langle x\rangle^2)\in L^1(\R)\setminus L^1_1(\R)$,
$\ker(H)\ne\{0\}$ (thus $0\not\in\sigma\sb{\mathrm{cont}}(H)$).
\end{remark}

\begin{theorem}[Schr\"odinger operators
with decaying potentials in $L^2(\R)$:
virtual states and LAP estimates]
\label{theorem-1d}
Let $V$ be a measurable function on $\R$
such that $V\in L^1_1(\R)$.
There is the following dichotomy:

\medskip

\noindent
Either there is a nontrivial solution
(in the sense of distributions)
to
\begin{eqnarray}
\label{dvp}
(-\Delta+V)\Psi=0,
\qquad
\Psi\in L^\infty(\R)\cap H^2_{\mathrm{loc}}(\R),
\end{eqnarray}

\medskip

\noindent
or
\[
R_V(z)=(-\Delta+V-z I)^{-1},
\qquad
z\in\C\setminus\sigma(-\Delta+V)
\]
is uniformly bounded
as a mapping $L^2_{s_1}(\R)\to L^2_{-s_2}(\R)$
for some $s_1,\,s_2>1/2$, $s_1+s_2\ge 2$,
for
$z\in\mathbb{D}_\delta\setminus\overline{\R_{+}}$
with some $\delta>0$,
and in this case $R_V(z)$ has a limit as
$z\to z_0=0$, $z\in\C\setminus(-\Delta+V)$,
in the following topologies:
\begin{enumerate}
\item
\label{theorem-1d-1a}
In the strong operator topology of
$\scrB\big(L^2_s(\R),L^2_{-s'}(\R)\big)$ for
all pairs of $s,\,s'$ such that
$s,\,s'>1/2$, $s+s'\ge 2$
(in the uniform operator topology if $s+s'>2$);
\item
\label{theorem-1d-1b}
In the weak$^*$ operator topology of
$\scrB\big(L^1_\varsigma(\R),L^\infty_{\varsigma-1}(\R)\big)$,
$0\le\varsigma\le 1$.
\end{enumerate}
\end{theorem}

We postpone the proof of Theorem~\ref{theorem-1d}
until after we formulate and prove Theorem~\ref{theorem-1d-general}.

\begin{lemma}\label{lemma-1d-positive-potential}
Let $V$ be a real-valued
nonnegative
locally $L^2$-integrable
function on $\R$ which is not identically zero.
\begin{enumerate}
\item
\label{lemma-1d-positive-potential-1}
There is no nontrivial solution to
\begin{eqnarray}\label{dvp-1}
\Delta\Psi=V\Psi
\end{eqnarray}
(with the left-hand side
understood in the sense of distributions)
such that
$\Psi\in L^2_{-3/2}(\R)$.
\item
\label{lemma-1d-positive-potential-2}
If, further,
$V\in L^1_1(\R)$
and $s_1,\,s_2>1/2$, $s_1+s_2\ge 2$,
then
there is a limit
\begin{eqnarray}\label{lim-1d}
(-\Delta+V-z_0 I)^{-1}_{L^2_{s_1},L^2_{-s_2},\cnor}
=
\lim\sb{z\to z_0,\,z\in\cnor}
(-\Delta+V-z I)^{-1},
\end{eqnarray}
with $z_0=0$,
in the strong operator topology of
$\scrB\big(L^2_{s_1}(\R),L^2_{-s_2}(\R)\big)$.
\item
\label{lemma-1d-positive-potential-3}
If, moreover,
$s_2\le 3/2$
and $u\in L^2_{-s_2}(\R)$
satisfies $(-\Delta+V)u\in L^2_{s_1}(\R)$
(in the sense of distributions),
then
\[
u=(-\Delta+V-z_0 I)^{-1}_{L^2_{s_1},L^2_{-s_2},\cnor}
(-\Delta+V)u,
\qquad
z_0=0.
\]
\end{enumerate}
\end{lemma}

\begin{proof}
Let
$\Psi\in L^2_{\mathrm{loc}}(\R)$
be a nontrivial solution to $\Psi''=V\Psi$
(with the left-hand side
understood in the sense of distributions).
It follows that
$\Psi''=V\Psi\in L^1_{\mathrm{loc}}(\R)$,
thus
$\Psi\in C^1(\R)$.
Since $\Psi$ is not identically zero
and $V$ is not identically zero,
there is a point $x_0\in\R$ such that
$\Psi(x_0)\ne 0$ and $\Psi'(x_0)\ne 0$.
Without loss of generality
(changing $\Psi$ to $-\Psi$
and/or $x$ to $-x$ if necessary),
we may assume that
$\Psi(x_0)>0$ and $\Psi'(x_0)>0$.
  From $\Psi''=V\Psi$ with $V\ge 0$
we conclude that
$\Psi'(x)\ge\Psi'(x_0)>0$ for all $x\ge x_0$,
and hence $\Psi$ grows at least linearly as $x\to+\infty$,
hence $\Psi\not\in L^2_{-3/2}(\R)$.
This completes the proof of
Part~\itref{lemma-1d-positive-potential-1}.

Now we further assume that
$\langle x\rangle V\in L^1(\R)$.
Then Part~\itref{lemma-1d-positive-potential-2}
follows from Theorem~\ref{theorem-1d},
since by Part~\itref{lemma-1d-positive-potential-1}
there is no nontrivial solution to
$(-\Delta+V)\Psi=0$, $\Psi\in L^\infty(\R)\cap H^2_{\mathrm{loc}}(\R)$.

Let us prove Part~\itref{lemma-1d-positive-potential-3}.
Let us assume that
$u\in L^2_{-s_2}(\R)$ satisfies
$(-\Delta+V)u\in L^2_{s_1}(\R)$
(with $s_1>1/2$, $1/2<s_2\le 3/2$, $s_1+s_2\le 2$ as above).
Appealing to Part~\itref{lemma-1d-positive-potential-2},
we denote
\begin{eqnarray}\label{wv}
w=(-\Delta+V-z_0 I)^{-1}_{L^2_{s_1},L^2_{-s_2},\cnor}
(-\Delta+V)u\in L^2_{-s_2}(\R).
\end{eqnarray}
We apply $-\Delta+V$ to both sides of \eqref{wv};
by Lemma~\ref{lemma-jk}~\itref{lemma-jk-jk1},
$(-\Delta+V)w
=(-\Delta+V)u$,
hence
$\Psi=u-w\in L^2_{-s_2}(\R)$
is a solution to $(-\Delta+V)\Psi=0$
and thus $\Psi=0$
by Part~\itref{lemma-1d-positive-potential-1},
leading to $u=w$.
\end{proof}

\begin{theorem}[Schr\"odinger operators
with relatively compact perturbations in $L^2(\R)$:
virtual states and LAP estimates]
\label{theorem-1d-general}
Let $s_1,\,s_2>1/2$, $s_1+s_2\ge 2$,
$s_2\le 3/2$,
and
assume that
$W:\,L^2_{-s_2}(\R)\to L^2_{s_1}(\R)$
is $\Delta$-compact
(with $\Delta$ considered in $L^2_{-s_2}(\R)$).
There is the following dichotomy:

\medskip

\noindent
Either there is a nontrivial solution
to
$(-\Delta+W)\Psi=0$,
$\Psi\in L^2_{-s_2}(\R)$,
and, moreover,
this solution satisfies the relation
\begin{eqnarray}\label{w-is-w}
\Psi
=
(-\Delta+V_0-z_0 I)^{-1}_{L^2_{s_1},L^2_{-s_2},\cnor}
(V_0-W)\Psi,
\qquad
z_0=0,
\end{eqnarray}
with any
$V_0\in C^\infty_{\mathrm{comp}}(\R)$, $V_0\geq 0$, $V_0\ne 0$,
\medskip

\noindent
or
there is $\delta>0$
such that
$\mathbb{D}_\delta\setminus\overline{\R_{+}}
\subset\C\setminus\sigma(-\Delta+W)$
and the resolvent
$
R_W(z)=(-\Delta+W-z I)^{-1},
$
$
z\in\C\setminus\sigma(-\Delta+W)
$
is bounded
as a mapping
$L^2_{s_1}(\R)\to L^2_{-s_2}(\R)$
uniformly in
$z\in\mathbb{D}_\delta\setminus\overline{\R_{+}}$,
and in this case $R_W(z)$ has a limit
as $z\to z_0=0$,
$z\in\C\setminus\sigma(-\Delta+W)$,
in the following topologies:
\begin{enumerate}
\item
\label{theorem-1d-general-1a}
Strong operator topology
of $\scrB\big(L^2_{s}(\R),L^2_{-s'}(\R)\big)$
as long as
$s>1/2$, $s'>1/2$, $s+s'\ge 2$
(uniform operator topology if $s+s'>2$),
and
$W$ extends to a mapping
$L^2_{-s'}(\R)\to L^2_s(\R)$
which is $\Delta$-compact,
with $\Delta$ considered in $L^2_{-s'}(\R)$
(this assumption on $W$ is not needed
if, additionally, one has $s\ge s_1$, $s'\ge s_2$,
and it is redundant if $s\le s_1$, $s'\le s_2$);
\item
\label{theorem-1d-general-1b}
Weak$^*$ operator topology of
$\scrB\big(L^1_\varsigma(\R),L^\infty_{\varsigma-1}(\R)\big)$
as long as
$0\le\varsigma\le 1$,
$s_1>1/2+\varsigma$,
$s_2>3/2-\varsigma$,
and
for some
$\delta>0$
and
$V_0\in C^\infty\sb{\mathrm{comp}}(\R)$,
$V_0(x)\ge 0$, $V_0\ne 0$,
the operator family
\begin{eqnarray}\label{of-varsigma}
W(-\Delta+V_0-z I)^{-1}:\,L^1_\varsigma(\R)\to L^1_\varsigma(\R),
\qquad z\in\mathbb{D}_\delta\setminus\overline{\R_{+}},
\end{eqnarray}
is collectively compact.
\end{enumerate}
\end{theorem}

\begin{remark}
The condition that $s_2\le 3/2$
in Theorem~\ref{theorem-1d-general}
is required for the dichotomy
since
for any $V_0\in C^\infty_{\mathrm{comp}}(\R)$
there is a nontrivial solution to
$-\Psi''+V_0\Psi=0$,
$\Psi\in L^2_{-S}(\R)$
with $S>3/2$ (with linear growth at infinity).
\end{remark}

\begin{remark}\label{remark-1d-better}
The relation \eqref{w-is-w}
can be used to improve the regularity properties of
the virtual state $\Psi$
under additional assumptions on the operator $W$.
For example, if $W$ is the multiplication by
a function from $L^\infty_\rho(\R)$, $\rho>2$,
then
we can take
$s_1>3/2$ and $s_2\in(1/2,3/2]$
such that $s_1+s_2\le\rho$
(and thus $W:\,L^2_{-s_2}(\R)\to L^2_{s_1}(\R)$
is $\Delta$-compact,
with $\Delta$ considered in $L^2_{-s_2}(\R)$).
By Theorem~\ref{theorem-1d-general},
a virtual state $\Psi\in L^2_{-s_2}(\R)$
satisfies \eqref{w-is-w};
since $(W-V_0)\Psi\in L^2_{s_1}(\R)\subset L^1_1(\R)$,
Theorem~\ref{theorem-1d}~\itref{theorem-1d-1b}
shows that $\Psi\in L^\infty(\R)$.
\end{remark}

\begin{example}
\label{example-1d}
Let $s,\,s'=1$
and fix $\alpha\in(0,1/2)$.
Let
$\Psi=\langle x\rangle^{\alpha}\in L^2_{-1}(\R)$,
\ $f:=\Psi''=O(\langle x\rangle^{\alpha-2})\in L^2_{1}(\R)$.
Define
\[
W=f\frac{1}{\langle\Psi,f\rangle}\langle\cdot,f\rangle
\,:\;\bfF:=L^2_{-1}(\R)\to L^2_1(\R)=:\bfE,
\]
so that $(-\p_x^2+W)\Psi=0$.
Then
$
(-\p_x^2+V_0)\Psi=(V_0-W)\Psi\in L^2_1(\R)$,
where we assume that
$V_0\in C^\infty\sb{\mathrm{comp}}(\R)$,
$V_0\ge 0$, $V_0\ne 0$.
By Lemma~\ref{lemma-1d-positive-potential}~\itref{lemma-1d-positive-potential-3},
one has
$
\Psi=(-\p_x^2+V_0-z_0 I)^{-1}_{\bfE,\bfF,\cnor}(-\p_x^2+V_0)\Psi
=(-\p_x^2+V_0-z_0 I)^{-1}_{\bfE,\bfF,\cnor}(V_0-W)\Psi$,
with $z_0=0$,
hence
$\Psi\in\range\big((-\p_x^2+V_0-z_0 I)^{-1}_{\bfE,\bfF,\cnor}\big)$
is a
(non-$L^\infty$)
virtual state
corresponding to virtual level $z_0=0$
of the operator $-\Delta+W$
relative to $(\bfE,\bfF,\cnor)$.
\end{example}

\bigskip

\begin{proof}[Proof of Theorem~\ref{theorem-1d-general}]
Let
$V_0\in C^\infty_{\mathrm{comp}}(\R)$, $V_0\geq 0$, $V_0\ne 0$.
One has
$\sigma(-\Delta+V_0)
=\sigma\sb{\mathrm{ess}}(-\Delta+V_0)
=\overline{\R_{+}}$.
The Laplace operator $-\Delta$
in $L^2(\R)$
is closable as a mapping
$L^2_{-s_2}(\R)\to L^2_{-s_2}(\R)$
(cf. Example~\ref{example-laplace-2d}
and Lemma~\ref{lemma-last}).
By
Lemma~\ref{lemma-1d-positive-potential}~\itref{lemma-1d-positive-potential-1},
there is no nontrivial solution $\Psi\in L^2_{-3/2}(\R)$
to $(-\Delta+V_0)\Psi=0$,
and hence no solution in $L^\infty(\R)\cap H^2_{\mathrm{loc}}(\R)$.
Therefore,
Theorem~\ref{theorem-1d}
shows that
$z_0=0$
is
a regular point 
of the essential spectrum of $-\Delta+V_0$
relative to $\big(L^2_{s_1}(\R),L^2_{-s_2}(\R),\cnor\big)$
and that
$R_{V_0}(z)=(-\Delta+V_0-z I)^{-1}$
is uniformly bounded as a mapping $L^2_{s_1}(\R)\to L^2_{-s_2}(\R)$
near $z_0=0$
and has a limit as $z\to z_0$ in the strong operator topology
of $\scrB\big(L^2_{s_1}(\R),L^2_{-s_2}(\R)\big)$
(uniform operator topology if $s_1+s_2>2$).
We note that $W-V_0:\,L^2_{-s_2}(\R)\to L^2_{s_1}(\R)$
is
$(-\Delta+V_0)$-compact
(with $-\Delta+V_0$ considered in $L^2_{-s_2}(\R)$)
since, by assumption, it is
$\Delta$-compact
(with $\Delta$ considered in $L^2_{-s_2}(\R)$);
therefore, the assumptions of Theorem~\ref{theorem-lap}
are satisfied
with $\bfX=L^2(\R)$, $\bfE=L^2_{s_1}(\R)$,
$\bfF=L^2_{-s_2}(\R)$, $A=-\Delta+V_0$, $z_0=0$,
$\varOmega=\cnor$,
and $\calB=W-V_0$.

If $\Psi\in L^2_{-s_2}(\R)$
is a solution to
$(-\Delta+W)\Psi=0$, 
then
$(-\Delta+V_0)\Psi=(V_0-W)\Psi\in L^2_{s_1}(\R)$,
hence
by Lemma~\ref{lemma-1d-positive-potential}~\itref{lemma-1d-positive-potential-3}
one has
$\Psi
=
(-\Delta+V_0-z_0 I_\bfX)^{-1}_{L^2_{s_1},L^2_{-s_2},\cnor}
(-\Delta+V_0)\Psi
$,
thus yielding the relation \eqref{w-is-w};
in particular, there is the inclusion
$\Psi\in
\range\big((-\Delta+V_0-z_0 I_\bfX)^{-1}_{L^2_{s_1},L^2_{-s_2},\cnor}\big)
$.
Having such $\Psi\ne 0$
is the opposite of the statement
of Theorem~\ref{theorem-lap}~\itref{theorem-lap-2a}.
Now the dichotomy
stated in the theorem
follows from the equivalence
of Parts~\itref{theorem-lap-2a} and~\itref{theorem-lap-2e}
of Theorem~\ref{theorem-lap}.

\smallskip

We now suppose that
$R_W(z)=(-\Delta+W-z I)^{-1}$,
$z\in\C\setminus\sigma(-\Delta+W)$,
is bounded as a mapping $L^2_{s_1}(\R)\to L^2_{-s_2}(\R)$
uniformly in
$z\in\mathbb{D}_\delta\setminus\overline{\R_{+}}$
with some $\delta>0$;
let us prove
the convergence of $R_W(z)$ stated in the theorem.
We start with Part~\itref{theorem-1d-general-1a}.
There is nothing to prove if
$s\ge s_1$, $s'\ge s_2$.
We assume that $W:\,L^2_{-s'}(\R)\to L^2_s(\R)$ is $\Delta$-compact
(with $\Delta$ considered in $L^2_{-s'}(\R)$).
Above, we already proved that
there is a convergence of
$(-\Delta+W-z I)^{-1}$
in $\scrB\big(L^2_{s_1}(\R),L^2_{-s_2}(\R)\big)$
as $z\to z_0=0$, $z\in\C\setminus\sigma(-\Delta+W)$.
Now we apply Theorem~\ref{theorem-a},
where we take
$A=-\Delta+W$,
$\bfE_1=L^2_{s_1}(\R)$,
$\bfF_1=L^2_{-s_2}(\R)$,
$\bfE_2=L^2_{s}(\R)$,
$\bfF_2=L^2_{-s'}(\R)$;
$\bfE=\bfE_1\cap\bfE_2=L^2_S(\R)$,
$\bfF=\bfF_1+\bfF_2=L^2_{-S'}(\R)$,
with $S=\max(s_1,s)$ and $S'=\max(s_2,s')$;
we infer that there is also a convergence
of $(-\Delta+W-z I)^{-1}$
in $\scrB\big(L^2_{s}(\R),L^2_{-s'}(\R)\big)$.
By Theorem~\ref{theorem-a}~\itref{theorem-a-1},
this convergence holds in the weak operator topology;
then Theorem~\ref{theorem-lap}~\itref{theorem-lap-2d}
and Theorem~\ref{theorem-1d} provide the
convergence in the strong operator topology
(uniform if $s+s'>2$).

For
Part~\itref{theorem-1d-general-1b},
the convergence in the weak$^*$ operator topology of
$\scrB\big(L^1_\varsigma(\R),L^\infty_{\varsigma-1}(\R)\big)$
for $0\le\varsigma\le 1$
follows from
Theorem~\ref{theorem-1d}
and
Theorem~\ref{theorem-factorization}~\itref{theorem-factorization-2}
where we take
$\bfE=L^2_{s_1}(\R)$,
$\bfF=L^2_{-s_2}(\R)$,
$\tilde\bfE=L^1_\varsigma(\R)$,
and
$\tilde\bfF=L^\infty_{\varsigma-1}(\R)$.
The restrictions $s_1>1/2+\varsigma$
and $s_2>3/2-\varsigma$
are to ensure that
$\bfE\hookrightarrow\tilde\bfE$
and
$\tilde\bfF\hookrightarrow\bfF$,
while the condition \eqref{of-varsigma}
corresponds to the assumption of collective compactness
of $\upvarepsilon\circ\calB\circ\upvarphi\circ\tilde\calR(z)$
in
Theorem~\ref{theorem-factorization}~\itref{theorem-factorization-2}.
This concludes the proof.
\end{proof}

\begin{proof}[Proof of Theorem~\ref{theorem-1d}]
There are expositions
of constructions of the Jost solutions
by many authors;
see \cite[Appendix]{faddeev1963inverse} for the early story of the subject.
We chose to give an elementary derivation
of the estimates on the Jost solutions
providing explicit constants;
we
build upon the argument from \cite[pp.~325--326]{chadan1989inverse}.

Let $V\in L^1(\R)$
be a measurable complex-valued function on $\R$.
We consider the spectral problem
for the Schr\"odinger operator $H=-\Delta+V$
in $L^2(\R)$
with domain $\dom(H)=H^2(\R)$:
\begin{eqnarray}\label{stationary-d1}
(-\Delta+V(x))\psi = \zeta^2\psi,
\qquad
\psi(x)\in\C,
\quad
x\in\R,
\quad
\zeta\in\overline{\C_{+}}.
\end{eqnarray}
Below, we will use the following notations:
\[
x^\pm=\abs{x}\unity_{\R_\pm}(x),
\quad x\in\R,
\quad
\mbox{so that}
\quad
\langle x^{-}\rangle
=
\begin{cases}
\langle x\rangle,&x<0,
\\
1,&x\ge 0,
\end{cases}
\qquad
\langle x^{+}\rangle
=
\begin{cases}
1,&x\le 0,
\\
\langle x\rangle,&x>0;
\end{cases}
\]
\[
M^{+}_\varsigma(x)
=\int_x^{+\infty}\langle y\rangle^\varsigma\abs{V(y)}\,dy,
\qquad
M^{-}_\varsigma(x)
=\int_{-\infty}^x\langle y\rangle^\varsigma\abs{V(y)}\,dy,
\qquad
x\in\R,
\quad
\varsigma=0,\,1.
\]
The values $M^{\pm}_1(x)$ are finite
if $V$ is of Faddeev class (that is, belongs to $L^1_1(\R)$).

\begin{lemma}\label{lemma-jost}
Let $V\in L^1(\R)$.
Equation \eqref{stationary-d1} has two
distributional solutions,
$\theta_{+}(x,\zeta)$ and $\theta_{-}(x,\zeta)$,
which are
continuous
for all $x\in\R$
and $\zeta\in\overline{\C_{+}}\setminus\{0\}$,
continuously differentiable in $x\in\R$
for each $\zeta\in\overline{\C_{+}}$,
for each $x\in\R$ are
analytic in $\zeta\in\C_{+}$,
and satisfy the following estimates
for all $x\in\R$ and for all $\zeta\in\overline{\C_{+}}\setminus\{0\}$:
\begin{eqnarray}
\abs{\theta_{+}(x,\zeta)}
&\le&
e^{\fra{M^{+}_0(x)}{\abs{\zeta}}}
e^{- x\Im\zeta},
\label{thetaestimate-nz}
\\[0.5ex]
\abs{\theta_{+}(x,\zeta)-e^{\jj\zeta x}}
&\le&
\abs{\zeta}^{-1}M^{+}_0(x)
e^{\fra{M^{+}_0(x)}{\abs{\zeta}}}
e^{- x\Im\zeta},
\label{thetaestimate-1-nz}
\\[0.5ex]
\abs{\p_x\theta_{+}(x,\zeta)-\jj\zeta e^{\jj\zeta x}}
&\le&
M^{+}_0(x)
e^{\fra{M^{+}_0(x)}{\abs{\zeta}}}
e^{-x\Im\zeta},
\label{theta-1-prime-nz}
\\[0.5ex]
\abs{\theta_{-}(x,\zeta)}
&\le&
e^{\fra{M^{-}_0(x)}{\abs{\zeta}}}
e^{x\Im\zeta},
\label{thetaestimate-minus-nz}
\\[0.5ex]
\abs{\theta_{-}(x,\zeta)-e^{\jj\zeta x}}
&\le&
\abs{\zeta}^{-1}M^{-}_0(x)
e^{\fra{M^{-}_0(x)}{\abs{\zeta}}}
e^{x\Im\zeta},
\\[0.5ex]
\abs{\p_x\theta_{-}(x,\zeta)+\jj\zeta e^{-\jj\zeta x}}
&\le&
M^{-}_0(x)
e^{\fra{M^{-}_0(x)}{\abs{\zeta}}}
e^{x\Im\zeta}.
\end{eqnarray}
If, moreover, $V\in L^1_1(\R)$, then
$\theta_{+}(x,\zeta)$ and $\theta_{-}(x,\zeta)$,
are
continuous
for all $x\in\R$
and $\zeta\in\overline{\C_{+}}$,
continuously differentiable in $x\in\R$
for each $\zeta\in\overline{\C_{+}}$,
and satisfy the following estimates
for all $x\in\R$ and for all $\zeta\in\overline{\C_{+}}$:
\begin{eqnarray}
\label{thetaestimate}
\abs{\theta_{+}(x,\zeta)}
&\le&
\langle x^{-}\rangle
e^{\fra{\sqrt{2}M^{+}_1(x)}{\langle\zeta\rangle}}
e^{-x\Im\zeta},
\\[0.5ex]
\abs{\theta_{+}(x,\zeta)-e^{\jj\zeta x}}
&\le&
\sqrt{2}\langle \zeta \rangle^{-1}\langle x^{-}\rangle
M^{+}_1(x)
e^{\fra{\sqrt{2}M^{+}_1(x)}{\langle\zeta\rangle}}
e^{- x\Im\zeta},
\label{theta-1-such-0}
\\[0.5ex]
\abs{\p_x\theta_{+}(x,\zeta)-\jj\zeta e^{\jj\zeta x}}
&\le&
M^{+}_1(x)
e^{\fra{\sqrt{2}M^{+}_1(x)}{\langle\zeta\rangle}}
e^{- x\Im\zeta}
,
\label{theta-1-prime}
\\[0.5ex]
\abs{\theta_{-}(x,\zeta)}
&\le&
\langle x^{+}\rangle
e^{\fra{\sqrt{2}M^{-}_1(x)}{\langle\zeta\rangle}}
e^{x\Im\zeta}
,
\label{thetaestimate-minus}
\\[0.5ex]
\abs{\theta_{-}(x,\zeta)-e^{-\jj\zeta x}}
&\le&
\sqrt{2}
\langle \zeta \rangle^{-1}
\langle x^{+}\rangle
M^{-}_1(x)
e^{\fra{\sqrt{2}M^{-}_1(x)}{\langle\zeta\rangle}}
e^{x\Im\zeta}
,
\quad
\label{theta-2-such-0}
\\[0.5ex]
\abs{\p_x\theta_{-}(x,\zeta)+\jj\zeta e^{-\jj\zeta x}}
&\le&
M^{-}_1(x)
e^{\fra{\sqrt{2}M^{-}_1(x)}{\langle\zeta\rangle}}
e^{x\Im\zeta}
.
\label{theta-2-prime}
\end{eqnarray}
\end{lemma}

\begin{proof}
We only prove the estimates on $\theta_{+}$;
the estimates on $\theta_{-}$ follow by reflection
$x\leftrightarrow -x$.
We first prove the following lemma.
\begin{lemma}\label{lemma-e-over-zeta}
The function
$
E(y-x,\zeta)=
\begin{cases}
\frac{e^{2 \jj \zeta (y-x)} - 1}{2\jj\zeta},&\zeta\ne 0
\\
y-x,&\zeta=0
\end{cases}
$
satisfies the following bound:
\begin{eqnarray}\label{e-over-zeta}
\abs{E(y-x,\zeta)}
\le
\min
\Big\{\frac{1}{|\zeta|},\abs{y-x}\Big\}
\leq
\sqrt{2}\langle x^{-}\rangle \langle y^{+}\rangle\langle \zeta \rangle^{-1},
\qquad
y\ge x,
\quad
\zeta\in\overline{\C_{+}}.
\end{eqnarray}
\end{lemma}

\begin{proof}
The case $\zeta=0$ is immediate;
now assume that $\zeta\ne 0$.
The first inequality
in \eqref{e-over-zeta}
is straightforward.
For the second inequality,
we note that $|y - x| \leq \langle x^{-}\rangle \langle y^{+}\rangle$ for $x \leq y$.
Then, denoting
$A:=\langle x^{-}\rangle\langle y^{+}\rangle\ge 1$,
we have:
\[
\min\Big\{\frac{1}{\abs{\zeta}},A\Big\}
\le
\begin{cases}
\abs{\zeta}^{-1}
\le\sqrt{2}/\langle\zeta\rangle\le A\sqrt{2}/\langle\zeta\rangle
,&\abs{\zeta}\ge 1;
\\
A\le A\sqrt{2}/\langle\zeta\rangle,&\abs{\zeta}\le 1.
\end{cases}
\qedhere
\]
\end{proof}

The notation
$F(x,\zeta) = e^{-\jj\zeta x}\theta_{+}(x,\zeta)$ renders
the standard integral equation for $\theta_{+}$,
\begin{equation}
\theta_{+}(x,\zeta)
=e^{\jj \zeta x}
+\int_x^{+\infty}
S(y-x,\zeta)
V(y)
\theta_{+}(y,\zeta)\,dy,
\qquad
S(y,\zeta)
=
\begin{cases}
\frac{\sin(\zeta y)}{\zeta},&\zeta\ne 0,
\\
y,&\zeta=0,
\end{cases}
\label{theta-0}
\end{equation}
into
\begin{equation}
F(x,\zeta) = 1+\int_x^{+\infty}
E(y-x,\zeta)
V(y)
F(y,\zeta)\,dy,
\qquad
E(y,\zeta)
=
\begin{cases}
\frac{e^{2 \jj \zeta y} - 1}{2 \jj \zeta}
,&\zeta\ne 0,
\\
y,&\zeta=0.
\end{cases}
\label{Theta}
\end{equation}
To solve equation \eqref{Theta}, one decomposes
$F(x,\zeta)=\sum_{n\in\N_0}F_n(x,\zeta)$,
defining $F_n$ by
\begin{eqnarray}\label{rec}
\displaystyle
F_{0}(x,\zeta)=1;
\qquad
F_{n}(x,\zeta) = \int_x^{+\infty}
E(y-x,\zeta)
V(y)
F_{n-1}(y,\zeta)\,dy,
\quad
n \in\N.
\end{eqnarray}
We claim that
\begin{equation}
\abs{F_n(x,\zeta)}\le\frac{\abs{\zeta}^{-n}}{n!}
\bigg[\int_x^{+\infty}\abs{V(y)}\,dy\bigg]^n,
\quad n \in\N_0.
\label{Theta_n-nz}
\end{equation}
The proof is by induction:
\eqref{Theta_n-nz} holds for $n=0$;
if \eqref{Theta_n-nz} holds for some $n\in\N_0$,
then
\begin{eqnarray}\label{Eq:BoundFn-nz}
&&
\abs{F_{n+1}(x,\zeta)}
\leq \int_x^{+\infty}
\abs{E(y-x,\zeta)}
\abs{V(y)} |F_{n}(y,\zeta)|\,dy
\nonumber
\\
&&
\le
\int_x^{+\infty}
\frac{1}{\abs{\zeta}}
\abs{V(y)} \frac{\abs{\zeta}^{-n}}{n!}
 \bigg[\int_y^{+\infty} \abs{V(t)}\,dt\bigg]^n\,dy
\le \frac{\abs{\zeta}^{-n-1}}{(n+1)!}
\bigg[\int_x^{+\infty}\abs{V(y)}\,dy\bigg]^{n+1},
\end{eqnarray}
justifying the induction step.
Here we used Lemma~\ref{lemma-e-over-zeta}
and the identity
\begin{equation}
\int_x^{+\infty} g(y)
\bigg[\int_y^{+\infty} g(t)\,dt\bigg]^n\,dy
=\frac{1}{n+1}\bigg[\int_x^{+\infty} g(y)\,dy\bigg]^{n+1}, \quad x \in\R, \; n \in\N_0,
\end{equation}
valid for any $g\in L^1(\R)$.
Thus,
$
F(x,\zeta) = \sum_{n\in\N_0}F_{n}(x,\zeta)
$
converges absolutely for $x\in\R$
and $\zeta\in\overline{\C_{+}}\setminus\{0\}$,
and uniformly
with respect to
$x\ge x_0$ for any fixed $x_0\in\R$
and with respect to
$\zeta\in\overline{\C_{+}}\setminus\mathbb{D}_\epsilon$,
for each $\epsilon>0$.
Now the bounds \eqref{thetaestimate-nz}
and \eqref{thetaestimate-1-nz}
on $\theta_{+}(x,\zeta)=e^{\jj \zeta x}F(x,\zeta)$
follow from \eqref{Theta_n-nz}
due to inequalities
\begin{eqnarray*}
&&
\abs{F(x,\zeta)}
\leq
\sum\limits_{n=0}^{\infty} |F_{n}(x,\zeta)|
\leq
\sum\limits_{n=0}^{\infty} \frac{1}{n!\abs{\zeta}^n}
\bigg[\int_x^{+\infty}\abs{V(y)}\,dy\bigg]^n
=
\exp\biggr(\frac{1}{\abs{\zeta}}
\int_x^{+\infty}\abs{V(y)}\,dy\biggr),
\\
&&
\ds
\abs{F(x,\zeta)-1}
\le
\frac{1}{\abs{\zeta}}
\int_x^{+\infty}
\abs{V(y)}\,dy
\,
\exp\bigg(
\frac{1}{\abs{\zeta}}
\int_x^{+\infty}
\abs{V(y)}\,dy\bigg);
\end{eqnarray*}
we used the inequality
$e^t - 1\leq t\,e^t$ valid for all
$t\ge 0$.
Moreover,
due to the continuity in $x\in\R$
and $\zeta\in\overline{\C_{+}}\setminus\{0\}$
and analyticity in
$\zeta\in\C_{+}$
of each of $F_n(x,\zeta)$ (proved recursively)
and due to the bounds
on $F_n(x,\zeta)$
which are uniform in $x\ge x_0$ and
$\zeta\in\overline{\C_{+}}\setminus\mathbb{D}_\epsilon$
(with fixed $\epsilon>0$ and $x_0\in\R$),
for each $n\in\N_0$,
$\theta_{+}(x,\zeta)$ is continuous in $x$
and $\zeta$
for $(x,\zeta)\in(\R,\overline{\C_{+}}\setminus\{0\})$
and analytic in $\zeta\in\C_{+}$ for any $x\in\R$.

In the case when $V\in L^1_1(\R)$,
one proves by induction the estimate
\begin{equation}
|F_{n}(x,\zeta)| \leq \frac{2^{n/2} \langle x^{-}\rangle}{n!\langle \zeta \rangle^n}
\bigg[\int_x^{+\infty} \langle y\rangle \abs{V(y)}\,dy\bigg]^n, \quad n \in\N_0.
\label{Theta_n}
\end{equation}
The proof follows that of \eqref{Theta_n-nz};
now, to justify the induction step,
one uses the bound
$\abs{E(y-x,\zeta)}
\le \sqrt{2}
\langle x^{-}\rangle\langle y^{+}\rangle$
from Lemma~\ref{lemma-e-over-zeta}
and the identity
$\langle y\rangle
=\langle y^{-}\rangle\langle y^{+}\rangle$, $y\in\R$:
\begin{align*}
&
|F_{n+1}(x,\zeta)|
\leq \int_x^{+\infty}
|E(y-x,\zeta)|
\abs{V(y)} |F_{n}(y,\zeta)|\,dy
\\
&
\leq
\int\limits_x^{+\infty}
\frac{\sqrt{2}\langle x^{-}\rangle \langle y^{+}\rangle}{\langle \zeta \rangle}
\abs{V(y)} \frac{2^{\frac{n}{2}} \langle y^{-}\rangle}
{n!\langle \zeta \rangle^n}
 \bigg[\int\limits_y^{+\infty} \langle t\rangle |V(t)|\,dt\bigg]^n dy
\leq \frac{2^{\frac{n+1}{2}}\langle x^{-}\rangle}{(n+1)!\langle \zeta \rangle^{n+1}}
\bigg[
\int\limits_x^{+\infty} \langle y\rangle \abs{V(y)}\,dy\bigg]^{n+1}.
\end{align*}
Using \eqref{Theta_n}, we arrive at
\begin{eqnarray*}
&
\ds
|F(x,\zeta)| \leq
\sum\limits_{n=0}^{\infty} |F_{n}(x,\zeta)|
\le
\langle x^{-}\rangle\exp\bigg(\frac{\sqrt{2}}{\langle\zeta\rangle}
\int_x^{+\infty}\langle y\rangle \abs{V(y)}\,dy\bigg),
\\
&
\ds
\abs{F(x,\zeta)-1}
\le
\langle x^{-}\rangle
\frac{\sqrt{2}}{\langle \zeta \rangle}
\int_x^{+\infty}
\langle y\rangle \abs{V(y)}\,dy\,
\exp\bigg(\frac{\sqrt{2}}{\langle \zeta \rangle}
\int_x^{+\infty}
\langle y\rangle\abs{V(y)}\,dy\bigg),
\end{eqnarray*}
which imply the bounds \eqref{thetaestimate}
and \eqref{theta-1-such-0}.
Due to the continuity in $x\in\R$
of each of $F_n(x,\zeta)$ (proved recursively)
and 
due to the  bounds
on $F_n(x,\zeta)$
which are uniform
in $x\ge x_0$ and $\zeta\in\overline{\C_{+}}$
(for a fixed $x_0\in\R$),
for each $n\in\N_0$,
$\theta_{+}(x,\zeta)$ is continuous in $x$
for $(x,\zeta)\in(\R,\overline{\C_{+}})$.

Finally we prove
the continuous differentiability in $x$.
One can see that the
second term in the right-hand side of \eqref{theta-0}
is differentiable
with respect to $x$, with the result given by
\begin{eqnarray*}
\lim\sb{h\to 0}
\frac{1}{h}
\Biggr[
\int_{x+h}^{+\infty} \frac{\sin(\zeta (y-x-h))}{\zeta}
V\theta_{+}
\,dy
-
\int_x^{+\infty} \frac{\sin(\zeta (y-x-h))}{\zeta}
V\theta_{+}
\,dy
\\
+
\int_x^{+\infty} \frac{\sin(\zeta (y-x-h))}{\zeta}
V\theta_{+}
\,dy
-\int_x^{+\infty} \frac{\sin(\zeta (y-x))}{\zeta}
V\theta_{+}
\,dy
\Biggr]
\\
=
-\lim\sb{h\to 0}
\int_{x}^{x+h}\frac{\sin(\zeta (y-x-h))}{h\zeta}
V\theta_{+}
\,dy
+
\int_x^{+\infty}\frac{\p}{\p x}
\frac{\sin(\zeta (y-x))}{\zeta}
V\theta_{+}
\,dy,
\end{eqnarray*}
where $V\theta_{+}=V(y)\theta_{+}(y,\zeta)$
is measurable.
The first term in the right-hand side
vanishes due to the inequality
$\sup\sb{y\in(x,x+h)}
\big|(h\zeta)^{-1}\sin(\zeta(y-x-h))\big|\le 1$.
We conclude that $\theta_{+}(x,\zeta)$
has continuous derivative in $x$.
We now differentiate \eqref{theta-0}
with respect to $x$:
\[
\big|\p_x\theta_+(x,\zeta) -\jj\zeta e^{\jj\zeta x}\big|
=
\Big|
\int\limits_x^{+\infty}\cos(\zeta(y-x)) V(y)\theta_+(y,\zeta)\,dy
\Big|
\le
\Big|
\int\limits_x^{+\infty}
e^{(y-x)\Im\zeta}
V(y)\theta_+(y,\zeta)\,dy
\Big|.
\]
Using for $\abs{\theta_+(y,\zeta)}$
the bounds
\eqref{thetaestimate-nz} or \eqref{thetaestimate}
(the latter available when $V\in L^1_1(\R)$),
one arrives at
\eqref{theta-1-prime-nz}
and
\eqref{theta-1-prime},
respectively.
This completes the proof of Lemma~\ref{lemma-jost}.
\end{proof}

Denote the Wronskian of $\theta_{+}$ and $\theta_{-}$
at $\zeta\in\overline{\C_{+}}$
by
\begin{eqnarray}\label{w-zeta}
w(\zeta)=
W(\theta_{+}(\cdot,\zeta),\theta_{-}(\cdot,\zeta))\at{x=0},
\qquad
\zeta\in\overline{\C\sb{+}},
\end{eqnarray}
with the Wronskian defined by
$W(f,g)(x)=f(x)g'(x)-f'(x)g(x)$,
$f,\,g\in C^1(\R)$.
We note that $w(\zeta)$ is nonzero
for $\abs{\zeta}$ is sufficiently large:

\begin{lemma}\label{Lem:OnTheWronskian}
For $\zeta\in\overline{\C_{+}}$,
one has
$\abs{w(\zeta)}
\ge
2\abs{\zeta}-2(2+\sqrt{2}) M e^{2\sqrt{2}M/\langle\zeta\rangle}.
$
\end{lemma}

\begin{proof}
For $\zeta\in\C_{+}$,
using the estimates
\eqref{theta-1-prime}, \eqref{theta-2-prime}
and then
\eqref{theta-1-such-0}, \eqref{theta-2-such-0}
from Lemma~\ref{lemma-jost},
we compute at $x=0$:
\begin{eqnarray*}
\abs{w(\zeta)}
&=&
\abs{\theta_{+}\p_x \theta_{-}
-
\p_x\theta_{+}\theta_{-}}
\ge
\abs{\jj\zeta\theta_{+}
-(-\jj\zeta)\theta_{-}}
-
\abs{\theta_{+}}M e^{\sqrt{2}M/\langle\zeta\rangle}
-
\abs{\theta_{-}}M e^{\sqrt{2}M/\langle\zeta\rangle}
\\
&\ge&
2\abs{\zeta}
-
2\abs{\zeta} \frac{\sqrt{2}M}{\langle\zeta\rangle}e^{\sqrt{2}M/\langle\zeta\rangle}
-e^{\sqrt{2}M/\langle\zeta\rangle}M e^{\sqrt{2}M/\langle\zeta\rangle}
-e^{\sqrt{2}M/\langle\zeta\rangle}M e^{\sqrt{2}M/\langle\zeta\rangle},
\end{eqnarray*}
finishing the proof.
Above,
$\theta_\pm=\theta_\pm(x,\zeta)$
and
$\p_x\theta_\pm=\p_x\theta_\pm(x,\zeta)$
are evaluated at $x=0$.
\end{proof}

If $\zeta\in\overline{\C_{+}}$
is such that $w(\zeta)\ne 0$, we define the operator
$
G(\zeta):\;\mathscr{D}(\R)\to\mathscr{D}'(\R)
$
specifying its integral kernel
$\calK(G(\zeta))(x,y)$, $(x,y)\in\R\times\R$,
\begin{eqnarray}\label{def-G-1d}
\calK(G(\zeta))(x,y)
=
\frac{1}{w(\zeta)}
\begin{cases}
\theta_{+}(y,\zeta)\theta_{-}(x,\zeta),&x\le y,
\\[0.2ex]
\theta_{-}(y,\zeta)\theta_{+}(x,\zeta),&x\ge y.
\end{cases}
\end{eqnarray}
For $\zeta^2\not\in\sigma(-\Delta+V)$
the operator $G(\zeta)$ coincides with
the resolvent
\[
(-\Delta+V-\zeta^2 I)^{-1}
:\,L^2(\R)\to L^2(\R),
\]
and for $\zeta\ge 0$
it is the boundary trace
of the resolvent
on $\R_{+}+\jj 0$.

Let $\zeta\in\C_{+}$ and assume that $w(\zeta)\ne 0$.
By Lemma~\ref{lemma-jost}
(estimates
\eqref{thetaestimate} and \eqref{thetaestimate-minus}),
the integral kernel of $G(\zeta)$,
$\calK(G(\zeta))(x,y)$, $(x,y)\in\R\times\R$,
satisfies the estimates
\begin{eqnarray}\label{G-estimates}
\abs{\calK(G(\zeta))(x,y)}
\le
\frac{e^{2\sqrt{2}M_1/\langle\zeta\rangle}}{\abs{w(\zeta)}}
\begin{cases}
\quad 1,&x\le 0\le y\quad\mbox{or}\quad y\le 0\le x,
\\[0.2ex]
\langle y\rangle
,&x\le y\le 0\quad\mbox{or}\quad 0\le y\le x,
\\[0.2ex]
\langle x\rangle
,&0\le x\le y\quad\mbox{or}\quad y\le x\le 0.
\end{cases}
\end{eqnarray}
Moreover, by
\eqref{thetaestimate-nz} and \eqref{thetaestimate-minus-nz},
if $\zeta\in\overline{\C_{+}}\setminus\{0\}$ and $w(\zeta)\ne 0$,
then
the integral kernel of $G(\zeta)$ satisfies the estimates
\begin{eqnarray}\label{G-estimates-nonzero}
\abs{\calK(G(\zeta))(x,y)}
\le
\fra{e^{2 M_0/\abs{\zeta}}}{\abs{w(\zeta)}},
\qquad
x,\,y\in\R.
\end{eqnarray}
Above, $M_0=\int_\R\abs{V(x)}\,dx$
and $M_1=\int_\R\langle x\rangle\abs{V(x)}\,dx$.

\begin{lemma}\label{lemma-d1}
\begin{enumerate}
\item
\label{lemma-d1-1}
Let $G:\mathscr{D}(\R)\to\mathscr{D}'(\R)$ be the operator
with integral kernel
$\calK(G)(x,y)$, $(x,y)\in\R\times\R$,
satisfying the estimate
\begin{eqnarray}\label{G-estimates-weak}
\abs{\calK(G)(x,y)}
\le
\min(\langle x\rangle,\,\langle y\rangle),
\qquad
x,\,y\in\R.
\end{eqnarray}
Then $G$ extends to continuous linear mappings
\begin{eqnarray*}
&&
G:\;L^1_\varsigma(\R)\to L^\infty_{\varsigma-1}(\R),
\qquad
0\le\varsigma\le 1;
\\[1ex]
&&
G:\;
L^2_{s}(\R)\to L^2_{-s'}(\R),
\qquad
s,\,s'>1/2,
\quad
s+s'\ge 2.
\end{eqnarray*}
\item
\label{lemma-d1-2}
Assume that
there is an operator family $G(\zeta)$,
$\zeta\in\upomega\subset\C$,
with integral kernels
$\calK(G(\zeta))(x,y)$, $(x,y)\in\R\times\R$,
satisfying the estimates
\eqref{G-estimates-weak}
and with the pointwise convergence
of $\calK(G(\zeta))(x,y)$
(to some integral kernel $\calK_0(x,y)$, $(x,y)\in\R\times\R$)
as $\zeta\to\zeta_0\in\p\upomega$, $\zeta\in\upomega$,
which is uniform on compact subsets of
$\R\times\R$.
Then
$G(\zeta)$ converges as $\zeta\to\zeta_0$,
$\zeta\in\upomega$,
in the weak$^*$ topology of
$L^1_\varsigma(\R)\to L^\infty_{\varsigma-1}(\R)$, $0\le\varsigma\le 1$,
in the strong operator topology
of
$\scrB\big(L^2_s(\R),L^2_{-s'}(\R)\big)$,
$s,\,s'>1/2$, $s+s'\ge 2$,
and in the uniform operator topology
of
$\scrB\big(L^2_s(\R),L^2_{-s'}(\R)\big)$, $s,\,s'>1/2$, $s+s'>2$.
\end{enumerate}
\end{lemma}

\begin{proof}
We start with $L^1_\varsigma\to L^\infty_{\varsigma-1}$ estimates,
$0\le\varsigma\le 1$.
By \eqref{G-estimates-weak},
\begin{eqnarray}\label{G-estimates-weak-1}
\abs{\calK(G)(x,y)}
\le
\langle x\rangle^{1-\varsigma}\langle y\rangle^\varsigma,
\qquad
x,\,y\in\R,
\qquad
0\le\varsigma\le 1.
\end{eqnarray}
Therefore, if $\phi\in L^1_\varsigma(\R)$,
then
$\abs{G\phi(x)}\le\langle x\rangle^{1-\varsigma}
\norm{\phi}_{L^1_\varsigma}$,
for any $x\in\R$,
so
$\norm{G}_{L^1_\varsigma\to L^\infty_{\varsigma-1}}\le 1$
for all $0\le\varsigma\le 1$.
Due to the embeddings
$L^2_s(\R)\hookrightarrow L^1_\varsigma(\R)$
if $s>1/2+\varsigma$
and $L^\infty_{\varsigma-1}\hookrightarrow L^2_{-s'}(\R)$
if $s'>3/2-\varsigma$, for all $0\le\varsigma\le 1$,
the $L^1_\varsigma\to L^\infty_{\varsigma-1}$ estimates
implies the $L^2_s\to L^2_{-s'}$ estimates in the case $s,\,s'>1/2$,
$s+s'> 2$.

To prove $L^2_s\to L^2_{-s'}$ estimates
in the case $s,\,s'>1/2$, $s+s'=2$,
let us decompose $G=\sum\sb{\sigma,\,\sigma'\in\{\pm\}}
\unity_{\R_\sigma}\circ G\circ\unity_{\R_{\sigma'}}$;
it suffices to consider $\unity_{\R_{+}}\circ G\circ\unity_{\R_{+}}$.
It is enough to show that the operators
$G_1,\,G_2:\,\mathscr{D}(\R_{+})\to\mathscr{D}'(\R_{+})$
with the integral kernels
\begin{eqnarray}\label{def-t-kernel}
\calK(G_1)(x,y)
=
\langle y\rangle
\unity_{\R_{+}}(x)
\unity_{[0,x]}(y),
\quad
\calK(G_2)(x,y)
=
\langle x\rangle
\unity_{[0,y]}(x)
\unity_{\R_{+}}(y),
\quad
x,\,y\in\R_{+},
\end{eqnarray}
have the regularity properties announced in the lemma.

We apply the notion of almost orthogonality.
For convenience, in the integral kernel $\calK(G_1)(x,y)$ of $G_1$
(see \eqref{def-t-kernel}),
we shift both $x$ and in $y$ by one;
we thus need to show that the operator $T$
defined by the integral kernel
\begin{eqnarray}\label{kt0}
\calK(T)(x,y)
:=
\langle x\rangle^{-s'}
\unity_{[1,+\infty)}(x)
\unity_{[1,x]}(y)
\langle y\rangle^{1-s},
\qquad
x,\,y\in\R_{+},
\end{eqnarray}
is bounded in $L^2((1,+\infty))$
as long as $s,\,s'>1/2$ and 
$s+s'=2$
(this in turn will imply that
the operator $G_1$ is bounded in $L^2(\R_{+})$).
We decompose the operator
$T$
into the infinite sum
\begin{eqnarray}\label{t-tra}
T=
\sum\sb{\substack{R=2^j,\,A=2^k\\j,\,k\in\N_0}}
T^{}_{R,A},
\end{eqnarray}
with the operators $T^{}_{R,A}$
defined by the integral kernels
\[
\calK(T^{}_{R,A})(x,y)
=
\unity_{[R,2R]}(x)
\langle x\rangle^{-s'}
\langle y\rangle^{1-s}
\unity_{[1,x]}(y)
\unity_{[A,2A]}(y),
\quad
R=2^j,
\ A=2^k,
\ j,\,k\in\N_0.
\]
Since
on the support of
$\calK(T^{}_{R,A})(x,y)$
one has
$R\le x\le 2R$ and $A\le y\le 2A$,
while on the support of $\calK(T)(x,y)$
one has $y\le x$,
we may assume that
\begin{eqnarray}\label{a-le-r}
A\le R
\end{eqnarray}
(or else $T^{}_{R,A}=0$).
We claim that
the operators
$T^{}_{R,A}$
are bounded in $L^2$,
uniformly in $R$ and $A$,
and moreover are \emph{almost orthogonal},
in the sense of Cotlar--Stein
\cite{cotlar1955,stein1993harmonic}.
To prove the almost orthogonality
of operators $T^{}_{R,A}$ with different $R=2^j$ and $A=2^k$, with $j,\,k\in\N_0$,
we consider the integral kernel of the operator
$T^{}_{R,A}T_{S,B}^*$, which is bounded by
\begin{eqnarray*}
&&
\abs{\calK(T^{}_{R,A}T_{S,B}^*)(x,z)}
\\
&&\le
\int\sb{\R}
\langle x\rangle^{-s'}
\langle y\rangle^{2-2s}
\langle z\rangle^{-s'}
\unity_{[R,2R]}(x)
\unity_{[1,x]}(y)
\unity_{[A,2A]}(y)
\unity_{[B,2B]}(y)
\unity_{[1,z]}(y)
\unity_{[S,2S]}(z)
\,dy
\\
&&\le
c\delta_{A B}
R^{-s'}A^{2-2s}S^{-s'}
\int\sb{\R}
\unity_{[R,2R]}(x)
\unity_{[1,x]}(y)
\unity_{[A,2A]}(y)
\unity_{[1,z]}(y)
\unity_{[S,2S]}(z)
\,dy,
\quad
x,\,z\in\R_{+},
\end{eqnarray*}
with some $c>0$.
Due to the support properties of the above expression,
$T^{}_{R,A}T_{S,B}^*=0$ unless
$A\le R$,
$A=B\le S$.
The Schur test yields
\[
\norm{T^{}_{R,A}T_{S,B}^*}
\le
c\delta_{A B}R^{-s'}A^{2-2s}S^{-s'}
A
\sqrt{RS}
=
c\delta_{A B}R^{\frac 1 2-s'}A^{3-2s}S^{\frac 1 2-s'},
\]
and hence
\begin{eqnarray}\label{above-0}
\sum\sb{\substack{S=2^j,B=2^k,\\j,k\in\N_0}}
\!\!\!\!\!\!
\norm{T^{}_{R,A}T_{S,B}^*}^{\frac{1}{2}}
\le
c
\!\!\!\!
\sum\sb{\substack{S=2^j,\,j\in\N_0,\\S\ge A}}
\!\!\!\!\!\!
\big(R^{\frac 1 2-s'}A^{3-2s}S^{\frac 1 2-s'}\big)^{\frac{1}{2}}
\le
c'\big(R^{\frac 1 2-s'}A^{3+\frac 1 2-2s-s'}\big)^{\frac{1}{2}}
\le
c',
\end{eqnarray}
with some $c'>0$.
Above, for the summation in $S=2^j$, $j\in\N_0$,
we took into account that $s'>1/2$ and that $S\ge A$;
the last inequality follows since
$1\le A\le R$ and
$s+s'=2$.

Next, we consider
the operator $T_{S,B}^* T^{}_{R,A}$,
with its integral kernel bounded by
\[
\hskip -1pt
\abs{\calK(T_{S,B}^* T^{}_{R,A})
(t,y)}
\!
\le
\!
c
R^{-2s'}
A^{1-s}
B^{1-s}
\!\!
\int\limits\sb{\R}
\!\!
\unity_{[B,2B]}(t)
\unity_{[1,x]}(t)
\unity_{[S,2S]}(x)
\unity_{[R,2R]}(x)
\unity_{[1,x]}(y)
\unity_{[A,2A]}(y)
\,dx,
\]
with $t,\,y\in\R_{+}$;
we noted that because of the support properties
one needs to have $R=S$.
By \eqref{a-le-r},
we assume that
\begin{eqnarray}\label{a-r-b-r}
A\le R,
\qquad
B\le S=R,
\end{eqnarray}
since otherwise $T_{S,B}^* T^{}_{R,A}=0$.
The Schur test then yields
\[
\norm{T_{S,B}^* T^{}_{R,A}}
\le
c R^{-2s'}A^{1-s}B^{1-s}
R
\sqrt{AB}
=
c R^{1-2s'}A^{\frac 3 2-s}B^{\frac 3 2-s}.
\]
Thus, for fixed $R=2^\ell$ and $A=2^m$, with $\ell, m\in\N_0$,
taking into account that the terms with $S\ne R$ vanish,
we estimate:
\begin{eqnarray}\label{number}
\sum\sb{\substack{S=2^j,\,B=2^k,\\j,k\in\N_0}}
\norm{T_{S,B}^* T^{}_{R,A}}^{\frac{1}{2}}
\le
\sum\sb{\substack{B=2^k,\\k\in\N_0}}
\norm{T_{R,B}^* T^{}_{R,A}}^{\frac{1}{2}}
\le
c\sum\sb{\substack{B=2^k,\,k\in\N_0,\\B\le R}}
\big(R^{1-2s'}A^{\frac 3 2-s}B^{\frac 3 2-s}\big)^{\frac{1}{2}},
\end{eqnarray}
where in the last summation we indicated that $B\le R$ due to \eqref{a-r-b-r}.
 The summation in the right-hand side
of \eqref{number} leads to
\begin{eqnarray}\label{above-1}
\sum\sb{\substack{S=2^j,\,B=2^k,\\j,k\in\N_0}}
\norm{T_{S,B}^* T^{}_{R,A}}^{\frac{1}{2}}
\le
c
\big(R^{1-2s'}A^{\frac 3 2-s}R^{\frac 3 2-s}\big)^{\frac{1}{2}}
\le
c'
\big(R^{4-2s-2s'}\big)^{\frac{1}{2}}
\le c';
\end{eqnarray}
in the second inequality,
we took into account
that
$A\le R$ (see \eqref{a-r-b-r})
and that
$3/2-s>0$ (since $s'=2-s>1/2$).
Since
\eqref{above-0}
and \eqref{above-1}
are bounded,
the operators
$T^{}_{R,A}$, with $R=2^j$ and $A=2^k$, $j,\,k\in\N_0$,
are almost orthogonal,
and the
convergence
of the summation in the right-hand side of \eqref{t-tra}
in the strong operator topology
follows from the Cotlar--Stein almost orthogonality lemma
(see e.g. \cite[Lemma II.18]{opus}).
This completes the proof of
Part~\itref{lemma-d1-1}.

Let us prove Part~\itref{lemma-d1-2}.
The convergence
of $G(\zeta)$ to
$G(\zeta_0)$
as $\zeta\to\zeta_0$,
$\zeta\in\upomega$,
in the weak$^*$ topology of
$\scrB\big(L^1_\varsigma(\R),L^\infty_{\varsigma-1}(\R)\big)$,
$0\le\varsigma\le 1$,
immediately follows from the
convergence of integral kernels
uniformly on compacts.

Let us consider the case
$L^2_s(\R)\to L^2_{-s'}(\R)$,
$s,\,s'>1/2$, $s+s'\ge 2$;
we start with the case $s+s'=2$.
It is enough to prove that if
\[
U(\zeta):\;\mathscr{D}(\R)\to\mathscr{D}'(\R),
\qquad
\zeta\in\upomega,
\]
are integral operators with kernels
$
\calK(U(\zeta))(x,y)
$
pointwise convergent
(to some $\calK_0(x,y)$)
for each $x,\,y\in\R$
as $\zeta\to\zeta_0$, $\zeta\in\upomega$,
with the convergence uniform on compacts
in $\R\times\R$,
and satisfying
\[
\abs{\calK(U(\zeta))(x,y)}
\le\calK(T)(x,y),
\qquad
\forall x,\,y\in\R,
\]
with $\calK(T)(x,y)$ from \eqref{kt0},
then,
by the dominated convergence theorem,
$U(\zeta)$ converges
as $\zeta\to\zeta_0$, $\zeta\in\upomega$,
in the strong operator topology
of
$\scrB\big(L^2(\R)\big)$.
For $R=2^j$, $A=2^k$, $j,\,k\in\N_0$,
define the
integral operators
$U_{R,A}(\zeta)$, $\zeta\in\upomega$,
by their integral kernels
\[
\calK(U_{R,A}(\zeta))(x,y)
=\unity_{[R,2R]}(x)
\calK(U(\zeta))(x,y)
\unity_{[A,2A]}(y).
\]
According to Part~\itref{lemma-d1-1},
for each $\zeta\in\upomega$,
the operators $U_{R,A}(\zeta)$ are almost orthogonal;
therefore, by the Cotlar--Stein almost orthogonality lemma,
for each $\phi\in L^2_s(\R)$,
there is a convergence of the series
\[
\sum_{\substack{R=2^j,\,A=2^k\\j,\,k\in\N_0}}
U_{R,A}(\zeta)\phi,
\qquad
\sum_{\substack{R=2^j,\,A=2^k\\j,\,k\in\N_0}}
T_{R,A}\phi_0,
\]
with $\phi_0(x)=\abs{\phi(x)}$, $x\in\R$,
and moreover
these series are commutatively convergent
(this follows from the proof given in \cite[Lemma II.18]{opus}).
Fix $\varepsilon>0$.
Let $I_\varepsilon\subset\N_0\times\N_0$ be a finite subset such that
$\Norm{\sum\limits\sb{\substack{R=2^j,\,A=2^k\\j,\,k\in
(\N_0\times\N_0)\setminus I_\varepsilon}}
T_{R,A}\phi_0}_{L^2_{-s'}}<\varepsilon/3$.
For $\zeta,\,\zeta'\in\upomega$, we have:
\begin{eqnarray}
&&
\norm{U(\zeta)\phi-U(\zeta')\phi}_{L^2_{-s'}}
\le
\Norm{
\sum_{\substack{R=2^j,\,A=2^k\\(j,k)\in I_\varepsilon}}
(U_{R,A}(\zeta)-U_{R,A}(\zeta'))
\phi}_{L^2_{-s'}}
\label{s-s-s}
\\
\nonumber
&&
\qquad\qquad\qquad
+\,
\Norm{\sum_{\substack{R=2^j,\,A=2^k\\(j,k)\in(\N_0\times\N_0)\setminus I_\varepsilon}}
U_{R,A}(\zeta)\phi}_{L^2_{-s'}}
+
\Norm{\sum_{\substack{R=2^j,\,A=2^k\\(j,k)\in(\N_0\times\N_0)\setminus I_\varepsilon}}
U_{R,A}(\zeta')\phi}_{L^2_{-s'}}.
\end{eqnarray}
Since
$\abs{U(\zeta)_{R,A}\phi(x)}\le (T_{R,A}\phi_0)(x)$
for each $x\in\R$ and $\zeta\in\upomega$,
the last two terms in the right-hand side of \eqref{s-s-s}
are bounded by $\varepsilon/3$,
while the first term in the right-hand side
becomes smaller than $\varepsilon/3$ if
both $\zeta$ and $\zeta'$
are sufficiently close to $\zeta_0$
due to the pointwise convergence
of integral kernels
of $U(\zeta)$ as $\zeta\to\zeta_0$
which is uniform on compacts.

In the case $s+s'>2$,
the improvement to convergence
of $G(\zeta)$, $\zeta\in\upomega$, as $\zeta\to\zeta_0$
in the uniform operator topology of
$\scrB\big(L^2_s(\R),L^2_{-s'}(\R)\big)$
is provided by the following lemma.

\begin{lemma}\label{lemma-sus}
Let there be an operator family
$T(\zeta)\in\scrB\big(L^2_s(\R),L^2_{-s'}(\R)\big)$,
$\zeta\in\upomega\subset\C$,
with integral kernel
$\calK(T(\zeta))(x,y)$, $(x,y)\in\R\times\R$.
Assume that $T(\zeta)$
converges as $\zeta\to\zeta_0\in\p\upomega$
in the weak operator topology of
$\scrB\big(L^2_s(\R),L^2_{-s'}(\R)\big)$
and that the integral kernels
of $T(\zeta)$,
$\calK(T(\zeta))(x,y)$,
converge pointwise and uniformly on compacts
in $\R\times\R$.
Then for any $S>s$ and $S'>s'$,
the operator family
$T(\zeta)$
converges as $\zeta\to\zeta_0$
in the uniform operator topology of
$\scrB\big(L^2_{S}(\R),L^2_{-S'}(\R)\big)$.
\end{lemma}

\begin{proof}
It follows from the assumptions
that there is $\delta>0$
and $C>0$
such that
$\sup\sb{\zeta\in\upomega\cap\mathbb{D}_\delta(\zeta_0)}
\norm{T(\zeta)}_{L^2_s\to L^2_{-s'}}\le C<\infty$.
Let $\varepsilon>0$.
We decompose
\begin{eqnarray*}
T(\zeta)
&=&\unity_{\abs{x}<R_1}\circ T(\zeta)\circ\unity_{\abs{y}<R_2}
\\
&&
+(1-\unity_{\abs{x}<R_1})\circ T(\zeta)\circ\unity_{\abs{y}<R_2}
+\unity_{\abs{x}<R_1}\circ T(\zeta)\circ(1-\unity_{\abs{y}<R_2}).
\end{eqnarray*}
We can choose $R_1,\,R_2\ge 1$ so that
\[
C\norm{1-\unity_{\abs{y}<R_2}}_{L^2_{S}\to L^2_{s}}<\varepsilon/3,
\qquad
C\norm{1-\unity_{\abs{x}<R_1}}_{L^2_{-s'}\to L^2_{-S'}}<\varepsilon/3.
\]
Due to the pointwise convergence of the integral kernels
which is uniform on compacts,
there is $\delta_\varepsilon\in(0,\delta)$
such that
$\norm{\unity_{\abs{x}<R_1}\circ T(\zeta)\circ\unity_{\abs{y}<R_2}}
_{L^2_s\to L^2_{-s'}}<\varepsilon/3$
for any
$\zeta\in\upomega\cap\mathbb{D}_{\delta_\varepsilon}(\zeta_0)$.
\end{proof}

This completes the proof.
\end{proof}

\begin{lemma}
\label{lemma-1d}
\begin{enumerate}
\item
\label{lemma-1d-1}
If 
$\theta_{+}(\cdot,\zeta)$
and
$\theta_{-}(\cdot,\zeta)$
are linearly independent at $\zeta\in\overline{\C_{+}}$,
so that
\[
w(\zeta)=W(\theta_{+}(\cdot,\zeta),\theta_{-}(\cdot,\zeta))\ne 0,
\]
then the operator $G(\zeta)$
defined by the integral kernel \eqref{def-G-1d}
extends to continuous linear mappings
\begin{eqnarray*}
&&G(\zeta):\;
L^1_\varsigma(\R)\to L^\infty_{\varsigma-1}(\R),
\qquad
0\le\varsigma\le 1;
\\[1ex]
&&
G(\zeta):\;
L^2_{s}(\R)\to L^2_{-s'}(\R),
\qquad s,\,s'>1/2,\; s+s'\ge 2.
\end{eqnarray*}
There is $c>0$
(which does not depend on $\zeta$)
such that
for all $0\le\varsigma\le 1$
one has
\[
 \norm{G(\zeta)}_{L^1_\varsigma\to L^\infty_{\varsigma-1}}
 \le
 \fra{c e^{2\sqrt{2}M/\langle\zeta\rangle}}{\abs{w(\zeta)}};
\]
there is $c_{s, s'}>0$, depending only on $s,\,s'>1/2$, $s+s'\ge 2$,
but not on $\zeta$, such that
\[
\norm{G(\zeta)}_{L^2_s\to L^2_{-s}}
\le
\fra{c_{s, s'} e^{2\sqrt{2}M/\langle\zeta\rangle}}{\abs{w(\zeta)}}.
\]
If, moreover, $\zeta\ne 0$,
then
$G(\zeta)$ extends to a continuous linear mapping
$L^1(\R)\to L^\infty(\R)$, and
\[
\norm{G(\zeta)}_{L^1\to L^\infty}
\le
\fra{e^{2M/\abs{\zeta}}}{\abs{w(\zeta)}}.
\]
\item
\label{lemma-1d-2}
If $w(0)\ne 0$,
then in the limit $\zeta\to\zeta_0=0$, $\zeta\in\C_{+}$,
the operators $G(\zeta)$
converge to $G(\zeta_0)$
in the weak$^*$ operator topology of
$\scrB\big(L^1_\varsigma(\R),L^\infty_{\varsigma-1}(\R)\big)$,
$0\le\varsigma\le 1$,
in the strong operator topology of
$\scrB\big(L^2_s(\R),L^2_{-s'}(\R)\big)$
with $s,\,s'>1/2$, $s+s'\ge 2$,
and in the uniform operator topology
of $\scrB\big(L^2_s(\R),L^2_{-s'}(\R)\big)$
with $s,\,s'>1/2$, $s+s'>2$.
\end{enumerate}
\end{lemma}

\begin{proof}
Part~\itref{lemma-1d-1}
follows from
Lemma~\ref{lemma-d1}~\itref{lemma-d1-1}
applied to the integral kernel of $G(\zeta)$
which satisfies the bounds \eqref{G-estimates}.
The $L^1\to L^\infty$ estimate
in the case $\zeta\ne 0$
follows from the estimate \eqref{G-estimates-nonzero}
which is valid for the integral kernel of $G(\zeta)$
from \eqref{def-G-1d}.
The convergence
$G(\zeta)\to G(\zeta_0)$
stated in Part~\itref{lemma-1d-2}
follows from the pointwise convergence
of $\theta_\pm(x,\zeta)$
to $\theta_\pm(x,\zeta_0)$
and from Lemma~\ref{lemma-d1}~\itref{lemma-d1-2}.
\end{proof}

Now let us prove the dichotomy stated in Theorem~\ref{theorem-1d}.
If there is a solution
$\Psi\in H^2_{\mathrm{loc}}(\R)\cap L^\infty(\R)$ to $(-\Delta+V)\Psi=0$,
then the Jost solutions
$\theta_{+}(\cdot,0)$
and 
$\theta_{-}(\cdot,0)$
(cf. Lemma~\ref{lemma-jost})
are linearly dependent
and thus $w(0)=0$
(with the Wronskian $w(\zeta)$ from \eqref{w-zeta}),
implying
that $R_V(z):\,L^2_{s_1}(\R)\to L^2_{-s_2}(\R)$
can not be bounded uniformly near $z_0=0$
(cf. \eqref{def-G-1d}).

If instead there is no solution
$\Psi\in H^2_{\mathrm{loc}}(\R)\cap L^\infty(\R)$ to $(-\Delta+V)\Psi=0$,
then the Jost solutions
$\theta_{+}(\cdot,0)$
and 
$\theta_{-}(\cdot,0)$
are linearly independent,
so that $w(0)\ne 0$.
By Lemma~\ref{lemma-1d}~\itref{lemma-1d-2},
the resolvent
$R_V(z):\,L^2_{s}(\R)\to L^2_{-s'}(\R)$,
$s,\,s'>1/2$, $s+s'\ge 2$,
converges in the strong operator topology
(uniform operator topology if $s+s'>2$)
and in the weak$^*$ operator topology
of $\scrB\big(L^1_\varsigma(\R),L^\infty_{\varsigma-1}(\R)\big)$,
$0\le\varsigma\le 1$,
and is bounded uniformly
in $z\in\mathbb{D}_\delta\setminus\overline{\R_{+}}$
for some $\delta>0$
by Lemma~\ref{lemma-uniform}~\itref{lemma-uniform-1}.

This completes the proof of Theorem~\ref{theorem-1d}.
\end{proof}

\section{Schr\"odinger operators in two dimensions}
\label{sect-2d}
Just like in dimension $d=1$,
the Laplace operator in dimension $d=2$
has a virtual level at $z_0=0$.
Similarly to the one-dimensional case,
this virtual level can be removed
by adding an operator of multiplication
by a nonnegative nonzero function
to remove the virtual level at zero
(cf. Lemma~\ref{lemma-2d-pr}).
We choose to present another approach through
a perturbation of rank $1$;
we notice that such an approach
could also be used in dimension $d=1$.

\begin{theorem}[Model Schr\"odinger operators in $L^2(\R^2)$: LAP estimates]
\label{theorem-2d}
Let
\begin{eqnarray}\label{def-varphi}
\varPhi\in L^2_{1,\nu_0}(\R^2),
\qquad
\nu_0>3/2;
\qquad
\int_{\R^2}\varPhi(x)\,dx=1.
\end{eqnarray}
Denote
\begin{eqnarray}\label{def-u}
U_\varPhi=\varPhi\otimes\varPhi\in
\scrB_{00}\big(L^2_{-1,-\nu_0}(\R^2),L^2_{1,\nu_0}(\R^2)\big),
\quad
U_\varPhi:\,\psi\mapsto
\varPhi\langle\psi,\varPhi\rangle
=\varPhi\int_{\R^2}\psi(y)\overline{\varPhi(y)}\,dy.
\end{eqnarray}
\begin{enumerate}
\item
\label{theorem-2d-1}
The operator $-\Delta+U_\varPhi$
with domain $\dom(-\Delta+U_\varPhi)=H^2(\R^2)$
is selfadjoint and satisfies
\[
\sigma(-\Delta+U_\varPhi)
=
\sigma\sb{\mathrm{ess}}(-\Delta+U_\varPhi)
=\overline{\R_{+}}.
\]
\item
\label{theorem-2d-2}
The resolvent
\[
R_\varPhi(z):=(-\Delta+U_\varPhi-z I)^{-1},
\qquad
z\in\cnor,
\]
has a limit as $z\to z_0=0$
in the following topologies:
\begin{enumerate}
\item
\label{theorem-2d-2a}
weak$^*$ operator topology of
$\scrB\big(L^1_{0,\mu}(\R^2),L^\infty(\R^2)\big)$,
for all $\mu>1$;
\item
\label{theorem-2d-2b}
strong operator topology of
$\scrB\big(L^1(\R^2),L^\infty_{0,-\mu'}(\R^2)\big)$,
for all $\mu'>1$;
\item
\label{theorem-2d-2c}
uniform operator topology of
$\scrB\big(L^1_{0,\mu}(\R^2),L^\infty_{0,-\mu'}(\R^2)\big)$ for
all $\mu,\,\mu'>0$, $\mu+\mu'>1$;
\item
\label{theorem-2d-2d}
uniform operator topology of
$\scrB\big(L^2_{1,\nu}(\R^2),L^2_{-1,-\nu'}(\R^2)\big)$
for all $\nu,\,\nu'>1/2$, $\nu+\nu'>2$.
\end{enumerate}
\item
\label{theorem-2d-3}
There are no nontrivial solutions
to $\Delta u=U_\varPhi u$
(with the left-hand side
understood in the sense of distributions)
such that either
$u\in L^\infty_{0,-\mu'}(\R^2)$
with $\mu'<1$
or
$u\in L^2_{-1,-3/2}(\R^2)$.
\item
\label{theorem-2d-4}
If $\nu>1/2$, $1/2<\nu'\le 3/2$, $\nu+\nu'>2$,
and if $u\in L^2_{-1,-\nu'}(\R^2)$ satisfies
$(-\Delta+U_\varPhi)u\in L^2_{1,\nu}(\R^2)$
(in the sense of distributions),
then
\[
(-\Delta+U_\varPhi-z_0 I)^{-1}_{L^2_{1,\nu},L^2_{-1,-\nu'},\cnor}
(-\Delta+U_\varPhi)u
=u.
\]
\end{enumerate}
\end{theorem}

\begin{remark}
The requirement that
the integral of $\varPhi(x)$ over $\R^2$
equals one is for convenience only;
the statements of Theorem~\ref{theorem-2d} remain true
as long as this integral is nonzero.
\end{remark}

We postpone the proof of Theorem~\ref{theorem-2d}
until after we formulate and prove Theorem~\ref{theorem-2d-general}.

\begin{theorem}[Schr\"odinger operators
with relatively compact perturbations in $L^2(\R^2)$:
virtual states and LAP estimates]
\label{theorem-2d-general}
Assume that there are $\nu_1,\,\nu_2>1/2$,
$\nu_1+\nu_2>2$,
$\nu_2\le 3/2$,
such that
$W:\,L^2_{-1,-\nu_2}(\R^2)\to L^2_{1,\nu_1}(\R^2)$
is $\Delta$-compact
(with $\Delta$ considered in $L^2_{-1,-\nu_2}(\R^2)$).
There is the following dichotomy:

\medskip

\noindent
Either there is a nontrivial solution
$\Psi\in L^2_{-1,-\nu_2}(\R^2)$
to the equation
$(-\Delta+W)\Psi=0$
(in the sense of distributions),
and moreover,
this solution satisfies the relation
\begin{eqnarray}\label{w-is-w-2d}
\Psi
=
(-\Delta+U_{\varPhi_0}-z_0 I)^{-1}_{L^2_{1,\nu_1},L^2_{-1,-\nu_2},\cnor}
(U_{\varPhi_0}-W)\Psi,
\qquad
z_0=0,
\end{eqnarray}
with
\begin{eqnarray}\label{def-u-1}
U_{\varPhi_0}=\varPhi_0\otimes\varPhi_0
\in\scrB_{00}\big(L^2_{-1,-\nu_2}(\R^2),L^2_{1,\nu_1}(\R^2)\big),
\end{eqnarray}
with any
\begin{eqnarray}\label{def-u-2}
\varPhi_0\in L^2_{\mathrm{comp}}(\R^2),
\qquad
\int_{\R^2}\varPhi_0(x)\,dx=1,
\end{eqnarray}

\medskip
\noindent
or
there is $\delta>0$
such that
$\mathbb{D}_\delta(z_0)\setminus\overline{\R_{+}}
\subset\C\setminus\sigma(-\Delta+W)$
and the resolvent
\[
R_W(z)=(-\Delta+W-z I)^{-1},
\qquad
z\in\C\setminus\sigma(-\Delta+W)
\]
is bounded
as a mapping $L^2_{1,\nu_1}(\R^2)\to L^2_{-1,-\nu_2}(\R^2)$
uniformly in $z\in\mathbb{D}_\delta\setminus\overline{\R_{+}}$,
and in this case $R_W(z)$ has a limit
as $z\to z_0=0$, $z\in\C\setminus\sigma(-\Delta+W)$,
in the following topologies:
\begin{enumerate}
\item
\label{theorem-2d-general-1a}
Uniform operator topology of
$\scrB\big(L^2_{1,\nu}(\R^2),L^2_{-1,-\nu'}(\R^2)\big)$
with $\nu,\,\nu'>1/2$
as long as
$W$ extends to a mapping
$L^2_{-1,-\nu'}(\R^2)\to L^2_{1,\nu}(\R^2)$
which is $\Delta$-compact,
with $\Delta$ considered in $L^2_{-1,-\nu'}(\R^2)$
(this assumption on $W$ is not needed
if, additionally, one has
$\nu\ge\nu_1$, $\nu'\ge\nu_2$,
and it is redundant if $\nu\le\nu_1$, $\nu'\le\nu_2$);

\item
\label{theorem-2d-general-1c}
Weak$^*$ operator topology of
$\scrB\big(L^1_{0,\mu}(\R^2),L^\infty(\R^2)\big)$
if
$\mu\in(1,\nu_1-1/2)$
is such that
for some
$\delta>0$ and
$\varPhi_0\in C_{\mathrm{comp}}^\infty(\R^2)$,
$\int_{\R^2}\varPhi_0(x)\,dx=1$,
 the operator family
 \begin{eqnarray}\label{of2}
 W(-\Delta+U_{\varPhi_0}-z I)^{-1}
 :\,L^1_{0,\mu}(\R^2)\to L^1_{0,\mu}(\R^2),
 \qquad
 z\in\mathbb{D}_\delta\setminus\overline{\R_{+}},
 \end{eqnarray}
with $U_{\varPhi_0}=\varPhi_0\otimes\varPhi_0$,
is collectively compact.
\end{enumerate}

\end{theorem}

\begin{proof}
Let us prove the dichotomy.
Let us assume that there is a nontrivial solution to
\begin{eqnarray}\label{nt-1}
(-\Delta+W)\Psi=0,
\qquad
\Psi\in L^2_{-1,-\nu_2}(\R^2);
\end{eqnarray}
so,
$(-\Delta+U_{\varPhi_0})\Psi=(U_{\varPhi_0}-W)\Psi$.
Let
$
\Theta
=
(-\Delta+U_{\varPhi_0}-z_0 I)^{-1}_{L^2_{1,\nu_1},L^2_{-1,-\nu_2},\cnor}(U_{\varPhi_0}-W)\Psi$,
where $z_0=0$.
By Theorem~\ref{theorem-2d}~\itref{theorem-2d-2d},
$\Theta\in L^2_{-1,-\nu_2}(\R^2)$.
According to Lemma~\ref{lemma-jk}~\itref{lemma-jk-jk1},
$(-\Delta+U_{\varPhi_0})\Theta=(U_{\varPhi_0}-W)\Psi$, hence
\[
(-\Delta+U_{\varPhi_0})(\Psi-\Theta)=0, 
\qquad
\Psi-\Theta\in L^2_{-1,-\nu_2}(\R^2).
\]
By Theorem~\ref{theorem-2d}~\itref{theorem-2d-3},
taking into account our assumption that $\nu_2\le 3/2$,
one has $\Theta=\Psi$,
hence
$\Psi
\in\ran\big((-\Delta+U_{\varPhi_0}-z_0 I)^{-1}_{L^2_{1,\nu_1},L^2_{-1,-\nu_2},\cnor}\big)$.
By Theorem~\ref{theorem-lap}~\itref{theorem-lap-2}
(equivalence of~\itref{theorem-lap-2a}
and~\itref{theorem-lap-2e}),
this implies that
$R_W(z)$
cannot be bounded
in $\scrB\big(L^2_{1,\nu_1}(\R^2),L^2_{-1,-\nu_2}(\R^2)\big)$
uniformly in $z\in\mathbb{D}_\delta\setminus\overline{\R_{+}}$
with some $\delta>0$;
that is, $z_0=0$ is a virtual level of $-\Delta+W$
relative to
$\big(L^2_{1,\nu_1}(\R^2),L^2_{-1,-\nu_2}(\R^2),\varOmega\big)$,
$\varOmega=\C\setminus\sigma(-\Delta+W)$,
with $\Psi$ a corresponding virtual state.

If, on the contrary,
there is no nontrivial solution to 
\eqref{nt-1},
then there is also no nontrivial solution to
$(-\Delta+W)\Psi=0$
with
$\Psi\in
\range\big((-\Delta+V_0-z_0 I_\bfX)^{-1}_{L^2_{s_1},L^2_{-s_2},\cnor}\big)$,
and again by Theorem~\ref{theorem-lap}~\itref{theorem-lap-2}
(equivalence of~\itref{theorem-lap-2a}
and~\itref{theorem-lap-2e})
there is $\delta>0$ such that
$\mathbb{D}_\delta\setminus\overline{\R_{+}}
\subset
\C\setminus\sigma(-\Delta+W)$
and
the mapping $R_W(z)=(-\Delta+W-z I)^{-1}:\,
L^2_{1,\nu_1}(\R^2)\to L^2_{-1,-\nu_2}(\R^2)$,
$z\in\C\setminus\sigma(-\Delta+W)$
is bounded uniformly in
$z\in\mathbb{D}_\delta\setminus\overline{\R_{+}}$.

This completes the proof of the dichotomy.

The relation \eqref{w-is-w-2d} is proved in the same way as
the relation \eqref{w-is-w}
in Theorem~\ref{theorem-1d-general}.

\smallskip

Let us now assume that
$R_W(z)=(-\Delta+W-z I)^{-1}$,
$z\in\varOmega=\C\setminus\sigma(-\Delta+W)$,
is bounded as a mapping
$L^2_{1,\nu_1}(\R^2)\to L^2_{-1,-\nu_2}(\R^2)$,
uniformly in
$z\in\mathbb{D}_\delta\setminus\overline{\R_{+}}$
with some $\delta>0$.
By Theorem~\ref{theorem-lap}~\itref{theorem-lap-2}
(equivalence of~\itref{theorem-lap-2d}
and~\itref{theorem-lap-2e}),
where we take $-\Delta+U_{\varPhi_0}$ and $W-U_{\varPhi_0}$
in place of $A$ and $B$,
$\bfE=L^2_{1,\nu_1}(\R^2)$,
$\bfF=L^2_{-1,-\nu_2}(\R^2)$,
the mapping $(-\Delta+W-z I)^{-1}$
converges as $z\to z_0=0$, $z\in\varOmega$,
in the uniform operator topology of
$\scrB\big(L^2_{1,\nu_1}(\R^2),L^2_{-1,-\nu_2}(\R^2)\big)$,
since it is this convergence that was proved
for $(-\Delta+U_{\varPhi_0}-z I)^{-1}$
(see Theorem~\ref{theorem-2d}~\itref{theorem-2d-2d}).

Let us prove
the convergence of $R_W(z)$ stated in the theorem.
We assume that
$\nu,\,\nu'>1/2$ and that
$W:\,L^2_{-1,-\nu'}(\R^2)\to L^2_{1,\nu}(\R^2)$ is $\Delta$-compact
(with $\Delta$ considered in $L^2_{-1,-\nu'}(\R^2)$).
We apply Theorem~\ref{theorem-a}~\itref{theorem-a-2};
its assumptions are satisfied if we take
$\bfE_1=L^2_{1,\nu_1}(\R^2)$,
$\bfF_1=L^2_{-1,-\nu_2}(\R^2)$,
$\bfE_2=L^2_{1,\nu}(\R^2)$,
$\bfF_2=L^2_{-1,-\nu'}(\R^2)$,
so that
$\bfE=\bfE_1\cap\bfE_2=L^2_{1,\max(\nu_1,\nu)}(\R^2)$
and
$\bfF=\bfF_1+\bfF_2=L^2_{-1,-\max(\nu_2,\nu')}(\R^2)$.
By that theorem,
since
$z_0=0$ is a regular point of the essential spectrum of
$-\Delta+W$ relative to
$\big(L^2_{1,\nu_1}(\R^2),L^2_{-1,-\nu_2}(\R^2),\varOmega\big)$,
the same is true relative to
$\big(L^2_{1,\nu}(\R^2),L^2_{-1,-\nu'}(\R^2),\varOmega\big)$,
and the convergence of
the resolvent holds in the uniform operator topology of
$\scrB\big(L^2_{1,\nu}(\R^2),L^2_{-1,-\nu'}(\R^2)\big)$,
completing Part~\itref{theorem-2d-general-1a}.

\begin{lemma}\label{lemma-lpsm}
For any $d\in\N$, $\mu\in\R$, and $\nu>\mu+1/2$,
there is a continuous embedding
$L^2_{d/2,\nu}(\R^d)\hookrightarrow L^1_{0,\mu}(\R^d)$.
\end{lemma}

\begin{proof}
It suffices to check that
for any $u\in L^2_{d/2,\nu}(\R^d)$,
\[
\Big(\int\limits_{\R^d}
\big(\ln(1+\langle x\rangle)\big)^\mu \abs{u(x)}\,dx
\Big)^2
\le
\int\limits_{\R^d}
\langle x\rangle^{d}
\big(\ln(1+\langle x\rangle)\big)^{2\nu} \abs{u(x)}^2\,dx
\,
\int\limits_{\R^d}
\frac{dx}{
\langle x\rangle^{d}
\big(\ln(1+\langle x\rangle)\big)^{2\nu-2\mu}},
\]
with the last factor finite as long as
$\nu>\mu+1/2$.
\end{proof}

The convergence
in the weak$^*$ operator topology
of $\scrB\big(L^1_{0,\mu}(\R^2),L^\infty(\R^2)\big)$
with $\mu\in(1,\nu_1-1/2)$
stated in Part~\itref{theorem-2d-general-1c}
follows from
a similar convergence for the resolvent
of the perturbed Schr\"odinger operator
stated in
Theorem~\ref{theorem-2d}~\itref{theorem-2d-2a},
from the convergence of the resolvent in
$\scrB\big(L^2_{1,\nu}(\R^2),L^2_{-1,-\nu_2}(\R^2)\big)$
with $\nu>\max(\nu_1,\mu+1/2)$
which follows from
Part~\itref{theorem-2d-general-1a},
and from
Theorem~\ref{theorem-factorization}~\itref{theorem-factorization-2}
with
$\bfE=L^2_{1,\nu}(\R^2)
\hookrightarrow\tilde\bfE=L^1_{0,\mu}(\R^2)$
(the embedding is continuous due to
$\nu>\mu+1/2$;
see Lemma~\ref{lemma-lpsm})
and
$\tilde\bfF=L^\infty(\R^2)\hookrightarrow \bfF=L^2_{-1,-\nu_2}(\R^2)$
(the embedding being continuous due to $\nu_2>1/2$).
Let us mention
that the assumption on the collective compactness of
the operator family \eqref{of2}
is needed in
Theorem~\ref{theorem-factorization}~\itref{theorem-factorization-2}.
This concludes the proof.
\end{proof}

\begin{example}
\label{example-2d}
Let us show that a virtual state $\Psi$
of $-\Delta+W$ in $L^2(\R^2)$,
with $W$ being $\Delta$-compact,
is not necessarily from $L^\infty(\R^2)$.
Pick $\varPhi_0\in C^\infty_{\mathrm{comp}}(\R^2)$,
$\int_{\R^2}\varPhi_0(x)\,dx=1$.
Let $\alpha\in(0,1)$
and denote
$\Psi(x)=\big(\ln(1+\langle x\rangle)\big)^\alpha$;
one has
$\Psi\in L^2_{-1,-\nu'}(\R^2)$ with $\nu'>\alpha+1/2>1/2$.

Denote
$f:=\Delta\Psi\sim
\big(\ln(1+\langle x\rangle)\big)^{\alpha-2}/r^2
$;
one has
$f\in L^2_{1,\nu}(\R^2)$,
with any
$\nu<3/2-\alpha$.

We require that $\nu>1/2$, $1/2<\nu'\le 3/2$, $\nu+\nu'>2$.
Let $g\in L^2_{1,\nu'}(\R^2)$
be such that $\langle\Psi,g\rangle\ne 0$ and define
\begin{eqnarray}\label{wfg-2}
W=c f\otimes g
=c f\langle\cdot,g\rangle\,:\:
L^2_{-1,-\nu'}(\R^2)\to L^2_{1,\nu}(\R^2).
\end{eqnarray}
We choose $c=1/\langle \Psi,g\rangle$;
then $W\Psi=f=\Delta\Psi$.

By Theorem~\ref{theorem-2d}~\itref{theorem-2d-2d},
since $\nu,\,\nu'>1/2$ and $\nu+\nu'>2$,
the mapping
$(-\Delta+U_{\varPhi_0}-z I)^{-1}:\,L^2_{1,\nu}(\R^2)
\to L^2_{-1,-\nu'}(\R^2)$
is continuous,
bounded
uniformly in $z\in\mathbb{D}\setminus\overline{\R_{+}}$
and convergent as $z\to z_0=0$;
so we can take
$\bfE=L^2_{1,\nu}(\R^2)$, $\bfF=L^2_{-1,-\nu'}(\R^2)$.
Since
$
(-\Delta+U_{\varPhi_0})\Psi=(U_{\varPhi_0}-W)\Psi
=U_{\varPhi_0}\Psi-f\in L^2_{1,\nu}(\R^2)=:\bfE,
$
by
Theorem~\ref{theorem-2d}~\itref{theorem-2d-4},
since $\nu'\le 3/2$,
one has
\begin{eqnarray}\label{sdfsdf}
(-\Delta+U_{\varPhi_0})^{-1}_{\bfE,\bfF,\cnor}(-\Delta+U_{\varPhi_0})\Psi
=\Psi;
\end{eqnarray}
it follows that
$\Psi\in\ran\big((-\Delta+U_{\varPhi_0})^{-1}_{\bfE,\bfF,\cnor}\big)$,
so it is a virtual state
corresponding to $z_0=0\in\sigma\sb{\mathrm{ess}}(-\Delta+W)$
relative to $(\bfE,\bfF,\C\setminus\sigma(-\Delta+W))$.

Let us point out that if $\alpha\in(0,1/2)$,
then $1<\nu<3/2$, and
one can choose $\nu'=\nu$, $g=f$, so that $W$ is selfadjoint
(note that in this case
$\langle\Psi,g\rangle
=\langle\Psi,\Delta\Psi\rangle
=-\langle\nabla\Psi,\nabla\Psi\rangle<0$).
\end{example}

\begin{example}\label{remark-2d-better}
Just like in the case $d=1$
(cf. Remark~\ref{remark-1d-better}),
the relation \eqref{w-is-w-2d}
can be used to deduce better regularity
of a virtual state $\Psi$
when additional information about
the operator $W$ is available.
For example, if $W$ is the operator of multiplication by
a function from $L^\infty_{2,\nu}(\R^2)$, $\nu>2$,
then
we can take
$\nu_1>1/2$ and $\nu_2\in(1/2,3/2]$,
$\nu_1+\nu_2>2$,
such that $\nu_1+\nu_2\le\nu$,
so that
$W:\,L^2_{1,\nu_1}(\R^2)\to L^2_{-1,-\nu_2}(\R^2)$
is $\Delta$-compact
(with $\Delta$ considered in $L^2_{-1,-\nu_2}(\R^2)$).
By Theorem~\ref{theorem-2d-general},
a virtual state $\Psi\in L^2_{-1,-\nu_2}(\R^2)$
satisfies \eqref{w-is-w};
since
$(W-U_{\varPhi_0})\Psi\in L^2_{1,\nu_1}(\R^2)\subset L^1_{0,\mu}(\R^2)$,
with any $\mu\in(1,\nu_1-1/2)$,
Theorem~\ref{theorem-2d}~\itref{theorem-2d-2a}
shows that $\Psi\in L^\infty(\R^2)$.
\end{example}

\begin{proof}[Proof of Theorem~\ref{theorem-2d}]
The selfadjointness is immediate since
$U_\varPhi$ is bounded and symmetric.
To prove the absence of discrete spectrum,
we notice that $U_\varPhi\geq0$ (in the sense of forms)
and is of rank one;
hence $\sigma(-\Delta+U_\varPhi)\subset\overline{\R_{+}}$
and,
by Weyl's theorem,
$\sigma\sb{\mathrm{ess}}(-\Delta+U_\varPhi)
=\sigma\sb{\mathrm{ess}}(-\Delta)=\overline{\R_{+}}$.
Therefore,
$\sigma(-\Delta+U_\varPhi)=\sigma_{\mathrm{ess}}(-\Delta+U_\varPhi)
=\overline{\R_{+}}$.
This completes the proof of Part~\itref{theorem-2d-1}.

Let us prove Part~\itref{theorem-2d-2}.
Our approach to the resolvent
will be based on the construction
from Theorem~\ref{theorem-construction}.
The idea is to invert the resolvent
on codimension one subspace
(namely, on the subspace of zero mean functions),
and then to use a sequence which converges weakly to
a virtual state
to explicitly invert the resolvent of the perturbed
operator.

We recall that
$\nu_0>3/2$,
$\varPhi\in L^2_{1,\nu_0}(\R^2)$, $\int_{\R^2}\varPhi(x)\,dx=1$.

\begin{lemma}
\label{lemma-2d-angular-vnew}
Let $\mu>1$
and let
$P=\varPhi\otimes 1\in\scrB_{00}\big(L^2_{1,\nu_0}(\R^2)\big)$,
$P:\,\phi\mapsto\varPhi\overline{\langle 1,\phi\rangle}$,
be a projector.
For any $\zeta\in\C_{+}$,
$(-\Delta-\zeta^2 I)^{-1}\circ(I-P)$
extends to
a continuous linear mapping
from $L^1_{0,\mu}(\R^2)$ to $L^\infty(\R^2)$.
Moreover, as $\zeta\to 0$, $\zeta\in\C_{+}$,
this mapping has a limit in the
uniform
operator topology
of $\scrB\big(L^1_{0,\mu}(\R^2),L^\infty(\R^2)\big)$.
\end{lemma}

\begin{proof}
It is enough to give a proof
assuming that $\mu\in(1,\nu_0-1/2)$.
We note that, under this restriction,
$\varPhi\in L^2_{1,\nu_0}(\R^2)\hookrightarrow L^1_{0,\mu}(\R^2)$
and hence
$P=\varPhi\otimes 1\in\scrB_{00}\big(L^1_{0,\mu}(\R^2)\big)$.
We will only consider values
$z\in\mathbb{D}_{1/16}\setminus\overline{\R_{+}}$,
thus we assume that $z=\zeta^2$ with
$\zeta\in\C_{+}\cap\mathbb{D}_{1/4}$.
Let
$\phi\in L^1_{0,\mu}(\R^2)$
and denote
\[
\phi_0=(I-P)\phi
\in L^1_{0,\mu}(\R^2),
\qquad
\int_{\R^2}\phi_0(x)\,dx=0;
\]
we note that
$\norm{\phi_0}_{L^1_{0,\mu}}
\le
\norm{\phi}_{L^1_{0,\mu}}
+
\norm{P\phi}_{L^1_{0,\mu}}
\le
c\norm{\phi}_{L^1_{0,\mu}}$,
with
$c=1+\norm{\varPhi}_{L^1_{0,\mu}}$.
For $\zeta\in\C\sb{+}\cap\mathbb{D}_{1/4}$,
we decompose
$(-\Delta-\zeta^2 I)^{-1}\circ(I-P)\phi
=
(-\Delta-\zeta^2 I)^{-1}\phi_0$
into
\begin{eqnarray}\label{i-i1-i2-i3-vnew}
(-\Delta-\zeta^2 I)^{-1}\phi_0
=T_1(\zeta)\phi_0+T_2(\zeta)\phi_0+T_3(\zeta)\phi_0,
\quad
T_i(\zeta)\in\scrB\big(L^2(\R^2)\big),
\quad
1\le i\le 3,
\end{eqnarray}
with
\begin{eqnarray*}
&
\ds
T_1(\zeta)\phi_0(x)
=
\int\limits_{\mathbb{B}^2_{2\abs{\zeta}}}
\frac{
e^{\jj \xi x}
\big(
\hat\phi_0(\xi)-\hat\phi_0(\abs{\zeta}\omega)\big)
}
{\xi^2-\zeta^2}\,d\xi,
\qquad
T_2(\zeta)\phi_0(x)
=
\int\limits_{\mathbb{B}^2_{2\abs{\zeta}}}
\frac{e^{\jj \xi x}\hat\phi_0(\abs{\zeta}\omega)}
{\xi^2-\zeta^2}\,d\xi,
\\
&
\ds
T_3(\zeta)\phi_0(x)
=
\int_{\abs{\xi}\ge 2\abs{\zeta}}
\frac{e^{\jj \xi x}\hat{\phi_0}(\xi)}{\xi^2-\zeta^2}\,d\xi,
\qquad
\zeta\in\C_{+}\cap\mathbb{D}_{1/4},
\end{eqnarray*}
where
$\omega=\xi/\abs{\xi}\in\mathbb{S}^1\subset\R^2$
and $\lambda=\abs{\xi}$.

Pick $\beta\in(1,\mu)$.
The $L^\infty$-norm of
$T_1(\zeta)\phi_0$
is estimated via the log-H\"older continuity of
$\hat\phi_0$.
For $\xi,\,\eta\in\R^d$, $\abs{\xi-\eta}<1/2$,
define
$R=\frac{1}{\abs{\xi-\eta}\abs{\ln\abs{\xi-\eta}}^\mu}$;
then
\begin{eqnarray}\label{phi-phi-holder}
&&
\abs{
\hat\phi_0(\xi)
-
\hat\phi_0(\eta)
}
\le
\int_{\mathbb{B}_R^d}
\abs{1-e^{\jj(\xi-\eta)x}}
\abs{\phi_0(x)}\,dx
+
2
\int_{\R^d\setminus\mathbb{B}_R^d}
\abs{\phi_0(x)}\,dx
\\
\nonumber
&&
\le
\abs{\xi-\eta}R
\norm{\phi_0}_{L^1_{0,\mu}}
+
\frac{2\norm{\phi_0}_{L^1_{0,\mu}}}{(\ln(1+\langle R\rangle))^{\mu}}
\le
\frac{C\norm{\phi_0}_{L^1_{0,\mu}}}
{\Abs{\ln
(\abs{\xi-\eta}
\abs{\ln\abs{\xi-\eta}}^\mu
)}^\mu}
\le
\frac{C\norm{\phi}_{L^1_{0,\mu}}}
{\abs{\ln\abs{\xi-\eta}}^\beta},
\end{eqnarray}
where $C$ only depends on $\mu$ and $\beta\in(1,\mu)$
but not on $\phi$, $\xi$, $\eta$.
Using \eqref{phi-phi-holder}
and taking into account the triangle inequality
$
\abs{\lambda^2-\zeta^2}
\ge
\bigabs{\abs{\lambda}^2-\abs{\zeta}^2},
$
we bound the $L^\infty$-norm of
\[
T_1(\zeta)\phi_0(x)
=
\int_{\mathbb{B}^2_{2\abs{\zeta}}}
\frac{
e^{\jj \lambda \omega x}
\big(\hat\phi_0(\xi)-\hat\phi_0(\abs{\zeta}\omega)\big)
}
{\xi^2-\zeta^2}\,d\xi
=
\int\limits_{0}^{2\pi}
\int\limits_{0}^{2\abs{\zeta}}
\frac{e^{\jj \lambda \omega x}
\big(
\hat\phi_0(\lambda\omega)-\hat\phi_0(\abs{\zeta}\omega)
\big)
}
{\lambda^2-\zeta^2}\lambda\,d\lambda\,d\omega
\]
by
\[
\norm{T_1(\zeta)\phi_0}_{L^\infty}
\le
\int_{0}^{2\abs{\zeta}}
\frac{
C
\norm{\phi_0}_{L^1_{0,\mu}}
\lambda\,d\lambda}{
\bigabs{\lambda^2-\abs{\zeta}^2}
\abs{\ln\abs{\lambda-\abs{\zeta}}}^\beta
}
\le
\int_{0}^{2\abs{\zeta}}
\frac{C\norm{\phi}_{L^1_{0,\mu}}
\,d\lambda}{
\bigabs{\lambda-\abs{\zeta}}
\abs{\ln\abs{\lambda-\abs{\zeta}}}^\beta
};
\]
one can see that
$\norm{T_1(\zeta)}_{L^1_{0,\mu}\shortto L^\infty}$
is bounded uniformly in $\zeta\in\C_{+}\cap\mathbb{D}_{1/4}$
and tends to zero as $\zeta\to 0$.

The value of
\[
T_2(\zeta)\phi_0(x)
=
\int_0^{2\pi}
\Big[
\int_{0}^{2\abs{\zeta}}
\frac{
e^{\jj \lambda \omega x}
\lambda\,d\lambda}{\lambda^2-\zeta^2}
\Big]
\hat\phi_0(\abs{\zeta}\omega)
\,d\omega
\]
is estimated via the Cauchy theorem,
modifying the contour of integration in $\lambda$
to a semicircle
centered at $\lambda=\abs{\zeta}$, of radius $\abs{\zeta}$;
a semicircle is chosen in $\C_\pm$
depending on the sign of $\omega x$,
so that $\abs{e^{\jj\lambda \omega x}}\le 1$
for $\lambda$ on that semicircle.
For definiteness, let
$\omega x\ge 0$,
and then the semicircle is chosen in $\C_{+}$.
We consider two cases:

\smallskip

\noindent
$(i)\ $
For
$\abs{\zeta^2-\abs{\zeta}^2}<\abs{\zeta}^2/2$,
\[
T_2(\zeta)\phi_0(x)
=
\begin{cases}
\ds
\int_0^{2\pi}
\bigg[
2\pi\jj
\frac{e^{\jj \zeta \omega x}\zeta}{2\zeta}
+
\int\limits_{\abs{\eta-\abs{\zeta}}=\abs{\zeta},\,\Im\eta\ge 0}
\frac{
e^{\jj \eta \omega x}
\eta\,d\eta}{\eta^2-\zeta^2}
\bigg]
\hat\phi_0(\abs{\zeta}\omega)\,d\omega,
&
\abs{\zeta-\abs{\zeta}}<\abs{\zeta},
\\[2.5ex]
\qquad\quad
\ds
\int_0^{2\pi}
\bigg[
\int\limits_{\abs{\eta-\abs{\zeta}}=\abs{\zeta},\,\Im\eta\ge 0}
\frac{
e^{\jj \eta \omega x}
\eta\,d\eta}{\eta^2-\zeta^2}
\bigg]
\hat\phi_0(\abs{\zeta}\omega)\,d\omega
,
&
\abs{\zeta+\abs{\zeta}}<\abs{\zeta}.
\end{cases}
\]
The integrals in the brackets
are bounded by a constant
independent of $\zeta\in\C_{+}\cap\mathbb{D}_{1/4}$ upon noting that
\[
\abs{\eta^2-\zeta^2}
\ge
\bigabs{\eta^2-\abs{\zeta}^2}
-
\bigabs{\zeta^2-\abs{\zeta}^2}
\ge
\abs{\eta-\abs{\zeta}}\cdot\abs{\eta+\abs{\zeta}}
- 
\abs{\zeta}^2/2
\ge
\abs{\zeta}^2- \abs{\zeta}^2/2
=\abs{\zeta}^2/2,
\]
\[
\frac{\abs{\eta}}{\abs{\eta^2-\zeta^2}}\leq  \frac{\abs{\eta}}{\abs{\abs{\xi}^2-\abs{\zeta}^2}}=\frac12\left|\frac1{\abs{\eta}-\abs{\zeta}}+\frac1{\abs{\eta}+\abs{\zeta}}\right|.
\]

\smallskip

\noindent
$(ii)\ $ For
$\abs{\zeta^2-\abs{\zeta}^2}\ge\abs{\zeta}^2/2$,
\begin{eqnarray*}
&&
\abs{T_2(\zeta)\phi_0(x)}
\le
\int_0^{2\pi}
\Big|
\hat\phi_0(\abs{\zeta}\omega)
\int_{0}^{2\abs{\zeta}}
\frac{
e^{\jj \lambda \omega x}
\lambda\,d\lambda}{\lambda^2-\zeta^2}
\Big|
\,d\omega
\\
&&
\le
\int_0^{2\pi}
\abs{\hat\phi_0(\abs{\zeta}\omega)}
\,d\omega
\int_{0}^{2\abs{\zeta}}
\frac{
2\abs{\zeta}\,d\lambda}{\dist(\zeta^2,\R\sb{+})}
\le C
\int_0^{2\pi}
\abs{\hat\phi_0(\abs{\zeta}\omega)}\,d\omega
\le
\frac{C\norm{\phi}_{L^1_{0,\mu}}}{\abs{\ln\abs{\zeta}}^\beta};
\end{eqnarray*}
we took into account that
$\dist(\zeta^2,\R_{+})=\abs{\zeta}^2$ if
$\Re\zeta^2\le 0$ and
$\dist(\zeta^2,\R_{+})\ge \abs{\zeta}^2/3$ if 
$\Re\zeta^2>0$,
and also that
$
\abs{\hat\phi_0(\xi)}
=
\abs{\hat\phi_0(\xi)-\hat\phi_0(0)}
\le C\norm{\phi}_{L^1_{0,\mu}}
\abs{\ln\abs{\xi}}^{-\beta}
$,
$\xi\in\mathbb{B}^2_{1/2}$
(see \eqref{phi-phi-holder}).
We see that
$\norm{T_2(\zeta)\circ(I-P)}_{L^1_{0,\mu}\shortto L^\infty}$
is bounded uniformly in $\zeta\in\C_{+}\cap\mathbb{D}_{1/4}$
and converges to zero as $\zeta\to 0$.

\begin{figure}
\begin{center}
\begin{tikzpicture}[>=stealth,thick, xscale=1, yscale=1 ]
  \begin{axis}[
  xmin=-0.1,xmax=3,
  ymin=-0.5,ymax=1.7,
  axis x line=middle,
  axis y line=middle,
  ylabel={},
  yticklabels={,,},
  xticklabels={,,},
  xlabel={},
  tick label style={major tick length=0pt},
  ]
  \draw (0,0) -- (2.4,1.16) node[pos=0.5,above]
  {$|z|$} node[right] {$z=\zeta^2$};
  \draw (0,0) -- (2.7,0) node[below] {$\quad\,\,|z|$};
  \draw (2.4,1.16) -- (2.1,0) node[pos=0.5,left] {$\frac{|z|}{2}$};
  \draw (2.4,1.16) -- (2.7,0) node[pos=0.5,right] {$\frac{|z|}{2}$};
  \draw (2.1,0) -- (2.7,0) node[pos=0.5,below]
   {$\underbrace{\qquad\,\,\,\,\,\,}\sb{|z|/4}$};
  \draw[dashed] (2.4,1.16) -- (2.4,0) node[pos=0.8,right] {$\!\!h$};
  \end{axis}
\end{tikzpicture}
\end{center}
\caption{
The worst case (the shortest distance
from $z$ to $\R_{+}$), when $z=\zeta^2$
is separated from $|z|$ by exactly $|z|/2$.
The two isosceles triangles are similar,
so the base of the smaller triangle is $|z|/4$,
and then one has:
$\mathrm{dist}\,(z,\R_{+})=h=\sqrt{(|z|/2)^2-(|z|/8)^2}
=|z|\sqrt{15/64}
>|z|/3
$.
}
\label{fig-obvious}
\end{figure}

Let us consider
$T_3(\zeta)$.
We first note that,
as long as $|\xi|\geq 2|\zeta|$,
\begin{eqnarray}\label{diff-small-vnew}
|\xi^2-\zeta^2|\geq|\xi^2-|\zeta|^2|
=(|\xi|+|\zeta|)(|\xi|-|\zeta|)
\ge
|\xi|
\big(|\xi|-|\zeta|\big)
\ge|\xi|^2/2.
\end{eqnarray}
Let $\delta_0\in(\abs{\zeta},1/4]$; we decompose
$T_3(\zeta)$
into
$T_{3,0}(\zeta,\delta_0)+T_{3,1}(\zeta,\delta_0)$
with
\[
T_{3,0}(\zeta,\delta_0)\phi_0(x)
=
\int_{2\abs{\zeta}\le\abs{\xi}\le 2\delta_0}
\frac{\hat\phi_0(\xi)e^{\jj\xi x}}{\xi^2-\zeta^2}\,d\xi,
\qquad
T_{3,1}(\zeta,\delta_0)\phi_0(x)
=
\int_{\abs{\xi}>2\delta_0}
\frac{\hat\phi_0(\xi)e^{\jj\xi x}}{\xi^2-\zeta^2}\,d\xi.
\]
By
\eqref{phi-phi-holder}
and
\eqref{diff-small-vnew},
\begin{eqnarray}\label{t-3-1-vnew}
\abs{T_{3,0}(\zeta,\delta_0)\phi_0(x)}
\le
\int\limits_{2\abs{\zeta}\le\abs{\xi}\le 2\delta_0}
\frac{\abs{\hat\phi_0(\xi)}\,d\xi}{\xi^2/2}
\le
\int\limits_{0}^{2\delta_0}
\frac{C
\norm{\phi}_{L^1_{0,\mu}}
\lambda\,d\lambda}
{\lambda^2\abs{\ln\abs{\lambda}}^{\beta}}
\le
\frac{C\norm{\phi}_{L^1_{0,\mu}}}{\abs{\ln\abs{2\delta_0}}^{\beta-1}},
\end{eqnarray}
for all $\zeta\in\C_{+}\cap\mathbb{D}_{\delta_0}$;
similarly,
\begin{eqnarray}
\abs{
T_{3,1}(\zeta,\delta_0)\phi_0(x)
}
\le
\Abs{
\int_{\abs{\xi}>2\delta_0}
\frac{\hat\phi_0(\xi)e^{\jj\xi x}\,d\xi}{\xi^2-\zeta^2}
}
\le
\int_{\abs{\xi}>2\delta_0}
\frac{\abs{\hat\phi_0(\xi)}}{\xi^2/2}
\,d\xi
\le
\frac{C}{\delta_0}
\norm{\phi}_{L^1_{0,\mu}},
\label{t-3-2-vnew}
\end{eqnarray}
and also
\begin{eqnarray}
\abs{
T_{3,1}(\zeta,\delta_0)\phi_0(x)
-
T_{3,1}(\zeta',\delta_0)\phi_0(x)
}
\le
\Abs{
\int_{\abs{\xi}>2\delta_0}
\Big(
\frac{1}{\xi^2-\zeta^2}
-\frac{1}{\xi^2-(\zeta')^2}\Big)
\hat\phi_0(\xi)e^{\jj\xi x}\,d\xi
}
\nonumber
\\
\le
\abs{\zeta^2-(\zeta')^2}
\int_{\abs{\xi}>2\delta_0}
\frac{\abs{\hat\phi_0(\xi)}}{(\xi^2/2)^2}
\,d\xi
\le
C
\norm{\phi}_{L^1_{0,\mu}}
\frac{\abs{\zeta^2-(\zeta')^2}}{\delta_0^3},
\qquad
\forall\zeta,\,\zeta'\in\C_{+}\cap\mathbb{D}_{\delta_0}.
\label{t-3-3-vnew}
\end{eqnarray}
In the above inequalities, we used \eqref{diff-small-vnew}
to estimate the denominators
and applied the Cauchy--Schwarz inequality.
One can see
from \eqref{t-3-1-vnew} and \eqref{t-3-2-vnew}
that the $L^1_{0,\mu}\to L^\infty$ norm of
$T_3(\zeta)\circ(I-P)
=T_{3,0}(\zeta,\delta_0)\circ(I-P)+T_{3,1}(\zeta,\delta_0)\circ(I-P)$,
with $\delta_0=1/4$,
is bounded uniformly in 
$\zeta\in\C_{+}\cap\mathbb{D}_{1/4}$.

Let us now prove the convergence of
$T_3(\zeta)\circ(I-P)$
as $\zeta\to 0$, $\zeta\in\C_{+}$,
in the uniform operator topology of
$\scrB\big(L^1_{0,\mu}(\R^2),L^\infty(\R^2)\big)$.
Let $\varepsilon>0$.
Fix $\delta_0\in(0,1/4]$ so that
the right-hand side of \eqref{t-3-1-vnew}
is bounded by $\varepsilon/3$.
There is $\delta\in(0,\delta_0)$
such that
the right-hand side
of \eqref{t-3-3-vnew} is smaller than $\varepsilon/3$
for all $\zeta,\,\zeta'\in\C_{+}\cap\mathbb{D}_\delta$.
Then from the identity
\[
T_3(\zeta)-T_3(\zeta')
=
T_{3,0}(\zeta,\delta_0)
-
T_{3,0}(\zeta',\delta_0)
+
\big(T_{3,1}(\zeta,\delta_0)-T_{3,1}(\zeta',\delta_0)\big)
\]
we see that
$\norm{(T_3(\zeta)-T_3(\zeta'))\circ(I-P)}_{L^1_{0,\mu}\shortto L^\infty}<\varepsilon$
as long as
$\zeta,\,\zeta'\in\C_{+}\cap\mathbb{D}_\delta$.
Since we already proved that
both
$\norm{T_1(\zeta)\circ(I-P)}_{L^1_{0,\mu}\shortto L^\infty}$
and
$\norm{T_2(\zeta)\circ(I-P)}_{L^1_{0,\mu}\shortto L^\infty}$
tend to zero
as $\zeta\to 0$, $\zeta\in\C\sb{+}$,
we conclude that
the operator family
$(-\Delta-\zeta^2 I)^{-1}\circ(I-P)$
converges in the uniform operator topology of
$\scrB\big(L^1_{0,\mu}(\R^2),L^\infty(\R^2)\big)$
as $\zeta\to\zeta_0=0$, $\zeta\in\C_{+}$.
\end{proof}

Now we complete the proof of
convergence
of the resolvent stated in Part~\itref{theorem-2d-2};
we start with Part~\itref{theorem-2d-2a}.
It is enough to give the proof in the case
\[
\mu\in(1,\nu_0-1/2),
\]
with $\nu_0>3/2$ from \eqref{def-varphi}.
We will
apply Theorem~\ref{theorem-construction}
with
\[
\bfE=L^2_{1,\nu_0}(\R^2),
\qquad
\tilde\bfE=L^1_{0,\mu}(\R^2),
\qquad
\tilde\bfF=L^\infty(\R^2),
\qquad
\bfF=L^2_{-1,-\nu_0}(\R^2),
\]
so that there are continuous embeddings
$\bfE\hookrightarrow\bfX\hookrightarrow\bfF$,
$\bfE\hookrightarrow\tilde\bfE$ (cf. Lemma~\ref{lemma-lpsm}),
and
$\tilde\bfF\hookrightarrow\bfF$.
Using
$\varPhi\in L^2_{1,\nu_0}(\R^2)$
from \eqref{def-varphi}
(which we consider as an element of $L^1_{0,\mu}(\R^2)$),
we define the projector
$\tilde{P}=\varPhi\otimes 1\in\scrB_{00}(\tilde\bfE)$.
By Lemma~\ref{lemma-2d-angular-vnew},
as $z\to z_0=0$, $z\in\cnor$,
the operator family
$\tilde\calS(z)=(-\Delta-z I)^{-1}\circ (I_{\tilde\bfE}-\tilde{P})$
converges
in the uniform operator topology of
$\scrB\big(L^1_{0,\mu}(\R^2),L^\infty(\R^2)\big)$,
satisfying the assumption \eqref{or1}
of Theorem~\ref{theorem-construction}.
We define
(cf. \eqref{r02})
\begin{eqnarray}
\label{def-psi-z}
\varPsi(x,z)
=
\varrho_0(\abs{x})+\varrho_1(\abs{x})
\frac{\pi H_0^{(1)}(\zeta\abs{x})}{2\jj\ln\zeta}
,
\quad x\in\R^2,
\quad z=\zeta^2,
\quad\zeta\in\C_{+},
\end{eqnarray}
where $\varrho_0,\,\varrho_1\in C^\infty(\R)$ are such that
$\supp\varrho_0=(-\infty,1]$,
$\supp\varrho_1=[1/2,+\infty)$,
$\varrho_i(t)\ge 0$,
$\varrho_0(t)+\varrho_1(t)=1$ for all $t\in\R$.
Due to the asymptotics \eqref{h01},
one has
$\varPsi(x,z)\to\varPsi^{0}(x)=1$
as $z\to z_0=0$,
uniformly on compact subsets of $\R^2$,
thus in the weak$^*$ topology of $L^\infty(\R^2)$.
Denote
\begin{eqnarray*}
&&
\varTheta(x,z)
=
(-\Delta-z I)\varPsi(x,z)
\\
&&
\phantom{\varTheta(x,z)}
=
(-\Delta-z I)\varrho_0(\abs{x})
+
\frac{\pi H_0^{(1)}(\zeta\abs{x})}{2\jj\ln\zeta}
(-\Delta-z I)\varrho_1(\abs{x})
+
\frac{\pi\nabla H_0^{(1)}(\zeta\abs{x})}{\jj\ln\zeta}
\cdot\nabla\varrho_1(\abs{x}),
\end{eqnarray*}
with $z=\zeta^2$, $\zeta\in\C_{+}$;
we notice that
$\varTheta(z)$ has support inside $\mathbb{B}^2_1$
and that
$\norm{\varTheta(z)}\to 0$ as $z\to z_0=0$
(cf. \eqref{h01};
let us mention that
$\nabla H_0^{(1)}(\zeta\abs{x})
=-\fra{2\jj x}{(\pi\abs{x}^2)}+O(1)$
for $\abs{\zeta}\abs{z}<1$
and
$\sqrt{\fra{2}{(\pi\upzeta)}}
\,e^{\jj(\zeta\abs{x}-\pi/4)}
\jj\zeta \fra{x}{\abs{x}}$
for
$\abs{\zeta}\abs{z}\gg 1$).
Since
$\varPsi(x,z)$ is exponentially decaying in $\abs{x}$
for 
$z=\zeta^2$, $\zeta\in\C_{+}$,
there is the relation
$\varPsi(x,z)
=(-\Delta-z I)^{-1}\varTheta(x,z)$.
We conclude that $\varPsi(x,z)$ and $\varTheta(x,z)$
satisfy assumptions
\eqref{a-psi} and \eqref{or2}
of Theorem~\ref{theorem-construction}.

The application of Theorem~\ref{theorem-construction}
shows that for $z\in\cnor$
the resolvent of $-\Delta+U_\varPhi$
with
\[
U_\varPhi=\varPhi\otimes\varPhi\in\scrB_{00}(\bfF,\bfE),
\qquad
U_\varPhi:\,\psi\mapsto\varPhi
\overline{\langle\varPhi,\psi\rangle}
=\varPhi(x)\int_{\R^2}\overline{\varPhi(y)}\psi(y)\,dy,
\]
is uniformly bounded as a mapping
from $\tilde\bfE=L^1_{0,\mu}(\R^2)$ to $\tilde\bfF=L^\infty(\R^2)$
and that
this resolvent converges
in the weak$^*$ operator topology of
$\scrB\big(L^1_{0,\mu}(\R^2),L^\infty(\R^2)\big)$,
completing Part~\itref{theorem-2d-2a}.
Considering the adjoint operator
$(-\Delta+U_\varPhi)^*=-\Delta+U_\varPhi$,
we conclude from
Theorem~\ref{theorem-construction}~\itref{theorem-construction-2b}
that
$(-\Delta+U_\varPhi-z I)^{-1}$ converges
as $z\to z_0=0$, $z\in\cnor$,
in the strong operator topology of
$\scrB\big(L^1(\R^2),L^\infty_{0,-\mu}(\R^2)\big)$,
$\mu>1$,
completing Part~\itref{theorem-2d-2b}.

Similarly,
taking in the above argument
$\tilde\bfF=L^\infty_{0,-\mu'}(\R^2)$
with $\mu'\in(0,\nu_0-1/2)$
(so that $L^\infty_{0,-\mu'}(\R^2)\hookrightarrow L^2_{1,-\nu_0}(\R^2)$),
when
$\varPsi(z)\to\varPsi_0$
(strongly)
in $L^\infty_{0,-\mu'}(\R^2)$ as $z\to z_0=0$,
we conclude from
Theorem~\ref{theorem-construction}~\itref{theorem-construction-1c}
that
the resolvent of $-\Delta+U_\varPhi$
converges
in the uniform operator topology of
$\scrB\big(L^1_{0,\mu}(\R^2),L^\infty_{0,-\mu'}(\R^2)\big)$,
for all $\mu>1$ and $\mu'>0$.
By duality,
we conclude that
$(-\Delta+U_\varPhi-z I)^{-1}$
converges
as $z\to z_0=0$, $z\in\cnor$,
in the uniform operator topology of
$\scrB\big(L^1_{0,\mu}(\R^2),L^\infty_{0,-\mu'}(\R^2)\big)$
with any $\mu>0$ and $\mu'>1$.
By interpolation,
$(-\Delta+U_\varPhi-z I)^{-1}$
converges
as $z\to z_0=0$, $z\in\cnor$,
in the uniform operator topology of
$\scrB\big(L^1_{0,\mu}(\R^2),L^\infty_{0,-\mu'}(\R^2)\big)$
for any $\mu,\,\mu'>0$, $\mu+\mu'>1$,
completing Part~\itref{theorem-2d-2c}.

Let us prove the convergence
stated in Part~\itref{theorem-2d-2d}.
Let
$\nu,\,\nu'>1/2$, $\nu+\nu'>2$.
There are $\mu,\,\mu'>0$ satisfying
$\mu>\nu-1/2$,
$\mu'>\nu'-1/2$,
$\mu+\mu'>1$;
now the proof follows from Part~\itref{theorem-2d-2c}
due to the continuity of the embeddings
\[
L^2_{1,\nu}(\R^2)\hookrightarrow L^1_{0,\mu}(\R^2),
\qquad
L^\infty_{0,-\mu'}(\R^2)
\hookrightarrow 
L^2_{-1,\nu'}(\R^2).
\]
This completes Part~\itref{theorem-2d-2}.

Let us prove Part~\itref{theorem-2d-3}.
Let us assume that
$u\in\mathscr{D}'(\R^2)$
is a solution to $\Delta u=c_0\varPhi$
(with the left-hand side
understood in the sense of distributions)
such that
either $u\in L^\infty_{0,-\mu'}(\R^2)$ with $\mu'<1$
or $u\in L^2_{-1,-3/2}(\R^2)$,
with $c_0=\langle u,\varPhi\rangle\in\C$.
We note that $c_0\ne 0$: indeed, if $c_0=0$,
then $\Delta u=0$, hence $u$
coincides almost everywhere with a (strongly) harmonic function
(which we also denote by $u$),
which leads to a contradiction:
either $u=u_0\in\C$ is constant,
resulting in
$c_0=\langle u,\varPhi\rangle=u_0\langle 1,\varPhi\rangle
=u_0\ne 0$,
or $u$ is a nonconstant harmonic function
which could not satisfy
the sublogarithmic growth estimate for large $x$.

Averaging both $u$
(which, by our assumptions, belongs to
$L^2\sb{\mathrm{loc}}(\R^2)$
and hence to
$H^2\sb{\mathrm{loc}}(\R^2)$)
and $\varPhi\in L^2_{1,\nu_0}(\R^2)\subset L^1(\R^2)$
(with $\nu_0>3/2$)
over the angular direction
and denoting the results by
$v$
and
$h\in L^2_{1,\nu_0}(\R^2)\subset L^1(\R^2)$,
respectively,
one arrives at
$\Delta v=(\p_r^2+r^{-1}\p_r)v=c_0 h$,
which we rewrite as
$\p_r(r\p_r v)=c_0 r h$;
thus $v\in C^1(\R_{+})$
as a function of $r$,
with $\p_r v$ absolutely continuous
for $r>0$.
So,
\begin{eqnarray}\label{rr}
\qquad
r_1 v'(r_1)
-r_0 v'(r_0)=c_0\int_{r_0}^{r_1} h(r)\,r\,dr
=c_0\int_{\mathbb{B}_{r_1}^2\setminus\mathbb{B}_{r_0}^2}
\varPhi(x)\,dx,
\qquad
\forall r_1\ge r_0>0,
\end{eqnarray}
where both $v$ and $h$ are considered as functions of $r$.
Fixing $r_1$ and sending $r_0$ to zero,
one can see from \eqref{rr} that there exists a limit
$c_1:=\lim_{r\to 0+} r v'(r)$.
Let us argue that $c_1=0$.
Given a radially symmetric function
$\varphi\in C^\infty_{\mathrm{comp}}(\R^2)$
with $\varphi(0)\ne 0$,
we have:
\begin{eqnarray*}
\langle\varphi,\Delta v\rangle
&=&
\langle\Delta\varphi,v\rangle
=\lim_{r_0\to 0+}
2\pi\int_{r_0}^\infty
(r\varphi')'v(r)\,dr
\\
&=&
2\pi\lim_{r_0\to 0+}
\Big(
-r_0\varphi'(r_0)v(r_0)
+r_0\varphi(r_0)v'(r_0)
+
\int_{r_0}^\infty
\varphi(r)(r v')'\,dr
\Big).
\end{eqnarray*}
Taking into account that
$
\lim\sb{r\to 0+}\frac{v(r)}{1/r}
=
\lim\sb{r\to 0+}\frac{v'(r)}{-1/r^2}=0
$
and that $\p_r(r \p_r\varphi)=c_0 r h(r)$,
we arrive at
\[
\langle\varphi,\Delta v\rangle
=2\pi c_1\varphi(0)+c_0\langle \varphi,h\rangle
=2\pi c_1\varphi(0)+\langle \varphi,\Delta v\rangle;
\]
it follows that $c_1=0$.
Therefore,
recalling that $c_0\ne 0$,
\[
c_0^{-1}r_1 v'(r_1)
=\int_0^{r_1}h(r)r\,dr.
\]
Since
$\lim\sb{r_1\to\infty}\int_0^{r_1}h(r)r\,dr\to (2\pi)^{-1}$,
there is $R>0$ such that
$\Re\big(c_0^{-1}r v'(r)\big)\ge (4\pi)^{-1}$
for all $r\ge R$,
hence
one has
$\Re\big(c_0^{-1}v(r)\big)\ge (4\pi)^{-1}\ln(r/R)-\abs{c_0^{-1}v(R)}$
for all $r\ge R$,
which in turn implies that
$u\not\in L^\infty_{0,-1+0}(\R^2)$
(or else there would be the same
sublogarithmic growth estimate for $v$)
and also that $u\not\in L^2_{-1,-3/2}(\R^2)$
(since for any $s,\,\mu\in\R$,
one has
$\norm{v}_{L^2_{-1,-3/2}(\R^2)}
\le
\norm{u}_{L^2_{-1,-3/2}(\R^2)}$).
This completes the proof of Part~\itref{theorem-2d-3}.

Let us prove Part~\itref{theorem-2d-4}.
Denote
$\phi:=(-\Delta+U_\varPhi)u\in L^2_{1,\nu}(\R^2)$.
By Theorem~\ref{theorem-2d}~\itref{theorem-2d-2},
$\Psi
:=(-\Delta+U_\varPhi-z_0 I)^{-1}_{L^2_{1,\nu},L^2_{-1,-\nu'},\cnor}\phi
\in L^2_{-1,-\nu'}(\R^2)$,
and
by Lemma~\ref{lemma-jk}~\itref{lemma-jk-jk1}
one has
$(-\Delta+U_\varPhi)\Psi=\phi$.
It follows that
$(-\Delta+U_\varPhi)(\Psi-u)=0$;
since $\Psi-u\in L^2_{-1,-\nu'}(\R^2)$ with $\nu'\le 3/2$,
one has $\Psi=u$ by Part~\itref{theorem-2d-3}.
This completes the proof of Theorem~\ref{theorem-2d}.
\end{proof}

\begin{remark}
Let $s,\,s'>1$.
According to Theorem~\ref{theorem-2d},
since there is a solution
$\Psi=1\in H^2_{-s'}(\R^2)\cap L^\infty(\R^2)$
to $-\Delta\Psi=0$,
$z_0=0$ is a singular point of the free Laplace operator
(relative to
$\big(L^2_s(\R^2),L^2_{-s'}(\R^2),\cnor\big)$,
$s,\,s'>1$).
Moreover,
as follows from Lemma~\ref{lemma-2d-angular-vnew}
(which we used in the proof of Theorem~\ref{theorem-2d}),
$z_0=0$ is a virtual level of rank $r=1$.
\end{remark}

Let us show that any nonnegative nonzero perturbation
destroys a virtual level at $z_0=0$.

\begin{lemma}\label{lemma-2d-positive-potential}
Let $V\in L^\infty_{\mathrm{loc}}(\R^2)$ be nonnegative
and assume  that $V$ is
nonzero on a nonempty open set of $\R^2$.
Then there is no nontrivial solution
to $(-\Delta+V)u=0$
with $u\in L^2_{-1,-1}(\R^2)$
or $u\in L^\infty_{-1/2+0}(\R^2)$.
\end{lemma}

\begin{proof}
Let $u\in L^2_{-1,-1}(\R^2)$ or $u\in L^\infty_{-1/2+0}(\R^2)$.
Due to our assumption on $V$,
one has
$V u\in L^2_{\mathrm{loc}}(\R^2)$,
and hence $u\in H^2_{\mathrm{loc}}(\R^2)$
is continuous.
Multiplying this relation by $\bar{u}$,
integrating over $\mathbb{B}^2_r$, $r>0$,
with the aid of Green's first identity,
and taking the real part,
we obtain:
\begin{eqnarray}\label{nabla-u-1}
\int_{\mathbb{B}^2_r} \abs{\nabla u(x)}^2\,dx
-\frac{r}{2}\partial_r
\int_{\mathbb{S}^1_1}
\abs{u(r\omega)}^2\,d\omega
+
\int_{\mathbb{B}^2_r}V(x)\abs{u(x)}^2\,dx
=0;
\end{eqnarray}
we used the identity
\[
\Re\int_{\p\mathbb{B}^2_r}(\bar u\nabla u)\cdot\omega\,d\ell
=
\frac 1 2
\int_{\mathbb{S}^1_1}
(\bar u\nabla u
+u\nabla\bar u)\cdot\omega r\,d\omega
=
\frac r 2
\p_r
\int_{\mathbb{S}^1_1}
\abs{u(\omega r)}^2\,d\omega,
\]
with
$\omega=x/\abs{x}\in\mathbb{S}^1_1$.
Then
$
\partial_r
\int_{\mathbb{S}^1_1}
\abs{u(r\omega)}^2\,d\omega
\ge 2 c(r)/r
$,
with
$c(r):= \int_{\mathbb{B}^2_{r}} V(x)\abs{ u(x)}^2\,dx$
nondecreasing nonnegative function,
and hence
\[
\int_{\mathbb{S}^1_1}
\abs{u(r\omega)}^2\,d\omega
\geq
2 c(r_0)\ln\Big(\frac{r}{r_0}\Big)
+
\int_{\mathbb{S}^1_1}
\abs{u(r_0\omega)}^2\,d\omega,
\qquad
\forall r\ge r_0,
\]
with any $r_0>0$.
This would lead to
$u\not\in L^2_{-1,-1}(\R^2)$
and $u\not\in L^\infty_{-1/2+0}(\R^2)$
unless $c(r_0)=0$ for all $r_0>0$,
hence we conclude that $V\abs{u}^2=0$.
Due to our assumptions on $V$,
$u$ vanishes on a nonempty open set.
By the unique continuation principle
(see Remark~\ref{remark-ucp} below),
$u=0$.
\end{proof}

\begin{remark}\label{remark-ucp}
For the Laplace operator,
for the relation of the form
$\abs{\Delta u}\le V\abs{u}$
valid almost everywhere in a connected open domain,
the (strong) unique continuation property
holds for $V\in L^{d/2}\sb{\mathrm{loc}}(\R^d)$
for any $d\ge 2$.
See \cite[Theorem 1]{koch2001carleman}.
\end{remark}

\begin{lemma}\label{lemma-2d-pr}
Let $\nu_1,\,\nu_2>1/2$, $\nu_1+\nu_2>2$,
$\nu_2\le 1$,
and let
$V\in L^\infty_{2,\nu_1+\nu_2}(\R^2)$,
$V\ge 0$, $V\ne 0$.
Then $z_0=0$ is a regular point of the
essential spectrum of $-\Delta+V$
relative to
$\big(L^2_{1,\nu_1}(\R^2),L^2_{-1,-\nu_2}(\R^2),\cnor\big)$.
\end{lemma}

\begin{proof}
We note that the continuous mapping
$V:\;L^2_{-1,-\nu_2}(\R^2)\to L^2_{1,\nu_1}(\R^2)$,
with $V$ considered as the operator of multiplication,
is $\Delta$-compact
(with $\Delta$ considered in $L^2_{-1,-\nu_2}(\R^2)$)
by Lemma~\ref{lemma-last}.
By Lemma~\ref{lemma-2d-positive-potential},
there is no nontrivial solution
to $(-\Delta+V)u=0$,
$u\in L^2_{-1,-\nu_2}(\R^2)$.
Therefore,
by Theorem~\ref{theorem-2d-general},
$z_0=0$ is a regular point of the
essential spectrum of $-\Delta+V$
relative to
$\big(L^2_{1,\nu_1}(\R^2),L^2_{-1,-\nu_2}(\R^2),\cnor\big)$.
\end{proof}

\section{Schr\"odinger operators in higher dimensions}
\label{sect-free}

The uniform boundedness of
$R_0^{(d)}(z)=(-\Delta-z I)^{-1}:\,L^2_s(\R^d)\to L^2_{-s'}(\R^d)$
in $z\in\cnor$
for $d\ge 3$, $s,\,s'>1/2$, $s+s'>2$
is proved
in \cite[Proposition 2.4]{ginibre1974hilbert}.
The sharp version ($s+s'\ge 2$)
follows from
\cite[Lemma 2.1]{nirenberg1973null};
see also
\cite[Lemma 1]{mcowen1979behavior},
\cite[Lemma 2.3]{jensen1980spectral},
and \cite[Lemma 7.4.2 and Proposition 7.4.3]{yafaev2010mathematical}.

\begin{theorem}[Laplace operator in $L^2(\R^d)$, $d\ge 3$: LAP estimates]~
\label{theorem-3d}
\begin{enumerate}
\item
\label{theorem-3d-1}
The resolvent of the free Laplacian,
\[
R_0^{(d)}(z):=(-\Delta-z I)^{-1},
\qquad
z\in\cnor,
\]
has a limit, denoted by $R_0^{(d)}(0)$,
as $z\to z_0=0$,
with the convergence
in the strong operator topology of
$\scrB\big(L^2_s(\R^d),L^2_{-s'}(\R^d)\big)$ for
all $s,\,s'>1/2$, $s+s'\ge 2$
(in the uniform operator topology if $s+s'>2$).

\item
\label{theorem-3d-2}
If $d=3$,
$R_0^{(3)}(z)$ converges to $R_0^{(3)}(0)$
as $z\to z_0=0$, $z\in\cnor$,
\begin{enumerate}
\item
\label{theorem-3d-2a}
in the weak$^*$ operator topology of
$\scrB\big(L^2_s(\R^3),L^\infty_1(\R^3)\big)$
with $s>1/2$;
\item
\label{theorem-3d-2b}
in the strong operator topology of
$\scrB\big(L^1_{-1}(\R^3),L^2_{-s}(\R^3)\big)$
with $s>1/2$;
\item
\label{theorem-3d-2c}
in the uniform operator topology of $\scrB\big(L^2_s(\R^3),L^\infty_\sigma(\R^3)\big)$ and
$\scrB\big(L^1_{-\sigma}(\R^3),L^2_{-s}(\R^3)\big)$,
$\sigma\in[0,1)$, $s>1/2+\sigma$.
\end{enumerate}
\item
\label{theorem-3d-3}
If $d\ge 4$, $R_0^{(d)}(0)$ extends to continuous mappings
$L^2_s(\R^d)\to L^2_{-s'}(\R^d)$
if $s,\,s'>0$
($s,\,s'\ge 0$ if $d\ge 5$)
and $s+s'\ge 2$.
\item
\label{theorem-3d-4}
There are no nontrivial solutions
to $\Delta u=0$, $u\in L^2_{-s'}(\R^d)$,
with $s'\le d/2$.
\end{enumerate}
\end{theorem}

Let us mention that
Theorem~\ref{theorem-3d}~\itref{theorem-3d-3}
is a reformulation of the continuity of the Riesz potentials
$I_\alpha:=(-\Delta)^{-\alpha/2}$,
$\alpha=2$,
proved in \cite[Lemma 2.3]{jensen1980spectral}.
We postpone the proof of Theorem~\ref{theorem-3d}
until after we formulate and prove Theorem~\ref{theorem-3d-general}.

\begin{theorem}[Schr\"odinger operators in $L^2(\R^d)$, $d\ge 3$:
virtual states and LAP estimates]
\label{theorem-3d-general}
Let $s_1,\,s_2>1/2$,
$s_1+s_2\ge 2$, $s_2\le d/2$, and assume that
$W:\;L^2_{-s_2}(\R^d)\to L^2_{s_1}(\R^d)$
is $\Delta$-compact
(with $\Delta$ considered in $L^2_{-s_2}(\R^d)$).
There is the following dichotomy:

\medskip

\noindent
Either there is a nontrivial solution
(in the sense of distributions)
to
\begin{eqnarray}
(-\Delta+W)\Psi=0,
\qquad
\Psi\in L^2_{-s_2}(\R^d),
\end{eqnarray}
and, moreover,
this solution satisfies the relation
\begin{eqnarray}\label{w-is-w-3d}
\Psi
=
-(-\Delta-z_0 I)^{-1}_{L^2_{s_1},L^2_{-s_2},\cnor}
W\Psi,
\qquad
z_0=0,
\end{eqnarray}

\medskip

\noindent
or
there is $\delta>0$
such that
$\mathbb{D}_\delta\setminus\overline{\R_{+}}
\subset\C\setminus\sigma(-\Delta+W)$
and the resolvent
\[
R_W(z)=(-\Delta+W-z I)^{-1},
\qquad
z\in\C\setminus\sigma(-\Delta+W)
\]
is bounded
as a mapping $L^2_{s_1}(\R^d)\to L^2_{-s_2}(\R^d)$
uniformly in $z\in\mathbb{D}_\delta\setminus\overline{\R_{+}}$,
and in this case $R_W(z)$ has a limit
$R_W(z_0)\in\scrB\big(L^2_{s_1}(\R^d),L^2_{-s_2}(\R^d)\big)$
as $z\to z_0=0$, $z\in\C\setminus\sigma(-\Delta+W)$,
in the following topologies:
\begin{enumerate}
\item
\label{theorem-3d-general-1}
Strong operator topology of
$\scrB\big(L^2_s(\R^d),L^2_{-s'}(\R^d)\big)$
for all $s,\,s'>1/2$, $s+s'\ge 2$
(uniform operator topology if $s+s'>2$)
as long as $W$ extends to a mapping
$L^2_{-s'}(\R^d)\to L^2_s(\R^d)$
which is $\Delta$-compact,
with $\Delta$ considered in $L^2_{-s'}(\R^d)$
(this assumption on $W$ is not needed
if, additionally, one has
$s\ge s_1$, $s'\ge s_2$,
and it is redundant if $s\le s_1$, $s'\le s_2$);
\item
If $d=3$,
then there is also convergence in the following topologies:
\begin{enumerate}
\item
\label{theorem-3d-general-2a}
weak$^*$ operator topology of
$\scrB\big(L^2_s(\R^3),L^\infty_1(\R^3)\big)$, $s>1/2$,
if for some $\delta>0$ the operator family
\begin{eqnarray}\label{of-3d-2a}
W(-\Delta-z I)^{-1}:\,
L^2_s(\R^3)\to L^2_s(\R^3),
\qquad
z\in\mathbb{D}_\delta\setminus\sigma(-\Delta+W)
\end{eqnarray}
is collectively compact;
\item
\label{theorem-3d-general-2b}
strong operator topology of
$\scrB\big(L^1_{-1}(\R^3),L^2_{-s}(\R^3)\big)$
if $\sigma\in[0,1)$ and $s>1/2+\sigma$
and the operator family
\begin{eqnarray}\label{of-3d-2b}
W(-\Delta-z I)^{-1}:\,L^1_{-1}(\R^3)\to L^1_{-1}(\R^3),
\qquad
z\in\mathbb{D}_\delta\setminus\sigma(-\Delta+W)
\end{eqnarray}
is collectively compact for some $\delta>0$;
\item
\label{theorem-3d-general-2c}
uniform operator topology of
$\scrB\big(L^2_{s}(\R^3),L^\infty_{\sigma}(\R^3)\big)$
if $\sigma\in[0,1)$,
$1/2+\sigma<s\le s_1$,
and $s_2>3/2-\sigma$,
and for some $\delta>0$ the operator family
\begin{eqnarray}\label{of-3d-2c}
W(-\Delta-z I)^{-1}:\,
L^2_s(\R^3)\to L^2_s(\R^3),
\qquad
z\in\mathbb{D}_\delta\setminus\sigma(-\Delta+W)
\end{eqnarray}
is collectively compact
(this assumption is not needed if
$s=s_1$)
and converges in the uniform operator topology
of $\scrB\big(L^2_s(\R^3)\big)$
as $z\to z_0$, $z\in\C\setminus\sigma(-\Delta+W)$
(cf. \eqref{factorization-3});
\item
\label{theorem-3d-general-2d}
uniform operator topology of
$\scrB\big(L^1_{-\sigma}(\R^3),L^2_{-s'}(\R^3)\big)$
if
$\sigma\in[0,1)$, $s_1>3/2-\sigma$,
and $1/2+\sigma<s'\le s_2$,
and for some $\delta>0$ the operator family
\begin{eqnarray}\label{of-3d-2d}
W(-\Delta-z I)^{-1}:\,
L^1_{-\sigma}(\R^3)\to L^1_{-\sigma}(\R^3),
\qquad
z\in\mathbb{D}_\delta\setminus\sigma(-\Delta+W)
\end{eqnarray}
is collectively compact
and converges in the uniform operator topology
of $\scrB\big(L^1_{-\sigma}(\R^3)\big)$
as $z\to z_0$, $z\in\C\setminus\sigma(-\Delta+W)$.
\end{enumerate}
\end{enumerate}
Moreover, if $d\ge 4$,
$s,\,s'>0$
($s,\,s'\ge 0$ if $d\ge 5$),
$s+s'\ge 2$,
$s\le s_1$, $s'\le s_2$,
and
\begin{eqnarray}\label{ebfr-3d}
W(-\Delta-z_0 I)^{-1}_{L^2_{s_1},L^2_{-s_2},\cnor}
\mbox{ \ extends to a compact mapping in $L^2_s(\R^d)$,}
\end{eqnarray}
then
$R_W(z_0)\in\scrB\big(L^2_{s_1}(\R^d),L^2_{-s_2}(\R^d)\big)$
extends to a continuous mapping
$L^2_{s}(\R^d)\to L^2_{-s'}(\R^d)$.
\end{theorem}

\begin{remark}
\label{remark-3d-better}
If $d=3$ and
$W$ is the operator of multiplication by a function
from $L^\infty_\rho(\R^3)$, $\rho>2$,
then,
taking $s_1=s_2=1$,
one uses Theorem~\ref{theorem-3d}~\itref{theorem-3d-2a}
(with $s=1$)
to conclude from \eqref{w-is-w-3d}
that $\Psi\in L^\infty_1(\R^3)$.
We note that $W:\,L^2_{-1}(\R^3)\to L^2_1(\R^3)$
is $\Delta$-compact
(with $\Delta$ considered in $L^2_{-1}(\R^3)$);
see Lemma~\ref{lemma-last}.

If $d\ge 5$ and
$W$ is the operator of multiplication by a function
from $L^\infty_\rho(\R^d)$, $\rho>2$,
then
one uses Theorem~\ref{theorem-3d}~\itref{theorem-3d-3}
with $s=2$ and $s'=0$
to conclude from \eqref{w-is-w-3d}
that $\Psi\in L^2(\R^d)$
(that is, for multiplicative potentials
with at least quadratic spatial decay,
virtual states are $L^2$-eigenvectors,
so if $z_0=0$ is a virtual level,
it is never genuine, always being an eigenvalue).
\end{remark}

\begin{proof}
We first check that
the assumptions of Theorem~\ref{theorem-lap}
are satisfied
with $\bfX=L^2(\R^d)$, $\bfE=L^2_{s_1}(\R^d)$,
$\bfF=L^2_{-s_2}(\R)$, $A=-\Delta$, $z_0=0$,
and $\varOmega=\cnor$.
Indeed,
the Laplace operator $-\Delta$ in $L^2(\R^d)$
is closable as a mapping $L^2_{-s_2}(\R^d)\to L^2_{-s_2}(\R^d)$
(cf. Example~\ref{example-laplace-2d}).
Then, by Theorem~\ref{theorem-3d}~\itref{theorem-3d-1},
$z_0=0\in\sigma(-\Delta)$
is a regular point of the essential spectrum
relative to $\big(L^2_{s_1}(\R^d),L^2_{-s_2}(\R^d),\cnor\big)$.
More precisely,
$R(z)=(-\Delta-z I)^{-1}$,
$z\in\cnor$,
is uniformly bounded as a mapping $L^2_{s_1}(\R^d)\to L^2_{-s_2}(\R^d)$
near $z_0=0$
and has a limit as $z\to z_0$
in the strong operator topology
(uniform operator topology if $s_1+s_2>2$).

Now let $W:\,L^2_{-s_2}(\R^d)\to L^2_{s_1}(\R^d)$
be $\Delta$-compact
(with $\Delta$ considered in $L^2_{-s_2}(\R^d)$).
If $(-\Delta+W-z I)^{-1}:\,L^2_{s_1}(\R^d)\to L^2_{-s_2}(\R^d)$
is not uniformly bounded
for $z\in\mathbb{D}_\delta(0)\setminus\overline{\R_{+}}$
with some $\delta>0$, then
from the equivalence
of Theorem~\ref{theorem-lap}~\itref{theorem-lap-2a}
and
Theorem~\ref{theorem-lap}~\itref{theorem-lap-2e}
we conclude that there is
$\Psi\in L^2_{-s_2}(\R^d)\setminus\{0\}$,
$(-\Delta+W)\Psi=0$,
such that
$\Psi=(-\Delta)^{-1}\phi$ with some $\phi\in L^2_{s_1}(\R^d)$,
where
\[
(-\Delta)^{-1}:=(-\Delta-z_0 I)_{\bfE,\bfF,\cnor},
\qquad
z_0=0.
\]
Alternatively, assume that
there is $\phi\in L^2_{s_1}(\R^d)\setminus\{0\}$
such that $(-\Delta+W)\phi=0$.
Denote $\Psi=(-\Delta)^{-1}W\phi\in L^2_{-s_2}(\R^d)$;
by Lemma~\ref{lemma-jk}~\itref{lemma-jk-jk1},
$-\Delta\Psi=W\phi$ and hence
$-\Delta(\Psi-\phi)=0$.
Since $\Psi-\phi\in L^2_{-s_2}(\R^d)$ with $s_2\le d/2$,
by Theorem~\ref{theorem-3d}~\itref{theorem-3d-4},
$\phi=\Psi\in\ran((-\Delta)^{-1})$,
so from the equivalence
of Parts~\itref{theorem-lap-2a} and~\itref{theorem-lap-2e}
of Theorem~\ref{theorem-lap}
we conclude that
$z_0=0\in\sigma(-\Delta+W)$ is a regular point
of the essential spectrum
relative to $(\bfE,\bfF,\C\setminus\sigma(-\Delta+W))$.
This completes the proof of the dichotomy.

The relation \eqref{w-is-w-3d} is proved in the same way as
the relation \eqref{w-is-w}
in Theorem~\ref{theorem-1d-general}.

\smallskip

We now suppose that
$R_W(z)=(-\Delta+W-z I)^{-1}$,
$z\in\varOmega=\C\setminus\sigma(-\Delta+W)$,
is bounded as a mapping $L^2_{s_1}(\R^d)\to L^2_{-s_2}(\R^d)$
uniformly in
$z\in\mathbb{D}_\delta\setminus\overline{\R_{+}}$
with some $\delta>0$.
Let us prove
the convergence stated in
Part~\itref{theorem-3d-general-1}.
There is nothing to prove if
$s\ge s_1$, $s'\ge s_2$.
We assume that $W:\,L^2_{-s'}(\R^d)\to L^2_s(\R^d)$ is $\Delta$-compact
(with $\Delta$ considered in $L^2_{-s'}(\R^d)$);
we note that this assumption
is redundant if $s\le s_1$, $s'\le s_2$.
We will apply Theorem~\ref{theorem-a},
where we take
$A=-\Delta+W$,
$\bfE_1=L^2_{s_1}(\R^d)$,
$\bfF_1=L^2_{-s_2}(\R^d)$,
$\bfE_2=L^2_{s}(\R^d)$,
$\bfF_2=L^2_{-s'}(\R^d)$.
Due to the uniform boundedness of $R_W(z):\,L^2_{s_1}(\R^d)\to L^2_{-s'}(\R^d)$
for $z\in\mathbb{D}_\delta\setminus\overline{\R_{+}}$,
by Theorem~\ref{theorem-lap},
$z_0=0$ is a regular point of the essential spectrum of
$-\Delta+W$ relative to
$\big(L^2_{s_1}(\R^d),L^2_{-s_2}(\R^d),\varOmega\big)$,
so $W\in\scrQ_{\bfE_1,\bfF_1,\varOmega}(A-z_0 I_\bfX)$.
By Theorem~\ref{theorem-a}~\itref{theorem-a-2a},
$z_0=0$ is also a regular point of the essential spectrum of
$-\Delta+W$ relative to
$\big(L^2_{s}(\R^d),L^2_{-s'}(\R^d),\varOmega\big)$
(since the corresponding spaces of virtual states are trivial),
so $W\in\scrQ_{\bfE_2,\bfF_2,\varOmega}(A-z_0 I_\bfX)$,
and the convergence of the resolvent $R_W(z)$ as $z\to z_0=0$ holds in
$\scrB\big(L^2_{s}(\R^d),L^2_{-s'}(\R^d)\big)$.
By Theorem~\ref{theorem-a}~\itref{theorem-a-1},
this convergence holds in the weak operator topology;
Theorem~\ref{theorem-lap}~\itref{theorem-lap-2d}
and Theorem~\ref{theorem-3d} provide the
convergence in the strong operator topology
(uniform if $s+s'>2$).

The convergence in the weak$^*$ operator topology of
$\scrB\big(L^2_s(\R^3),L^\infty_1(\R^3)\big)$
stated in Part~\itref{theorem-3d-general-2a}
follows from
Theorem~\ref{theorem-factorization}~\itref{theorem-factorization-2},
where we take
$\bfE=L^2_{S}(\R^3)$ with $S=\max(s_1,s)$,
$\tilde\bfE=L^2_s(\R^3)$,
$\bfF=L^2_{-s_2}(\R^3)$,
and
$\tilde\bfF=L^\infty_1(\R^3)$,
and Theorem~\ref{theorem-3d}~\itref{theorem-3d-2a}.
We note that due to $s_2>1/2$
the embedding
$
L^\infty_1(\R^3)
=
\tilde\bfF\hookrightarrow\bfF
=L^2_{-s_2}(\R^3)
$
is continuous.
We also notice that
the condition that the operator family
\eqref{of-3d-2a}
is collectively compact
corresponds to the assumption of collective compactness
of $\upvarepsilon\circ\calB\circ\upvarphi\circ\tilde\calR(z)$
in
Theorem~\ref{theorem-factorization}~\itref{theorem-factorization-2}.
The convergence in the weak operator topology of
$\scrB\big(L^1_{-1}(\R^3),L^2_{-s}(\R^3)\big)$, $s>1/2$,
follows by duality.
It is improved to
the convergence in the strong operator topology of
$\scrB\big(L^1_{-1}(\R^3),L^2_{-s}(\R^3)\big)$
stated in Part~\itref{theorem-3d-general-2b}
as follows:
from the equivalence
of statements
~\itref{theorem-lap-2d} and~\itref{theorem-lap-2e}
in Theorem~\ref{theorem-lap},
$R_W(z)$ is bounded in 
$\scrB\big(L^2_s(\R^3),L^\infty_1(\R^3)\big)$
uniformly in
$z\in\mathbb{D}_\delta(0)\setminus\overline{\R_{+}}$
for some $\delta>0$.
Again the equivalence of statements
\itref{theorem-lap-2e} and~\itref{theorem-lap-2d}
in Theorem~\ref{theorem-lap}
and the convergence of $R_0^{(3)}(z)$
as $z\to z_0=0$,
$z\in\mathbb{D}_\delta(0)\setminus\overline{\R_{+}}$,
in the strong operator topology of
$\scrB\big(L^1_{-1}(\R^3),L^2_{-s}(\R^3)\big)$
stated in Theorem~\ref{theorem-3d}~\itref{theorem-3d-2b}
prove the likewise convergence
of $R_W(z)$.

Let us prove Part~\itref{theorem-3d-general-2c}.
The assumption $s_2>3/2-\sigma$
provides that there is a continuous embedding
$L^\infty_\sigma(\R^3)\hookrightarrow L^2_{-s_2}(\R^3)$.
Now the convergence in the strong operator topology of
$\scrB\big(L^2_s(\R^3),L^\infty_{\sigma}(\R^3)\big)$
for $\sigma\in[0,1)$, $s>1/2+\sigma$
follows from
Theorem~\ref{theorem-factorization}~\itref{theorem-factorization-2},
where we take
$\bfE=L^2_{s_1}(\R^3)$,
$\tilde\bfE=L^2_s(\R^3)$,
$\tilde\bfF=L^\infty_\sigma(\R^3)$,
and $\bfF=L^2_{-s_2}(\R^3)$,
and
Theorem~\ref{theorem-3d}~\itref{theorem-3d-2c}.

We notice that
the condition that the operator family
\eqref{of-3d-2c}
is collectively compact and converges
in the uniform operator topology
corresponds to the assumptions of
likewise compactness and convergence
of
$\upvarepsilon\circ\calB\circ\upvarphi\circ\tilde\calR(z)$
in
Theorem~\ref{theorem-factorization}~\itref{theorem-factorization-2}.
We also notice that
if $s_1\le s$
(so that in the above proof
we increase $s_1$ to the value of $s$),
the assumption that
\eqref{of-3d-2c} is collectively compact
is not needed
since one applies
Theorem~\ref{theorem-factorization}~\itref{theorem-factorization-2} with $\tilde\bfE=\bfE$.

Let us prove Part~\itref{theorem-3d-general-2d}.
The assumption $s_1>3/2-\sigma$
provides that there is a continuous embedding
$L^2_{s_1}(\R^3)\hookrightarrow L^1_{-\sigma}(\R^3)
$.
The convergence in the strong operator topology of
$\scrB\big(L^1_{-\sigma}(\R^3),L^2_{-s'}(\R^3)\big)$
for $\sigma\in[0,1)$, $s'>1/2+\sigma$
follows from
Theorem~\ref{theorem-factorization}~\itref{theorem-factorization-2}
(where we take
$\bfE=L^2_{s_1}(\R^3)$,
$\tilde\bfE=L^1_{-\sigma}(\R^3)$,
$\tilde\bfF=L^2_{-s'}(\R^3)$,
and $\bfF=L^2_{-s_2}(\R^3)$)
and
Theorem~\ref{theorem-3d}~\itref{theorem-3d-2c}.

Finally,
let us prove the extension of
$R_W(z_0)\in\scrB\big(L^2_{s_1}(\R^d),L^2_{-s_2}(\R^d)\big)$
(with $z_0=0$),
$d\ge 4$,
to a continuous mapping
$L^2_{s}(\R^d)\to L^2_{-s'}(\R^d)$
if $s,\,s'>0$
($s,\,s'\ge 0$ if $d\ge 5$)
and $s+s'\ge 2$,
$s\le s_1$, $s'\le s_2$,
follows from
Theorem~\ref{theorem-factorization}~\itref{theorem-factorization-1}
(where we take
$\bfE=L^2_{s_1}(\R^d)$, $\tilde\bfE=L^2_{s}(\R^d)$,
$\tilde\bfF=L^2_{-s'}(\R^d)$, $\bfF=L^2_{-s_2}(\R^d)$)
and
Theorem~\ref{theorem-3d}~\itref{theorem-3d-3}.
We note that the assumption \eqref{ebfr} in
Theorem~\ref{theorem-factorization}~\itref{theorem-factorization-1}
is satisfied due to the assumption
\eqref{ebfr-3d}.

This concludes the proof
of Theorem~\ref{theorem-3d-general}.
\end{proof}

\begin{proof}[Proof of Theorem~\ref{theorem-3d}]
For $d=3$,
the boundedness of
$R_0^{(3)}(z)=(-\Delta -z I)^{-1}:\,L^2_s(\R^3)\to L^2_{-s'}(\R^3)$
for $s,\,s'>1/2$,
$s+s'\ge 2$,
uniformly in $z\in\cnor$,
is obtained by noticing that
the integral kernel of the resolvent,
$R_0^{(3)}(x,y;z)=\fra{e^{-\abs{x-y}\sqrt{-z}}}{(4\pi\abs{x-y})}$,
$\Re\sqrt{-z}>0$,
satisfies
\[
\abs{R_0^{(3)}(x,y;z)}\le R_0^{(3)}(x,y;0):=
1/(4\pi\abs{x-y}),
\qquad
\forall x,\,y\in\R^3,
\quad z\in\cnor,
\]
while the
$L^2_s(\R^3)\to L^2_{-s'}(\R^3)$
boundedness of the operator
with the integral kernel $1/\abs{x-y}$
follows from \cite[Lemma 2.1]{nirenberg1973null}.

The boundedness
of $R^{(d)}_0(z)=(-\Delta-z I)^{-1}:\,
L^2_s(\R^d)\to L^2_{-s'}(\R^d)$
for $d\ge 3$,
$s,\,s'>1/2$, $s+s'\ge 2$,
$z\in\cnor$,
uniformly in $z\in\cnor$,
is due to
\cite[Proposition (2.4)]{ginibre1974hilbert}
(the case $s+s'>2$
and \cite[Lemma VI.19]{opus}.

Let us prove the convergence of $R_0^{(d)}(z)$, $d\ge 3$,
as $z\to z_0=0$, $z\in\cnor$,
in the uniform operator topology of
$\scrB\big(L^2_s(\R^d),L^2_{-s'}(\R^d)\big)$
for $s+s'>2$, $s,\,s'>1/2$.
Fix $\varepsilon>0$.
Choose $t\in(1/2,s)$
and $t'\in(1/2,s')$, $t+t'>2$.
Due to the uniform boundedness
of $(-\Delta-z I)^{-1}:\,L^2_t(\R^d)\to L^2_{-t'}(\R^d)$
in $z\in\cnor$,
there is $R>0$ such that
$(1-\unity_{\abs{x}\le R})\circ R_0^{(d)}(z)$,
$R_0^{(d)}(z)\circ(1-\unity_{\abs{x}\le R})$,
and
$(1-\unity_{\abs{x}\le R})\circ R_0^{(d)}(z)\circ
(1-\unity_{\abs{x}\le R})$
are each bounded
as mappings $L^2_s(\R^d)\to L^2_{-s'}(\R^d)$
by $\varepsilon/4$, uniformly in $z\in\cnor$.
Finally, let us consider
$\unity_{\abs{x}\le R}\circ R_0^{(d)}(z)\circ
\unity_{\abs{x}\le R}$.
The integral kernel of
$R_0^{(d)}(z)$
is given by
\begin{eqnarray}\label{rxyz}
R_0^{(d)}(x,y;z)
=
\frac{\jj}{4}
\Big(
\frac{z^{1/2}}{2\pi\abs{x-y}}
\Big)^{\frac{d}{2}-1}
H^{(1)}_{\frac{d}{2}-1}(z^{1/2}\abs{x-y})
=
\frac{\jj}{4\abs{x-y}^{d-2}}
\Big(
\frac{\upzeta}{2\pi}
\Big)^{\frac{d}{2}-1}
H^{(1)}_{\frac{d}{2}-1}(\upzeta);
\end{eqnarray}
here $H_\nu^{(1)}(\upzeta)$ are the modified Hankel functions,
$z\in\cnor$,
and $\upzeta=z^{1/2}\abs{x-y}$,
with the branch $\Im z^{1/2}>0$.
By \cite[9.1.7,\,9.1.9]{AS72},
\[
H_\nu^{(1)}(\upzeta)
=
J_\nu(\upzeta)
+
\jj Y_\nu(\upzeta)
\sim
\frac{(\upzeta/2)^\nu}{\Gamma(\nu+1)}
-\frac{\jj}{\pi}\Gamma(\nu)(\upzeta/2)^{-\nu},
\qquad
\nu>0,
\quad
\upzeta\to 0,
\quad
\upzeta>0,
\]
showing that there is a finite nonzero limit of
$\upzeta^\nu H^{(1)}_\nu(\upzeta)$ as $\upzeta\to 0$,
$\Im\upzeta>0$;
due to the operator with the integral kernel
$\abs{x-y}^{-(d-2)}$, $d\ge 3$, being bounded
as a mapping
$L^2_s(\R^d)\to L^2_{-s'}(\R^d)$
for $s+s'\ge 2$
(see \cite[Lemma 2.1]{nirenberg1973null}),
we conclude from \eqref{rxyz} that
the norm of
$\unity_{\abs{x}\le R}\circ
(R_0^{(d)}(z)-R_0^{(d)}(0))
\circ
\unity_{\abs{x}\le R}$
is bounded by $\varepsilon/4$
if $\abs{z}$ is small enough.
Since $\varepsilon>0$ was arbitrary,
this concludes the proof
of the convergence
of $R_0^{(d)}(z)$
in the uniform operator topology.

Let us prove that $R_0^{(d)}(z)=(-\Delta-z I)^{-1}$, $z\in\cnor$,
converges as $z\to z_0=0$,
in the strong operator topology of $\scrB\big(L^2_s(\R^d), L^2_{-s'}(\R^d)\big)$
when $s+s'=2$, $s,\,s'>1/2$.
Fix $u\in L^2_s(\R^d)$ and $\varepsilon>0$.
Due to the uniform boundedness of $(-\Delta-z I)^{-1}:\,L^2_s(\R^d)\to L^2_{-s'}(\R^d)$
in $z\in\cnor$,
there is $R>0$ such that
$\norm{(-\Delta-z I)^{-1}(\unity_{\abs{x}>R}\,u)}_{L^2_{-s'}}<\varepsilon/3$,
$\forall z\in\cnor$.
It remains to notice that
one has
$\norm{\unity_{\abs{x}\le R}\,u}_{L^2_{s+1}}
\le R\norm{\unity_{\abs{x}\le R}\,u}_{L^2_s}<\infty$,
therefore
$
\norm{
((-\Delta-z I)^{-1}-(-\Delta-z' I)^{-1})
(\unity_{\abs{x}\le R}\,u)}_{L^2_{-s'}}<\varepsilon/3
$
since
$(-\Delta-z I)^{-1}:\,L^2_{s+1}(\R^d)\to L^2_{-s'}(\R^d)$
converges
in the uniform operator topology as $z\to z_0$
by the above argument.

\begin{lemma}\label{lemma-strange-three}
\begin{enumerate}
\item
\label{lemma-strange-three-1}
For $z\in\cnor$,
the following mappings are continuous:
\[
R_0^{(3)}(z):\;L^1_{-\sigma}(\R^3)\to L^2_{-s}(\R^3),
\qquad
L^2_s(\R^3)\to L^\infty_\sigma(\R^3),
\qquad
0\le \sigma\le 1,
\quad
s>1/2+\sigma,
\]
with bounds uniform in $z\in\cnor$.
\item
\label{lemma-strange-three-2}
As $z\to z_0=0$, $z\in\cnor$,
$R_0^{(3)}(z)$ converges
\begin{enumerate}
\item
\label{lemma-strange-three-2a}
in the weak$^*$ operator topology of $\scrB\big(L^2_s(\R^3),L^\infty_1(\R^3)\big)$
with $s>1/2$;
\item
\label{lemma-strange-three-2b}
in the strong operator topology of $\scrB\big(L^1_{-1}(\R^3),L^2_{-s}(\R^3)\big)$
with $s>1/2$;
\item
\label{lemma-strange-three-2c}
in the uniform operator topology of
$\scrB\big(L^2_s(\R^3),L^\infty_\sigma(\R^3)\big)$
and
$\scrB\big(L^1_{-\sigma}(\R^3),L^2_{-s}(\R^3)\big)$,
$\sigma\in[0,1)$, $s>1/2+\sigma$.
\end{enumerate}
\end{enumerate}
\end{lemma}

\begin{proof}
Fix $z\in\cnor$.
For $u\in L^2_s(\R^3)$, $s>1/2$, one has:
\begin{eqnarray}\label{r00-new}
\norm{R_0^{(3)}(z)u}\sb{L^\infty_\sigma}^2
\le
\sup\sb{x\in\R^3}
\left(
\langle x\rangle^{\sigma}
\int\sb{\R^3}\frac{\abs{u(y)}\,dy}{4\pi\abs{x-y}}
\right)^2
\le
\frac{\norm{u}_{L^2_s}^2}{16\pi^2}
\sup\sb{x\in\R^3}
\int\sb{\R^3}\frac{
\langle x\rangle^{2\sigma}}{\abs{x-y}^2}\frac{dy}{\langle y\rangle^{2s}}.
\end{eqnarray}
We claim that the integral in the right-hand side
of \eqref{r00-new}
is bounded
uniformly in $x\in\R^3$.
The result is immediate if $\abs{x}<1$;
from now on, we assume that $\abs{x}\ge 1$.

We split the integration into two regions:

\noindent
$\bullet$
$\abs{y}<\abs{x}/2$.
In this case,
\begin{eqnarray}\label{xxyy}
\int\sb{\R^3}
\frac{
\unity_{\abs{y}<\abs{x}/2}
\langle x\rangle^{2\sigma}}{\abs{x-y}^2}\frac{dy}{\langle y\rangle^{2s}}
\le
C
\int\sb{\R^3}
\frac{dy}{\abs{x-y}^{2-2\sigma}\langle y\rangle^{2s}},
\end{eqnarray}
where
$
C:=\sup\sb{\substack{\abs{y}<\abs{x}/2,\,\abs{x}\ge 1}}
\fra{\langle x\rangle^{2\sigma}}{\abs{x-y}^{2\sigma}}
=
\sup\sb{\abs{x}\ge 1}
\fra{\langle x\rangle^{2\sigma}}{\abs{x/2}^{2\sigma}}
=8^\sigma$.
The contribution into the integral in the
right-hand side of \eqref{xxyy}
from the region $\abs{x-y}<1$ is uniformly bounded
by
$\int_{\mathbb{B}^3_1}\abs{y}^{-2(1-\sigma)}\,dy
\le 4\pi$.
To bound the contribution
from the region $\abs{x-y}\ge 1$
in the case
$0\le\sigma<1$, 
we apply H\"older's inequality:
\begin{eqnarray}\label{if-not}
\int\sb{\R^3}\frac{\unity\sb{\abs{x-y}\ge 1}}{\abs{x-y}^{2-2\sigma}}
\frac{dy}{\langle y\rangle^{2s}}
\le
\bigg(\int\sb{\R^3}\frac{dy}{\langle y\rangle^{2\alpha s}}
\bigg)^{\frac{1}{\alpha}}
\bigg(\int\sb{\abs{x-y}>1}\frac{dy}{\abs{x-y}^{2(1-\sigma)\beta}}
\bigg)^{\frac{1}{\beta}},
\quad
\frac 1 \alpha + \frac 1 \beta =1,
\end{eqnarray}
choosing $\alpha > \frac{3}{2s}$, $\beta >\frac{3}{2(1-\sigma)}$
such that $\alpha^{-1} + \beta^{-1}=1$
(which is possible since $s>1/2+\sigma$).
In the case $\sigma=1$ and $s>3/2$,
the finiteness of the left-hand side of \eqref{if-not}
is straightforward.

\noindent
$\bullet$
$\abs{y}\ge\abs{x}/2$.
In this case,
\[
\int\sb{\R^3}
\frac{
\unity\sb{\abs{y}\ge\abs{x}/2}
\langle x\rangle^{2\sigma}}{\abs{x-y}^2}\frac{dy}{\langle y\rangle^{2s}}
\le
\int\sb{\R^3}
\frac{2^{2\sigma}\,dy}{\abs{x-y}^{2}\langle y\rangle^{2(s-\sigma)}}.
\]
The analysis of the last integral
is the same as in \eqref{xxyy}
(with $\sigma=0$ and with $s-\sigma>1/2$ in place of $s$).
We conclude that
the right-hand side of \eqref{r00-new} is finite,
proving the uniform boundedness of
$R_0^{(3)}(z):\,L^2_s(\R^3)\to L^\infty_\sigma(\R^3)$
for $z\in\cnor$.
This completes the proof of Part~\itref{lemma-strange-three-1}.

Let us now consider the convergence of the resolvent
$R_0^{(3)}(z)$ as $z\to z_0=0$.
For the convergence in the weak$^*$ operator topology of
$\scrB\big(L^2_s(\R^3),L^\infty_1(\R^3)\big)$
stated in Part~\itref{lemma-strange-three-2a},
it is enough to prove that for any $R>0$
and any $u\in L^2_s(\R^3)$,
$s>1/2$,
\begin{eqnarray}\label{w-w-n-new}
\norm{
R_0^{(3)}(z)u-R_0^{(3)}(z')u}_{L^\infty_1(\mathbb{B}^3_R)}
\to 0
\qquad
\mbox{as $z,\,z'\to z_0$,
\quad
$z,\,z'\in\cnor$}.
\end{eqnarray}
Let $u\in L^2_s(\R^3)$ and fix $\varepsilon>0$.
For $N\ge 1$,
the left-hand side of 
\eqref{w-w-n-new} is bounded by
\begin{eqnarray}\label{two-terms}
\norm{(R_0^{(3)}(z)-R_0^{(3)}(z'))
\unity_{\R^3\setminus\mathbb{B}^3_N}u}_{L^\infty_1(\mathbb{B}^3_R)}
+
\norm{(R_0^{(3)}(z)-R_0^{(3)}(z'))\unity_{\mathbb{B}^3_N}u}_{L^\infty_1(\mathbb{B}^3_R)}.
\end{eqnarray}
Due to $L^2_s\to L^\infty_1$ boundedness
of $R_0^{(3)}(z)$
uniformly in $z\in\cnor$
that we
proved in Part~\itref{lemma-strange-three-1} of the lemma,
we can choose $N$ large enough so that
the first term in \eqref{two-terms}
is bounded by $\varepsilon/2$
for all
$z,\,z'\in\cnor$.
To bound the last term, we use the inequality
$
\abs{R_0^{(3)}(x,y;z)-R_0^{(3)}(x,y;z')}
\le
(4\pi)^{-1}\Abs{\sqrt{-z}-\sqrt{-z'}}$,
valid for all $x,\,y\in\R^3$
and $z,\,z\in\cnor$
(the branch of the square root
is such that $\Re\sqrt{-z}>0$);
it follows that
\[
\norm{
(R_0^{(3)}(z)-R_0^{(3)}(z'))
\unity_{\mathbb{B}^3_N}u
}_{L^\infty_1(\mathbb{B}^3_R)}
\le
\langle R\rangle
\sup\sb{x\in\R^3}
\int_{\R^3}
\Abs{
(R_0^{(3)}(x,y;z)-R_0^{(3)}(x,y;z'))
\unity_{\mathbb{B}^3_N}u(y)
}\,dy
\]
\[
\le
(4\pi)^{-1}
\langle R\rangle
\big|\sqrt{-z}-\sqrt{-z'}\big|
\langle R\rangle
4\pi N^3\norm{u}_{L^2}/3
,
\]
which is smaller than $\varepsilon/2$ as long as
$z,\,z'$ are both close enough to $0$.

The convergence of
$R_0^{(3)}(z)$ as $z\to z_0=0$,
$z\in\cnor$,
in the weak operator topology of $\scrB\big(L^1_{-1}(\R^3),L^2_{-s}(\R^3)\big)$
(Part~\itref{lemma-strange-three-2b})
follows by
Theorem~\ref{theorem-adjoint}~\itref{theorem-adjoint-0}
since $\left(L^1_{-1}(\R^3)\right)^*=L^\infty_{1}(\R^3)$.
Therefore, by Lemma~\ref{lemma-uniform}~\itref{lemma-uniform-1}, there exists $\delta>0$ such that
$\big\{R_0^{(3)}(z)\big\}_{z\in\mathbb{D}_\delta(0)\setminus\overline{\R_{+}}}$
is uniformly bounded
in $\scrB\big(L^1_{-1}(\R^3),L^2_{-s}(\R^3)\big)$.
This allows us to improve the convergence to the strong operator topology
of $\scrB\big(L^1_{-1}(\R^3),L^2_{-s}(\R^3)\big)$:
indeed, given $\varepsilon>0$, for $u\in L^1_{-1}(\R^3)$ and $z$, $z'$ in $\mathbb{D}_\delta(0)\setminus\overline{\R_{+}}$,
\[
\norm{(R_0^{(3)}(z)-R_0^{(3)}(z'))
\unity_{\R^3\setminus\mathbb{B}^3_N}u}_{L^2_{-s}(\R^3)}
\leq 2\sup_{z\in \mathbb{D}_\delta(0)\setminus\overline{\R_{+}}}\norm{R_0^{(3)}(z)}_{\scrB(L^1_{-1}(\R^3), L^2_{-s}(\R^3)}
\norm{\unity_{\R^3\setminus\mathbb{B}^3_N}u}_{L^1_{-1}(\mathbb{B}^3_R)}
\]
is less than $\varepsilon/3$ for $N$ large since
$u\in L^1_{-1}(\R^3)$;
the term
$ \norm{\unity_{\R^3\setminus\mathbb{B}^3_N}((R_0^{(3)}(z)-R_0^{(3)}(z'))
u}_{L^2_{-s}}$ is also smaller than $\varepsilon/3$ for $N$ large enough,
while
$\norm{\unity_{\mathbb{B}^3_N}((R_0^{(3)}(z)-R_0^{(3)}(z'))
\unity_{\mathbb{B}^3_N}u}_{L^2_{-s}(\mathbb{B}^3_R)}$
is smaller than $\varepsilon/3$
for $z,\,z'$ close enough to $0$ by dominated convergence theorem,
since the integral kernel
$R_0^{(3)}(x,y;z)=\fra{e^{-\abs{x-y}\sqrt{-z}}}{(4\pi\abs{x-y})}$
is continuous in $z$ for any $x,\,y\in \R^3$, $x\neq y$, and
uniformly bounded by
$R_0^{(3)}(x,y;0)$.

Finally, for Part~\itref{lemma-strange-three-2c},
the convergence
of the resolvent $R_0^{(3)}(z)$ as $z\to z_0=0$,
$z\in\cnor$,
in the uniform operator topology of
$\scrB\big(L^2_s(\R^3),L^\infty_\sigma(\R^3)\big)$
for $\sigma\in[0,1)$ and $s>1/2+\sigma$
(and, by duality, in the uniform operator topology of
$\scrB\big(L^1_{-\sigma}(\R^3),L^2_{-s}(\R^3)\big)$)
follows, like in the proof of Part~\itref{theorem-3d-1},
by varying $\sigma$ and $s$.
\end{proof}

Part~\itref{theorem-3d-3} is proved in \cite[Lemma~VI.18]{opus}.

Part~\itref{theorem-3d-4}
follows by the Liouville theorem:
if $u$ is an entire harmonic function,
then its Fourier transform $\hat u(\xi)$
is a finite linear combination of the Dirac $\delta$-function
in $\R^d$ and its derivatives;
hence $u$ is a finite degree polynomial
and therefore the inclusion $u\in L^2_{-s'}(\R^d)$ with $s'\le d/2$
leads to $u=0$.
This completes the proof of Theorem~\ref{theorem-3d}.
\end{proof}

\appendix

\section{Appendix:
convergence for families of compact and collectively compact operators
}

Here we collect several results
on relatively compact operators
and on families of compact and collectively compact operators.
Below,
$\bfE$ and $\bfF$ are some Banach spaces over $\C$.

\begin{definition}\label{def-compact}
Let $A:\,\bfF\to\bfF$
and $\calB:\,\bfF\to\bfE$
be linear operators,
with
$\dom(\calB)\supset\dom(A)$.
We say that $\calB$ is $A$-compact if
it maps the unit ball
$\mathbb{B}_1\big(\dom(A)\big)$
in the graph norm of $\dom(A)$,
\[
\norm{x}_{\dom(A)}
:=
\norm{x}_\bfF+\norm{A x}_\bfF,
\qquad
x\in\dom(A),
\]
into a precompact set in $\bfE$.
Equivalently,
$\calB:\,\bfF\to\bfE$
is $A$-compact
if
$\calB\circ\updelta_{A} :\,\dom(A)\to\bfE$
is compact,
where $\updelta_{A}:\,\dom(A)\hookrightarrow\bfF$
is the canonical embedding.
\end{definition}

\begin{remark}
Theorem~\ref{theorem-m} below shows that if
the limit
\eqref{lim} in Definition~\ref{def-virtual}
exists
in the weak or weak$^*$ operator topology
with some $\hatA$-compact perturbation
$\calB:\,\bfF\to\bfE$
which is not of finite rank,
then this limit also exists
in the weak or weak$^*$ operator topology, respectively,
with some
$\calB_1\in\scrB_{00}(\bfF,\bfE)$.
\end{remark}

\begin{lemma}
\label{lemma-kato-4.1.11-new}
Let
$\bfE\mathop{\longhookrightarrow}\limits\sp{\kappa}\bfF$
be a continuous embedding,
let $A:\,\bfF\to\bfF$
and $\calB:\,\bfF\to\bfE$
be linear operators
with
$\dom(\calB)\supset\dom(A)$,
and let $\calB$ be $A$-compact
(in the sense of Definition~\ref{def-compact}).
Denote
\[
B=\kappa\circ\calB:\,\bfF\to\bfF.
\]
If $A$ is closable, then $A+B$
is also closable,
the closures $\overline{A}$ of $A$ and $\overline{A+B}$ of $A+B$ have the same domain,
then $\dom\big(\overline{A}\big)=\dom\big(\overline{A+B}\big)$
as Banach spaces with the corresponding graph norms,
and $\calB$ is $(A+B)$-compact.
In particular, if $A$ is closed, then so is $A+B$.
\end{lemma}

The proof follows from
\cite[Theorem IV.1.11]{kato1995perturbation}
since $B$ is $A$-compact.

\medskip

\begin{lemma}\label{lemma-br}
Assume that the operator family
$\big\{\calR(z)\in\scrB(\bfE,\bfF)\big\}_{z\in\varOmega}$,
$\varOmega\subset\C$,
converges
to $\calR(z_0)\in\scrB(\bfE,\bfF)$
in the weak$^*$ or weak operator topology
of $\scrB(\bfE,\bfF)$
as $z\to z_0\in\p\varOmega$,
and that $\calB:\,\bfF\to\bfE$
is a linear operator
such that
the operator family
$\big\{\calB\calR(z)\in\scrB(\bfE)\big\}_{z\in\varOmega}$
is collectively compact
in the sense that the set
\[
\bigcup_{z\in\varOmega}
\calB\calR(z)\mathbb{B}_1(\bfE)\subset\bfE
\]
is precompact.
Then
there is the convergence $\calB\calR(z)\to\calB\calR(z_0)$
as $z\to z_0$, $z\in\varOmega$,
in the strong operator topology
of $\scrB(\bfE)$,
and
$\calB\calR(z_0)\in\scrB_0(\bfE)$.
\end{lemma}

\begin{proof}
It is enough to give the proof
for the case when the convergence
$\calR(z)\to\calR(z_0)$ as $z\to z_0$ holds
in the weak$^*$ operator topology
(under the assumption that $\bfF$ has a pre-dual);
the proof in the case of weak convergence
is verbatim.
Given $\phi\in\bfE$, the family
$\big\{\calB\calR(z)\phi\in\bfE\big\}_{z\in\varOmega}$
is precompact in $\bfE$.
At the same time,
$\calB\calR(z)\phi$
converges
in the weak$^*$ operator topology
to $\calB\calR(z_0)\phi$
as $z\to z_0$, $z\in\varOmega$;
therefore, $\calB\calR(z)\phi$
converges to $\calB\calR(z_0)\phi$ in $\bfE$.
It follows that 
\[
\calB\calR(z_0)\mathbb{B}_1(\bfE)
\subset
\overline{
\cup\sb{z\in\varOmega
}
\calB\calR(z)\mathbb{B}_1(\bfE)
}.
\]
Since
the set
$\cup\sb{z\in\varOmega
}\calB\calR(z)\mathbb{B}_1(\bfE)$
is precompact,
the operator $\calB\calR(z_0)$ is compact.
\end{proof}

\begin{remark}
Let us point out that even if
$\calR(z)$
converges as $z\to z_0$
in the uniform operator topology of $\scrB(\bfE,\bfF)$,
it is possible that $\calB\calR(z)$
converges as $z\to z_0$
in the strong operator topology of $\scrB(\bfE)$
but not in the uniform operator topology.
For example, for $n\in\N$,
consider
$\calR(n):\,\e_j\mapsto n^{-1}\delta_{j n}\e_n$,
$\calB:\,\e_j\mapsto j\e_1$,
$j\in\N$.
Then
$\calR(n)\to 0$ in the uniform operator topology
of $\scrB(\ell^2(\N))$
as $n\to\infty$,
while
$\calB\calR(n):\,\e_j\mapsto\delta_{j n}\e_1$
is collectively compact in $n\in\N$ and
converges to $0$
as $n\to\infty$
in the strong but not in the uniform operator topology.
\end{remark}

\begin{lemma}\label{lemma-rc}
Assume that the operator family
$\big\{\calR(z)\in\scrB(\bfE,\bfF)\big\}_{z\in\varOmega}$,
$\varOmega\subset\C$,
converges
as $z\to z_0\in\p\varOmega$
to $\calR(z_0)\in\scrB(\bfE,\bfF)$
in the strong operator topology
of $\scrB(\bfE,\bfF)$.
Let $\calC\in\scrB_0(\bfE)$.
Then
there is the convergence $\calR(z)\calC\to\calR(z_0)\calC$
as $z\to z_0$, $z\in\varOmega$,
in the uniform operator topology
of $\scrB(\bfE,\bfF)$.
\end{lemma}

\begin{proof}
By our assumptions,
there is $\delta>0$ such that
$M:=
\sup_{z\in\varOmega\cap\mathbb{D}_\delta(z_0)}
\norm{\calR(z)}_{\scrB(\bfE,\bfF)}<\infty$.
Let $\varepsilon>0$. Since $\calC\mathbb{B}_1(\bfE)$ is precompact in $\bfE$,
there exist $k\in\N$ and a finite set
$\big\{\phi_i\in\bfE\big\}_{1\le i\le k}$
such that
$
\calC\mathbb{B}_1(\bfE)\subset \cup_{i=1}^k \mathbb{B}_\varepsilon\big(\phi_i,\bfE\big)$.
Thus,
for any
$z\in\varOmega\cap\mathbb{D}_\delta(z_0)$,
we have:
\begin{eqnarray*}
&
\sup\limits_{\psi\in\mathbb{B}_1(\bfE)}\norm{(\calR(z)-\calR(z_0))\calC\psi}_\bfF
=
\sup\limits_{\phi\in\calC\mathbb{B}_1(\bfE)}
\norm{(\calR(z)-\calR(z_0))\phi}_\bfF
\\[1ex]
&
\leq
\max\limits\sb{1\le i\le k}
\norm{(\calR(z)-\calR(z_0))\phi_i}_\bfF + 2M\varepsilon.
\end{eqnarray*}
Since $\calR(z)$ converges
in the strong operator topology
as $z\to z_0$, $z\in\varOmega$,
there is
$\delta_\varepsilon\in(0,\delta)$
such that
$
\max_{1\le i\le k}
\norm{(\calR(z)-\calR(z_0))\phi_i}_\bfF \leq \varepsilon$
for
$z\in\varOmega\cap\mathbb{D}_{\delta_\varepsilon}(z_0)$,
and hence
\[
\sup_{\psi\in\mathbb{B}_1(\bfE)}
\norm{(\calR(z)-\calR(z_0))\calC\psi}_\bfF
\leq (1+2M)\varepsilon,
\qquad
z\in\varOmega\cap\mathbb{D}_{\delta_\varepsilon}(z_0).
\]
Since
$\varepsilon>0$ is arbitrarily small,
this concludes the proof.
\end{proof}

\begin{lemma}\label{lemma-kz}
Assume that
the operator family
$\big\{\calK(z)\in\scrB_0(\bfE)\big\}_{z\in\varOmega}$,
$\varOmega\subset\C$,
converges
as $z\to z_0\in\p\varOmega$
to $\calK(z_0)\in\scrB_0(\bfE)$
in the strong
or uniform
operator topology
of $\scrB(\bfE)$.
Then:
\begin{enumerate}
\item
\label{lemma-kz-1}
Let $\lambda_0\in\C\setminus\sigma(\calK(z_0))$, $\lambda_0\ne 0$.
If $\delta>0$ is small enough, then
there exist $\epsilon>0$ and $C>0$ such that
$\mathbb{D}_{\epsilon}(\lambda_0)\cap\sigma(\calK(z))
=\emptyset$
for $z\in\varOmega\cap\mathbb{D}_\delta(z_0)$,
and moreover
\begin{eqnarray}\label{more-1}
&&
\norm{(\calK(z)-\lambda I_{\bfE})^{-1}}_{\scrB(\bfE)}
<C,
\quad
\forall z\in\varOmega\cap\mathbb{D}_\delta(z_0),
\quad
\forall \lambda\in\mathbb{D}_\epsilon(\lambda_0).
\end{eqnarray}
The operator family
$\big\{(\calK(z)-\lambda I_{\bfE})^{-1}\big\}_{z\in\varOmega}$
converges to
$(\calK(z_0)-\lambda I_{\bfE})^{-1}$
as $z\to z_0$
in the strong
or uniform, respectively
operator topology of $\scrB(\bfE)$,
uniformly in $\lambda\in\mathbb{D}_{\epsilon}(\lambda_0)$.
\item
\label{lemma-kz-2}
There is a relation
\[
\lim_{z\to z_0,\,z\in\varOmega}
r(\calK(z))=r(\calK(z_0)),
\]
where $r(T)=\{\abs{\lambda}\sothat \lambda\in\sigma(T)\}$
is the spectral radius of $T\in\scrB(\bfE)$.
\end{enumerate}
\end{lemma}

\begin{proof}
Let us prove Part~\itref{lemma-kz-1}.
It is enough to give the proof assuming the convergence
$\calK(z)\to\calK(z_0)$
in the strong operator topology
since the extension to the case when
the convergence holds in the uniform
operator topology
is immediate.
Assume that, contrary to \eqref{more-1},
there are sequences
$z_j\in\varOmega$, $\lambda_j\in\C$, and $\phi_j\in\bfE$,
$j\in\N$,
such that $z_j\to z_0$, $\lambda_j\to\lambda_0$,
$\norm{\phi_j}_\bfE=1$,
while
$u_j:=(\lambda_j I_\bfE-\calK(z_j))\phi_j\to 0$
in $\bfE$ as $j\to\infty$.
Without loss of generality, we may assume that
$\abs{\lambda_j}\ge\abs{\lambda_0}/2>0$ for all $j\in\N$;
then the sequence
$\phi_j=\lambda_j^{-1}(\calK(z_j)\phi_j+u_j)$,
$j\in\N$,
is precompact.
Passing to a subsequence if necessary,
we may assume that $\phi_j\to\phi_0$ in $\bfE$;
we infer that $\norm{\phi_0}_\bfE=1$ and
$\calK(z_0)\phi_0=\lambda_0\phi_0$,
in contradiction to
our assumption that
$\lambda_0\not\in\sigma(\calK(z_0))$.
We conclude that
there are $\delta>0$, $\epsilon>0$,
and $C>0$
such that for all
$z\in\varOmega\cap\mathbb{D}_\delta(z_0)$,
$\lambda\in\mathbb{D}_\epsilon(\lambda_0)$,
and $\phi\in\bfE$ one has
$\norm{\lambda\phi-\calK(z)\phi}_\bfE\geq C \norm{\phi}_\bfE$.

Let us prove the convergence
$(\calK(z)-\lambda I_{\bfE})^{-1}$
in the strong operator topology.
Fix $\phi\in\bfE$.
The continuity of the mapping
$\mathbb{D}_{\epsilon}(\lambda_0)\to\bfE$,
$\lambda\mapsto (\calK(z_0)-\lambda I_{\bfE})^{-1}\phi$,
provides that the set
$\big\{(\calK(z_0)-\lambda I_{\bfE})^{-1}\phi,\,\lambda\in\mathbb{D}_{\epsilon}(\lambda_0)\big\}$
is precompact in $\bfE$ and hence
is a subset of a compact set $\scrK_{\phi,\epsilon}\subset\bfE$.
Since
the operator family
$\big\{\calK(z)\in\scrB_0(\bfE)\big\}_{z\in\varOmega}$
converges to $\calK(z_0)\in\scrB(\bfE)$
in the strong operator topology
as $z\to z_0$,
the family
$\big\{(\calK(z)-\calK(z_0))\theta\in\bfE\big\}_{z\in\varOmega}$
converges to zero in $\bfE$
uniformly in $\theta\in\scrK_{\phi,\epsilon}$
as $z\to z_0$.
Finally, by Part~\itref{lemma-kz-1},
the factor
$(\calK(z)-\lambda I_{\bfE})^{-1}$
in the right-hand side of the relation
\begin{align}\label{Eq:Difference}
&(\calK(z)-\lambda I_{\bfE})^{-1}\phi-(\calK(z_0)-\lambda I_{\bfE})^{-1}\phi\nonumber\\
&
\qquad=(\calK(z)-\lambda I_{\bfE})^{-1}
(\calK(z_0)-\calK(z))(\calK(z_0)-\lambda I_{\bfE})^{-1}\phi
\end{align}
is uniformly bounded
for $\lambda\in\mathbb{D}_{\epsilon}(\lambda_0)$
and $z\in\varOmega\cap\mathbb{D}_\delta(z_0)$.
This shows that the above expression converges to zero
in $\bfE$
as $z\to z_0$
uniformly in $\lambda\in\mathbb{D}_{\epsilon}(\lambda_0)$.

Let us prove Part~\itref{lemma-kz-2}.
Denote
\[
r_0:=\liminf_{z\to z_0,\,z\in\varOmega}r(\calK(z)),
\qquad
r_1:=\limsup_{z\to z_0,\,z\in\varOmega}r(\calK(z)).
\]
There is a sequence
$\big(z_j\in\varOmega\cap\mathbb{D}_\delta(z_0)\big)_{j\in\N}$,
$z_j\to z_0$,
such that
$\lim_{j\to\infty}r(\calK(z_j))=r_0$.
Pick $\lambda_*\in\C$ with $|\lambda_*|>r_0$
and let $\epsilon\in (0,|\lambda_*|-r_0)$
be such that
$\partial\mathbb{D}_\epsilon(\lambda_*)\cap\sigma(\calK(z_0))
=\emptyset$
(such $\epsilon$ exists since $\calK(z_0)$ is compact).
Dropping finitely
many terms from the sequence
$(z_j)_{j\in\N}$ if needed,
we may assume that
\begin{eqnarray}\label{d-e-l}
\overline{\mathbb{D}_\epsilon(\lambda_*)}
\cap\sigma(\calK(z_j))
\subset
\overline{\mathbb{D}_\epsilon(\lambda_*)}
\cap\mathbb{D}_{r(\calK(z_j))}(0)
=\emptyset,
\qquad
\forall j\in\N.
\end{eqnarray}
So the Riesz projector
$
P_{\lambda_*}(z)=-\frac{1}{2\pi\jj}\oint_{\partial\mathbb{D}_\epsilon(\lambda_*)}(\calK(z)-\zeta I_\bfE)^{-1}\,d\zeta
$
is well-defined
as an element of $\scrB_{00}(\bfE)$
for
$z=z_0$ and also for $z=z_j$, $j\in\N$.
By \eqref{d-e-l}, one has $P_{\lambda_*}(z_j)=0$ for all $j\in\N$,
while if $\phi\in\bfE$ is such that
$\calK(z_0)\phi=\lambda_*\phi$,
then $P_{\lambda_*}(z_0)\phi=\phi$. From
Part~\itref{lemma-kz-1},
$P_{\lambda_*}(z_j)$ converges to $P_{\lambda_*}(z_0)$ as $j\to\infty$
in the strong operator topology of $\scrB(\bfE)$;
we conclude that $\phi=0$
and hence $\lambda_*\not\in\sigma(\calK(z_0))$.
This proves that
$r(\calK(z_0))\leq r_0$.

If $r_1=0$, then also $r_0=0$,
hence, by the preceding argument,
$r(\calK(z_0))=0$,
concluding the proof in this case.
Let us now assume that $r_1>0$.
There is a sequence
$(z_j\in\varOmega)_{j\in\N}$, $z_j\to z_0$,
such that
$r_1=\lim_{j\to\infty}r(\calK(z_j))>0$.
Hence there exists a sequence
$\lambda_j\in\sigma(\calK(z_j))$,
$\abs{\lambda_j}\to r_1$.
Discarding finitely many $z_j$ if needed,
we may assume that
$|\lambda_j|\geq r_1/2$ for all $j\in\N$
and that $\lambda_j$ converge to some $\lambda_*\in\C$,
$\abs{\lambda_*}=r_1$.
Let
$(\phi_j\in\bfE)_{j\in\N}$,
be the corresponding eigenvectors:
\[
\phi_j=\lambda_j^{-1}\calK(z_j)\phi_j,
\qquad
\norm{\phi_j}_\bfE=1
\qquad \forall j\in\N.
\]
Since the set
$\{\calK(z_j)\}_{j\in\N}$
is collectively compact
while $\abs{\lambda_j}^{-1}\le 2/r_1$
for all $j\in\N$,
the set $(\phi_j)_{j\in\N}$ is precompact.
Passing to a subsequence of $(z_j)_{j\in\N}$,
we may assume that, as $j\to\infty$,
$\phi_j$ converge to some $\phi_*\in\bfE$,
$\norm{\phi_*}_{\bfE}=1$,
and then
$\lambda_*\phi_*=\calK(z_0)\phi_*$,
hence $\lambda_*\in\sigma(\calK(z_0))$.
Therefore,
$\abs{\lambda_*}=r_1\leq r(\calK(z_0))$.

It follows that
$r_0=r_1$,
hence $\lim_{z\to z_0,\,z\in\varOmega}r(\calK(z))$
exists and is equal to $r(\calK(z_0))$.
Let us mention that this result
is similar to \cite[Theorem 3]{newburgh1951variation}
where the continuity of the spectrum with respect to the
uniform operator topology is established
in the case when the spectrum is totally disconnected.
For counterexamples in the general case,
see~\cite[Introduction \S 3]{newburgh1951variation}.
The above proof
shows the validity of
\cite[Corollary of Theorem 3]{newburgh1951variation}
in the algebra $\scrB(\bfE)$
for weak$^*$ operator topology
instead of the
uniform operator topology.
\end{proof}

\section{Appendix: closure of the Laplacian in $L^2_s$}

\begin{lemma}\label{lemma-last}
\begin{enumerate}
\item
\label{lemma-last-1}
For any $d\in\N$ and $s\in\R$,
the Laplace operator
is closable in $L^2_s(\R^d)$.
\item
\label{lemma-last-2}
Let $s,\,s'\in\R$, $s>s'$, and let
$V\in L^{q}_{\rho}(\R^d)$, with
$q\in [2, \frac{2d}{d-2})$ if $d\geq 3$,
$q\in [2, +\infty)$ if $d=2$,
or $q\in [2, +\infty]$ if $d=1$,
and $\rho\ge s+s'$.
Then multiplication by $V$ is
a $\Delta$-compact operator
from $L^2_{-s'}(\R^d)$ to $L^2_s(\R^d)$
(with $\Delta$ considered in $L^2_{-s'}(\R^d)$).
\end{enumerate}
\end{lemma}

\begin{proof}
Let us show that the domain of
the closure of
the extension of $\Delta$ onto $L^2_s(\R^d)$
is
\[
\bfD=
\left\{u\in L^2_{s}(\R^d),\, \Delta u\in L^2_{s}(\R^d)\right\}.
\]
This is the maximal domain of $\Delta$ in $L^2_s(\R^d)$.
Therefore, the operator $\Delta$ with domain $\bfD$ is closed.
The space $\bfD$ is endowed with the topology generated by
the graph norm of $\Delta$
in $L^2_{s}(\R^d)$
and there are continuous inclusions
$
\mathscr{D}(\R^d)
\hookrightarrow
\dom(\Delta_{L^2_s\shortto L^2_s})
\hookrightarrow\bfD
$,
where
\[
\dom(\Delta_{L^2_s\shortto L^2_s})
=\big\{
u\in L^2_s(\R^d)\cap L^2(\R^d)
\sothat
\Delta u\in L^2_s(\R^d)\cap L^2(\R^d)
\big\}.
\]
Let us prove that
$\mathscr{D}(\R^d)$ is dense in $\bfD$
in the graph norm topology
of $\Delta$ in $L^2_{s}(\R^d)$.
Let $\varrho\in C^\infty\sb{\mathrm{comp}}(\R)$,
$0\le\varrho\le 1$,
$\supp\varrho\subset[-2,2]$,
$\varrho\at{[-1,1]}=1$.
For $n\in\N$, we define
$\chi_n(x)=\varrho(\abs{x}/n)$.
Given $u\in\bfD$,
one has
\begin{equation}\label{Equation:ChainRule}
\Delta \left(\chi_n u\right) = \chi_n \Delta u+ \frac2n \nabla\chi_1(\frac{\cdot}{n})\cdot\nabla u+
\frac1{n^2} \big(\Delta\chi_1(\frac{\cdot}{n})\big) u.
\end{equation}
We notice that
\[
\norm{\langle x\rangle^{s} \nabla \left(\chi_n u\right)}_{L^2}^2
=-\big\langle \langle x\rangle^{s}
\chi_n u, \langle x\rangle^{s} \Delta \left(\chi_n u\right)\big\rangle_{L^2}
-
2\big\langle \langle x\rangle^{s}\chi_n u, s\langle x\rangle^{s-2} x\cdot \nabla \left(\chi_n u\right)\big\rangle_{L^2},
\]
\[
\left|\big\langle \langle x\rangle^{s}
\chi_n u, s\langle x\rangle^{s-2} x\cdot \nabla \left(\chi_n u\right)\big\rangle_{L^2}\right|^2
\leq
4s^2\norm{ \langle x\rangle^{s}
\chi_n u}_{L^2}
+
\frac{1}{4}
\norm{\langle x\rangle^{s-2} x\cdot \nabla \left(\chi_n u\right)}_{L^2}^2,
\]
and so
\begin{align*}
\norm{\langle x\rangle^{s} \nabla \left(\chi_n u\right)}_{L^2}^2
&\leq 2
\left|\big\langle \langle x\rangle^{s}
\chi_n u, \langle x\rangle^{s} \Delta \left(\chi_n u\right)\big\rangle_{L^2}\right|
+16s^2\norm{\langle x\rangle^{s}\chi_n u}_{L^2}^2
\\
&\le
\norm{\langle x\rangle^{s} \Delta \left(\chi_n u\right)}_{L^2}^2
+(1+16s^2)
\norm{\langle x\rangle^{s}\chi_n u}_{L^2}^2.
\end{align*}
Since
$\nabla \left(\chi_n u\right)=\frac1n\left(\nabla \chi_1\right)(\frac{\cdot}{n})u+\chi_n \nabla u$,
we infer:
\begin{eqnarray*}
\norm{\langle x\rangle^{s} \chi_n \nabla  u}_{L^2}
&\le&
\norm{\langle x\rangle^{s} \Delta \left(\chi_n u\right)}_{L^2}
+(1+4\abs{s})\norm{ \langle x\rangle^{s}\chi_n u}_{L^2}+\norm{\langle x\rangle^{s} \frac1n\left(\nabla \chi_1\right)(\frac{\cdot}{n})u}_{L^2}
\\
&\leq&
\norm{\langle x\rangle^{s} \chi_n \Delta u}_{L^2}
+ \norm{\langle x\rangle^{s} \frac2n \nabla\chi_1(\frac{\cdot}{n})\cdot\nabla u}_{L^2}
+
\norm{\langle x\rangle^{s} \frac1{n^2}
\big(\Delta\chi_1(\frac{\cdot}{n})\big) u}_{L^2}
\\
&&
\quad
+
(1+4\abs{s})\norm{ \langle x\rangle^{s}\chi_n u}_{L^2}
+\norm{\langle x\rangle^{s} \frac1n\left(\nabla \chi_1\right)(\frac{\cdot}{n})u}_{L^2},
\end{eqnarray*}
where for the second inequality we used \eqref{Equation:ChainRule}.
Subtracting $\frac2n \norm{\nabla\chi_1}_{L^\infty(\R^d)}\norm{\langle x\rangle^{s} \nabla u}_{L^2}$, we deduce in the limit $n\to\infty$, by the Fatou lemma, that $\nabla u\in L^2_{s}(\R^d,\C^d)$.
This shows that,
in \eqref{Equation:ChainRule}, $\chi_n u  \to u$ in the graph norm topology of $-\Delta$ in $L^2_{s}(\R^d)$
corresponding to the norm
$\norm{u}_{L^2_s}+\norm{\Delta u}_{L^2_s}$.
Since $\chi_n u\in H^2(\R^d)$ and $\mathscr{D}(\R^d)$ is dense in $H^2(\R^d)$, taking the closure,
we conclude that $ \doma=\bfD$.
This concludes the proof of Part~\itref{lemma-last-1}.

For Part~\itref{lemma-last-2},
let $V_0\in L^\infty(\R^d)$, with $\supp V_0$ compact.
Then the operator of multiplication by $V_0$
defines a compact operator from $H^2(\R^d)$ to $L^2(\R^d)$.
Let $\chi$ be a smooth function on $\R^d$ with compact support
and such that $\chi V_0=V_0$.
The operator of multiplication by $\chi$
defines a bounded operator from
$\{u\in L^2_{-s'}\sothat \Delta u\in L^2_{-s'}\}$
to $H^2(\R^d)$, therefore
multiplication by
$\chi V_0=V_0$
is a compact operator
from $\{u\in L^2_{-s'}(\R^d)\sothat \Delta u\in L^2_{-s'}(\R^d)\}$
to $L^2_s(\R^d)$.

Let now $V\in L^{q}_{\rho}(\R^d)$, with
$q\in [2, \frac{2d}{d-2})$ if $d\geq 3$,
$q\in [2, +\infty)$ if $d=2$,
or $q\in [2, +\infty]$ if $d=1$,
and let $\rho\ge s+s'$,
so that multiplication by $V$
is an operator from $\scrB(L^2_{-s'}(\R^d),L^2_s(\R^d))$.
Since
$\big(V_n:=V\unity_{\mathbb{B}^d_n}\unity_{|V|\leq n}\big)_{n\in\N}$
converges to $V$ in $L^{q}_{s+s'}(\R^d)$, multiplication by $V$ is a norm limit of $\Delta$-compact operators.
Therefore, multiplication by $V$ is $\Delta$-compact.
\end{proof}

\bibliographystyle{sima-doi}
\bibliography{bibcomech}
\end{document}